\documentclass[a4paper,11pt,oneside]{book}
\usepackage[T1]{fontenc}
\usepackage[utf8]{inputenc}
\fontfamily{phv}\selectfont 
\usepackage{lmodern}
\usepackage[Sonny]{fncychap}
\usepackage{framed} 
\usepackage[colorlinks=true,citecolor=blue,linkcolor=black,]{hyperref}
\usepackage[new]{old-arrows}
\usepackage{tcolorbox}
\usepackage{lipsum}
\usepackage{hyperref,xcolor}
\usepackage[all]{xy}
\usepackage{amsfonts,amsthm,amsmath,amssymb,color}
\usepackage{mathabx}
\usepackage{float}

\usepackage[toc,page]{appendix}
\usepackage[makeroom]{cancel}
\theoremstyle{plain}
\newtheorem{lemma}{Lemma}[section]
\newtheorem{theorem}[lemma]{Theorem}
\newtheorem{proposition}[lemma]{Proposition}
\newtheorem{prop}[lemma]{{Proposition}}
\newtheorem{corollary}[lemma]{Corollary}
\newtheorem{convention}[lemma]{Convention}

\newtheorem*{convention*}{Convention}

\theoremstyle{definition}
\newtheorem{definition}[lemma]{Definition}

\newtheorem{example}[lemma]{Example}
\newtheorem{remark}[lemma]{Remark}
\newtheorem*{definition*}{Definition}

\theoremstyle{remark}

\newcommand{\E}{\mathcal{E}}
\newcommand{\lb}{\left[ \cdot\,,\cdot\right] }
\usepackage[left=2cm, right=2cm, top=3cm, bottom=1.7cm]{geometry}
\newcommand{\dd}{\mathrm{d}}
\newcommand{\dif}{\text{d}}
\newcommand{\X}{\mathfrak {X} }
\usepackage{tikz}
\usepackage[]{titlesec} 
\usepackage{fancyhdr}
\newcommand{\HRule}{\rule{\linewidth}{0.5mm}}

\begin{document}
\begin{titlepage}

\center
		\begin{center}
			\begin{tabular}{c@{\hskip 3cm}c}
				\includegraphics[height=1.5cm]{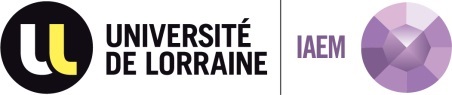} &
			\includegraphics[height=1.5cm]{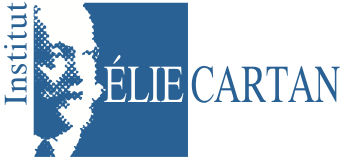}\\
			\end{tabular}
		\end{center}
	
		\vfill

\textsc{\LARGE	Université de Lorraine}\\[1.5cm]
	
\textsc{\Large }\\[0.5cm] 

	 \HRule \\[0.2cm]{ \Large \bfseries UNIVERSAL HIGHER LIE ALGEBRAS OF SINGULAR SPACES AND THEIR SYMMETRIES}\\[0.2cm]\HRule\\[1.5cm]

	\begin{minipage}{0.4\textwidth}
		\begin{flushleft}
			\large
			\textit{Auteur}\\
			Ruben \textsc{Louis} 
		\end{flushleft}
	\end{minipage}
	~
	\begin{minipage}{0.4\textwidth}
		\begin{flushright}
			\large
			\textit{Directeur}\\
			 Prof. Camille \textsc{Laurent-Gengoux}
		\end{flushright}
	\end{minipage}
		
		\vfill
			
		\begin{center}
	\textbf{Thèse}\\\vspace{0.5cm}{ présentée  et soutenue publiquement le 12 novembre 2022 pour l’obtention du}\\\vspace{0.5cm}
       {
\textbf{Doctorat de l’Université de Lorraine}
}\\\vspace{0.5cm} {\textbf{(mention Mathématiques)}}
\\[1cm]
\end{center}\vspace{1cm}
\begin{flushleft}
    \textbf{Membres du jury :}
\end{flushleft}		
\begin{center}
\begin{tabular}{ |c|c|c|c| } 
\hline
\textbf{Directeur de thèse} : & M. Camille Laurent-Gengoux & Professeur, Université de Lorraine, Metz\\
\hline
\textbf{Président de jury} :  & M. Robert Yuncken &Professeur, Université de Lorraine, Metz  \\ \hline
\textbf{Rapporteurs} : & M. Marco Zambon & Professeur,  KU Leuven, Leuven \\ \hline
& Mme. Chenchang Zhu  & Professeure,  Université de Göttingen, Göttingen\\ 
\hline
\textbf{Examinateurs} :& Mme. Claire Debord  & Professeure, Université de Paris, Paris\\ 
\hline
& M. Pol Vanhaecke & Professeur, Université de Poitiers, Poitiers  \\ 
\hline
\textbf{Membre invité} :&M. Rajan Mehta & Professeur, Smith College, Massachusetts \\ 
\hline
\end{tabular}
\end{center}

	\end{titlepage}
	
\newpage
\begin{center}
    {\textbf{ABSTRACT}}
\end{center}

This thesis breaks into two main parts. Let me describe them.

\begin{itemize}
    \item We show that there is an equivalence of categories between Lie-Rinehart algebras over a commutative algebra $\mathcal O$ and homotopy equivalence classes of negatively graded acyclic Lie $\infty$-algebroids. Therefore, this result makes sense of the universal Lie $\infty$-algebroid of every singular foliation, without any additional assumption, and for Androulidakis-Zambon singular Lie algebroids. This extends to a purely algebraic setting the construction of the universal $Q$-manifold of a locally real analytic singular foliation of \cite{LLS,LavauSylvain}.  Also, to any ideal $\mathcal I \subset \mathcal O $ preserved by the anchor map of a Lie-Rinehart algebra $\mathcal A $, we associate a homotopy equivalence class of negatively graded Lie $\infty $-algebroids over complexes computing ${\mathrm{Tor}}_{\mathcal O}(\mathcal A, \mathcal O/\mathcal I) $.
Several explicit examples are given.

\item The second part is dedicated to  some applications of the results on Lie-Rinehart algebras. 

\begin{enumerate}
    \item We associate to any affine variety a universal Lie $\infty$-algebroid of the Lie-Rinehart algebra of its vector fields. We study the effect of some common operations on affine varieties such as blow-ups, germs at a point, etc.
    \item We give an interpretation of the blow-up of a singular foliation $\mathfrak{F}$ in the sense of Mohsen \cite{MohsenOmar} in terms of the universal Lie $\infty$-algebroid of $\mathfrak{F}$, in fact an almost Lie algebroid over a geometric resolution of $\mathfrak{F}$.
    \item We introduce the notion  of longitudinal vector fields on a graded manifold over  a singular foliation, and study their cohomology. We prove that the cohomology groups of the latter vanish. 
    \item We study symmetries of singular foliations through universal Lie $\infty$-algebroids. More precisely, we prove that a weak symmetry action of a Lie algebra $\mathfrak{g}$ on a singular foliation $\mathfrak F$ (which is morally an action of $\mathfrak g$ on the leaf space $M/\mathfrak F$) induces a unique up to homotopy Lie $\infty$-morphism from $\mathfrak{g}$ to the Differential Graded Lie Algebra (DGLA) of vector fields on a universal Lie $\infty$-algebroid of $\mathfrak F$ (such morphim is known under the name "$L_\infty$-algebra action" in \cite{MEHTA2012576}). We deduce from this general result several geometrical consequences. For instance, We give an example of a Lie algebra action on an affine sub-variety which cannot be extended on the ambient space.  Last, we present the notion of bi-submersion towers over a singular foliation and lift symmetries to those. 
\end{enumerate}
\end{itemize}

\newpage
\begin{center}
    \textbf{{ACKNOWLEDGEMENTS}}
\end{center}
I wish to express my gratitude to God and to several people for their help on the accomplishment of this thesis.

\begin{itemize}
    \item[]First, I would like to express my gratitude to my supervisor  Camille Laurent-Gengoux for accepting me as his Ph.D. student. He was patient, attentive to me. He was a great support for me and the one on which I could lean when I feel lost in my ideas. He encouraged me in my research. He gave me confidence. He knew how to motivate me. He gave me valuable advice and suggested directions to take, articles to read in order to  lead better my research. He transmitted to me knowledge and the ethics of a researcher, he is therefore my mathematic father. I also thank him for the time he devoted to me and for his comments and his suggestions throughout of this Ph.D. I feel very lucky for this opportunity I had  
    to work with him.

\item[] My gratitude also goes  to Chenchang Zhu and Marco Zambon for accepting
to be reporters for my thesis. I also thank Claire Debord, Pol Vanhaecke, Robert Yuncken, and Rajan Amit Mehta for taking part in the jury. 

\item[]I would like to acknowledge the full financial support for this Ph.D from R\'egion Grand Est. Likewise, I thank the managers and staff of Université de Lorraine for their collaboration and their help in the administrative procedures. I particularly would like to warmly thank Prof. Philippe Bonneau for helping me find accommodation and settling me comfortably in Metz to begin my PhD.

\item[]Thanks to the CNRS MITI 80Prime project GRANUM also to Franco-German PHC project Procope. I would like to thank the Institut Henri Poincaré for hosting me in november 2021. I want to thank S. Lavau for his advice and comments on my paper "Symmetries of singular foliations". I would like to thank C. Ospel, P. Vanhaecke and V. Salnikov for giving the possibility to present the results on Lie-Rinehart algebras at the \textquotedblleft Rencontre Poisson à La Rochelle, 21-22 October 2021\textquotedblright. I was glad to meet my mathematic grand father Claude Roger at this conference and I would like to thank him  for his articles on the Lie-Rinehart algebras that he gave me. Also, I thank the organizers of the Geometry Seminar at the Aristotle  
University of Thessaloniki, in particular Panagiotis Batakidis for inviting me in January 28th, 2022 to give a talk on my work.

I would like to thank Iakovos Androulidakis with whom I had the honor of discussing my work at the "Foliations, pseudodifferential operators and groupoids" school, Göttingen, Germany in February 28th-March 4th, 2022. I also would like to thank Leonid Ryvkin for always being ready to discuss with me when I needed it.

I am also grateful to the organizers of Poisson 2022 Advanced School at  CRM (Centre de Recerca Matemàtica) Barcelona, for giving this precious opportunity  to be in charge of the problems sessions for the Lecture  "Singular Foliations". Likewise, I would like to thank the organizers of Poisson 2022 Conference - ICMAT, Madrid for giving the possibility to present my results at  the Poster session. Last, I would like to express sincere gratitude to Université d'État d'Haïti and more precisely the department of mathematics of \'Ecole Normale Supérieure (ENS), for giving a golden opportunity to meet mathematics. I would like to mention a few names of ENS professors who have contributed to this achievement: Prof. Bérard Cenatus, Prof. Lesly Dejean, Prof. Dr. Antonine Phigareau, Prof. Dr. Pierre Timothe, Prof. Dr. J.B Antenord, Prof. Dr. OLguine Yacinthe, Prof. Dr. Oscar Walguen, Prof. Steeve Germain,  Prof. Dr. Aril Milce, Prof. Dr. Yvesner Marcelin and Prof. Dr. J. K Innocent and Prof. Dr. Dieuseul Predelus. I specially would like to thank Prof. Dr. Achis Chery and Prof. Dr. Pol Vanhaecke for supporting me by writing letters in my favor to obtain this grant.

\item[] 
\end{itemize}
\newpage
\begin{center}
    \begin{enumerate}
    \item[]Je tiens à remercier Anderson Augusma ainsi que Dor Dieunel pour m'avoir aidé à vérifier les erreurs typographiques dans le manuscrit.
    \item []Je remercie chaleureusement Wanglaise Fateon qui m'a également aidé à relire mes articles et à identifier les coquilles. La rencontrer a été l'une des meilleures choses qui me soient arrivées.
\item[] Enfin, je dédie cette thèse à mes parents, Madame et Monsieur Wisner Louis qui ont toujours été présents pour m'encourager au tout début et tout au long de ces années de thèse. Je remercie mon seul et unique frère et petit frère Benjamin Louis pour ses blagues drôles qui m'ont aidé à ne pas devenir fou dans la foulée. Je t'aime mon frère. 
    \end{enumerate}
    \vspace{5cm}
    \item[]\hspace{5cm}À tous mes proches !
\end{center}

\tableofcontents 

\chapter*{Introduction}


The recent  studies about Lie $\infty $-algebras  or Lie $\infty $-groups, their morphisms and their -oids equivalent (i.e. Lie $\infty $-algebroids \cite{Campos,Voronov,Voronov2} and \textquotedblleft  higher groupoids\textquotedblright\,\cite{Severa}) is usually justified by their use in various fields of theoretical physics and  mathematics. Lie $\infty $-algebras or -oids often appear where, at first look, they do not seem to be part of the story, but end up to be  needed to answer natural questions, in particular questions where no higher-structure concept seems a priori involved.
Among examples of such a situation, let us cite deformation quantization of Poisson manifolds \cite{Kontsevich} and many recent developments of BV operator theory, e.g. \cite{CamposBV}, deformations of coisotropic submanifolds \cite{CattaneoFelder},  integration problems of Lie algebroids by stacky-groupoids \cite{Zhu}, complex submanifolds and Atiyah classes  \cite{ChenStienonXu,Kapranov,LSX}.  The list could be longer.


For instance, it appears in theory of singular foliations.  Singular foliations arise frequently in differential or algebraic geometry. Here following \cite{AndroulidakisIakovos,AndroulidakisZambon,Cerveau, Debord,LLS} we define a singular foliation on a smooth, complex, algebraic, real analytic manifold $M$  with sheaf of functions $\mathcal O$ to be a subsheaf $\mathfrak F \colon \mathcal U\longrightarrow\mathfrak F(\mathcal U)$ of the sheaf of vector fields $\mathfrak X$, which is closed under the Lie bracket and locally finitely generated as an $\mathcal O$-module. By Hermann's theorem \cite{Hermann}, this is enough to induce a partition of the manifold $M$ into embedded submanifolds of possibly different dimensions, called \emph{leaves} of the singular foliation. Singular foliations appear for instance as orbits of Lie group actions with possibly different dimensions or as symplectic leaves of a Poisson structure.  When all the leaves have the same dimension, we recover the usual \textquotedblleft regular foliations\textquotedblright \cite{DiederichKlas,LLL}.\\

In \cite{LLS}-\cite{LavauSylvain}, it is proven that \textquotedblleft behind\textquotedblright\,several singular foliations $ \mathfrak F$ there is a natural homotopy class of Lie $\infty $-algebroids, called \emph{universal Lie $\infty $-algebroid of $\mathfrak F $}, and that the latter answers natural basic questions about the existence of \textquotedblleft good\textquotedblright\,generators and relations for a singular foliation.The first part of  this thesis  is mainly an algebraization of \cite{LLS}-\cite{LavauSylvain}, that allows to enlarge widely the classes of examples. More precisely, Theorems \ref{thm:existence} and \ref{th:universal} are similar to the main theorems Theorem 2.8. and Theorem 2.9 in \cite{LLS}:
\begin{enumerate}
    \item Theorem \ref{thm:existence} equips  any free $\mathcal O $-resolution of a Lie-Rinehart algebra $\mathcal A $ with a Lie $\infty $-algebroid structure (Theorem 2.8. in \cite{LLS} was a statement for geometric resolutions of locally real analytic singular foliation on an open subset with compact closure). This is a sort of homotopy transfer theorem, except that no existing homotopy transfer theorem applies in the context of \underline{generic} $\mathcal O $-modules (for instance, \cite{Campos} deals only with projective $\mathcal O $-modules).
    { The difficulty is that we cannot apply the explicit transfer formulas that appear in the homological perturbation lemma because there is in general no $\mathcal O$-linear section of $\mathcal A $ to its projective resolutions.}
    \item Theorem \ref{th:universal} states that any Lie $\infty $-algebroid structure that terminates in $\mathcal A $ comes equipped with a unique up to homotopy Lie $\infty $-algebroid morphism to any structure as in the first item (Theorem 2.8. in \cite{LLS} was a similar statement for Lie $\infty $-algebroids whose anchor takes values in a given singular foliation).
\end{enumerate}
As in \cite{LLS}, an immediate corollary of the result is that any two Lie $\infty $-algebroids as in the first item are homotopy equivalent in a unique up to homotopy manner, defining therefore a class canonically associated to the Lie-Rinehart algebra, that deserve in view of the second item to be called \textquotedblleft universal\textquotedblright.\\

However:
\begin{enumerate}
 \item  While \cite{LLS} dealt with Lie $\infty $-algebroids over projective resolutions of finite length and finite dimension, we work here with  Lie $\infty $-algebroids over any \emph{free} resolution -even those of infinite length and of infinite dimension in every degree.
    \item In particular, since we are in a context where taking twice the dual does not bring back the initial space, we can not work with $Q$-manifolds (those being the \textquotedblleft dual\textquotedblright\,of Lie $\infty $-algebroids): it is much complicated to deal with morphisms and homotopies.  
\end{enumerate}
By doing so, several limitations of \cite{LLS} are overcome. While \cite{LLS} only applied to singular foliations which were algebraic or locally real analytic on a relatively compact open subset, the present thesis associates a natural homotopy class of Lie $ \infty$-algebroids to any Lie-Rinehart algebra, and in particular
\begin{enumerate}
    \item[a)]  to any singular foliation on a smooth manifold, (finitely generated or not). This construction still works with singular foliations in the sense of Stefan-Sussmann for instance, or to any involutive $C^\infty(M)$-module in $\Gamma(A)$.
    \item[b)] to any affine variety, to which we associate its Lie-Rinehart algebra of vector fields),
     and more generally to derivations of any commutative algebra,
    \item[c)] to singular Lie algebroids in the sense of Androulidakis and Zambon  \cite{ZambonMarco2},
    \item[d)] to unexpected various contexts, e.g. Poisson vector fields of a Poisson manifold, seen as a Lie-Rinehart algebra over Casimir functions, or symmetries of a singular foliation, seen as a Lie-Rinehart algebra over functions constant on the leaves. 
\end{enumerate}
These Lie $\infty $-algebroids are constructed on $\mathcal O$-free resolutions of the initial Lie-Rinehart algebra
over $\mathcal O $. They are  universal in some sense (see Section \ref{sec:main}), and they also are in particular unique up to homotopy equivalence.

A similar algebraization of the main results of \cite{LLS}, using semi-models category, appeared recently in Ya\"el Fr\'egier and Rigel A. Juarez-Ojeda \cite{Fregier}.
There are strong similarities between our results and theirs, but morphisms and homotopies in \cite{Fregier} do not match ours. It is highly possible, however, that  Theorem \ref{thm:existence} could be recovered using their results. Luca Vitagliano \cite{VitaglianoLuca} also constructed Lie $\infty$-algebra  structures out of regular foliations, which are of course a particular case of Lie-Rinehart algebra. These constructions do not have the same purposes. For regular foliations, our Lie $\infty$-algebroid structure is trivial in the sense that it is a Lie algebroid, while his structures become trivial when a good transverse submanifold exists. Lars Kjeseth \cite{LARSKJESETH,LARSKJESETHBRST} also has a notion of resolutions of Lie-Rinehart algebras. But his construction is more in the spirit of Koszul-Tate resolution: Definition 1. in \cite{LARSKJESETH} defines Lie-Rinehart algebras resolutions as resolutions of their Chevalley-Eilenberg coalgebra, not of the Lie-Rinehart algebra itself as a module. It answers a different category of questions, related to BRST and the search of cohomological model for Lie-Rinehart algebra cohomology.\\

These results can be used to understand the geometry of singular foliations such as their symmetries. More precisely, 
Let $(M,\mathfrak F)$ be a foliated manifold. A \emph{global symmetry} of a singular foliation $\mathfrak F$ on $M$ is a diffeomorphism $\phi\colon M\longrightarrow M$ which preserves $\mathfrak F$, that is, $\phi_*(\mathfrak F)=\mathfrak F$. The image of a leaf through a global symmetry is again a leaf (not necessarily the same leaf). For $G$ a Lie group, a \emph{strict symmetry action} of $G$ on a foliated manifold $(M,\mathfrak F)$ is a smooth action $G\times M\longrightarrow M$ that acts by global symmetries \cite{GarmendiaAlfonso2}. Infinitesimally, it corresponds to a Lie algebra morphism $\mathfrak{g} \longrightarrow \mathfrak X(M)$ between the Lie algebra $(\mathfrak{g},\lb_\mathfrak{g})$ of $G$ and the Lie algebra of symmetries of $\mathfrak F$.

A strict symmetry action of $G$ on $M$ goes down to the leaf space $M/\mathfrak F$, even though the latter space is not a manifold. The opposite direction is more sophisticated, since a strict symmetry action of $G$ on $M/\mathfrak F$ does not induce a strict action over $M$ in general. However, it makes sense to consider linear maps $\varrho\colon\mathfrak g\longrightarrow \mathfrak X(M)$ that satisfy $[\varrho(x),\mathfrak F]\subset\mathfrak F$ for all $x\in\mathfrak{g}$, and which are Lie algebra morphisms up to $\mathfrak F$, namely, $\varrho([x,y]_\mathfrak{g})-[\varrho(x),\varrho(y)]\in\mathfrak F$ for all $x,y\in\mathfrak{g}$. The latter linear maps are called \textquotedblleft \emph{weak symmetry actions}\textquotedblright. These actions induce a \textquotedblleft strict action\textquotedblright on the leaf space i.e.  a Lie algebra morphism $\mathfrak g\longrightarrow \mathfrak X(M/\mathfrak F)$, whenever $M/\mathfrak F$ is a manifold.

Let us emphasize on the following observation:  An infinitesimal action of a Lie algebra $\mathfrak{g}$ on a manifold $M$ is a Lie algebra morphism $\mathfrak g\longrightarrow \mathfrak{X}(M)$. Replacing $M$ by a Lie $\infty$-algebroid $(E,Q)$, one expects to define them as Lie $\infty$-algebra morphisms $\mathfrak{g}\longrightarrow \mathfrak{X}(E,Q)$, the latter space being a DGLA. Various results about those are given in Mehta-Zambon \cite{MEHTA2012576}. In particular, these authors give several equivalent definitions and interpretations of those.

In view of \cite{LLS,LavauSylvain} it is shown that behind every singular foliation or more generally any Lie-Rinehart algebras \cite{CLRL} there exists a Lie $\infty$-algebroid structure which is unique up to homotopy called the \emph{universal} Lie $\infty$-algebroid. Here is a natural question:  what does a symmetry of a singular foliation $\mathfrak F$ induce on an universal Lie $\infty$-algebroid of $\mathfrak F$?  Theorem \ref{main} of this thesis gives an answer to that question. It states that any weak symmetry action of a Lie algebra on a singular foliation $\mathfrak F$ can be lifted to a Lie $\infty$-morphism valued in the DGLA of vector fields on an universal Lie $\infty$-algebroid of $\mathfrak F$. 
Such Lie $\infty$-morphisms were studied by Mehta and Zambon \cite{MEHTA2012576} as "$L_\infty$-algebra actions". This goes in the same direction as \cite{GarmendiaAlfonso2} who already underlined Lie-2-group structures associated  to strict symmetry action of Lie groups. Furthermore, Theorem \ref{main} says this lift is unique modulo homotopy equivalence.

This result gives several geometric consequences. Here is an elementary question: can a Lie algebra action $\mathfrak{g} \to \mathfrak X(W)$ on an affine variety $W \subset \mathbb C^d$ be extended to a Lie algebra action $\mathfrak{g} \to \mathfrak X(\mathbb C^d)$ on $\mathbb{C}^d $?
    Said differently: it is trivial that any vector field on $W$ extends to $\mathbb C^d $, but can this extension be done in such a manner that it preserves the Lie bracket? Although no \textquotedblleft $\infty$-oids\textquotedblright\, appears in the question, which seems to be a pure algebraic geometry question, we claim that the answer goes through Lie $\infty $-algebroids and singular foliations. More precisely, the idea is then to say that any $\mathfrak g $-action on $W$ induces a weak symmetry action on the singular foliation $\mathcal I_W\mathfrak{X}(\mathbb{C}^d)$ of all vector fields vanishing on $W$ (here $\mathcal I_W$ is the ideal that defines $W$). By Theorem \ref{main}, we know that it is possible to lift any weak symmetry action of singular foliation into a Lie $\infty $-morphism. But is it possible to build such a Lie $\infty $-morphism where the polynomial-degree $-1$ of the second order Taylor coefficient is zero? There are cohomological obstructions. In some specific cases, obstruction classes appear on some cohomology, although in general the obstruction is rather a Maurer-Cartan-like equations that may or may not have solutions. We show that both questions are in fact related.\\

    Here is another interesting question where we would like to apply our results: Can we desingularize a singular affine variety $W\subseteq \mathbb C^d$ by making use of the universal Lie $\infty$-algebroid of the singular foliation $\mathfrak{F}=\mathfrak X(W)$ of vector fields tangent $W$? In section \ref{sec:blow-up-procedure}, we use the geometric resolution of the singular foliation (i.e. the resolution on which the universal Lie $\infty$-algebroid is built) to recover several notions of resolution of singularities: on being due to Nash \cite{D.T} and a second one to Mohsen \cite{MohsenOmar}. But the meaning of these spaces are unclear. We would like to relate them with the higher brackets of the universal Lie $\infty$-algebroid, e.g. to understand the role of the $3$-ary bracket in this procedure.

    
    In the last chapter of the thesis, we introduce the  notion of "bi-submersion towers"  over singular foliations that we denote by $\mathcal{T}_B$. The latter notion as the name suggests is a family of "bi-submersions" which are built one over the other. The concept of bi-submersion over singular foliations has been introduced in \cite{AndroulidakisIakovos} and it is used  in $K$-theory \cite{Androulidakis-Iakovos-Georges} or differential geometry \cite{AndroulidakisIakovos2011Pcoa,AndroulidakisZambon,GarmendiaAlfonsoMarcoZambon}. We show that such a bi-submersion tower over a singular foliation $\mathfrak{F}$ exists if  $\mathfrak{F}$ admits a geometric resolution. Provided that it exists, we show in Theorem \ref{thm:final-sym} that  any infinitesimal action of a Lie algebra $\mathfrak{g}$ on the singular foliation $\mathfrak{F}$ lifts to the bi-submersion tower $\mathcal{T}_B$. This lift looks like the beginning of a kind of Lie $\infty$-morphism. We wonder if we can continue the construction in Theorem \ref{thm:final-sym} to a Lie $\infty$-morphism.\\

    The thesis is organized as follows. In chapter \ref{chap:1}, we recall some basics on  graded co-algebra structures on the graded symmetry algebra and their morphisms. We also review the notion of co-derivations and their properties. In Chapter \ref{chap:oid-alg-ver}, we introduce the notion of Lie $\infty$-algebroid on an arbitrary commutative unital algebra $\mathcal{O}$ over a field of characteristic zero, and also define their morphisms in terms of co-derivations. We define the notion of homotopy between them. Some technical Lemmas and Propositions are given. In  Chapter \ref{chap-Lie-Rinehrt}, we  fix notations and review definitions,
examples, and give main properties of Lie-Rinehart algebras. Besides, we construct a Lie $2$-algebroid structure over any Lie-Rinehart algebra that admits a free resolution of length less or equal to $2$. In Chapter \ref{Chap:main}, we state and prove
the main results of the first part of the thesis, i.e. the equivalence of categories between Lie-Rinehart algebras
and homotopy classes of free acyclic Lie $\infty$-algebroids, which justifies the name universal Lie
$\infty$-algebroid of a Lie-Rinehart algebra. Also, we describe  the
universal Lie $\infty$-algebroids of several Lie-Rinehart algebras. The complexity reached by the
higher brackets in these examples should convince us that it is not a trivial structure, even for relatively simple Lie-Rinehart algebras. The Chapter \ref{chap:5} is devoted to the applications of the results of the previous chapters to affine varieties. We recall some basics definitions and theorems on affine varieties $W$. We present three main constructions of Lie-Rinehart associated to $W$ and relate their universal Lie $\infty$-algebroids together. Also, some examples of the universal Lie $\infty$-algebroids are given, such as blow-ups, vector fields vanishing on a complete intersection, etc. In Chapter \ref{chap:oid1}, we recall the notion of $Q$-manifolds and apply the results on Lie-Rinehart algebras to recover the universal Lie $\infty$-algebroids of a singular foliation \cite{LLS}. This chapter ends with a result on the cohomology of longitudinal vector fields. In Chapter \ref{chap:isotropy}, we recall the definition of Androulidakis and Skandalis isotropy Lie algebra
of a singular foliation at a point and recall from \cite{LLS} its relation with the Universal Lie $\infty$-algebroids of a singular foliation. We end with a blow-up procedure for a singular foliation inspired by O. Mohsen. In Chapter \ref{chap:symmetries}, we study symmetries of singular foliations. We present some definitions and facts on weak symmetry actions of Lie algebras on singular foliations and give some examples. We state the main results of the second part of the thesis and present their proofs. In  Chapter \ref{chap:obstruction-theory}, we define an obstruction class for extending a Lie algebra action on an affine variety to ambient space and also give some examples. In Chapter \ref{chap:tower}, we look at symmetries of bi-submersions. Afterwards, we introduce the notion of bi-submersion towers over a singular foliation and point out some observations related to their symmetries.

Finally, in Appendix \ref{appendix:tensor}, we recall the definition of tensor algebra and fix notations. In Appendix \ref{appendix:mod}, we recall some general facts on homological algebra and give some geometric constructions.



\part{Lie-Rinehart algebras $\equiv$ Acyclic Lie $\infty$-algebroids}

\chapter{Preliminaries}\label{chap:1}
This chapter sets the ground for the whole thesis, especially  chapters \ref{chap:oid-alg-ver}, \ref{chap-Lie-Rinehrt} and \ref{Chap:main}. It fixes terminologies, conventions and notations. For more details, we also refer the reader to Appendix \ref{appendix:tensor} and \ref{appendix:mod}.\\

Throughout this thesis, $ \mathbb K$ is a field of characteristic zero, and  $\mathcal O$ is an associative commutative unital $\mathbb{K}$-algebra unless otherwise mentioned. Also, an $\mathcal{O}$-module $\mathcal E$ is seen as $\mathbb{K}$-vector space in the natural way, $\lambda\cdot e:=(\lambda\cdot 1_\mathcal{O})\cdot e$, where $1_\mathcal{O}\equiv 1$ is the unit of $\mathcal{O}$. In the sequel, we will drop the notation "$\cdot$".\\

Geometrically, $\mathcal{O}$ can be understood as the algebra of smooth functions on a manifold $M$, or on an open subset $\mathcal{U}\subset M$ of a complex manifold, or the coordinate ring of an affine variety $W$.

\section{Graded symmetric algebras}\label{conventions}
For $\E=\displaystyle{\oplus_{i\in \mathbb Z}\E_{i}}$ be a $\mathbb{Z}$-graded module, we denote by $\lvert x\rvert \in \mathbb{Z}$ the degree of a homogeneous element $x\in \mathcal{E}$. 
\begin{itemize}
    \item We denote by $\bigodot^\bullet\E  $ and call \emph{(reduced) graded symmetric algebra of $\E $ over $\mathcal O$} the quotient $$T^\bullet_\mathcal O \E/\langle x\otimes_{\mathcal O} y-(-1)^{\lvert x\rvert\lvert y\rvert}y\otimes_{\mathcal O}x\rangle$$ of the tensor algebra (see Appendix \ref{appendix:tensor}) over $\mathcal O $ of $\E$, namely
$$ T^\bullet_\mathcal O \E:= \bigoplus_{k =1}^\infty
\underbrace{ \E \otimes_\mathcal O \cdots \otimes_\mathcal O \E}_{\hbox{\small{$k$ times}}}
$$
 by the ideal generated by $x\otimes_{\mathcal O} y-(-1)^{\lvert x\rvert\lvert y\rvert}y\otimes_{\mathcal O} x$, with $x,y$ arbitrary homogeneous elements of $\E $. This quotient is a graded commutative algebra. We denote its  product by $\odot$.
\item Similarly, we denote by $ S^\bullet_\mathbb K( \E ) $ and call \emph{(reduced) graded symmetric algebra} of $\E $ over \emph{the field $\mathbb K$} the quotient of the tensor algebra (over $\mathbb K $) of $\E$, i.e.,
$$ T^\bullet_\mathbb K \E:= \bigoplus_{k =1}^\infty
\underbrace{ \E \otimes_\mathbb{K} \cdots \otimes_\mathbb{K} \E}_{\hbox{\small{$k$ times}}}$$  by the ideal generated by $x\otimes_\mathbb{K}  y-(-1)^{\lvert x\rvert\lvert y\rvert}y\otimes_\mathbb{K}  x$, with $x,y$ arbitrary homogeneous elements of $\E $. We denote by $\cdot$ the product in $S^\bullet_\mathbb K (\E)$.

\end{itemize}

The algebras $\bigodot^\bullet\E$ and $ S^\bullet_\mathbb{K} (\E)$ come equipped with two different \textquotedblleft grading\textquotedblright\,. Let us make them explicit. 
\begin{enumerate}
    \item  We define the \emph{degree} of $x= x_1 \odot \cdots  \odot x_n\in \bigodot^n\E$ or $x=x_1 \cdot \cdots  \cdot x_n\in S^n_\mathbb{K}(\E)$ by 
$$| x_1 \cdot \cdots  \cdot x_n | =  | x_1 \odot \cdots  \odot x_n | = |x_1|+ \cdots  + |x_n|$$
for any homogeneous $ x_1, \dots, x_n\in\E$. With respect to this degree, $\bigodot^\bullet\E $ and $S^\bullet_\mathbb{K}(\E)$ are graded commutative algebras. 
    \item The second grading is called "polynomial-degree". The \emph{polynomial-degree} of $ x_1 \odot \cdots  \odot x_n\in \bigodot^n\E$ or $ x_1 \cdot \cdots  \cdot x_n\in S^n_\mathbb{K}(\E)$ is defined to be $n$. 
    We have the following  polynomial-degree decomposition
$ \bigodot^\bullet \E = \oplus_{k \geq 1} \bigodot^k \E $ and $ S^\bullet_\mathbb{K} (\E) = \oplus_{k \geq 1} S^k_\mathbb{K} (\E)$,
where $ \bigodot^k\E$ and $ S^k_\mathbb{K}(\E)$ stand for the $\mathcal O$-module of elements of polynomial-degree~$k$ {and the $\mathbb{K}$-vector space of elements of polynomial-degree~$k$, respectively}.
\end{enumerate}
\begin{convention}
For $ \E$ a graded $ \mathcal O$-module,
the set of elements of \emph{polynomial-degree} $k$ and degree $d$ in $ \bigodot^\bullet\E$ (resp. $S^k_\mathbb{K}(\E)$) shall be denoted by $\bigodot^k \mathcal E_{|_d}$ (resp. $S^k_\mathbb{K}(\E)_{|_d}$). For example, $$\displaystyle{\bigodot^k \mathcal E_{|_d}=\bigoplus_{i_1+\cdots +i_k=d}\E_{i_1} \otimes_\mathcal O \cdots \otimes_\mathcal O \E_{i_k}.}$$
\end{convention}
\begin{itemize}
\item For any homogeneous elements $x_1 , \ldots , x_k \in \mathcal{E}$ and $\sigma\in\mathfrak{S}_k$ a permutation of $\{1, \ldots, k\}$, the \emph{Koszul sign} $\epsilon(\sigma; x_1 , \ldots , x_k )$ is defined by:
$$  x_{\sigma(1)} \odot \cdots \odot x_{\sigma(k)}= \epsilon(\sigma; x_1 , \ldots , x_k ) \, x_1 \odot \cdots \odot x_k.$$ We often write $\epsilon(\sigma )$ for $ \epsilon(\sigma; x_1 , \ldots , x_k )$.\\
 
\item For $i, j \in\mathbb{N}$, a \emph{$(i, j)$-shuffle} is a permutation
$\sigma \in \mathfrak{S}_{i+j}$ such that $\sigma(1) <\ldots < \sigma(i)$ and $\sigma(i + 1) < \ldots < \sigma(i + j)$, and the set of
all $(i, j)$-shuffles is denoted by $\mathfrak{S}(i,j)$. More generally, a $(i_1,\ldots,i_k)$-\emph{shuffle} is a permutation $\sigma\in \mathfrak{S}_{i_1+\cdots+i_k}$ such that $$\begin{array}{c}
    \sigma(1)<\cdots< \sigma(i_1 ) \\
     \sigma(i_1+1)<\cdots <\sigma(i_1 +i_2),\\
\vdots\\
\sigma(i_1 +i_2 +\cdots+i_{j-1}+1)<\cdots<\sigma(i_1 +i_2 +\cdots+i_j),\\\vdots\\\sigma(i_1 +i_2 +\cdots+i_{k-1}+1)<\cdots<\sigma(i_1 +i_2 +\cdots+i_k)
\end{array}$$The set of
all $(i_1,\ldots i_k)$-shuffles is denoted by $\mathfrak{S}(i_1,\ldots, i_k)$.
\end{itemize}

\section{Graded symmetric co-algebras and their morphisms}
Graded co-algebra structures are the dual version of graded algebra structures, i.e., they are obtained  by reversing all arrows and permutes the order of composition. Let us recall the definition of a co-algebra structure on an $\mathcal O$-module. We refer the reader to, e.g \cite{Loday-Vallette,Kassel,Marco-Manetti} or Chapter 8 of  \cite{ManettiMarco2005Lodo} for more details on this topic.

\begin{definition}\label{def:coproduct}
    A \emph{coassociative graded cocommutative co-algebra} structure on an graded $\mathcal{O}$-module $\mathcal{C}=\displaystyle{\oplus_{i\in \mathbb{Z}}\mathcal{C}_i}$ is \begin{itemize}
    
    \item a linear map of degree zero called \emph{(graded) coproduct} $\Delta\colon\mathcal{C}\longrightarrow \mathcal{C}\otimes \mathcal{C}$ with $$\Delta(\mathcal{C}_k)\subset \bigoplus_{i+j=k}\mathcal{C}_i\otimes \mathcal{C}_j$$ is such that $\Delta(\mathcal{C}_k)$ has non-zero\footnote{Note that this condition is automatically satisfied when the module is concentrated in positive degree.} intersection with only finitely many spaces $\mathcal{C}_i\otimes \mathcal{C}_j$

     \item $(\Delta \otimes \mathrm{id} )\circ\Delta = (\mathrm{id} \otimes \Delta)\circ\Delta \colon \mathcal{C} \rightarrow \mathcal{C} \otimes \mathcal{C} \otimes \mathcal{C}$,\, called \emph{graded coassociativity}.
       
       \item  $\tau\circ \Delta = \Delta$, called \emph{graded cocommutativity},\item[] where $\tau\colon \mathcal{C}\otimes\mathcal{C}\rightarrow \mathcal{C}\otimes\mathcal{C}$ is defined by $\tau(a\otimes b)= (-1)^{|a||b|}b\otimes a$.
\end{itemize}
    Equivalently, the following diagrams commute
\begin{align*}
    \xymatrix{\mathcal{C}\ar[d]_\Delta\ar[rr]^\Delta&&\mathcal{C}\otimes \mathcal{C}\ar[d]^{\Delta\otimes \mathrm{id}}\\\mathcal{C}\otimes\mathcal{C}\ar[rr]^{\mathrm{id}\otimes\Delta}&&\mathcal{C}\otimes\mathcal{C}\otimes\mathcal{C}}\qquad\qquad \xymatrix{\mathcal{C}\ar[d]_{\Delta}\ar[dr]^\Delta&\\\mathcal{C}\otimes\mathcal{C}\ar[r]^\tau&\mathcal{C}\otimes\mathcal{C}.}
\end{align*} The pair $(\mathcal C, \Delta)$ is called \emph{graded $\mathcal O$-co-algebra}. We say that $(\mathcal C, \Delta)$ is \emph{counital} if there is a morphism  of graded vector spaces $u\colon \mathcal{C}\rightarrow \mathbb K$ such that $(u\otimes\mathrm{id})\circ \Delta=(\mathrm{id}\otimes u)\circ\Delta=\mathrm{id}$ i.e. the diagram below commutes

$$ \xymatrix{\mathcal{C}\ar[drr]^{\mathrm{id}}\ar[d]_\Delta\ar[rr]^\Delta&&\mathcal{C}\otimes \mathcal{C}\ar[d]^{\mathrm{id}\otimes u}\\\mathcal{C}\otimes\mathcal{C}\ar[rr]<2pt>^{u\otimes\mathrm{id}}&&\mathbb{K}\otimes\mathcal{C}\simeq \mathcal{C}\simeq \mathcal{C}\otimes\mathbb{K}.}$$
\end{definition}

\begin{prop}
    Let $(\mathcal{C},\Delta)$ and $(\mathcal{C}',\Delta')$ be two co-algebras over $\mathcal{O}$. Then, the tensor product $\left(\mathcal{C}\otimes \mathcal{C}', (\mathrm{id}_\mathcal{C}\otimes\tau\otimes \mathrm{id}_{\mathcal{C}'})\circ \Delta\otimes\Delta'\right)$ is a co-algebra  over $\mathcal{O}$.
\end{prop}

     \begin{example}
     The polynomial ring $\mathbb{K}[t]$ in one indeterminate $t$ is a co-algebra with the coproduct $\Delta$ given on monomials\, $t^n,n\geq 0$\, by:
 $$ \Delta (t^{n})=\sum _{k=0}^{n}{\dbinom {n}{k}}t^{k}\otimes t^{n-k}$$and extends by linearity.
     \end{example}

The graded symmetric algebra is also an example of co-algebra.
\begin{lemma}\label{lemma:coprod}Let $\E$ be a $\mathcal{O}$-module. Both $\bigodot^\bullet\E$ and $S^\bullet _\mathbb K (\mathcal E) $ admit natural co-algebra structures with respect to the deconcatenation $\Delta$ defined by:\begin{equation}
\label{eq:coproduct}\Delta(x_1\odot\cdots\odot x_n)=\sum_{i=1}^{n-1}\sum_{\sigma\in\mathfrak{S}({i,n-i})}\epsilon(\sigma)x_{\sigma(1)}\odot\cdots \odot x_{\sigma(i)}\otimes x_{\sigma(i+1)}\odot\cdots\odot x_{\sigma(n)}\end{equation}for every $x_1,\ldots, x_n\in \E$. 
\end{lemma}
\begin{example}
 For small $n$. The formula of Equation \eqref{eq:coproduct} on homogeneous elements $x,y,z\in \E$ means that

\begin{itemize}
    \item $\Delta(x)=0$
    
    \item $\Delta(x\odot y)=x\otimes y + (-1)^{|x||y|}y\otimes x$
    
    \item $\Delta(x\odot y\odot z)=x\odot y \otimes z + (-1)^{|x||z|}x\odot z \otimes y  +(-1)^{|x|(|z|+|y|)} y\odot z \otimes x +\\ \phantom{xxxxxxxxxxxxxxxxxxxxxxxxxxxx}x\otimes y \odot z +(-1)^{|x||y|}y\otimes x \odot z+(-1)^{|z|(|x|+|y|)}z\otimes x \odot y$
\end{itemize}
\end{example}
\begin{remark}
According to Definition \ref{def:coproduct}, the coproduct $\Delta$ is graded with respect to the positive grading given by the "polynomial-degree" of monomials in  the symmetric algebras $\bigodot^\bullet\E$ and $S^\bullet_\mathbb{K}(\E)$. 
\end{remark}

\begin{example}
Consider $M=\mathbb R^d$ with global  coordinates $(x_1,\ldots,x_d)$. The algebra of differential forms $(\Omega(M), \wedge)$ over $M$ has a co-algebra structure $\Delta_\Omega\colon \Omega(M)\rightarrow \Omega(M)\otimes \Omega(M) $ given as in Lemma \ref{lemma:coprod}. Namely, 
\begin{equation*}
\Delta(dx_1\wedge\cdots\wedge dx_n)=\sum_{i=1}^{n-1}\sum_{\sigma\in\mathfrak{S}({i,n-i})}\epsilon(\sigma)dx_{\sigma(1)}\wedge\cdots\wedge dx_{\sigma(i)}\otimes dx_{\sigma(i+1)}\wedge\cdots\wedge dx_{\sigma(n)}.\end{equation*}
This matches \eqref{eq:coproduct} if we consider $\mathcal{E}=\Omega^1(M)$ is concentrated in degree $+1$.
\end{example}

\subsection{Co-morphisms and co-derivations}
We recall the definitions of co-morphisms, co-derivations of  graded co-algebra structures. 
\begin{definition}\label{def:co-morphism}
Let $(\mathcal C', \Delta')$ and $(\mathcal C, \Delta)$ be two graded co-algebras.  
\begin{enumerate}
    \item A \emph{morphism of co-algebras} or \emph{co-morphism} from $(\mathcal C', \Delta')$ to $(\mathcal C, \Delta)$ is a linear map $\Phi\colon \mathcal C'\longrightarrow \mathcal C$ of degree $0$ such that the following diagram commutes
\begin{equation}
    \xymatrix{\mathcal C'\ar[d]_{\Delta'}\ar[r]^\Phi& \mathcal C\ar[d]^{\Delta}\\\mathcal C'\otimes\mathcal C'\ar[r]^{\Phi\otimes \Phi}&\mathcal C\otimes\mathcal C}
\end{equation}
In formula: $\Delta\circ \Phi=(\Phi\otimes \Phi)\circ \Delta'$.

\item Let $\Phi\colon \mathcal C'\longrightarrow \mathcal C $ be a co-morphism. A \emph{$\Phi$-co-derivation of degree $l$}  is a linear map $\mathcal{H}\colon \mathcal C'\longrightarrow \mathcal C$ of degree $l$ such that the following diagram commutes,
\begin{equation}
    \xymatrix{&\mathcal C'\ar[d]_{\Delta'}\ar[rr]^{\mathcal H}&& \mathcal C\ar[d]^{\Delta}\\&\mathcal C'\otimes\mathcal C'\ar[rr]^{\mathcal H\otimes \Phi + \Phi \otimes\mathcal H}&&\mathcal C\otimes\mathcal C}
\end{equation}
that is, the so-called \emph{co-Leibniz identity} is satisfied:
\begin{equation} \label{eq:Phicoder} 
	\Delta \circ \mathcal H = (\mathcal H\otimes \Phi + \Phi \otimes\mathcal H) \circ \Delta'.
\end{equation}When $\mathcal{C'}=\mathcal{C}$ and $\Phi=\mathrm{id}$ we speak of a \emph{co-derivation} (of degree $l$).

\end{enumerate}
\end{definition}
\begin{remark}
It is easily checked that the composition of two co-morphisms $\mathcal C''\stackrel{\Psi}{\rightarrow} \mathcal C'\stackrel{\Phi}{\rightarrow} \mathcal{C}$ is again a co-morphism.
\end{remark}

\begin{prop}\label{prop:Coderbracket}
  We denote by $\mathrm{CoDer}(\mathcal{C})$ the linear space of co-derivations of $(\mathcal C, \Delta)$. The set $\mathrm{CoDer}(\mathcal{C})$ is a graded Lie sub-algebra of $\mathrm{Hom}(\mathcal{C},\mathcal{C})$ when equipped with the graded commutator.
 \end{prop}
\begin{proof}
For two co-derivations $\mathcal{H}_1,\mathcal{H}_2\in \mathrm{CoDer}(\mathcal{C})$ one has,\begin{align*}
    \Delta\circ \mathcal{H}_1\circ \mathcal{H}_2&=(\mathcal{H}_1\otimes \mathrm{id}+\mathrm{id}\otimes\mathcal{H}_1)\circ(\mathcal{H}_2\otimes \mathrm{id}+\mathrm{id}\otimes\mathcal{H}_2)\circ \Delta\\&=(\mathcal{H}_1\circ\mathcal{H}_2\otimes \mathrm{id}  +\mathcal{H}_1\otimes\mathcal{H}_2+(-1)^{|\mathcal{H}_1||\mathcal{H}_2|}\mathcal{H}_2\otimes\mathcal{H}_1+\mathrm{id}\otimes \mathcal{H}_1\circ\mathcal{H}_2)\circ \Delta. 
\end{align*}
Also, we obtain a similar equation for $(-1)^{|\mathcal{H}_1||\mathcal{H}_2|} \Delta\circ \mathcal{H}_2\circ \mathcal{H}_1$ by changing the roles of $\mathcal{H}_1$ and $\mathcal{H}_2$. Subtracting both terms, one get the co-Leibniz identity for $[\mathcal{H}_1,\mathcal{H}_2]= \mathcal{H}_1\circ \mathcal{H}_2 - (-1)^{|\mathcal{H}_1||\mathcal{H}_2|}\mathcal{H}_2\circ \mathcal{H}_1$.
\end{proof}

Let us describe Definition \ref{def:co-morphism} in the context of the thesis. Let $\E'$ and $\E$ be graded $\mathcal O$-modules. 
\begin{definition}
A linear map $\Phi\colon S^\bullet_\mathbb{K} (\E')\rightarrow S^\bullet_\mathbb{K}(\E)$ is said to be of \emph{{polynomial-degree}/degree $r\in \mathbb{Z}$}, if it sends polynomials of $S^{k}_\mathbb{K}(\E)$ to those of  $S^{k-r}_\mathbb{K} (\E)$/elements of $S^{\bullet}_\mathbb{K}(\E)_{|_d}$ to $S^{\bullet}_\mathbb{K}(\E)_{|_{d+r}}$. In formulas, for very $k\geq 0$, (resp. $d\in \mathbb{Z}$)  one has $$\Phi\left(S^{k}_\mathbb{K}(\E)\right)\subseteq S^{k-r}_\mathbb{K}(\E),\quad\text{resp.}\qquad\Phi\left(S^{\bullet}_\mathbb{K}(\E)_{|_d}\right)\subseteq S^{\bullet}_\mathbb{K}(\E)_{|_{d+r}}.$$\end{definition}

\begin{remark}
Any linear map $\Phi\colon S^\bullet_\mathbb{K} (\E')\rightarrow S^\bullet_\mathbb{K}(\E)$ of degree $l$ can be decomposed with respect to the polynomial-degree as formal sum:\begin{equation}\label{eq:arity-decomposition}
    \Phi=\sum_{r\in\mathbb{Z}}\Phi^{(r)}
\end{equation}
{where, for all $k\in \mathbb{N}_0$}, \,$\Phi^{(r)}\colon S^\bullet_\mathbb{K}(\E')\rightarrow S^{\bullet-r}_\mathbb{K}(\E)$ is a linear map of polynomial-degree $r$ and of degree $l$.\\

\begin{itemize}
\item It should be understood that $\Phi^{(r)}(v)=0$ for all $v\in S^{k}_\mathbb{K}(\E')$ with $k\leq r$. Also, the decomposition \eqref{eq:arity-decomposition} makes sense 
in particular when the polynomial-degrees of $\Phi$ are lower bounded that is, there is an integer $N\in \mathbb{Z}$, such that $\Phi^{(r)}=0$ for all $r< N$. In that case, for every $v\in S^k_\mathbb{K}(\E)$, $$\sum_{r\in\mathbb{Z}}\Phi^{(r)}(v)=\sum_{N\leq r<k}\Phi^{(r)}(v),$$ is a finite sum. Hence, Equation \eqref{eq:arity-decomposition} becomes $$\Phi=\sum_{r\geq N}\Phi^{(r)},$$
for some $N\in \mathbb{Z}$, since the l.h.s of \eqref{eq:arity-decomposition} is finite. Notice that a linear map $\Phi\colon S^\bullet_\mathbb{K} (\E')\rightarrow~S^\bullet_\mathbb{K}(\E)$ is of polynomial-degree $N$ if and only if $\Phi^{(N)}$ is the unique non-zero term, namely $\Phi^{(r)}=0$, for $r\neq N$.
    \item Also, for the composition of two linear maps $S_\mathbb{K}^\bullet (\E'')\stackrel{\Psi}{\rightarrow} S_\mathbb{K}^\bullet (\E')\stackrel{\Phi}{\rightarrow} S_\mathbb{K}^\bullet (\E)$,  we have \begin{equation}
    (\Phi\circ \Psi)^{(r)}=\sum_{i+j=r}\Phi^{(i)}\circ \Psi^{(j)}.
\end{equation}
for every $r\in \mathbb{Z}$. Here, although the sum is infinite, it becomes finite when applied to a given element in $S^\bullet_\mathbb{K}(\E')$.
\end{itemize}

\end{remark}

\subsubsection{Co-morphisms on symmetric algebras}
Let us have a discussion on co-morphisms from the symmetric graded co-algebras $(S_\mathbb{K}^\bullet (\E'),\Delta')$ and $(S_\mathbb{K}^\bullet (\E),\Delta)$, where $\Delta'$ and $\Delta$ are the respective coproducts like in Lemma \ref{lemma:coprod}. Given any linear map $F\colon S_\mathbb{K}^\bullet (\E')\rightarrow \E$. Denoting by $F_r\colon S_\mathbb{K}^{r+1} (\E')\rightarrow \E$ for $r\in \mathbb{N}_0$, the restriction of $F$ to $S_\mathbb{K}^{r+1} (\E')$. The linear map $F$ can be extended to a unique co-morphism $\Bar F\colon (S_\mathbb{K}^\bullet (\E'),\Delta')\rightarrow (S_\mathbb{K}^\bullet (\E),\Delta)$ by taking for $k\in \mathbb{N}$ the component on $S_\mathbb{K}^k (\E)$ to be, for any homogeneous  $x_1,\ldots,x_n\in \E'$
\begin{equation}\label{eq:co-morphism_component}
    \sum_{i_1 +\cdots+i_k =n}\,\sum_{\sigma\in \mathfrak{S}(i_1,\ldots,i_r )}\epsilon(\sigma)\frac{1}{k!}\prod_{j=1}^kF_{i_j}(x_{\sigma(i_{1}+\cdots+i_{j-1}+1)},\ldots, x_{\sigma(i_{1}+\cdots+i_j)}).
\end{equation}
where $\mathfrak{S}(i_1,\ldots,i_k )$ is the set of $(i_1,\ldots,i_k)$-shuffles, with $i_1,\ldots,i_k\in \mathbb{N}$.

\begin{example}
Let us compute $\Bar{F}\colon \left(S^{k}_\mathbb{K}(\E'), \Delta'\right)\longrightarrow \left(S^{\bullet}_\mathbb{K}(\E), \Delta\right)$ for $k=1,2,3$. For $x,y,z\in \E'$,
\begin{itemize}
    \item $\Bar{F}(x)=F_0(x)$.
    \item $\Bar{F}(x\cdot y)=F_1(x\cdot y)+ F_0(x)\cdot F_0(y)$.
    \item $\Bar{F}(x\cdot y\cdot z)=F_2(x\cdot y\cdot z)+ F_0(x)\cdot F_1(y\cdot z) + (-1)^{|x||y|}F_0(y)\cdot F_1(x\cdot z)+\\\phantom{cccccccccccccccccccccccccccccccccccccccc} (-1)^{(|x|+|y|)|z|}F_0(z)\cdot F_1(x\cdot y) + F_0(x)\cdot F_0(y)\cdot F_0(z)$.
\end{itemize}
We can easily check that $\Delta\circ \Bar{F}= (\Bar{F}\otimes \Bar{F})\circ \Delta'$, e.g

\begin{itemize}
    \item $\Delta\circ \Bar{F}(x)=\Delta\circ F_0(x)=0=(\Bar{F}\otimes \Bar{F})\circ \Delta'(x)$
    \item On one side, \begin{align*}
        \Delta\circ \Bar{F}(x\cdot y)&=\Delta\circ F_1(x\cdot y)+ \Delta(F_0(x)\cdot F_0(y))\\& =F_0(x)\otimes F_0(y) + (-1)^{|x||y|}F_0(y)\otimes F_0(x).
    \end{align*}
    On the other side, \begin{align*}
        (\Bar{F}\otimes \Bar{F})\circ \Delta'(x,y)&=(\Bar{F}\otimes \Bar{F})(x\otimes y + (-1)^{|x||y|}y\otimes x)\\&=(\Bar{F}\otimes \Bar{F})(x\otimes y) + (-1)^{|x||y|}(\Bar{F}\otimes \Bar{F})(y\otimes x) \\&=F_0(x)\otimes F_0(y) + (-1)^{|x||y|}F_0(y)\cdot F_0(x).
    \end{align*}
\end{itemize}
\end{example}

{The following proposition claims that every co-morphism from $S_\mathbb{K}^\bullet (\E')$ to $S_\mathbb{K}^\bullet (\E)$ is of the form described in \eqref{eq:co-morphism_component}. For a proof, see \cite{Loday-Vallette,Kassel,ManettiMarco2005Lodo}}.
\begin{prop}\label{prop:co-morphism} Let $\E', \E$ be two graded $\mathcal{O}$-modules. For any morphism of graded vector space $F\colon S_\mathbb{K}(\E')\longrightarrow \E$ there exists an (unique) morphism of graded co-algebras $\Bar{F}\colon (S_\mathbb{K}(\E'), \Delta')\longrightarrow (S_\mathbb{K}(\E), \Delta)$ that satisfies $\mathrm{pr}\circ \Bar{F}=F$.\\

\noindent
Here $\mathrm{pr}$ is the canonical projection onto the term of polynomial-degree $1$, i.e.,
$\mathrm{pr}\colon S^\bullet_\mathbb{K} (\mathcal E) \rightarrow S^{1}_\mathbb K (\mathcal E) \simeq \E$.
\end{prop}

 \begin{remark}The following facts are crucial, and will be used in the sequel.
 
 \begin{enumerate}
 \item By Proposition \ref{prop:co-morphism}, a co-morphism $\Phi\colon S_\mathbb{K}^\bullet (\E')\rightarrow S_\mathbb{K}^\bullet (\E)$ is entirely determined by the collection indexed by $r \in \mathbb N_0 $ of maps called its \emph{$r$-th Taylor coefficients}:
\begin{equation}
\label{eq:Taylor}
    \Phi_r \colon S^{r+1}_\mathbb{K}(\E')\stackrel{\Phi}{\longrightarrow} S^\bullet_\mathbb{K}(\E)\stackrel{\text{pr}}{\longrightarrow} \E,
\end{equation}
 Notice that the component $\Phi^{(r)}$ of polynomial-degree $r\geq 0$ of $\Phi$ coincides with $r$-th Taylor coefficient $\Phi_r$ on $S^{r+1}_\mathbb{K} (\mathcal E')$. In fact we have, $\Phi_r=(\mathrm{pr}\circ \Phi^{(r)})_{|{S^{r+1}_\mathbb{K} (\mathcal E')}}$.
 \item By item 1, a co-morphism $\Phi\colon S_\mathbb{K}^\bullet (\E')\rightarrow S_\mathbb{K}^\bullet (\E)$ is completely determined by its components of polynomial-degree $r\geq 0$, hence it admits a polynomial decomposition of the form:\begin{equation}
    \Phi=\sum_{r\geq0}\Phi^{(r)}.
\end{equation}



 \item   Similar results as in Proposition \ref{prop:co-morphism} hold for $\Phi$-co-derivations or for co-derivations, e.g. see Corollary VIII.34
  in \cite{ManettiMarco2005Lodo}  or \cite{Loday-Vallette}. Given a co-morphism $\Phi\colon S^\bullet_\mathbb{K} (\E')\rightarrow S^\bullet_\mathbb{K} (\E)$ a $\Phi$-co-derivation $\mathcal H$ of degree $l$ is  entirely determined by the \emph{Taylor coefficients} defined as in \eqref{eq:Taylor} $$\mathcal H_r\colon S^{r+1}_\mathbb K (\E') \rightarrow \E, \;r\in\mathbb{N}_0.$$ For $k\in \mathbb{N}_0$ and $x_1,\ldots,x_n\in \E'$, the component on $S_\mathbb{K}^k (\E)$ of $\mathcal{H}(x_1\cdot\cdots\cdot x_n)$ takes the form
\begin{equation}\label{ex:formulacoder}
    \sum_{i_1+\cdots+i_k=n}\sum_{\sigma\in \mathfrak{S}(i_1,\ldots,i_{k})}\epsilon(\sigma)\frac{1}{k!}\mathcal{H}_{i_1}(x_{\sigma(1)},\cdots,x_{\sigma(i_1)})\cdot\prod_{j=1}^{k-1}\Phi_{i_j-1}(x_{\sigma(i_{1}+\cdots+i_{j-1}+1)},\ldots, x_{\sigma(i_{1}+\cdots+i_j)}).
\end{equation}
 \end{enumerate}
 
Also, $\mathcal{H}$ has a decomposition into polynomial-degree of the form
:\begin{equation}
    \mathcal{H}=\sum_{r\geq0}\mathcal H^{(r)}.
\end{equation}
\end{remark}
\begin{example}
Every linear map $H\colon S^\bullet_\mathbb{K} (\E)\rightarrow \E$ of degree $l$ admits an unique extension to a co-derivation $\Bar{H}\colon S^\bullet_\mathbb{K} (\E)\rightarrow S^\bullet_\mathbb{K} (\E) $ of degree $l$. The latter is given explicitly by the formula
\begin{equation}\label{coder:formula}
    \Bar{H}(x_1\cdot\cdots\cdot x_n)=\sum_{i=1}^n\sum_{\sigma\in \mathfrak{S}(i,n-i)} \epsilon(\sigma) H(x_{\sigma(1)}\cdot\cdots \cdot x_{\sigma(i)})\cdot x_{\sigma(i+1)}\cdot\cdots \cdot x_{\sigma(n)}.
\end{equation}
\end{example}

The following lemma-definition is helpful in the sequel.

\begin{lemma}[\textbf{Definition}]{}\label{Tay-C}
We say that a co-algebra morphism  $\Phi \colon S^\bullet_\mathbb{K} (\mathcal E') \rightarrow S^\bullet_\mathbb{K} (\mathcal E)$ is \emph{$\mathcal O $-multilinear} when  the equivalent conditions below are satisfied:
\begin{enumerate}
    \item[(i)] For every $n \geq 0$, the $n$-th Taylor coefficient $ \Phi^{(n)} \colon S^{n+1}_\mathbb K (\mathcal E' ) \longrightarrow \mathcal E$ of  $ \Phi$ is $\mathcal O$-multilinear
    \item[(ii)] There exists an induced co-algebra morphism $\Phi_\mathcal O\colon \bigodot^\bullet\mathcal E' \longrightarrow \bigodot^\bullet\mathcal E $ making the following diagram commutative~:
    $$ \xymatrix{  \ar[d] S^\bullet_\mathbb{K} (\mathcal E') \ar[r]^{\Phi} &   \ar[d]  S^\bullet_\mathbb{K} (\mathcal E) \\ \bigodot^\bullet\mathcal E'  \ar[r]^{\Phi_\mathcal O} & \bigodot^\bullet\mathcal E}  
    $$
\end{enumerate}
When it is the case we  still write $\Phi$ for $\Phi_\mathcal O$.
\end{lemma}

\begin{proof}
The equivalence is straightforward.
\end{proof}
\begin{remark}
Similarly, the formula \eqref{ex:formulacoder} shows that the Taylor coefficients of $\mathcal H $ are $ \mathcal O$-multilinear if and only if $\mathcal H$ induces a $ {\Phi_\mathcal O}$-co-derivation ${\mathcal H}_\mathcal{O}\colon \bigodot^\bullet\E'\longrightarrow \bigodot^\bullet\E$.
\end{remark}

\subsection{Richardson-Nijenhuis bracket and co-derivations}
We need to use the Richardson-Nijenhuis bracket to express our results and explain some proofs in the coming chapters. This bracket was defined mainly to understand Lie algebra structures on a vector space and their deformations \cite{Richardson-cohomology} as well as Poisson structures and their generalizations. Let us recall the definition in our context.\\

Let $\E$ be an $\mathcal{O}$-module. 
For $k\in \mathbb{N}_0$, homomorphisms of  degree $j$ from $ S_\mathbb{K}^{k+1}(\E)$ to $\E$ shall be, by definition, the space
  $$ \text{Hom}_{\mathbb K}^{j} \left( S^{k+1}_\mathbb K  (\E) \,  , \E \right) := \oplus_{m \in \mathbb Z}\text{Hom}_{\mathbb K} \left( S^{k+1}_\mathbb K ( \E )\, _{|_{m-j}} , \E_{m}  \right).$$
 We refer the reader for example to \cite{MR1202431,YKS} for the following definition.
  \begin{definition}\label{def:RN}
The \emph{Richardson-Nihenhuis bracket} is a $\mathbb K$-bilinear map $$[\,\cdot\,,\cdot\,]_\mathrm{RN}\colon \mathrm{Hom}_{\mathbb K}^{i} \left( S^{p+1}_\mathbb K  (\E) \,  , \E \right)\otimes  \mathrm{Hom}^{j}_{\mathbb K} \left( S^{q+1}_\mathbb K  (\E) \,  , \E \right)\rightarrow  \mathrm{Hom}_{\mathbb K}^{i+j} \left( S^{p+q+1}_\mathbb K  (\E) \,  , \E \right)$$ which is defined on homogeneous elements $P\in \text{Hom}_{\mathbb K}^{\bullet} \left( S^{p+1}_\mathbb K  (\E) \,  , \E \right)$  and $R\in \text{Hom}_{\mathbb K}^{\bullet} \left( S^{q+1}_\mathbb K  (\E) \,  , \E \right)$ by \begin{equation}[P,R]_{\hbox{\tiny{RN}}}=P\circ R-(-1)^{|P||R|}R\circ P,\end{equation} with the understanding that $P\circ R$ is defined on $x_0\cdot\cdots x_{p+q}\in S^{p+q+1}_\mathbb K  (\E)$ by
 \begin{equation}( P\circ R) (x_0\cdot\cdots x_{p+q} ):=\sum_{\sigma \in \mathfrak{S}(q+1,p)} \epsilon( \sigma )P( R ( x_{\sigma ( 0)}\cdots x_{\sigma (q)} )\cdot x_{\sigma ( q+1 )}\cdot \cdots x_{\sigma ( p+q)} ). \end{equation}
 The bracket is extended by bilinearity.
  \end{definition}
  
  \begin{remark}
  The bracket $\lb_\mathrm{RN}$ should not be confused with the graded commutator $\lb$ of the graded vector space $\mathrm{Hom}_{\mathbb K}\left( S^\bullet_\mathbb K  (\E),S^{\bullet}_\mathbb K  (\E)\right)$. In fact, the graded commutator of two elements $P,R\in \text{Hom}_{\mathbb K}^{\bullet} \left( S^{\bullet}_\mathbb K  (\E), S^{\bullet}_\mathbb K  (\E) \right) $ is  $[P,R]=P\circ R-(-1)^{|P||R|}R\circ P$, but here $P\circ R$ is the usual composition of homomorphism.
  \end{remark}
 It is easy to check that
 \begin{lemma}
The graded vector space $\mathrm{Hom}_{\mathbb K}^{\bullet} \left( S^{\bullet}_\mathbb K  (\E) \,  , \E \right)$ together with the Richardson-Nihenhuis bracket $[\,\cdot\,,\cdot\,]_\mathrm{RN}$ is  a graded Lie algebra structure.
 \end{lemma}

 For a given $i \in \mathbb Z $, and a given $P = \sum_{k \geq 0} P_{k}$ with $P_{k}\in {\mathrm{Hom}}^i_\mathbb K  \left( S^{k+1}_\mathbb K (\E ), \E \right) $, we denote by $\Bar{P}\in\mathrm{CoDer}(S^{\bullet}_\mathbb K (\E )) $ the unique co-derivation of degree $i$ with Taylor coefficients $(P_{k} )_{k \in\mathbb{N}_0} $. The relation between the Richardson-Nijenhuis bracket an co-derivations is described in the following lemma:

\begin{lemma} \label{lem:RN}
The map $P\mapsto \Bar{P}$ is a graded Lie algebra morphism that is, for every $P,R$ of degrees $l,r$ as above, we have
 $$\widebar{[P,R]}_{\mathrm{\tiny{RN}}}=[\Bar{P},\Bar{R}]= \Bar{P} \circ \Bar{R} - (-1)^{lr}\Bar{R}  \circ \Bar{P}.$$
\end{lemma}

\begin{proof}
Linearity follows from uniqueness of the extension $\Bar{P}$ described in Proposition \ref{prop:co-morphism}. Indeed, for $P, R$ as above, one should have $\mathrm{pr}\circ (\Bar{P}+ \Bar{Q})= P+R=\mathrm{pr}\circ(\widebar{P+R})$. By unicity, we must have $\widebar{P+R}=\Bar{P}+\Bar{R}$. Likewise, one has for any $\lambda\in \mathbb{K}$,\, $\widebar{\lambda P}=\lambda \Bar{P}$.\\

\noindent
To show that it is a graded Lie algebra morphism, it suffices to check that $\mathrm{pr}\circ[\Bar{P},\Bar{R}]=[P,R]_{\mathrm{RN}}$. For $x=x_0\cdot\cdots\,x_{r}\in S^{r+1}_\mathbb{K}(\E)$

\begin{equation}\label{eq:commutator&Richardson}
\left(\mathrm{pr}\circ[\Bar{P},\Bar{R}]\right)_r(x)=P\circ\Bar{R}(x)-(-1)^{|P||R|}R\circ\Bar{P}(x)
\end{equation}
By using  Formula \eqref{coder:formula}, we obtain for example
\begin{align*}
   P\circ\Bar{R}(x) &=\sum_{j=0}^{k}\sum_{\sigma\in \mathfrak{S}(j+1,k-j)}\epsilon(\sigma) P(R_{j}(x_{\sigma(0)}\cdot\cdots \cdot x_{\sigma(j)})\cdot x_{\sigma(j+1)}\cdot\cdots \cdot x_{\sigma(k)})\\&=\sum_{\tiny{\begin{array}{c}
    i+j=k\\ i,j\geq 0
   \end{array}}}\sum_{\sigma\in \mathfrak{S}(j+1,i)}\epsilon(\sigma) P_i(R_{j}(x_{\sigma(0)}\cdot\cdots \cdot x_{\sigma(j)})\cdot x_{\sigma(j+1)}\cdot\cdots \cdot x_{\sigma(r)})\\&=  \sum_{\tiny{\begin{array}{c}
    i+j=r\\ i,j\geq 0
   \end{array}}} (P_i\circ R_j)(x)
\end{align*}
with the understanding that $P_i(x_0\cdot\cdots\, x_{m})=0$ for $m\neq i$. As a consequence, the right-hand side of Equation \eqref{eq:commutator&Richardson} becomes

\begin{align*}
         \sum_{\tiny{\begin{array}{c}
    i+j=r\\ i,j\geq 0
   \end{array}}}\left((P_i\circ R_j)(x)-(-1)^{|P||R|}(R_j\circ P_i)(x)\right)&=\sum_{\tiny{\begin{array}{c}
    i+j=r\\ i,j\geq 0
   \end{array}}}[P_i,R_j]_{\tiny{{\mathrm{RN}}}}(x)\\&=\left([P,R]_{\tiny{{\mathrm{RN}}}}\right)_r(x).
\end{align*}
\end{proof}
\vspace{2cm}
\begin{tcolorbox}[colback=gray!5!white,colframe=gray!80!black,title=Conclusion:]
In this chapter, we recalled the co-algebra language, in particular comorphisms, Richardson-Nijenhuis bracket. An important point is not to confuse the tensor products $\otimes_\mathbb{K}$ and $\otimes_\mathcal{O}$, hence the graded symmetric algebra $S^\bullet_\mathbb{K}(\E)$ and $S^\bullet_\mathcal{O}(\E)$, which is denoted by $S^\bullet(\E)$. This introduction is completed in Appendix \ref{appendix:tensor}.
\end{tcolorbox}

\chapter{Lie $\infty$-algebroids and their morphisms}\label{chap:oid-alg-ver}
 
\section{Definitions}
In this section, we present the notion Lie $\infty$-algebroid over the algebra $\mathcal{O}$. We also give geometrical examples. We refer the reader to Appendix \ref{appendix:mod} for some generalities on graded modules.\\ 

By definition, Lie $\infty$-algebras on a graded vector space $V$ are co-derivations of degree $-1$ squaring to $0$ of the graded symmetric algebra $S(V)$ (e.g see \cite{Stasheff,Ryvkin}). Lie $\infty$-algebroids generalize Lie $\infty$-algebras and Lie
algebroids, but the situation is more sophisticated, because the $2$-ary bracket is not $\mathcal{O}$-linear. Let us make some preparations before introducing the definition.

\vspace{1cm}
\begin{definition}
A \emph{derivation of $\mathcal O$} is a $\mathbb K$-linear map $X\colon\mathcal{O}\rightarrow \mathcal{O}$ that satisfies the so-called Leibniz identity$$X(fg)=X(f)g+ fX(g)$$
for all $f,g\in \mathcal{O}$. The set of all derivations of $\mathcal{O}$ inherits a Lie $\mathbb{K}$-algebra structure from the Lie algebra $\mathrm{End}_{\mathbb K}(\mathcal{O})$, whose Lie bracket is the commutator, that is, $$[X,Y]=X\circ Y-Y\circ X,$$ 
for all $X,Y$ derivations of $\mathcal{O}$.\\

We denote by $\mathrm{Der}(\mathcal O)$ the  Lie algebra of derivations of $\mathcal{O}$. Also, for $X\in\mathrm{Der}(\mathcal O)$, $X[f]$
stands for the derivation $X$ applied to $f\in \mathcal{O}$.
\end{definition}

\begin{example}
Geometrically, derivations of $\mathcal{O}$ can be understood as follows:
\begin{enumerate}
    \item when $\mathcal{O}=C^\infty(M)$ is the algebra of functions of a smooth manifold $M$, as vector fields on a smooth manifold $M$;
    \item when $\mathcal{O}$ is the algebra of holomorphic functions on an open subset $\mathcal{U}$ of $\mathbb{C}^d$, as vector fields on $\mathcal{U}$; 
    \item  when $\mathcal{O}=\mathcal O_W$ is the coordinate ring of  an affine variety $W$, as vector fields on the affine variety~$W$.
\end{enumerate}
\end{example}

\subsection{$dg$-almost Lie algebroids over $\mathcal{O}$}
\begin{definition}
\cite{LavauSylvain, LLS}
\label{def:almost}
A \emph{(negatively graded) almost differential graded Lie algebroid $(\E_\bullet,\ell_1, \ell_2, \rho)$ over $\mathcal{O}$} is a complex 
$$(\E, \ell_1, \rho): \quad\cdots\stackrel{\ell_1}{\longrightarrow}\mathcal E_{-3}\stackrel{\ell_1}{\longrightarrow} \mathcal E_{ -2}  \stackrel{\ell_1}{\longrightarrow}  \mathcal E_{ -1} \stackrel{\rho}{\longrightarrow} \mathrm{Der}(\mathcal{O}).$$
of projective $\mathcal O $-modules equipped with a graded symmetric degree $+1$ $\mathbb{K}$-bilinear \emph{bracket} $$ \ell_2 = [\cdot\,, \cdot]\colon \odot^2\E\rightarrow \E $$ such that:
\begin{enumerate}
\item  $\ell_2$ satisfies the \emph{Leibniz identity} with respect to  $\rho \colon \mathcal E_{-1} \longrightarrow {\mathrm{Der}}(\mathcal O) $, i.e.
\begin{equation}
    \ell_2(x, fy) = f\ell_2(x,y)+\rho(x)[f]y
\end{equation}
for all $x\in \E_{-1},y\in \E$ and $f\in\mathcal{O}$.
\item $\ell_1$ is degree $+1 $-derivation of $\ell_2 $, i.e. for all  $x \in \mathcal E_{-i}, y \in \mathcal E$:
$$ \ell_1 (\ell_2(x,y)) +  \ell_2( \ell_1( x),y) + (-1)^{i}  \ell_2( x  , \ell_1( y)) = 0,$$ 

\item $\rho$ is a morphism, i.e. for all  $x,y \in \mathcal E_{-1}$
$$ \rho (\ell_2 (x,y)) = [\rho (x), \rho (y)].$$
 \end{enumerate}
The $\mathcal{O}$-linear map $\rho$ is called the \emph{anchor map}, and $\ell_1$ the \emph{differential}.
\end{definition}

\begin{remark}\label{rmk:basicLR}
For any almost graded Lie-algebroid $(\E_\bullet,\ell_1, \ell_2, \rho)$ define the \emph{Jacobiator} \begin{align}
    Jac(x,y,z)
    =\ell_2(\ell_2(x,y), z) +(-1)^{|y||z|} \ell_2(\ell_2(x,z), y) + (-1)^{|x|(|y|+|z|)}\ell_2(\ell_2(y,z), x), \quad x,y,z\in \E
\end{align} For any triple of elements $x,y,z\in \E_{-1}$ of degree $-1$. Equivalently: $$Jac(x,y,z)=\ell_2(\ell_2(x,y), z)+\circlearrowleft(x,y,z),$$where $\circlearrowleft(x,y,z)$ means that we sum on the circular permutations of $x,y,z$ with Koszul signs. It is graded symmetric of degree $-2$. By axiom 3, the Jacobiator takes values in the kernel of the anchor map, since
\begin{align*}
\rho(Jac(x,y,z))&=\rho(\ell_2(\ell_2(x,y), z)) + \circlearrowleft(x,y,z)\\&=[\rho(\ell_2(x,y)), \rho(z)]+ \circlearrowleft(x,y,z)\\&=[[\rho(x),\rho(y)],\rho(z)]+ \circlearrowleft(x,y,z)=0,\quad\text{since $\lb$ satisfies Jacobi identity}.
\end{align*}
It is easy to see that the bracket $\ell_2$ goes to the quotient and endows, $\frac{\E_{-1}}{\ker\rho}$ and $\frac{\E_{-1}}{\ell_1(\E_{-2})}$ with a Lie $\mathbb{K}$-algebra bracket $\bar{\ell_2}$.
\end{remark}

\begin{remark}
The map $(x,y,z)\mapsto Jac(x,y,z)$ is $\mathcal{O}$-trilinear: indeed, it is clear if $x,y,z\in \E_{\leq -2}$. By graded symmetry, we need to verify this point when the at least one of the variables is of degree $-1$. Let $f\in \mathcal O$, \\

\noindent
\textbf{Case 1:} $x\in \E_{-1}$ and $y,z\in \mathcal{E}_{\leq -2}$
\begin{align*}
    Jac(x,y,fz)&=\ell_2(\ell_2(x,y), fz) + (-1)^{|z||y|}\ell_2(\ell_2(x,fz), y) + (-1)^{|z|+|y|}\ell_2(\ell_2(y,fz),x)\\&= f\ell_2(\ell_2(x,y), z) + (-1)^{|z||y|}(\ell_2(f\ell_2(x,z), y) +\ell_2(\rho(x)[f]z,y))+ (-1)^{|z|+|y|}f\ell_2(\ell_2(y,z),x)+\\&\quad(-1)^{|z|+|y|}(-1)^{|z|+|y|+1}\rho(x)[f]\ell_2(y,z)\\&=fJac(x,y,z)+\cancel{(-1)^{|z||y|}\rho(x)[f]\ell_2(z,y)} -\cancel{\rho(x)[f]\ell_2(y,z)}
\end{align*}

\noindent
\textbf{Case 2:} $x,y\in \E_{-1}$ and $z\in \mathcal{E}_{\leq -2}$\begin{align*}
    Jac(x,y,fz)&= \ell_2(\ell_2(x,y), fz)+(-1)^{|z|}\ell_2(\ell_2(x,fz),y) + (-1)^{|z| +1}\ell_2(\ell_2(y,fz),x)\\&= f\ell_2(\ell_2(x,y), z) + \rho(\ell_2(x,y))[f]z+(-1)^{|z|}\ell_2(f\ell_2(x,z)+ \rho(x)[f]z,y)) +\\&\hspace{8.5cm}(-1)^{|z|+1} \ell_2(f\ell_2(y,z)+\rho(y)[f]z,x)\\&\\&=fJac(x,y,z)+{\rho(\ell_2(x,y))[f]z}+(-1)^{|z|}((-1)^{|z|}\cancel{\rho(y)[f]\,\ell_2(x,z)}+ \cancel{\rho(x)[f]\ell_2(z,y)}\\&\qquad+(-1)^{|z|}{\rho(y)[\rho(x)[f]]z})+\cancel{(-1)^{|z|+1}((-1)^{|z|}\rho(x)[f]\,\ell_2(y,z))}+ \cancel{\rho(y)[f]\ell_2(z,x)}\\&\qquad +(-1)^{|z|}\rho(x)[\rho(y)[f]]z)\\&\\&=fJac(x,y,z) +\cancel{(\rho(\ell_2(x,y))-[\rho(x), \rho(y)])}[f]z \\&=fJac(x,y,z).
\end{align*}
  
 \noindent
\textbf{Case 3:} for $x,y,z\in \E_{-1}$, the proof is identical to case $2$.

\end{remark}


\begin{example}\label{ex:DGLA1}\label{ex:AGLA}
A \emph{differential graded Lie algebra (DGLA)} (see e.g \cite{ManettiMarco2005}) is a graded vector space ${\displaystyle \mathfrak{g}=\bigoplus_{j\in \mathbb{Z}} \mathfrak{g}_{j}}$ over $\mathbb{K}$ together with a bilinear map called \emph{graded Lie bracket} $\lb\colon \mathfrak{g}_i\times \mathfrak{g}_j\rightarrow \mathfrak{g}_{i+j}$ and a \emph{differential map} ${\displaystyle \dd\colon \mathfrak{g}_{i}\to \mathfrak{g}_{i-1}}$ such that for all homogeneous elements $v_1, v_2, v_3\in \mathfrak{g}$ has the properties\\
\begin{itemize}
    \item\textbf{graded commutativity}: ${\displaystyle [v_1,v_2]=-(-1)^{|v_1||v_2|}[v_2,v_1],}$

\item
\textbf{Jacobi's identity}: ${\displaystyle (-1)^{|v_1||v_3|}[v_1,[v_2,v_3]]+(-1)^{|v_2||v_1|}[v_2,[v_3,v_1]]+(-1)^{|v_3||v_2|}[v_3,[v_1,v_2]]=0,}$

\item
\textbf{Leibniz's identity}: ${\displaystyle \dd[v_1,v_2]=[\dd v_1,v_2]+(-1)^{|v_1|}[v_1,\dd v_2]}$.\end{itemize}

DGLAs are examples of almost differential graded Lie algebroid (ADGLA) with $\mathcal{O}=\mathbb{K}$\, (and $\mathfrak{g}_j=0$ for $j\geq 0$). We take $\E_{-i}:=\mathfrak{g}_{i}[1]$ for $i\geq 1$, $\ell_1(x):=-\dd(x)$, and $\ell_2(x,y):=(-1)^{|x|}[x,y]$. Of course, here $\rho=0$.  Note that ADGLAs are more general than DGLAs, since it does not impose Jacobi identity.

\end{example}
\begin{example}
\cite{HuebschmannJohannes2003HhaM} An \emph{almost-Lie algebroid over a manifold $M$} is a triple $(A\rightarrow M, \lb_A,\rho_A)$ made of a vector bundle $A\rightarrow M$, a skew-symmetric bracket
$\lb_A\colon \Gamma(A)\times \Gamma(A)\rightarrow \Gamma(A)$, fulfilling the Leibniz identity, i.e.  for all $a, b\in \Gamma(A), f\in C^\infty(M )$
$$ [a, f b]_A = f [a,b]_A + \rho_A(a)[f]b,$$ and
a vector bundle morphism $\rho\colon A \rightarrow T M$, 
that satisfies,
$$\rho([a, b]_A
) = [\rho_A(a), \rho_A(b)].$$We say that $\rho_A$ is the \emph{anchor} of $(A\rightarrow M, \lb_A,\rho_A)$. When $\lb_A$ satisfies Jacobi's identity, we speak of a \emph{Lie algebroid over $M$} \cite{Mackenzie}.\\

Almost Lie algebroids over a manifold $M$  give examples of almost differential graded Lie algebroid over $\mathcal{O}=C^\infty(M )$ with $\E_{-1}=\Gamma(A)$, $\E_{-i}=0$ for $i\neq 1$,  $\ell_2=\lb_A$, and the anchor is $\rho_A$.
\end{example}

\subsection{Lie $\infty$-algebroids over $\mathcal{O}$}
Lie $\infty $-algebroids over manifolds were introduced (explicitly or implicitly) by various authors, e.g. \cite{MR2966944}, \cite{Voronov2}, and \cite{SeveraTitle}. We refer to Giuseppe Bonavolont\`a and Norbert Poncin for a complete overview of the matter \cite{Poncin}. It extends the notion of almost differential graded Lie algebroids.
\begin{definition}\label{def:Linfty}
\cite{Stasheff} A \emph{negatively graded Lie $ \infty$-algebroid over $ \mathcal O$} is a collection $\E=(\mathcal E_{-i})_{i \geq 1} $ of projective $ \mathcal O$-modules, equipped with:
\begin{enumerate}
    \item a collection of linear maps $\ell_i:S_\mathbb{K}^i (\E) \longrightarrow \E$ of degree $ +1$ called $i$-ary \emph{brackets}
    \item a $\mathcal O$-linear map $\E_{-1}  \longrightarrow {\mathrm{Der}}(\mathcal O) $ called \emph{anchor map}
\end{enumerate}
satisfying the following axioms~:
\begin{enumerate}\label{def:Jacobi}
    \item[$(i)$] the \emph{higher Jacobi identity}:
    \begin{equation} \sum_{i=1}^{n}\sum_{\sigma\in \mathfrak{S}_{i,n-i}}{{\epsilon(\sigma)}} \, \, \ell_{n-i+1}(\ell_i( x_{\sigma(1)},\ldots,x_{\sigma(i)}),x_{\sigma(i+1)},\ldots,x_{\sigma(n)})=0, \end{equation} for all $n\geq1$ and homogeneous elements $x_1,\ldots,x_n\in\E$,
    \item[$(ii)$] for $ i \neq 2$, the bracket $\ell_i $ is $\mathcal O $-linear, while for $ i=2$,
     $$   \ell_2 ( x,  f y )= \rho (x ) [f] \, y + f \ell_2(x,y)  \hbox{ for all $ x,y \in \mathcal E , f \in \mathcal O$ } ,$$
     where, by convention, $  \rho_{\mathcal E}$ is extended by zero on $ \mathcal E_{-i}$ for all $i \geq 2$,
     \item[$(iii)$]  $\rho\circ\ell_1=0$ on $\E_{-2}$,
     \item[$(iv)$] $\rho$ is a morphism of brackets, i.e., $\rho(\ell_2(x, y)) = [\rho(x), \rho(y)]$ for all $x,y\in\E_{-1}$.
\end{enumerate}
We denote it by $(\E_\bullet, \ell_\bullet, \rho)$. 
\end{definition}
\begin{convention}
From now on, we simply say \textquotedblleft Lie $\infty$-algebroid\textquotedblright\,for \textquotedblleft negatively graded Lie $ \infty$-algebroid\textquotedblright.
\end{convention}
Let us explain these axioms. 
\begin{remark}It follows from Definition \ref{def:Linfty} that: 
\begin{enumerate}
\item The higher Jacobi identity is equivalent to
 $$ \sum_{i=1}^n [\ell_i, \ell_{n+1-i}]_{\hbox{\tiny{RN}}} =0,$$
for all positive integer $n$. See Definition \ref{def:RN} a definition of $\lb_{\hbox{\tiny{RN}}}$.
    \item  For $n=1$, higher Jacobi identity yields $\ell_1\circ \ell_1=0$. Thus, 
  \begin{equation}
      \label{eq:oid-complex}
       \cdots\stackrel{\ell_1}{\longrightarrow} \mathcal E_{ -3}\stackrel{\ell_1}{\longrightarrow} \mathcal E_{-2}  \stackrel{\ell_1}{\longrightarrow}  \mathcal E_{ -1} 
  \end{equation}
  is a complex of projective $\mathcal O$-modules. 
  \item Higher Jacobi identity for $n=2$ reads ${\displaystyle \ell_1(\ell_2(x,y))+\ell_2(\ell_1( x),y)+(-1)^{|x|}\ell_2(x,\ell_1(y))}=0$. 
  \item The third and the fourth axiom are  consequences of item $(i)$, and $(ii)$ if $\mathcal{O}$ has no zero divisors\footnote{i.e. for every $f\in \mathcal{O}$ the linear map $\xymatrix{\E_{-1}\ar[r]^{f}&\E_{-1}}$ given by multiplication by $f$, is injective.} on $\E_{-1}$. Indeed, for all $x\in \E_{-2}$ and $y\in\E_{-1}$ and $f\in \mathcal{O}$.
  Higher Jacobi identity specialized at $n=2$ and Leibniz identity read \begin{align*}
      \ell_1(\ell_2(x,fy))&=\ell_2(\ell_1(x), fy)\\&=f\ell_2(\ell_1(x),y)+ \rho(\ell_1(x))[f]y
  \end{align*}
 Using Leibniz identity again, the left-hand side of the equation above also reads $\ell_1(\ell_2(x,fy))=f\ell_1(\ell_2(x,y))$. Hence: $$\rho(\ell_1(x))[f]y=0.$$ 
 
 If $\mathcal{O}$ has no zero divisors on $\E_{-1}$, then $\rho(\ell_1(x))[f]=0$. Since $x$ and $f$ are arbitrary, we have  ${\rho\circ\ell_1}_{|_{\E_{-2}}}=0$.\\

 Also, by writing higher Jacobi for $n=3$ on elements $x,y,z\in \E_{-1}$ and $f\in\mathcal{O}$ while using Leibniz identity we get
 \begin{align*}
     0&=\ell_1(\ell_3(x,y,fz))+\ell_2(\ell_2(x,y),fz)- \ell_2(\ell_2(x,fz),y)+ \ell_2(\ell_2(y,fz),x)\\&=f\ell_1(\ell_3(x,y,z))+f\ell_2(\ell_2(x,y),z)+\rho(\ell_2(x,y))[f]z-\ell_2(f\ell_2(x,z),y)-\ell_2(\rho(x)[f]z,y)\\&\hspace{8,8cm}+ \ell_2(f\ell_2(y,z),x)+\ell_2(\rho(y)[f]z,x) \\&=f\underbrace{(\ell_1(\ell_3(x,y,fz))+\ell_2(\ell_2(x,y),z)-\ell_2(\ell_2(x,z),y)+ \ell_2(\ell_2(y,z),x))}_{=0}+ \rho(\ell_2(x,y))[f]z\\& \quad+\rho(y)[\rho(x)[f]]z-\rho(x)[\rho(y)[f]]z+ \cancel{\rho(y)[f]\ell_2(x,z)}+ \cancel{\rho(y)[f]\ell_2(z,x)} -\cancel{\rho(x)[f]\ell_2(z,y)}\\&\quad-\cancel{\rho(x)[f]\ell_2(y,z)}.
 \end{align*}
 This implies, $(\rho(\ell_2(x,y))[f]-[\rho(x),\rho(y)][f])z=0$. Therefore, when $\mathcal{O}$ has no zero divisors on $\E_{-1}$, we obtain that $$\rho(\ell_2(x,y))=[\rho(x),\rho(y)].$$
\end{enumerate}
\end{remark}
\begin{definition}
   A Lie $\infty $-algebroid is said to be \emph{acyclic} if the complex \eqref{eq:oid-complex} has no cohomology in degree $\leq -2 $.

\end{definition}

\begin{example}
When $\mathcal{O}=\mathbb{K}$  we have only the axiom $(i)$ of \ref{def:Linfty}, since the other axioms are trivial. Therefore, we recover the axioms that define Lie $\infty$-algebras \cite{Jim-Stasheff,Stasheff}.
\end{example}
\begin{remark}\label{rmk:NGLA}
Geometrically, we are in the case where the projective modules $(\E_{-i})_{i\geq 1}$ of Definition \ref{def:Linfty} are finitely generated. Serre-Swan theorem \cite{SwanRichardG} assures  for every $i\geq 1$, $\mathcal{E}_{-i}=\Gamma(E_{-i})$ for some vector bundle $E_{-i}\rightarrow M$ over a manifold $M$. Hence, Definition \ref{def:Linfty} is rewritten word by word as follows: A \emph{(finitely generated) negatively graded Lie $\infty$-algebroid} $\left(E,(\ell_k)_{k\geq 1}, \rho\right)$  over a manifold $M$ is
\begin{enumerate}
    \item a collection of vector bundles $E =(E_{-i})_{i\geq 1}$ over $M$
    \item together with a sheaf of Lie $\infty$-algebra
structures $(\ell_k)_{k\geq 1}$ over the sheaf of sections of $E$

\item   that comes  with a vector bundle morphism $\rho\colon E_{-1}\longrightarrow TM$, called the \emph{anchor}, 
\item such that the $k$-ary-brackets are all $\mathcal O$-multilinear  except when $k=2$ and at least one of the arguments is of degree $-1$. The $2$-ary bracket  satisfies the Leibniz identity

\begin{equation}
    \ell_2(x, f y) = \rho(x)[f]y + f\ell_2(x, y), x \in \Gamma(E_{-1}), y\in\Gamma(E).
\end{equation}
\end{enumerate}
\end{remark}

\section{Lie $\infty$-algebroids as homological co-derivations}\label{sec:co-version}
When it comes to manipulating morphisms of Lie $\infty$-alegbroids, it quickly becomes quite tedious. Therefore, it is very useful to dualize in order to make the notion of morphisms clearer.

In the finite dimensional case \cite{Voronov2}, 
it is usual to see it as a derivation of the symmetric algebra of the dual, i.e. as a $Q$-manifold (see Chapter \ref{Q-manifold}). The duality
finite rank Lie $\infty$-algebroids  and $Q$-manifolds is especially efficient to deal with morphisms.

In infinite dimension case, we cannot  dualize Lie $\infty$-algebroids in the sense of T. Voronov \cite{Voronov2, Poncin} anymore, since the identification of $S(V)^*$ with $S(V^*)$ does not hold. What I do to deal with this problem in the infinite case, is to stay in the world of squared to zero co-derivations and impose some particular additionnal properties on their Taylor coeficients, related to $\mathcal{O}$-linearity and the Leibniz rule. \\

Now we give an alternative description of Lie $ \infty$-algebroids in terms of co-derivation (extending the usual \cite{Voronov, Voronov2, Poncin} correspondence between Lie $ \infty$-algebroids and $Q$-manifolds in the finite rank case).


\begin{definition}
  A co-derivation $Q\in \mathrm{CoDer}(S_\mathbb{K}(\E))$ of degree $+1$ is said to be an \emph{homological co-derivation} or \emph{co-differential} when $Q^2=0$.\\
  
  Given such a co-derivation $Q$, the triplet $(S_\mathbb{K}(\E), \Delta, Q)$ is then called a \emph{differential graded co-algebra}.

\end{definition}
The following lemma is important.
\begin{lemma}
    A co-derivation $Q\in \mathrm{CoDer}(S_\mathbb{K}(\E))$ of degree $+1$ is an {homological co-derivation} if and only if $\mathrm{pr}\circ Q^2=0$.
\end{lemma}

\begin{proof}
By Proposition \ref{prop:Coderbracket}, the bracket $$[Q,Q]= Q\circ Q- (-1)^{|Q||Q|}Q\circ Q= 2Q^2$$ is a co-derivation, since co-derivations of a graded co-algebra is closed under the graded commutator bracket $\lb$. Therefore, $Q^2=\frac{1}{2}[Q,Q]$ is a co-derivation, and it is completely determined by $\mathrm{pr}\circ Q^2$.
\end{proof}
\begin{remark}
The composition of two co-derivations is not a co-derivation in general. In particular, the composition of a co-derivation with itself is not a co-derivation unless it is of odd degree.
\end{remark}


For a good understanding of the next proposition, see notations of Taylor coefficients in Equation \eqref{eq:Taylor}.

 \begin{proposition}\label{co-diff}
  Given a collection $\E=(\mathcal E_{-i})_{i \geq 1} $ of projective $ \mathcal O$-modules, there is a one-to-one correspondence between Lie $\infty$-algebroid structures $(\E_\bullet,\ell_\bullet,\rho)$ on $\E$ and pairs $(Q_\E, \rho)$ made of an homological co-derivation $Q_\E\colon S^\bullet_\mathbb{K} (\mathcal E) \rightarrow S^\bullet_\mathbb{K} (\mathcal E)$ 
  and a $\mathcal O$-linear morphism, $\rho\colon\E_{-1}\rightarrow \emph{Der}(\mathcal O)$ called the \emph{anchor}, such that 
 \begin{enumerate}
     \item for $k\neq 1$ the $k$-th Taylor coefficient {$Q_\E^{(k)}\colon S_\mathbb{K}^{k+1}(\E')\longrightarrow \E$} of $Q_\E$ is $\mathcal O$-multilinear,
     \item for all $x,y \in \mathcal E$ and $ f \in \mathcal O$, we have, $ Q_\mathcal E^{(1)} (x\cdot  fy) = f   Q_\mathcal E^{(1)} (x\cdot y)  + \rho(x)[f] \, y $,\; 
     \item $\rho \circ Q_\mathcal E^{(0)} =0 $ on $\mathcal E_{-2} $,
     \item {$\rho\circ Q^{(1)}_\E(x\cdot y)=[\rho(x),\rho(y)]$, for all $x,y \in \mathcal E_{-1}$}.
 \end{enumerate}
 The correspondence consists in assigning to a Lie $\infty $-algebroid $(\E_\bullet, \ell_\bullet, \rho)=(\ell_1,\ell_2,\ell_3,\cdots\,)$ the co-derivation $Q_\E$ whose  $k$-th Taylor coefficient  is the $k$-ary bracket $\ell_{k+1} $ for all $k \in \mathbb{N}_0$.
 \end{proposition}
 \begin{proof}
Take the $i$-th Taylor coefficient of co-derivation $Q_\E$ that satisfies the requirements 1. 2. 3. and 4. of Proposition \ref{co-diff} as, $\mathrm{pr}\circ Q_\E^{(i)}=\ell_{i+1}\colon S^{i+1}_\mathbb{K}(\E)\rightarrow \E$. By Lemma \ref{lem:RN} we have \begin{align*}
   2Q_\E^2= [Q_\E,Q_\E]=\sum_{n\geq 1}\sum_{i+j=n+1} \widebar{[\ell_i,\ell_j]}_{\hbox{\tiny{RN}}}.
\end{align*}Thus, by uniqueness of the extension as co-derivation, $Q_\E^2=0$ is equivalent to
 $$0=\sum_{i=1}^n [\ell_i, \ell_{n+1-i}]_{\hbox{\tiny{RN}}}\in   \mathrm{Hom}_{\mathbb K}^{2}\left( S^{n}_\mathbb K  (\E) \,  , \E \right).$$
For all integer $n\geq 1$.
\end{proof}
\begin{remark}\label{rmk:co-diff}
Note that, e.g. if $\mathcal{O}=\mathbb{K}$, we  recover the equivalence between Lie $\infty$-algebras and co-differentials. 
\end{remark}
\begin{convention}
 From now, when relevant, we sometime denote an underlying structure of Lie $\infty$-algebroid $(\E_\bullet,\ell_\bullet,\rho)$ on $\E$ by  $(\E,Q_\E,\rho)$ instead. 
\end{convention}

 \begin{remark}
 Notice that $Q_\E\colon S^\bullet_\mathbb{K}(\E)\longrightarrow S^\bullet_\mathbb{K}(\E)$ does not induce a co-derivation on $\bigodot^\bullet\E$ unless $\rho=0$.
 \end{remark}

Let us make the correspondence given by Proposition \ref{co-diff} explicit on the following examples.

\begin{example}{\it{Differential Graded Lie Algebras.}}\label{ex:dgla}
We come back to Example \ref{ex:DGLA1}. Any differential graded Lie algebra $(\mathfrak{g}=\oplus_{\in \mathbb{Z}}\mathfrak{g}_i,\lb, \dd)$ is Lie $\infty$-algebroid with trivial anchor, where the unary bracket $\ell_1$ and the binary bracket $\ell_2$ are obtained by adding a sign to the differential map $\dd$ and the graded skew-symmetric Lie bracket $\lb$ respectively, and the other brackets $\ell_k$ for $k\geq 3$ vanish. For $k\geq 3$, the $k$-th Taylor coefficient  of the corresponding co-derivation $Q_\mathfrak{g}$ is zero, hence  it can be written as  $Q_\mathfrak{g}=Q_\mathfrak{g}^{(0)}+Q_\mathfrak{g}^{(1)}$. For every homogeneous monomial $x_1\cdots x_k\in S^{k}\mathfrak{g}$

\begin{align*}
    Q_\mathfrak{g}^{(0)}(x_1\cdots x_k)&=\sum_{i=1}^k 
    (-1)^{(|x_1|+\cdots+|x_{i-1}|)|x_i|}\ell_1(x_i)\cdot x_1 \cdots\widehat{x}_i\cdots x_k,\\Q_\mathfrak{g}^{(1)}(x_1\cdots x_k)&=\sum_{1\leq i<j\leq k}(-1)^{(|x_1|+\cdots+|x_{j-1}|)|x_j|+(|x_1|+\cdots+|x_{i-1}|)|x_i|}\ell_2(x_i,x_j)\cdot x_1 \cdots\widetilde{x}_i\cdots \widehat{x}_j\cdots x_k.
\end{align*}
Proposition \ref{co-diff} and Remark \ref{rmk:co-diff} say that  $(\mathfrak{g}=\oplus_{\in \mathbb{Z}}\mathfrak{g}_i,\lb, \dd)$ is a DGLA if and only if
\begin{itemize}
\item $Q_\mathfrak{g}^{2}=0$;
\item []or
    \item  $(Q_\mathfrak{g}^{(0)})^{2}=(Q_\mathfrak{g}^{(1)})^2 = 0$\; and\; $Q_\mathfrak{g}^{(0)}\circ Q_\mathfrak{g}^{(1)} +Q_\mathfrak{g}^{(1)}\circ Q_\mathfrak{g}^{(0)}= 0$.
\end{itemize}

\end{example}

\section{Morphisms of Lie $\infty$-algebroids and homotopies}
This section extends Section 3.4 of \cite{LLS} to the infinite dimensional setting.
\subsection{Morphisms of Lie $\infty$-algebroids}
\begin{definition}	
\label{def:morph}
	A \emph{Lie $\infty$-algebroid morphism (or Lie $\infty$-morphism)} from a Lie $\infty$-algebroid $(\E', Q_{\E'},\rho')$ to a Lie $\infty$-algebroid $(\E, Q_\E,\rho)$, is a co-morphism $\Phi\colon S^\bullet_\mathbb{K} (\mathcal E') \longrightarrow S^\bullet_\mathbb{K} (\mathcal E)$ such that
	\begin{equation}\label{def:LM}
	\Phi\circ Q_{\E'}=Q_{\E}\circ\Phi
	\end{equation}
	and
	\begin{enumerate}
	
	    \item  all Taylor coefficients are $\mathcal O $-multilinear,
	    \item which satisfies  $\rho\circ \Phi_0=\rho'$ on $\E'_{-1}$.
	\end{enumerate}
	Above, $\Phi_0=\Phi^{(0)}_{|_{\E'}}\colon \mathcal E'\to\E$ is the $0$-th Taylor coefficient of $\Phi$. Notice that item 1. is equivalent to saying $\Phi$ is $\mathcal{O}$-multilinear in the sense of Lemma \ref{Tay-C}.

\end{definition}	

\begin{remark} Recall from \cite{Stasheff} that Lie $ \infty$-algebra morphisms from $(S^\bullet_\mathbb{K}(\E), Q_\E)$ 
to $(S^\bullet_\mathbb{K}(\E'), Q_{\E'})$ are defined to be co-algebra morphisms $ \Phi\colon S^\bullet_\mathbb{K} (\mathcal E') \longrightarrow S^\bullet_\mathbb{K} (\mathcal E)$ that satisfy \eqref{def:LM}. Definition \ref{def:morph} adds two additional assumptions to turn a Lie $\infty$-algebr\underline{a} morphism into a Lie $\infty$-algebr\underline{oid} morphism.
\end{remark}

\begin{remark}
For a Lie $\infty$-algebroid morphism $\Phi\colon S^\bullet_\mathbb{K}(\mathcal E')\longrightarrow S^\bullet_\mathbb{K}(\mathcal E)$, the $0$-th Taylor coefficient $\Phi_0\colon(\E',\ell'_1)\longrightarrow(\E,\ell_1)$ of $\Phi$ is the chain map, that is, the following diagram commutes 
\begin{equation}
\xymatrix{\cdots\ar[r]&\E_{-3}'\ar[d]^{\Phi_0}\ar[r]^{\ell_1'}&\E_{-2}'\ar[d]^{\Phi_0} \ar[r]^{\ell_1'}&\E_{-1}'\ar[d]^{\Phi_0}\ar[r]^{\rho'}&\mathrm{Der}(\mathcal{O})\ar[d]^{\mathrm{id}}\\
\cdots\ar[r]&\E_{-3}\ar[r]^{\ell_1}&\E_{-2} \ar[r]^{\ell_1}&\E_{-1}\ar[r]^{\rho}&\mathrm{Der}(\mathcal{O})}
\end{equation}
namely,  $\Phi_0\circ \ell_1'=\ell_1\circ \Phi_0$.\\

Also, by going a bit further, it satisfies for all $x,y \in \E'$
 \begin{equation}
\Phi_{0}\circ\ell'_2(x,y)+\Phi_{1}\circ\ell_1'(x\odot y)=\ell_1\circ \Phi_{1}(x\odot y)+\ell_2(\Phi_{0}(x),\Phi_{0}(y)).
\end{equation}	
\end{remark}

\begin{example}
Let $(A, \lb_A, \rho_A)$ and $(B,\lb_B, \rho_B)$ be two Lie algebroids over a manifold $M$. A Lie algebroid morphism $\phi\colon A\longrightarrow B$ over the identity \cite{Mackenzie} induces a Lie $\infty$-algebroid morphism $S^\bullet(\phi)\colon S^\bullet(\Gamma(A))\rightarrow S^\bullet(\Gamma(B))$ in the section level: where corresponding co-algebra morphism $S^\bullet(\phi)$ is defined as $$a_1\cdots\cdot a_n \mapsto \phi(a_1)\cdots\cdot\phi(a_n),$$ for all $a_1,\ldots, a_n \in A$. 
\end{example}

\begin{example}{\it{Lie $\infty$-morphisms of Differential Graded Lie Algebras.}}\label{ex:dgla-morph}
Let us consider $(\mathfrak{g},\lb_\mathfrak{g}, \dd_\mathfrak{g})$ and $(\mathfrak{h},\lb_\mathfrak h, \dd_\mathfrak h)$ two differential graded Lie algebras (see Example \ref{ex:AGLA}), where $Q_\mathfrak{g}$ and $Q_\mathfrak{h}$ denote their respective associated co-derivations. Let $\Phi\colon (S^\bullet(\mathfrak{g}[1]),Q_\mathfrak{g})\longrightarrow (S^\bullet(\mathfrak{h}[1]),Q_\mathfrak{h})$ be a Lie $\infty$-morphism, i.e. one has \begin{equation}\label{eq:dgla-morph}
    \Phi\circ Q_\mathfrak g = Q_\mathfrak{h}\circ \Phi.
\end{equation}In particular, $\Phi_0 \circ \dd_\mathfrak g = \dd_\mathfrak h\circ \Phi_0$. Let us write down what the restriction of Equation \eqref{eq:dgla-morph} to low Taylor coefficients: Let $x,y\in \mathfrak{g}$ be two homogeneous elements. One has, \begin{align*}
    \Phi(x\cdot y)&=\Phi_1(x\cdot y) +\Phi_0(x)\cdot \Phi_0(y),\\ \Phi(x)&=\Phi_0(x).
\end{align*}A direct computation of the LHS of Equation \eqref{eq:dgla-morph} applied to $x\cdot y$ gives,

\begin{align*}
  \Phi\circ Q_\mathfrak{g}(x\cdot y)&= \Phi\left(-\dd_\mathfrak g(x)\cdot y)-(-1)^{(|x|-1)(|y|-1)}\dd_\mathfrak g(y)\cdot x +(-1)^{|x|}[x,y]_\mathfrak g\right)\\&=-\Phi_1(\dd_\mathfrak g(x)\cdot y) -\Phi_0(\dd_\mathfrak g(x))\cdot \Phi_0(y)-(-1)^{(|x|-1)(|y|-1)}\left(\Phi_1(\dd_\mathfrak g(y)\cdot x) -\Phi_0(\dd_\mathfrak g(y))\cdot \Phi_0(x)\right) \\&\quad+ (-1)^{|x|}\Phi_0([x,y]_\mathfrak g)
\end{align*}Also, the RHS of Equation \eqref{eq:dgla-morph} applied to $x\cdot y$ gives,
\begin{align*}
    Q_\mathfrak{h}\circ\Phi(x\cdot y)&= Q_\mathfrak h(\Phi_1(x\cdot y) +\Phi_0(x)\cdot \Phi_0(y))\\&=-\dd_\mathfrak h(\Phi_1(x))\cdot y-\dd_\mathfrak h(\Phi_0(x))\cdot\Phi_0(y) -(-1)^{(|x|-1)(|y|-1)}\dd_\mathfrak h(\Phi_0(y))\cdot\Phi_0(x)\\&\quad+(-1)^{|x|}[\Phi_0(x), \Phi_0(y)]_\mathfrak{h}
\end{align*}Since both sides are equal, and $\Phi_0 \circ \dd_\mathfrak g=\dd_\mathfrak h\circ \Phi_0$ we obtain,
\begin{align*}
    \Phi_0([x,y]_\mathfrak{g})-[\Phi_0(x),\Phi_0(y)]_\mathfrak{h}=(-1)^{(\lvert x\rvert -1)\lvert y\rvert)}\Phi_1(\dd_\mathfrak g(y)\cdot x)-(-1)^{\lvert x\rvert}\left(\Phi_1(\dd_\mathfrak{g}(x)\cdot y)-\dd_\mathfrak{h}(\Phi_1(x\cdot y))\right).
\end{align*}
\end{example}
\subsection{Homotopies}\label{sec:Homtopies1}
In this section, we define homotopy between Lie $\infty$-morphisms, and between Lie $\infty$-algebroids. This extends \cite{LLS} from finite dimensional $Q$-manifolds to arbitrary Lie $\infty $-algebroids.

\subsubsection{A practical definition}
Let us consider  $(\Omega^{\bullet}([a,b]),\wedge, \dd_{dR})$ the differential graded algebra made of the forms on $[a,b]$ together with the wedge product and  the Rham differential. If we denote by $t$ the coordinate on $[a,b]$, we have $\Omega^{\bullet}([a,b])= C^\infty([a,b])\oplus C^\infty([a,b])\,dt$. We equip $\Omega^{\bullet}([a,b])$ with a co-associative co-algebra structure $\Delta\colon \Omega^{\bullet}([a,b])\rightarrow \Omega^{\bullet}([a,b])\otimes_{\Omega^{\bullet}([a,b])} \Omega^{\bullet}([a,b]) $, given by $\Delta_\Omega(1)= 1\otimes 1$, where we extend by $\Omega^{\bullet}([a,b])$-linearity. \\

Let $(\E',Q_{\E'},\rho')$ and $(\E,Q_{\E},\rho)$ be Lie $\infty$-algebroids over $\mathcal{O}$. 

\begin{lemma}\label{lemma:dg-colagebra}
The triplet $$\left( S_\mathbb{K}^\bullet(\E)\otimes_{\mathbb{K}}\Omega^{\bullet}([a,b]), \bar{\Delta}=(\mathrm{id}\otimes \tau \otimes \mathrm{id})\circ \Delta\boxtimes \Delta_\Omega, \, \partial= Q_{\E'}\otimes \mathrm{id}+\mathrm{id}\otimes \dd_{{dR}}\right)$$ is a differential graded co-algebra.
\end{lemma}
\begin{proof}One should understand that for $\alpha\in C^{\infty}([a,b])$ and  for an homogeneous element  $v\in S_\mathbb{K}^\bullet(\E)$, 
\begin{align*}
    (Q_{\E}\otimes\text{id}+\text{id}\otimes \dd_{{dR}})(v\otimes \alpha)=Q_{\E}(v)\otimes\alpha+(-1)^{|v|}v\otimes \alpha'dt
\end{align*}also,

\begin{align*}
    (Q_{\E}\otimes\text{id}+\text{id}\otimes \dd_{{dR}})(v\otimes \alpha dt)&=Q_{\E}(v)\otimes\alpha dt + v\otimes \underbrace{\dd_{dR}(\alpha dt)}_{=0}\\&=Q_{\E}(v)\otimes\alpha dt.
\end{align*} 
Clearly, $\partial^2=0$, see the definition of tensor product of complexes, Appendix \ref{appendix:mod}. Let us check that $\partial$ is indeed a co-derivation w.r.t the co-product $\bar{\Delta}$. Take $\alpha\in C^\infty([a,b])$ and a homogeneous element  $v\in S_\mathbb{K}^\bullet(\E)$, we will use the Sweedler notation, $\Delta(v)=v_{(1)}\otimes v_{(2)}$, to avoid a long useless text. On one hand,

\begin{align*}
   \bar{\Delta}\circ \partial(v\otimes\alpha)&=(\mathrm{id}\otimes \tau \otimes \mathrm{id})\circ \Delta\boxtimes \Delta_\Omega(Q_{\E}(v)\otimes\alpha+(-1)^{|v|}v\otimes \alpha'dt)\\&=\mathrm{id}\otimes \tau \otimes \mathrm{id}\left(\Delta\circ Q_\E(v)\boxtimes \alpha\otimes 1 +(-1)^{|v|}\Delta(v)\boxtimes \alpha'dt\otimes 1\right)\\&=(\mathrm{id}\otimes \tau \otimes \mathrm{id})\circ \left( (Q_\E\otimes\mathrm{id}+\mathrm{id}\otimes Q_\E)\circ \Delta(v)\boxtimes \alpha \otimes 1+ (-1)^{|v|}\Delta(v)\boxtimes \alpha'dt\otimes 1\right)\\&= Q_\E(v_{(1)})\otimes \alpha\boxtimes v_{(2)}\otimes 1 +(-1)^{|v_{(1)}|} v_{(1)}\otimes\alpha\boxtimes Q_\E(v_{(2)})\otimes 1+\\&\hspace{7.8cm} (-1)^{|v|+|v_{(2)}|}v_{(1)}\otimes (\alpha'dt)\boxtimes v_{(2)}\otimes 1.
\end{align*}
 On the other hand, 
 
\begin{align*}
    (\partial\boxtimes \mathrm{id} + \mathrm{id}\boxtimes \partial)\circ \bar{\Delta}(v\otimes \alpha)&=(\partial\boxtimes \mathrm{id} + \mathrm{id}\boxtimes \partial)\circ (\mathrm{id}\otimes \tau \otimes \mathrm{id})(\Delta(v)\boxtimes \alpha\otimes 1)\\&= (\partial\boxtimes \mathrm{id}+ \mathrm{id}\boxtimes \partial)(v_{(1)}\otimes\alpha\boxtimes v_{(2)}\otimes 1)\\&=\partial(v_{(1)}\otimes\alpha)\boxtimes v_{(1)}\otimes 1+ (-1)^{|v_{(1)}|}v_{(1)}\otimes\alpha\boxtimes\partial(v_{(2)}\otimes 1)\\&= \left(Q_\E(v_{(1)}\otimes \alpha +(-1)^{|v_{(1)}|}v_{(1)}\otimes\alpha'dt)\right)\boxtimes v_{(2)}\otimes 1+\\&\hspace{6cm}(-1)^{|v_{(1)}|}v_{(1)}\otimes \alpha\boxtimes Q_\E(v_{(2)})\otimes 1.
\end{align*}
Both sides are obviously equal. Likewise, a straightforward computation shows that $\bar{\Delta}\circ \partial(v\otimes \alpha dt)=(\partial\boxtimes \mathrm{id} + \mathrm{id}\boxtimes \partial)\circ \bar{\Delta}(v\otimes \alpha dt)$. Hence

 $$\bar{\Delta}\circ \partial=(\partial\boxtimes \mathrm{id} + \mathrm{id}\boxtimes \partial)\circ \bar{\Delta}$$ on $S_\mathbb{K}^\bullet(\E)\otimes_{\mathbb{K}}\Omega^{\bullet}([a,b])$.
\end{proof}

\begin{remark}
Elements of degree $k$ of $S_\mathbb{K}^\bullet(\E)\otimes_{\mathbb{K}}\Omega^{\bullet}([a,b])$ are $\mathbb{K}$-linear combinations of elements of the form $v\otimes \alpha$ and $  w\otimes \beta dt$, with $\alpha,\beta\in C^{\infty}([a,b])$ and $v\in S_\mathbb{K}^\bullet(\E)_{|k} $ and $w\in S_\mathbb{K}^\bullet(\E)_{|k-1} $. 
The latter can be seen as maps,  $t\in [a,b]\mapsto \alpha(t)v\otimes 1$ and $t\in [a,b]\mapsto \beta(t)w\otimes dt$. Hence, elements of degree $k$ of this complex can be considered as element of $S_\mathbb K(\E)\otimes_\mathbb{K} \Omega^\bullet([a,b])$  that depend on $t$, i.e. as  $t\in [a,b]\mapsto v_t\otimes 1+ w_t\otimes dt$, with  $v_t\in S_\mathbb{K}^\bullet(\E')_{|k} $ and $w_t\in S_\mathbb{K}^\bullet(\E')_{|k-1}$.
\end{remark}

The following is a temporary definition. 
It  will be generalized later to another more practical for gluing homotopies.

\begin{definition}\label{def:general-idea-homotopy}
   A \emph{homotopy} that joins two Lie $\infty$-algebroid morphisms $\Phi,\Psi\colon S_\mathbb{K}^\bullet(\E')\rightarrow S_\mathbb{K}^\bullet(\E)$ is the datum made of an interval $[a,b]\subset \mathbb{R}$ and a chain map\begin{align*}
    (S_\mathbb{K}^\bullet(\E'), Q_{\E'})&\stackrel{\mathcal{H}}{\longrightarrow}( S_\mathbb{K}^\bullet(\E)\otimes_\mathbb{K}\Omega^{\bullet}([a,b]), Q_{\E}\otimes \text{id}+\text{id}\otimes \dd_{{dR}})\\v&\longmapsto \left(t\in[a,b]\mapsto J_t(v)\otimes 1-(-1)^{|v|}H_t(v)\otimes d t\right)
\end{align*}
    \begin{enumerate}
        \item which is a co-algebra morphism, 
        
        \item and that coincides with $\Phi$ and $\Psi$ at $t=a$ and $b$ respectively, i.e. for all $v\in S_\mathbb{K}^\bullet(\E')$, one has $\mathcal{H}(v)(a)=\Phi(v)\otimes 1$ and $\mathcal{H}(v)(b)=\Psi(v)\otimes 1$.
    \end{enumerate} 
\end{definition}

\begin{remark}In the definition above, the map $\mathcal{H}$ induces for every $t\in [a,b]$ two different $\mathcal{O}$-multilinear maps $$\begin{cases} J_t\colon S_\mathbb{K}^\bullet(\E')\longrightarrow  S_\mathbb{K}^\bullet(\E)\\H_t\colon S_\mathbb{K}^\bullet(\E')\longrightarrow  S_\mathbb{K}^\bullet(\E).
     \end{cases}.$$Since $\mathcal{H}$ is of degree $0$, one of the maps is of degree zero and the other one of degree $-1$ respectively. By using respectively the property of co-algebra morphisms and chain map  property,  we obtain the following for every $t\in[a,b]$: 
\begin{enumerate}

    \item Let $\Bar{\Delta}$ be the co-product on $S_\mathbb{K}^\bullet(\E)\otimes_{\mathbb K}\Omega^{\bullet}([a,b])$ like in Lemma \ref{lemma:dg-colagebra}. We have, $$\Bar{\Delta}\circ \mathcal{H}(v)=\mathcal{H}\boxtimes \mathcal{H}\circ\Delta'(v).$$ Let $v=v_1\cdots v_n\in S^\bullet_\mathbb{K}(\E')$, a direct computation of  gives us\begin{align*}
       \Delta\boxtimes \Delta_\Omega\circ\mathcal{H}(v)(t)&=\Delta\circ J_t(v)\boxtimes 1\otimes 1 -(-1)^{|v|}\Delta\circ H_t(v)\boxtimes dt\otimes 1.
    \end{align*}
 Let us use the Sweedler notation just like in Lemma \ref{lemma:dg-colagebra} to compute  $\mathcal{H}\boxtimes \mathcal{H}\circ\Delta'(v)(t)$. We have,
 \begin{align*}
     \mathcal{H}\boxtimes \mathcal{H}\circ\Delta'(v)(t)&=\mathcal{H}\boxtimes \mathcal{H}(v_{(1)}\otimes v_{(2)})(t)\\&=\mathcal{H}(t)(v_{(1)})\boxtimes \mathcal{H}(v_{(1)})(t)\\&=\left(J_t(v_{(1)})\otimes 1-(-1)^{|v_{(1)}|}H_t(v_{(1)})\otimes dt\right)\boxtimes\left(J_t(v_{(2)})\otimes 1-(-1)^{|v_{(2)}|}H_t(v_{(2)})\otimes dt\right)
 \end{align*}
 

Hence,  \begin{align*}
    (\mathrm{id}\otimes \tau \otimes \mathrm{id})^{-1}\circ(\mathcal{H}\otimes& \mathcal{H})\circ\Delta'(v)=\\&(J_t\otimes J_t)\circ \Delta'(v)\boxtimes 1\otimes 1 + \cancel{H_t\otimes H_t\circ\Delta'(v)\boxtimes dt\otimes dt}\\ &+ (-1)^{|v|}\left(H_t\otimes J_t\circ\Delta'(v) +J_t\otimes H_t\circ \Delta'(v)\right)\boxtimes {dt\times 1}
    \end{align*}
    By equating both sides, we obtain equations that say that $J_t$ is  a co-morphism and $H_t$ a $J_t$-co-derivation.
    
    \item and for any homogeneous element  $v\in S^\bullet_\mathbb{K}(\E')$ we have: in one side \begin{align*}
        \mathcal{H}\circ Q_{\E'}(v)(t)=J_t\circ Q_{\E'}(v)\otimes 1-(-1)^{|v|+1}H_t(v)\circ Q_{\E'}(v)\otimes dt.   
            \end{align*}

     to the other side\footnote{It should be understood that e.g,  $(Q_{\E}\otimes\text{id}+\text{id}\otimes \dd_{{dR}})(\alpha(t)v\otimes 1)=\left(Q_{\E}(
     v)\otimes \alpha+v\otimes \dd_{{dR}}\alpha\right)(t)$.},   \begin{align*}
         (Q_{\E}\otimes\text{id}+\text{id}\otimes \dd_{{dR}})\circ\mathcal{H}(v)(t)&=(Q_{\E}\otimes \text{id}+\text{id}\otimes \dd_{{dR}})\circ\left(J_t(v)\otimes 1-(-1)^{|v|}H_t(v)\otimes d t\right)\\&=Q_{\E}\circ J_t(v)\otimes 1+ (-1)^{|v|}\frac{d J_t}{dt}(v)\otimes dt - (-1)^{|v|}Q_\E\circ H_t(v)\otimes dt
     \end{align*}
By equating both sides we see that $J_t$ and $H_t$ satisfy the following  condition
\begin{equation}\label{eq:hom-reformulation}
    \begin{cases}J_t\circ Q_{\E'}(v)=Q_{\E}\circ J_t(v)\\\frac{\dd J_t}{\dd t}(v)=Q_\E\circ H_t(v)+H_t\circ Q_{\E'}(v).
\end{cases}
\end{equation}
\end{enumerate}
\end{remark}
In addition, we have for $v\in \E_{-1}'$ that \begin{equation*}\label{eq-diffRecursion}
		\frac{\dd J_t^{(0)}}{\dd t}(v)=Q_{\E}^{(0)}\circ H_t^{(0)}(v)+\underbrace{H_t^{(0)}\circ Q_{\E'}^{(0)}}_{=0}
		(v)\end{equation*}and 
		
		\begin{align} \nonumber J_t^{(0)}(v)&=\Phi^{(0)}(v)+\int_a^t\ell_1\circ H_s^{(0)}(v)\dd s\\ \label{sol}&=\Phi^{(0)}(v)+\ell_1\circ\int_a^t
		H_s^{(0)}(v)\dd s\end{align}
This implies that $\rho\circ J_t^{(0)}(v)=\rho\circ \Phi^{(0)}(v)=\rho'(v)$, since $\rho\circ \ell_1=0$.\\
		
Gluing homotopies as defined in Definition \ref{def:general-idea-homotopy} is however not so easy. We can reformulate the definition of homotopies between Lie $\infty$-morphisms in the following manner. Before we do this, we would like to fix some vocabulary.\\

Let us recall some facts on vector-valued functions. Let $V$ be a vector space. Unless a topology on $V$ is chosen, the notion of $V$-valued continuous or differentiable or smooth function, and the concept of a limit on an interval  $ I =[a,b] \subset \mathbb R$ do not make any sense. However, we can always define the following  notion

\begin{definition}\label{def:piecewise}
   A vector-valued map $\gamma \colon I \longrightarrow V$ is said to be a   \emph{piecewise rational path on $I$} if there exists a finite increasing sequence $a=t_0 \leq \dots \leq t_N=b   $ of \emph{gluing points}, such that for all $i=0, \dots, N-1$  the restriction $\gamma^i$ of $\gamma$ to $[t_i,t_{i+1}] $ is of the form  $$\gamma^i(t)=\sum_{j=1}^n\beta^i_j(t) v_j,\quad \text{for some}\; n\in\mathbb{N}$$ where for every $j=1,\ldots,n$, $v_j \in V$ and $\beta^i_j\colon I\rightarrow \mathbb{R}$ a real rational function on $[t_i,t_{i+1}]$ that has no pole on $[t_i,t_{i+1}]$.

   \begin{enumerate}
       \item We say that $\gamma$ is \emph{continuous} on $I$, if for all $i=1,\ldots, N-1$, $\gamma_i$  and $ \gamma_{i+1}$ coincide at the gluing point $ t_{i+1}$.
       
       \item When $V$ is a space of linear maps between the vector spaces $S$ and $T$,  we shall say that a $V$-valued map $\Xi\colon I\rightarrow V$ is a \emph{piecewise rational (continuous)} if the map  $(t\in I\mapsto \Xi_t(s))$ is a piecewise rational (continuous) $T$-valued path for all  $s \in S$.
       
       \item We define the limit of $\gamma$ at $u\in [t_i,t_{i+1}]$ to be equal to $$\sum_{j=1}^n\lim_{t\to u}\beta^i_j(t) v_j$$
       
       when, for every $j=1,\ldots, n$, $\beta^i_j$ admits a limit a $u$. One can assume that $\{v_1, v_2,v_3\ldots\,\}$ is a basis of $V$.
   \end{enumerate}
   \end{definition}

 Let us recall the following.
\begin{remark}\label{rmk:Riemann-integral}
It is a very classical fact that the integral of a piecewise-$C^1$ function $\beta \colon I\longrightarrow \mathbb{R}$ on a compact interval $I=[a,b]\subset \mathbb{R}$ which is subordinate to a subdivision $a=t_0<\cdots<t_N=b$  admits  primitives which  are piecewise-$C^1$ on the same subdivision  $a=t_0<\cdots<t_N=b$. An important fact is that continuous piecewise-$C^1$ functions $\beta \colon I\longrightarrow \mathbb{R}$, i.e. piecewise-$C^1$ functions which are also continuous (even at the junction points) admit piecewise continuous derivatives $\beta'(t)$,  and  $\beta(b)-\beta(a)=\int_a^b\beta'(t)dt$.
\end{remark}
   
Definition \ref{def:piecewise} has this important consequence.
\begin{lemma} \label{lem:primitivesDerivatives}
A piecewise rational continuous function $\gamma\colon I\rightarrow V$ is differentiable at every point which is not a gluing point, and the latter is again  piecewise rational on $I$. 

Conversely, every piecewise rational functions admit a piecewise rational continuous primitive, unique up to a constant. 
\end{lemma}

\begin{proof}
The derivative of $\gamma$ can be defined in the usual way using the item 3 of Definition \ref{def:piecewise}. Likewise, for their primitives. The proof is then immediate.
\end{proof}

We can now give the following definition.
\begin{definition}
   A family $ ( J_t)_{t \in I}$ of co-algebra morphisms are said to be \emph{piecewise rational continuous } if its Taylor coefficients $  J_t^{(n)}$ are piecewise rational continuous for all $n\in\mathbb{N}$.

   For such a family $( J_t)_{t\in I}$, a family $(H_t)_{t\in I} $ made of $  J_t$-co-derivations is said to be \emph{piecewise rational} if all its Taylor coefficients are.
\end{definition}

 We are now ready to define formulate the definition of homotopies as follows and extend Definition 3.53 in \cite{LLS} to the infinite rank case.

\begin{definition}
\label{def:homotopy}
	Let $\Phi$ and $\Psi$ be Lie $\infty$-algebroid morphisms from $(\E',Q_{\E'},\rho' )$ to $(\E, Q_{\E},\rho)$. A \emph{homotopy between} or that \emph{joins} $\Phi$ and $\Psi$ is a pair $( J_t , H_t)_{t\in[a,b]}$ consisting of:
	\begin{enumerate}
		\item  a piecewise rational continuous  path $t\mapsto J_t$ valued in Lie $\infty$-algebroid morphisms between $S^\bullet_\mathbb{K}(\E')$ and $S^\bullet_\mathbb{K}(\E)$ satisfying the boundary condition:
		$$\Phi_a = \Phi\quad
		\text{and}\quad
		\Phi_b = \Psi,$$
		\item a piecewise rational  path $t\mapsto H_t$, with $ H_t$ a $J_t$-co-derivations of degree $-1$ from $S^\bullet_\mathbb{K}(\E')$ to $S^\bullet_\mathbb{K}(\E)$, such that the
		following equation:
		\begin{equation}\label{eq-diff}
		\frac{\dd J_t}{\dd t}=Q_{\E}\circ H_t+H_t\circ Q_{\E'}
		\end{equation}
		holds for every $t \in ]a\,, b[$ where it is defined (that is, not a gluing point for the Taylor coefficients). More precisely, for every  $v \in S^{\leq n}_\mathbb{K} (\mathcal E')$, \begin{equation}\label{eq-diffv}
		\frac{\dd J_t}{\dd t}(v)=Q_{\E}\circ H_t(v)+H_t\circ Q_{\E'}(v)
		\end{equation}for all $t$ which is not a gluing point of the Taylor coefficient of $ J_t^{(k)},H_t^{(k)}$ for $k=0,\ldots,n$.
	\end{enumerate}
	
	When these conditions are satisfied, we  say that $\Phi$ and $\Psi$ are \emph{homotopy equivalent}, and we write  $\Phi\sim \Psi$.
\end{definition}
\begin{remark}
This clearly extends \eqref{eq:hom-reformulation}.
\end{remark}

\begin{remark}
In the above definitions, it is not required that the gluing points of the various Taylor coefficients $  J_t^{(n)}$ or $H^{(n)}_t $ to be the same for all $n\in\mathbb{N}_0$. 
\end{remark}

The following Proposition shows that the notion of homotopy given in Definition \ref{def:homotopy} implies the usual notion of homotopy between chain maps (see Appendix \ref{appendix:mod}).
\begin{proposition}\label{prop:homotpy-usual}
Let $\Phi$ and $\Psi$ be Lie $\infty$-algebroid morphisms from $(\E',Q_{\E'},\rho' )$ to $(\E, Q_{\E},\rho)$ which are homotopic. Then, \begin{equation}\label{eq:comp-homtopy}
    \Psi-\Phi=Q_{\E}\circ H+H\circ Q_{\E'}
\end{equation}for some $\mathcal{O}$-linear map $H\colon S_\mathbb{K}^\bullet(\E')\longrightarrow S^\bullet_\mathbb{K}(\E)$ of degree $-1$.
\end{proposition}

\begin{proof}
Take $( J_t , H_t)_{t\in[a,b]}$ a {homotopy between} $\Phi$ and $\Psi$ as in Definition \ref{def:homotopy}. By definition, $t\mapsto J_t$ is piecewise rational continuous, therefore it is continuous on $[a,b]$ (even at the junctions points), we can use Remark  \ref{rmk:Riemann-integral} and write
\begin{align*}
    \Phi_b-\Phi_a&=\int_a^b	\frac{\dd J_t}{\dd t}\,dt\\\Psi-\Phi&=\int_a^b\left( Q_{\E}\circ H_t+H_t\circ Q_{\E'}\right) dt\\&=Q_{\E}\circ \left(\int_a^b H_t\,dt\right) + \left(\int_a^b H_t\,dt\right)\circ Q_{\E'}.
\end{align*}
Thus, the $\mathcal{O}$-multilinear map $H:=\int_a^b H_t\,dt$  satisfies Equation \eqref{eq:comp-homtopy}.\end{proof}

\begin{definition}
   We say that  two Lie $\infty$-algebroids $(\E',Q_{\E'},\rho' )$ to $(\E, Q_{\E},\rho)$ over $\mathcal O$ are \emph{homotopy equivalent} or \emph{homotopic} if there exists two Lie $\infty$-algebroid morphisms
  \begin{equation*}
    \xymatrix{(\E',Q_{\E'},\rho' )\ar@<2pt>[r]^\Phi&(\E,Q_{\E},\rho )\ar@<2pt>[l]^\Psi}
\end{equation*} 
   such that  $\Phi\circ\Psi\colon (\E,Q_{\E},\rho )\rightarrow (\E, Q_{\E},\rho)$ and $\Psi\circ\Phi\colon (\E',Q_{\E'},\rho' )\rightarrow (\E', Q_{\E'},\rho')$ are homotopy equivalent to the identity map of respective space.\end{definition}

The following proposition justifies Definition \ref{def:homotopy}.

\begin{prop}
\label{prop:justify}
Let $\Phi $ be a Lie $\infty $-algebroid morphism from $(\E',Q_{\E'},\rho' )$ to $(\E,Q_{\E},\rho)$. For all $t\in [a\,,b]$, let $ H_t^{(n)}\colon S^{n+1}_\mathbb K (\mathcal E') \rightarrow ~\mathcal E$ be a family $\mathcal O $-multilinear piecewise rational maps indexed by $n\in\mathbb{N}_0$. Then,
\begin{enumerate}
    \item There exists a unique piecewise rational continuous family of co-algebra morphisms $ J_t $ such that
      \begin{enumerate}
          \item $J_a = \Phi $ 
          \item $( J_t,H_t )$ is a solution of the differential equation \eqref{eq-diff}, where $  H_t$ is the $ J_t $-co-derivation whose $n$-th Taylor coefficient is $H_t^{(n)} $ for all $ n \geq 0$.  
      \end{enumerate}
    \item Moreover, for all $ t \in [a,b] $, $(J_s, H_s)_{s\in[a,t]}$ is a homotopy that joins  $\Phi $ and $ J_t$.
 \end{enumerate}
\end{prop}
\begin{proof}
Let us show item 1. We claim that equation \eqref{eq-diff} is a differential equation that can be solved recursively. In polynomial-degree zero, it reads,\begin{equation}\label{eq-diffRecursion2}
		\frac{\dd J_t^{(0)}}{\dd t}=Q_{\E}^{(0)}\circ H_t^{(0)}+H_t^{(0)}\circ Q_{\E'}^{(0)}
		\end{equation}and \begin{equation}\label{sol2} J_t^{(0)}=\Phi^{(0)}+\int_a^t\left(Q_{\E}^{(0)}\circ H_s^{(0)}+H_s^{(0)}\circ Q_{\E'}^{(0)}\right)\dd s\end{equation} is defined for all $t\in[a,b]$. Also, $\frac{\dd}{\dd t} J_t^{(n+1)}\colon S^{n+2}(\E')\rightarrow\E$ is an algebraic expression of $Q_\E^{(0)},\ldots,Q_\E^{(n+1)}$, $Q_{\E'}^{(0)},\ldots, Q_{\E'}^{(n+1)}$ $ J_t^{(0)},\ldots, J_t^{(n)},H_t^{(0)},\ldots,H_t^{(n+1)}$. But $ J_t^{(n+1)}$ does not appear in the $(n+1)$-th Taylor coefficient of $Q_{\E}\circ H_t+H_t\circ Q_{\E'}$ by Equation \eqref{ex:formulacoder}. 
By Lemma \ref{lem:primitivesDerivatives}, there exists a unique piecewise rational continuous solution $ J_t^{(n+1)} $ such that $J_a^{(n+1)} = \Phi^{(n+1)}$. The construction of the Taylor coefficients of the co-algebra morphisms $  J_t$ then goes by recursion. Recursion formulas also show that $  J_t$ is unique.\\

\noindent
Let us show item 2. i.e. that $ J_t $ is a $\mathcal O $-multilinear chain map for all $ t \in [a,b]$: For the same reason as in Equation \eqref{sol}, Equation \eqref{sol2} implies that $\rho\circ{ J_t^{(0)}}_{|_{\E'}}=\rho'$. The function given by $$\Lambda_k(t)=(Q_{\E}\circ  J_t- J_t\circ Q_{\E'})^{(k)}\quad\text{for all}\quad t\in[a,b],\, k\in\mathbb{N}_0$$ are differentiable at all points $t$ except for a finitely many  $t \in [a,b]$ and  are piecewise rational continuous. The map $\frac{\dd J_t}{\dd t}$ is a Lie $\infty$-morphism because $ Q_\E^2=0$ and $ Q_{\E'}^2=0$, hence $\Lambda'(t)=0$. By continuity, $\Lambda_k(t)$ is constant over the interval $[a,b]$. Since $J_a=\Phi$ is a Lie $\infty$-algebroid morphism, we have $\Lambda_k(a)=0$. Thus, $\Lambda_k(t)=0$ and, $$Q_{\E}\circ  J_t= J_t\circ Q_{\E'},\quad \text{for all}\; t\in[a,b].$$ \end{proof}

Let us show that homotopy in the sense above defines an equivalence relation $\mathtt{\sim}$ between Lie $\infty$-morphisms. We have the following lemma.

\begin{lemma}\label{Homotopy-lemma}{
 A pair $( J_t , H_t)$ is a homotopy between Lie $\infty$-algebroid morphisms $J_a$ and $J_b $  if and only if for all rational function, $g\colon[a,b]\rightarrow[c,d]$ without poles on $[a,b]$, the pair $(J_{g(t)} , g'(t)H_{g(t)})$ is a homotopy between $ J_{g(a)}$ and $ J_{g(b)}$.}
 
\end{lemma}

\begin{proof}
 Let $g\colon [a,b]\rightarrow[c,d]$  be a rational function without poles on $[a, b]$. A straightforward computation  gives:
 $$
 \begin{array}{rrcll}
    &\dfrac{\dd J_t}{\dd t}&=&Q_\E\circ H_t+H_t\circ Q_{\E'}&\hspace{.2cm} \hbox{(by definition}) \\ 
    \Rightarrow&\frac{\dd J}{\dd t}(g(t))&=&Q_\E\circ H_{g(t)}+H_{g(t)}\circ Q_{\E'}& \hspace{.2cm}  \hbox{(by replacing $t$ by $g(t)$)} \\ \Rightarrow& \frac{\dd J_{g(t)}}{\dd t}&=&Q_\E\circ \left(g'(t)H_{g(t)}\right)+\left(g'(t)H_{g(t)}\right)\circ Q_{\E'} &  \hspace{.2cm} \hbox{(by multiplying both sides by $g'(t)$)}.
 \end{array}
 $$
  The last equation means that $(J_{g(t)} , g'(t)H_{g(t)})$ is a homotopy between $ J_{g(a)}$ and $J_{g(b)} $. {The backward implication is obvious, it suffices to consider $a=c$, $b=d $ and $g={\mathrm{id}}$}.
\end{proof}
\begin{proposition}
Homotopy between Lie $\infty$-morphisms is an equivalence relation. In addition, it is compatible with composition, that is, if $\Phi,\Psi\colon S^\bullet_\mathbb{K}(\E')\rightarrow S^\bullet_\mathbb{K}(\E)$ are homotopic Lie $\infty$-algebroid morphisms and $ \hat{\Phi},\hat{\Psi}\colon S^\bullet_\mathbb{K}(\E)\rightarrow S^\bullet_\mathbb{K}(\E'')$ are homotopic Lie $\infty$-algebroid morphisms, then, so are their compositions $\hat{\Phi}\circ\Phi$  and $\hat{\Psi}\circ\Psi$.
\end{proposition}
\begin{proof}We first show that this notion of homotopy is an equivalence relation.
 Let $\Phi,\Psi$ and $\Xi\colon S^\bullet_\mathbb{K}(\E')\longrightarrow S^\bullet_\mathbb{K}(\E)$ be three Lie $\infty$-morphisms of algebroids. \begin{itemize}
	\item [$\bullet$] Reflexivity: The pair $( J_t=\Phi,H_t=0)_{t\in[0,1]}$ defines a homotopy between $\Phi$ and $\Phi$.
	\item [$\bullet$]Symmetry: Let $( J_t,H_t)_{t\in[0,1]}$ be a homotopy between $\Phi$ to $\Psi$. By applying Lemma \ref{Homotopy-lemma} with $g(t)=1-t$, we obtain a homotopy between $\Psi$ and $\Phi$ via the pair $(\Phi_{1-t},-H_{1-t})_{t\in[0,1]}$.
	\item [$\bullet$]Transitivity: Assume $\Phi\mathtt{\sim}\Psi$ and $\Psi\mathtt{\sim}\Xi$ and let  $( J_t,H_{1,t})_{t\in[0,\frac{1}{2}]}$ be a homotopy between $\Phi$ and $\Psi$ and let $(\bar{J}_t,H_{2,t})_{t\in[\frac{1}{2},1]}$ be a homotopy between $\Psi$ and $\Xi$. By gluing $ J_t$ and $\bar{J}_t$, respectively $H_{1t}$ and $H_{2,t}$ we obtain a homotopy $(\tilde{J}_t,H_t)_{t\in[0,1]}$ between $\Phi$ and $\Xi$.
\end{itemize}  

We then show it is compatible with composition. Let us denote by $( J_t , H_t )$ the homotopy between $\Phi$ and $\Psi$, and $( \hat{J}_t , \hat{H}_t )$ the homotopy between $\hat{\Phi}$ and $\hat{\Psi}$ . We obtain, \begin{align*}
    \frac{\dd  \hat{J}_t\circ J_t}{\dd t}&=\frac{\dd  \hat{J}_t }{\dd t}\circ J_t+ \hat{J}_t\circ\frac{\dd  \hat{J}_t }{\dd t}\\&=Q_{\E''}\circ\left(\hat{H}_t\circ J_t+ \hat{J}_t\circ H_t \right) +\left(\hat{H}_t\circ J_t+ \hat{J}_t\circ H_t \right)\circ Q_{\E'}.
\end{align*}Hence, $\hat{\Phi}\circ\Phi$  and $\hat{\Psi}\circ\Psi$ are homotopic via the pair $(  \hat{J}_t\circ J_t,\hat{H}_t\circ J_t+ \hat{J}_t\circ H_t)$ which is easily checked to satisfy all axioms. This concludes the proof.
\end{proof}

We conclude this section with a lemma that will be useful in the sequel. Notice that this is the lemma that forces to extend Definition \ref{def:general-idea-homotopy}, for it would not be true anymore in the smooth setting.


\begin{lemma}\label{gluing-lemma}
 Let $( J_t ,  H_t)_{t\in[c,+\infty[}$ be a homotopy such that for all $n\in\mathbb{N}_0$ and for every $t\geq n$, $H_t^{(n)}=0$. Then the $n$-th Taylor coefficient
  $ J_t^{(n)}$ is constant on $[n, +\infty [$ and
 the co-algebra morphism $J_\infty$ whose $n$-th Taylor coefficient is  $ J_t^{(n)}$ for any $n\in\mathbb{N}_0$ and $t \in [n, +\infty[$  is a Lie $\infty $-algebroid morphism.
 
 {Moreover, for $ g:[a,b[\rightarrow[c,+\infty[$ a rational function with no pole on $[a,b[$ and such that $\displaystyle\lim_{t\to b}g(t)=+\infty$, the pair $(J_{g(t)} , g'(t)H_{g(t)})$ is a homotopy between $J_c$ and $J_\infty$. }
\end{lemma}
\begin{proof}
Since the $n$-th Taylor coefficient of the $ J_t $-co-derivation $\frac{\dd  J_t^{(n)}}{\dd t}=(Q_\E\circ H_t +H_t\circ Q_{\E'})^{(n)}$ depends only on $H_t^{(i)} $ for $i=0, \dots, n-1 $, 
we have by assumption  $\frac{\dd  J_t^{(n)}}{\dd t}=0$ for all $t\geq n$. As a consequence, $ J_t^{(n)} $ is constant on $[n, +\infty[ $.   
It follows from Proposition \ref{prop:justify} that $J_\infty $ is a Lie $\infty $-algebroid morphism since for every $n\in\mathbb{N}_0$ and $t\in[n,+\infty[$\begin{align*}
    (Q_\E\circ J_\infty-J_\infty\circ Q_{\E'})^{(n)}&=\sum_{i+j=n}(Q_\E^{(i)}\circ J^{(i)}_t- J_t^{(j)}\circ Q_{\E'}^{(j)})\\&=0\hspace{.4cm}\hbox{(since $ J_t$ is a Lie $\infty$-algebroid morphism)}.
\end{align*}

Let us prove the last part of the statement. By assumption, there exists $a \leq b_n \leq b $ such that for all $t \in [b_n ,b]$, we have $g(t)\geq n$, so that $J_{g(t)}^{(n)}=J_\infty^{(n)}$ and $ g'(t)H_{g(t)}^{(n)}=0$ on $[b_n,b]$. The function $J_{t}^{(n)}$ (resp. $H_t^{(n)} $) being  piecewise rational continuous (resp. piecewise rational) on $[c,n]$, the same holds for $J_{g(t)}^{(n)}$ (resp. $g'(t)H_{g(t)}^{(n)}$) on $[a,b_n]$. By gluing with a constant function $J_\infty$ (resp. with $0$), we see that all Taylor coefficients of $J_{g(t)} $ (resp. $g'(t)H_{g(t)}$) are piecewise rational continuous (resp. piecewise rational) with finitely many gluing points. This completes the proof.
\end{proof}

\begin{remark}
 Lemma \ref{gluing-lemma} explains how to glue infinitely many homotopies, at least when for a given $n\in \mathbb{N}_0$, only finitely of them affects the $n$-th Taylor coefficient. This would not be possible using only Definition \ref{def:general-idea-homotopy}.
\end{remark}

\subsection{More technical lemmas and propositions}
Let us state and prove these technical assertions for later use.\\

Let $(\E',Q_{\E'}, \rho')$ and $(\E, Q_\E, \rho)$ be a Lie $\infty$-algebroid over $\mathcal{O}$.
\begin{proposition}\label{linearity} Let $ \Phi\colon S^\bullet_\mathbb{K} (\E')\rightarrow S^\bullet_\mathbb{K} (\E)$ be a  Lie $\infty $-algebroid morphism.
For $\mathcal H\colon S^\bullet_\mathbb{K} (\E')\rightarrow S^\bullet_\mathbb{K} (\E)$ a $\mathcal O $-multilinear $ \Phi$-co-derivation of degree $k\leq -1$, 
$\mathcal H \circ Q_{\mathcal E'} -(-1)^k  Q_{\mathcal E} \circ \mathcal H $ is a $\mathcal O $-multilinear $ \Phi$-co-derivation of degree $ k+1$.
\end{proposition}
\begin{proof}
We first check that    $\mathcal H \circ Q_{\mathcal E} - (-1)^k  Q_{\mathcal E'} \circ \mathcal H$ is  a $ \Phi$-co-derivation: \begin{align*}
   \Delta\circ\mathcal H \circ Q_{\E'}&=(\mathcal H\otimes\Phi+\Phi\otimes\mathcal H)\circ\Delta'\circ Q_\E',\;\text{by definition of $\mathcal H$}\\&=(\mathcal H\otimes\Phi+\Phi\otimes\mathcal H)\circ(Q_{\E'}\otimes \text{id}+\text{id}\otimes Q_{\E'})\circ\Delta',\;\text{by definition of $Q_\E'$}\\&=(\mathcal H\circ Q_{\E'}\otimes\Phi+(-1)^k\Phi\circ Q_{\E'}\otimes\mathcal H+\mathcal H\otimes\Phi\circ Q_{\E'}+\Phi\otimes\mathcal H\circ Q_{\E'})\circ\Delta'
\end{align*}
Subtracting a similar equation for $(-1)^k\Delta\circ Q_{\E}\circ\mathcal H$ and using \eqref{def:LM}, one obtains the $\Phi$-co-derivation
property for $\mathcal H \circ Q_{\mathcal E'} - (-1)^k  Q_{\mathcal E'} \circ \mathcal H$.
We now check that $\mathcal H \circ Q_{\mathcal E'} -(-1)^k  Q_{\mathcal E} \circ \mathcal H$ is $ \mathcal O$-multilinear, for which it suffices to check that its Taylor coefficients are $\mathcal{O}$-multilinear by Lemma \ref{Tay-C}. Let $x_1,\ldots,x_n\in\E'$ be homogeneous elements. Assume $x_i\in\E_{-1}'$ 
(if we have more elements of degree $-1$ the same reasoning holds). To verify $\mathcal{O}$-multilinearity it suffices to check that for all $f\in \mathcal O $:
\begin{align*}\text{pr}\circ(\mathcal H \circ Q_{\mathcal E'} -(-1)^k  Q_{\mathcal E} \circ \mathcal H) (x_1,\ldots,&x_i, \ldots, f x_j \ldots,x_n)=\\&  f \text{pr}\circ(\mathcal H \circ Q_{\mathcal E'} -(-1)^k  Q_{\mathcal E} \circ \mathcal H) (x_1,\ldots,x_i, \ldots,  x_j \ldots,x_n).\end{align*}
Only the terms where the $2$-ary bracket with a degree $-1 $ element on one-side and  $f $ on the other side may forbid $f$ to go in front. There are two such terms:
$$  \epsilon(x_i, x_j, x_{I^{ij}})   \mathcal H^{(n-1)} (\ell_2'(x_i, f x_j), x_{I^{ij}}  ) \hbox{ and } 
-(-1)^k  \epsilon(x_i,  x_{I^{i}})\ell_2(\Phi_{0}(x_i), f\mathcal H^{(n-1)} ( x_{I^{i}}  ))
$$ 
where $x_{I^i}$ and $ x_{I^{ij}}$ stand for the list $ x_1, \dots, x_n$ where $x_i$ and $x_i,x_j $ are missing, respectively, and $ \mathcal H^{(n-1)}$ is the $(n-1)$-th Taylor coefficient of $\mathcal H$. Since  $\rho\circ\Phi_{0}=\rho'$, in both terms $\rho(x_i)[f] $ appears, and these two terms add up to zero.
\end{proof}
\begin{remark}
If the degree of $\mathcal H$ is non-negative, then $\mathcal H \circ Q_{\mathcal E'} -(-1)^k  Q_{\mathcal E} \circ \mathcal H $ may not be $\mathcal O$-multilinear anymore, since there may exist extra terms where the anchor map appears, e.g. terms of the form $\ell_2(\mathcal H( x_{I^j} ) , f \Phi^{(0)} (x_j) ) $.
\end{remark}

\begin{lemma}\label{O-linearity-lemma}
Let $\Phi\colon S_{\mathbb K}^\bullet ( \mathcal E')\rightarrow S_{\mathbb K}^\bullet ( \mathcal E )$ be a co-algebra morphism such that
\begin{enumerate}
     \item $\Phi$ is $\mathcal O $-multilinear,
    \item $\rho \circ \Phi_{0} = \rho'$ on $\E'$,
\end{enumerate}
If for every \footnote{$\Phi\circ Q_{\E'}-Q_{\E}\circ\Phi$ being a $\Phi$-co-derivation, its component of polynomial-degree $i$ is zero for $0\leq i\leq n$ if only if its $i$-th Taylor coefficient is zero for $0\leq i\leq n$.} $n \in \mathbb{N}_0$, $(\Phi\circ Q_{\E'}-Q_{\E}\circ\Phi)^{(i)}=0$  for each $0 \leq i \leq n$, then the map $S_{\mathbb K} ( \mathcal E' )\rightarrow  S_{\mathbb K} ( \mathcal E)$ given by:
$$ (\Phi\circ Q_{\E'}-Q_{\E}\circ\Phi)^{(n+1)}  $$ 
\begin{enumerate}
    \item is a $\Phi^{(0)} $-co-derivation of degree $+1$,
    \item is $\mathcal O $-multilinear,
    \item and the induced  $\Phi^{(0)} $-co-derivation $ \left(\bigodot^\bullet \mathcal E', Q_{\E'}^{(0)}\right) \longrightarrow \left(\bigodot^\bullet\mathcal E,Q_{\E}^{(0)}\right)$ satisfies:
     $$    Q^{(0)}_\E\circ(\Phi\circ Q_{\E'}-Q_{\E}\circ\Phi)^{(n+1)}  =
          (Q_{\E}\circ\Phi-\Phi\circ Q_{\E'})^{(n+1)}\circ Q_{\E'}^{(0)} . $$
\end{enumerate}
\end{lemma}

\begin{proof}
A straightforward computation yields:
\begin{align*}
    \Delta(\Phi\circ Q_{\E'}-Q_{\E}\circ\Phi)&=(\Phi\otimes\Phi)\circ\Delta'\circ Q_{\E'}-(Q_{\E}\otimes \text{id}+\text{id}\otimes Q_{\E})\circ\Delta\circ\Phi\\&=\left((\Phi\circ Q_{\E'}-Q_{\E}\circ\Phi)\otimes\Phi+\Phi\otimes(\Phi\circ Q_{\E'}-Q_{\E}\circ\Phi) \right)\circ \Delta'.
\end{align*}
Now, $\Delta$ preserves polynomial-degree, i.e. $\displaystyle{\Delta:S^n_\mathbb K(\E)\longrightarrow \oplus_{i+j=n} S^i_\mathbb K(\E)\otimes S^j_\mathbb K(\E)}$ and so does $ \Delta'$. 
Taking into account the assumption  $(\Phi\circ Q_{\E'}-Q_{\E}\circ\Phi)^{(i)}=0$  for every $0 \leq i \leq n$, we obtain:
  \begin{eqnarray*}\Delta \circ (\Phi\circ Q_{\E'}-Q_{\E}\circ\Phi)^{(n+1)}\\ = 
   \left((\Phi\circ Q_{\E'}-Q_{\E}\circ\Phi)^{(n+1)}\otimes\Phi^{(0)}+\Phi^{(0)}\otimes(\Phi\circ Q_{\E'}-Q_{\E}\circ\Phi)^{(n+1)}\right)\circ \Delta'.\end{eqnarray*} 
   All the other terms disappear for polynomial-degree reasons Hence, $(\Phi\circ Q_{\E'}-Q_{\E}\circ\Phi)^{(n+1)}$ is a $\Phi^{(0)}$-co-derivation. 

Let us prove that it is $\mathcal{O}$-linear. It suffices to check  $\mathcal{O}$-linearity of\, $T_\Phi:= \Phi\circ Q_{\E'}-Q_{\E}\circ\Phi$. Let us choose homogeneous elements $x_1,\ldots, x_{N}\in\E'$ and let us assume that $x_i\in\E'_{-1}$ is the only term of degree $-1$: The proof in the case where there is more than one such an homogeneous element of degree $-1$ is identical. We choose $j \neq i$ and we compute $T_\Phi (x_1,\ldots,x_i,\ldots,fx_j,\ldots, x_{N})$ for some $f\in\mathcal{O}$. The only terms in the previous expression which are maybe non-linear in $f$ are those for which the $2$-ary brackets of a term containing $fx_j$ with $x_i$ or $\Phi_{0}(x_i)$ appear (since $\Phi $ and all other brackets are $\mathcal O$-linear). There are two such terms. The first one appears when we apply $Q_{\E'} $  first, and then $\Phi $: this forces
$\Phi\left(\ell'_2(x_i,fx_j),x_{I^{ij}}\right)$ to appear, and the non-linear term is then:
\begin{align}\label{term1}
 \epsilon (\sigma_i) \rho'(x_i)[f] \,   \Phi(x_{I^i})
\end{align}
with $ \sigma_i$ the permutation that let $i$ goes in front and leave the remaining terms unchanged. There is a second term that appears when one applies $ \Phi$ first, then $ Q_\E$. Since it is a co-morphism, $\Phi ( x_1 \ldots x_i \ldots,fx_j \ldots x_{N} )  $ is the product of several terms among which only one is of degree $-1$, namely the term 
 $$ \epsilon (x,\sigma_i)\Phi_{0}(x_i ) \Phi(f x_{I^i} ).$$
 Applying $Q_\E $ to this term yields the non-linear term 
  \begin{align}\label{term2}
      \epsilon (\sigma_i)\rho( \Phi_{0} (x_i))  [f] \, \Phi(x_{I^i}),
  \end{align} 

\noindent
where $I^i$ and $I^{ij}$ are as in Proposition \ref{linearity}. Since $\rho\circ\Phi_{0}=\rho'$, we see that the terms \eqref{term1} and \eqref{term2} containing an anchor add up to zero. 

Let us check that $(\Phi\circ Q_{\E'}-Q_{\E}\circ\Phi)^{(n+1)}$ is a chain map, in the sense that it satisfies item 3).
Considering again $T_\Phi:= \Phi\circ Q_{\E'}-Q_{\E}\circ\Phi$, we have that $T_\Phi^{(k)}=0$, for all $k=0,\ldots,n$. Since $T_\Phi\circ Q_{\E'}=Q_\E\circ T_\Phi$, one has
\begin{align*}
0=\left( T_\Phi\circ Q_{\E'}+Q_{\E}\circ T_\Phi\right)^{(n+1)}= T_\Phi^{(n+1)}\circ Q_{\E'}^{(0)}+Q_{\E}^{(0)}\circ T_\Phi^{(n+1)}+\sum_{\overset{i+j=n+1}{i,j\geq 1}}\underbrace{\left(T_\Phi^{(i)}\circ Q_{\E'}^{(j)}+Q_{\E}^{(j)}\circ T_\Phi^{(i)}\right)}_0
\end{align*}
By consequent, the $\mathcal O$-linear map $(\Phi\circ Q_{\E'}-Q_{\E}\circ\Phi)^{(n+1)}$ satisfies item 3.
\end{proof}

\begin{lemma}\label{imp-lemma}
Let $\Phi,\Xi: S^\bullet_\mathbb{K} (\E') \longrightarrow S^\bullet_\mathbb{K}(\E)$ be $\mathcal O $-linear Lie $\infty$-algebroid morphisms. Let $n \in \mathbb{N}_0$.
If $\Xi^{(i)}= \Phi^{(i)}$ for every $0 \leq i \leq n$, then
$(\Xi-\Phi)^{(n+1)} \colon S^\bullet_\mathbb{K} (\E') \longrightarrow S^\bullet_\mathbb{K} (\E )$
 \begin{enumerate}
     \item  is a $\Phi^{(0)}$-co-derivation
     \item is $\mathcal O $-multilinear
     \item  and the induced $\Phi^{(0)}$-co-derivation $\left(\bigodot^\bullet\E',Q_{\E'}^{(0)}\right) \longrightarrow \left(\bigodot^\bullet\E,Q_\E^{(0)}\right)$ satisfies:
      $$   Q^{(0)}_\E \circ(\Xi-\Phi)^{(n+1)}= (\Xi-\Phi)^{(n+1)} \circ Q_{\E'}^{(0)} .$$
 \end{enumerate}
\end{lemma}

 \begin{proof}
  For all $x_1,\ldots,x_k\in\E'$, one has: \begin{align*}
\Delta(\Xi-\Phi)(x_1\odot\cdots\odot x_k)&=\sum^k_{j=1}\sum_{\sigma\in\mathfrak{S}_{(j,k-j)}}\epsilon(\sigma)  \Xi(x_{\sigma(1)}\odot\cdots\odot x_{\sigma(j)})\otimes\Xi(x_{\sigma(j+1)}\odot\cdots\odot x_{\sigma(k)})\\&-\sum^k_{j=1}\sum_{\sigma\in\mathfrak{S}_{(j,k-j)}}  \epsilon(\sigma) ( \Phi(x_{\sigma(1)}\odot\cdots\odot x_{\sigma(j)})\otimes\Phi(x_{\sigma(j+1)}\odot\cdots\odot x_{\sigma(k)})\\&=\left( (\Xi-\Phi)\otimes\Phi+\Xi\otimes(\Xi-\Phi)\right)\circ\Delta'(x_1\odot\cdots\odot x_k).
\end{align*}
Since $\Delta$ has polynomial-degree $0$ and $(\Xi-\Phi)^{(i)}=0$ for all $0\leq i\leq n$, we obtain\begin{equation}
\Delta(\Xi-\Phi)^{(n+1)}(x_1\odot\cdots\odot x_k)=\left( (\Xi-\Phi)^{(n+1)}\otimes\Phi^{(0)}+\Phi^{(0)}\otimes(\Xi-\Phi)^{(n+1)}\right)\circ\Delta'(x_1\odot\cdots\odot x_k).
\end{equation} This proves the first item. 
Since both $\Phi $ and $\Xi $ are $\mathcal O$-multilinear, $(\Xi-\Phi)^{(n+1)}$ is $\mathcal O$-multilinear. This proves the second item.  Since, $\Phi$ and $\Xi$ are  Lie ${\infty}$-morphisms:\begin{equation}
(\Xi-\Phi)\circ Q_{\E'}-Q_{\E}\circ(\Xi-\Phi)=0.
\end{equation}
By looking at the component of polynomial-degree $n+1$, one obtains, $(\Xi-\Phi)^{(n+1)}\circ Q_{\E'}^{(0)}-Q_{\E}^{(0)}\circ(\Xi-\Phi)^{(n+1)}=0.$ This proves the third item.
\end{proof}
\vspace{2cm}
\begin{tcolorbox}[colback=gray!5!white,colframe=gray!80!black,title=Conclusion:]
This chapter recapitulates classical definition of Lie $\infty$-algebroids, and in particular describes it as co-derivation, which is not usual. Morphisms and homotopies are described. About homotopies, two versions are given: one uses smooth maps, and is way easier (see Definition \ref{def:general-idea-homotopy}). Unfortunately, to allow infinitely many gluings, we have to introduce a more complicated notion of homotopy, which uses continuous locally $C^1$-maps (see Definition \ref{def:homotopy}).
\end{tcolorbox}

\chapter{Lie-Rinehart algebras and their morphisms}\label{chap-Lie-Rinehrt}

Except for Remark \ref{loc:res}, this section is essentially a review of the literature on the subject, see, e.g. \cite{MR2075590,MR1625610}. 
\section{Definitions}
 Lie-Rinehart algebras are the algebraic encoding of the notion of Lie algebroids over a manifold. We owe this concept to \cite{RinehartGeorgeS}.
\begin{definition}
A \emph{Lie-Rinehart algebra} over an algebra  $ \mathcal O$ is a triple $(\mathcal A , [\cdot, \cdot]_\mathcal A , \rho_{\mathcal A})$ with $ \mathcal A$ an $ \mathcal O$-module, $[\cdot, \cdot]_\mathcal A  $ a Lie algebra bracket on  $\mathcal A $, and $ \rho_\mathcal A \colon \mathcal A \longrightarrow  {\mathrm{Der}}(\mathcal O)$ an $ \mathcal O$-linear Lie algebra morphism called \emph{anchor map}, satisfying the so-called \emph{Leibniz identity}:
 $$   [  a,  f b ]_\mathcal A  = \rho_\mathcal A (a ) [f] \, b + f [a,b]_\mathcal A  \hbox{ for all $ a,b \in \mathcal A, f \in \mathcal O$}.$$

 \begin{enumerate}
     

\item
  Let $\eta\colon\mathcal O\longrightarrow\mathcal{O}'$ be an algebra morphism.  A \emph{Lie-Rinehart algebra morphism over $\eta$} is a Lie algebra morphism $\phi\colon\mathcal A\longrightarrow\mathcal{A}'$  such that for every $a\in\mathcal A$ and $f\in\mathcal O$:\begin{enumerate}
    \item $\phi(fa)=\eta(f)\phi(a)$
    \item $\eta(\rho_\mathcal A(a)[f])=\rho_{\mathcal A'}(\phi(a)[\eta(f)]$,
\end{enumerate}When $\mathcal O =\mathcal O'$ and $\eta=\text{id}$, we say that $\phi$ is a \emph{Lie-Rinehart algebra morphism} \emph{over $\mathcal{O}$}.
\item A submodule $\mathcal{B}\subseteq \mathcal{A}$ is a said to be a \emph{Lie-Rinehart subalgebra of $\mathcal A$} if it carries a Lie-Rinehart algebra structure over $\mathcal{O}$ whose Lie bracket and anchor map are the restriction of the bracket $\lb_A$ and the anchor $\rho_A$ respectively to $\mathcal{B}$. 
\end{enumerate}

\end{definition}
\vspace{0.3cm}
Lie-Rinehart algebras over $\mathcal{O}$ form a category that we denote by \textbf{Lie-Rhart-alg}/$\mathcal O$
\begin{remark}
A Lie-Rinehart algebra is said to be a \emph{Lie algebroid} if $ \mathcal A$ is a projective $ \mathcal O$-module. See Example \ref{ex:LieAlg} for the relation with usual Lie algebroid as vector bundles.
\end{remark}
\begin{remark}\label{rmk:basic}
For every Lie $\infty$-algebroid $(\E_\bullet,\ell_\bullet,\rho)$ over $\mathcal O$, the quotient space $\frac{\E_{-1}}{\ell_1(\E_{-2})}$ comes equipped with a natural Lie-Rinehart algebra over $\mathcal O$. The $2$-ary bracket $\ell_2$ goes to quotient to $\frac{\E_{-1}}{\ell_1(\E_{-2})}$ to define a Lie algebra since for all $x\in\E_{-2}$ and  $y\in \E_{-1}$ we have $$\ell_2(\ell_1(x),y)=\ell_1(\ell_2(x,y)).$$ Also, the Jacobi identity holds, since $$\ell_2(\ell_2(x,y),y)+\circlearrowleft(x,y,z) =-\ell_1(\ell_3(x,y,z))\qquad \forall x,y,z\in\E_{-1}.$$In addition, the anchor map $\rho$ goes to quotient to a Lie algebra morphism $\frac{\E_{-1}}{\ell_1(\E_{-2})}\rightarrow \mathrm{Der}(\mathcal O)$, since $\rho\circ \ell_1=0$. This Lie-Rinehart algebra  is called the \emph{basic} Lie-Rinehart algebra of $(\E_\bullet,\ell_\bullet,\rho)$.

To summarize, every Lie $\infty$-algebroid induces a Lie-Rinehart algebra. The opposite direction, that is, wondering whether a Lie-Rinehart algebra $(\mathcal{A},\lb_\mathcal A,\rho)$ over $\mathcal{O}$ is the basic Lie-Rinehart algebra of some  almost differential graded Lie algebroid or more generally a Lie $\infty$-algebroid over $\mathcal O$ is part of the questions that I discussed in this thesis (see Chapter \ref{Chap:main}). This extends the main results of \cite{LLS} from locally real analytic finitely generated singular foliations (see Example \ref{ex:singfoliation}) to arbitrary Lie-Rinehart algebras. 
\end{remark}

Let us fix some vocabulary that will be used in the sequel.
\begin{definition}
We say that an almost differential  graded Lie algebroid $(\E_\bullet,\ell_1, \ell_2, \rho)$ or a Lie $\infty$-algebroid $(\E_\bullet,\ell_\bullet,\rho)$ over $\mathcal O$

\begin{enumerate}
    \item \emph{covers (through a hook $\pi$)} a Lie-Rinehart algebra $(\mathcal A,  [\cdot\,, \cdot]_\mathcal A, \rho_\mathcal{A})$, if there exists a morphism of brackets $\pi\colon \E_{-1}\rightarrow \mathcal A$ such that $\pi\circ \rho_\mathcal{A}=\rho$ and  $\pi(\E_{-1})=\mathcal{A}$,
    
    \item \emph{terminates in $\mathcal{ A}$
through the hook $\pi$}, when $\pi(\E_{-1})\subseteq\mathcal{A}$.
\end{enumerate}
\end{definition}
According to Remark \ref{rmk:basic}, any Lie $\infty$-algebroid $(\E_\bullet,\ell_\bullet,\rho)$ covers its basic Lie-Rinehart algebra through the hook $\pi$ which is given by the projection $\pi \colon \E_{-1} \longrightarrow \mathcal E_{-1} / \ell_1 (\mathcal E_{-2})$, since $\pi$ respects the brackets and $\pi(\E_{-1})=\mathcal E_{-1} / \ell_1 (\mathcal E_{-2})$.
\begin{definition}
  Let $(\E_\bullet',\ell'_\bullet,\rho')$ and $(\E_\bullet,\ell_\bullet,\rho)$ and be Lie $\infty$-algebroids over $\mathcal O$ that terminate in a Lie-Rinehart algebra $(\mathcal A,  [\cdot\,, \cdot]_\mathcal A, \rho_\mathcal{A})$ through the hooks $\pi'$ and, $\pi$ respectively.  We say that a Lie $\infty $-algebroid morphism $\Phi\colon S_\mathbb{K}^\bullet(\E')\rightarrow S_\mathbb{K}^\bullet(\E)$  is \emph{over}  $\mathcal A $, if $\pi \circ \Phi_{0}= \pi'$.
\end{definition}
We will need the following lemma.
\begin{lemma}
Let $(\E'_\bullet,\ell_\bullet',\rho')$ and $(\E_\bullet,\ell_\bullet,\rho)$ be Lie $\infty$-Lie algebroids that terminate in some Lie-Rinehart algebra  $(\mathcal A,\lb_A,\rho_A)$ through hooks $\pi'$ and $\pi$. Let $(J_t,H_t)_{t\in[a,b]}$ be a homotopy that joins $J_a$ and $J_b$. If $J_a$ is Lie $\infty$-algebroid morphism that terminates at $\mathcal A$ (i.e., $\pi \circ {J_a^{(0)}}_{|_{\E'}}=\pi'$), then so is the $\infty$-algebroid morphism $ J_t$ for all $t\in[a,b]$.
\end{lemma}
\begin{proof}
This is a direct consequence of Equation \eqref{sol2},
since ${Q_{\mathcal E}^{(0)}}_{|\E}= \ell_1 : \mathcal E_{-2} \to \mathcal E_{-1} $ and ${Q_{\mathcal E'}^{(0)}}_{|\E'}= \ell_1' :  \mathcal E_{-2}' \to \mathcal E_{-1}' $ are valued in the kernels of $ \pi$ and, $ \pi'$ respectively.

Last, $\mathcal{O}$-multilinearity of $ J_t$ follows from the $\mathcal{O}$-multilinearity of $Q_{\E}\circ H_t+H_t\circ Q_{\E'}$, which is granted by Proposition \ref{linearity}. This completes the proof.
\end{proof}

\section{Algebraic and geometric examples}\label{sec:Alg-Geo-ex}
\begin{example}
For every commutative $\mathbb K $-algebra $ \mathcal O$, the Lie algebra $\mathcal A= {\mathrm Der}({\mathcal O}) $ of derivations of a commutative algebra $ \mathcal O$ is a Lie-Rinehart algebra over $ \mathcal O$, with the identity as an anchor map. This Lie-Rinehart algebra is a terminal object in the category \textbf{Lie-Rhart-alg}/$\mathcal O$: for every Lie-Rinehart $(\mathcal{A},\lb_\mathcal{A}, \rho_\mathcal{A})$ the anchor $\mathcal A \overset{ \rho_\mathcal A}{\longrightarrow}  \mathrm{Der}(\mathcal O)$ is obviously a Lie-Rinehart algebra morphism over $\mathcal{O}$.

In particular, vector fields on a smooth or Stein manifold or an affine variety are examples of
Lie-Rinehart algebras over their respective natural algebras of functions. 	
\end{example}

\begin{example} \label{ex:LieAlg}
Let $M$ be a smooth manifold. By Serre-Swan Theorem, Lie algebroids over $M$ are precisely Lie-Rinehart algebras over $C^\infty(M)$
of the form $(\Gamma(A), [\cdot, \cdot], \rho) $
	where $A$ is a vector bundle over $M$ and  $\rho : A\rightarrow T M$ is a vector bundle morphism.
\end{example}
\begin{example}
Sections of a Lie algebroid that have values in the kernel of the anchor map form a Lie-Rinehart algebra ${\mathrm{Ker}}(\rho_A) $ for which the anchor map is zero.
\end{example}

The next two examples are part of our the main source of motivation to study Lie-Rinehart algebras. The second is more general than the first one.
\begin{example}{\it{Singular foliations.}}\label{ex:singfoliation}
 \cite{AndroulidakisIakovos,AndroulidakisZambon,Cerveau, Debord,LLS,LLL} A \emph{singular foliation} on a   smooth, real analytic, or complex manifold $M$ or Zariski open subset $\mathcal{U} \subseteq \mathbb C^d$  is a subsheaf $\mathfrak{F}\subseteq\mathfrak{X}(M)$ that fulfills the following conditions
\begin{itemize}
     \item \textbf{Stability under Lie bracket}: $[\mathfrak{F},\mathfrak{F}]\subseteq \mathfrak{F}$.
     \item \textbf{Locally finitely generateness}: every $m\in M$ admits an
open neighborhood $\mathcal{U}$ together with a finite number of vector fields  $X_1, \ldots, X_r \in \mathfrak{X}(\mathcal U)$ such that for
every open subset $\mathcal V \subseteq \mathcal U$ the vector fields ${X_1}_{|_\mathcal{V}} ,\ldots, {X_r}_{|_\mathcal{V}}$ generates $\mathfrak{F}$ on $\mathcal{V}$ as a $C^{\infty}(\mathcal{V})$-module.
 \end{itemize}
 
There are several other ways to define singular foliations on a manifold $M$ \cite{AZ,Cerveau,Dazord,Debord}.
All these definitions have in common to define then as Lie-Rinehart sub-algebras $\mathfrak F$ of the Lie-Rinehart algebra $\X (M) $ of vector fields on $M$ (or compactly supported vector fields $\X_c(M) $ on $M$). Here are some important consequences of the above definition.

\begin{enumerate}
     \item Singular foliation admits leaves: there exists a partition of $M$ into submanifolds called leaves such that for all $m\in M$, the image of the evaluation map $\mathfrak F \to T_m M  $  is the tangent space of the leaf through $m$. When $\mathfrak F $ coincides with the space of vector fields tangent to all leaves at all points, we shall speak of a \emph{\textquotedblleft Stefan-Sussman singular foliation\textquotedblright}.
     \item \emph{Singular foliations are self-preserving }: the flow of vector fields in $\mathfrak F$, whenever defined, preserves $\mathfrak F$ \cite{AndroulidakisIakovos,GarmendiaAlfonso}.
\end{enumerate}

Notice that in Remark \ref{rmk:NGLA}, 
 $\mathfrak{F}:=\rho(\Gamma(E_{-1}))$ is a singular foliation  in the sense above, called the \emph{basic singular foliation} of $\left(E,(\ell_k)_{k\geq 1}, \rho\right)$. We say, then the Lie $\infty$-algebroid $\left(E,(\ell_k)_{k\geq 1}, \rho\right)$ is \emph{over $\mathfrak{F}$}.
\end{example}

\begin{example}\cite{ZambonMarco2,ZambonMarco3,ZambonMarco}
Let $(A,\lb_A, \rho_A)$ be a Lie algebroid over a manifold $M$. A	\emph{singular subalgebroid $\mathcal{B}$ of $A$}, is a $C^{\infty}(M)$-submodule of $\Gamma(A)$ that is  locally finitely generated  and  involutive i.e.   stable under the  Lie bracket $\lb_A$. This notion is a generalization of singular foliations in the sense of Example \ref{ex:singfoliation}, since  singular foliations on $M$ are singular subalgebroids of $T M$.

A singular subalgebroid $\mathcal B$ is an  example of Lie-Rinehart algebra: its bracket and anchor are the restrictions of $\lb_A$ and $\rho_A$ to $\mathcal{B}$ respectively.

\end{example}

\begin{example}
	For a singular foliation $\mathfrak{F}$ on a manifold $M$, consider $\mathcal{S}:=\left\lbrace X\in\mathfrak{X}(M)\mid [X,\mathfrak{F}]\subseteq\mathfrak{F} \right\rbrace $ (i.e. infinitesimal symmetries of $\mathfrak F$) and $$\mathcal{C}:=\left\lbrace f\in\mathcal{C}^\infty(M)\mid Y[f]=0, \;\text{for all}\;Y\in\mathfrak{F}\right\rbrace $$ (that can be thought of as functions constant along the leaves of $\mathfrak F$). The quotient $\frac{\mathcal{S}}{\mathfrak{F}}$ is Lie-Rinehart algebra over  $\mathcal{C}$.
\end{example}

\begin{example} [\emph{Poisson manifold}]\label{ex:OverCasimir} \cite{CPA} We recall that a \emph{Poisson manifold} is a manifold $M$ together with a biderivation $\{\cdot\,,\cdot\}$ on its algebra of smooth functions $C^\infty(M)$ that satisfies Jacobi's identity, i.e. 

\begin{enumerate}
    \item $(C^\infty(M),\{\cdot\,,\cdot\})$ is Lie algebra.
    \item The biderivation $\{\cdot\,,\cdot\}$ is compatible with the product of functions in the following sense: for all $f,g,h\in C^\infty(M)$ $$\{f.g,h\}=f.\{g,h\}+g.\{f,h\}.$$
\end{enumerate}

This is equivalent to giving a bivector field $\pi\in \Gamma(\wedge^2 TM)$ such that the Schouten-Nijenhuis bracket with itself vanishes, that is, $[\pi,\pi]_{SN}=0$.\\

For every a Poisson manifold $(M,\pi)$, the operator $\delta_\pi\colon \Gamma(\wedge^\bullet TM)\rightarrow \Gamma(\wedge^{\bullet+1}TM), P\mapsto -[P,\pi]_{SN}$ defines a complex 

$$\xymatrix{\cdots\ar[r] &\Gamma(\wedge^{p-1}TM)\ar[r]^{\delta^{p-1}_\pi}&\Gamma(\wedge^{p}TM)\ar[r]^{\delta_\pi^p }&\Gamma(\wedge^{p+1}TM)\ar[r]&\cdots},$$

since $[\pi,\pi]_{SN}=0$ implies $\delta_\pi\circ \delta_\pi=0$. For every $p\in\mathbb{N}_0$, the quotient $H^p_\pi(M):= \frac{\ker \delta_\pi^p}{\mathrm{Im}\delta_\pi^{p-1}}$ is called the \emph{$p$-th Poisson cohomology of $(M,\pi)$}. We define $\mathcal{A}:=\text{H}_\pi^1(M)$ to be the first Poisson cohomology of $\pi$ and $\mathcal{O}:=\text{H}_\pi^0(M)=\text{Cas}(\pi)$ to be the algebra of Casimir functions. The bracket of vector fields makes $\mathcal A $ a Lie-Rinehart algebra over $\text{Cas}(\pi)$.
\end{example}

\subsection{Basic constructions}
\label{LR-construction}
\label{ex:vanishing}\label{loc:res}
Let $(\mathcal A,\lb_\mathcal A,\rho_\mathcal A)$ a Lie-Rinehart algebra over an algebra $\mathcal O$ and $\mathcal I\subset\mathcal O$ be an ideal. The submodule $\mathcal I\mathcal A$ is a Lie-Rinehart subalgebra of $\mathcal A$ and its anchor is given by the restriction of $\rho_\mathcal A$ over $\mathcal I \mathcal A\subset\mathcal A$. This follows easily from $$ [ f a, g b]_\mathcal A = fg [a,b]_\mathcal A + f \rho_\mathcal A (a) [g] \, b - g \rho_\mathcal A (b) [f] \, a  \hbox{ for all $a,b \in \mathcal A, f,g \in \mathcal I$}.$$

\begin{enumerate}
    \item \emph{Restriction.}
Consider a Lie-Rinehart algebra $(\mathcal A , [\cdot, \cdot]_\mathcal A , \rho_{\mathcal A})$  over $\mathcal O $.
For every \emph{Lie-Rinehart ideal} $\mathcal I \subset \mathcal O$, i.e. any ideal
such that
$$ \rho_\mathcal A (\mathcal A) [\mathcal I] \subset \mathcal I$$
the quotient space $\mathcal A / \mathcal I \mathcal A $ inherits a natural Lie-Rinehart algebra structure over $\mathcal O / \mathcal I$. We call this Lie-Rinehart algebra the \emph{restriction w.r.t the Lie-Rinehart ideal $ \mathcal I$}. In the context of affine varieties, when $\mathcal I$ is the ideal of functions vanishing on an affine sub-variety $W$, we shall denote $\frac{\mathcal A}{\mathcal{I}\mathcal A}$ by $\mathfrak i_W^* \mathcal A $. 
    \item \emph{Localization.}\label{def:localisation} Localizing w.r.t  a multiplicative subset $S\subset\mathcal O$ (i.e.  $S\subseteq \mathcal{O}\setminus \{0\}$ is closed under multiplication and $1\in S$) is a very powerful tool in commutative algebra and in algebraic geometry to deal with "global" problems by reducing them to "local" ones. Let us recall the construction (e.g. see \cite{stacks-project}, Section 10.9  or \cite{A.Gathmann} for details). \begin{enumerate}
        \item Let $\mathcal V$ be a $\mathcal{O}$-module. The \emph{localization of $\mathcal V$ at $S$} is defined as follows: consider the equivalence relation on $S\times \mathcal V$ that is given by
        $$(s,v)\sim (s',v') \Longleftrightarrow \exists u\in S,\; u(s'v-sv')=0.$$The class of an element $(s,v)\in S\times \mathcal  V$ by $\frac{v}{s}$, and by $S^{-1}\mathcal V:=S\times \mathcal V/_\sim$ the set of equivalence classes.\\
        
        In particular, the \emph{localization of $\mathcal{O}$ at $S$} is the localization of
        $\mathcal{O}$ as a $\mathcal{O}$-module. In that case, $S^{-1}\mathcal O$ is a $\mathbb{K}$-algebra together with the addition and multiplication defined by

        \begin{align*}
            \frac{f}{s}+\frac{g}{s'}:=\frac{fs'+sg}{ss'}\quad \text{and}\quad \frac{f}{s}\frac{g}{s'}:=\frac{fg}{ss'}.
        \end{align*}
      Similarly, $S^{-1}\mathcal V$ is a $S^{-1}\mathcal O$-module with addition and scalar product multiplication defined in an obvious way. Also, we have $S^{-1}\mathcal V\simeq  S^{-1}\mathcal{O}\otimes_\mathcal{O}\mathcal V$.
        \item A derivation $D\in \mathrm{Der}(\mathcal{O})$ admits a unique extension to a derivation $S^{-1}D\in\mathrm{Der}(S^{-1}\mathcal O)$ which given for $(s,f)\in S\times \mathcal{O}$ by the classical formula
        
        $$(S^{-1}D)\left(\frac{f}{s}\right):= \frac{D(f)s-D(s)f}{s^2}.$$
        
    \end{enumerate}

Let $(\mathcal{A},\lb_\mathcal{A},\rho)$ be a Lie-Rinehart algebra over $\mathcal O$. The localization module $S^{-1}\mathcal A= S^{-1}\mathcal O\otimes_\mathcal O\mathcal A$ comes equipped with a natural structure of Lie-Rinehart algebra over the localization algebra $S^{-1}\mathcal O$. The new anchor map is defined by $$\frac{a}{s}\in S^{-1}\mathcal{A}\mapsto \frac{1}{s}S^{-1}(\rho_\mathcal{A}(a))\in\mathrm{Der}(S^{-1}\mathcal O),$$ and the new Lie algebra bracket is given by $$\left[\frac{1}{s}a,\frac{1}{u} b\right]_{S^{-1}\mathcal{A}} =\frac{1}{su}[a, b]_\mathcal{A}- \frac{\rho_{\mathcal{A}}(a) [u]}{su^2} \, b + \frac{\rho_{\mathcal{A}}(b) [s]}{s^2u} \, a$$
for $a,b\in\mathcal{A}, (s,u)\in S^2$. The localization map $\mathcal A\hookrightarrow S^{-1}\mathcal A$ is a Lie-Rinehart algebra morphism over the localization map $\mathcal O\hookrightarrow S^{-1}\mathcal O$.
    Since localization exists, the notion of sheaf of Lie-Rinehart algebras \cite{VillatoroJoel} over a projective variety, or a scheme, makes sense.
      \item \emph{Algebra extension.} Assume that the algebra $\mathcal O $  has no zero divisor, and let $\mathbb O $ be its field of fraction. For any subalgebra $\tilde{\mathcal O}$ with $\mathcal O \subset \tilde{\mathcal O} \subset \mathbb O$ such that $\rho(a) $ is for any $a \in \mathcal A$ valued in derivations of $ {\mathbb O}$ that preserves $ \tilde{\mathcal O} $, there is a natural Lie-Rinehart algebra structure  over  $\tilde{\mathcal O} $ on the space $ \tilde{\mathcal O}\otimes_{\mathcal O}\mathcal A$. 
\item {\emph{Blow-up at the origin.}} Let us consider a particular case of the previous construction, when $\mathcal O $ is the algebra $\mathbb C [x_1, \dots, x_N] $. If  the anchor map of a Lie-Rinehart algebra $\mathcal A $ over $\mathcal O $ takes values in vector fields on $\mathbb C^N $ vanishing at the origin, then for all $i=1, \dots, N$,   the polynomial algebra ${\mathcal O}_{U_i} $ generated by $ \frac{x_1}{x_i} , \dots,\frac{x_{i-1}}{x_i},$ $  x_i, \frac{x_{i+1}}{x_i},\dots, \frac{x_N}{x_i} $ satisfies the previous condition, and  ${\mathcal O}_{U_i} \otimes_\mathcal O \mathcal A$ comes equipped with a Lie-Rinehart algebra. Geometrically, this operation corresponds to taking the blow-up of $\mathbb C^N $ at the origin, then looking at the $i$-th natural chart $U_i $ on this blow-up: $\mathcal O_{U_i} $ are the polynomial functions on $U_i $. The family $ {\mathcal O}_{U_i} \otimes_{\mathcal O}\mathcal A$ (for $i=1, \dots,N$) is therefore an atlas for a sheaf of Lie-Rinehart algebras (in the sense of \cite{VillatoroJoel}) on the blow-up of $\mathbb C^N $ at the origin, referred to as the \emph{blow-up of $\mathcal A $ at the origin}.
\end{enumerate}

\subsection{On free resolutions of length $\leq 2$ and Lie-Rinehart algebras}\label{Lie 2-algebras}

In this section, we discuss the case when a Lie-Rinehart algebra $\mathcal{A}$ admits a free resolution of length $1$ and $2$. In those cases, we claim that there are Lie algebra-like structures on them. See \ref{app:complexes-modules} for the notion of free resolution of modules. These results will be soon generalized.\\

I start with the following remark (owed to Marco Zambon).

\begin{remark}\label{rmk:almost}
Any Lie-Rinehart algebra $(\mathcal A, \lb_\mathcal{A}, \rho_\mathcal{A})$ is the image of an almost differential graded Lie algebroid $(\mathcal{E}_{-1}, \rho)$ concentrated in degree $-1$: to see this, let $\{a_i\in \mathcal{A}\mid i\in I\}$ be a set of generators of $\mathcal{A}$ (take all elements of $\mathcal A$ if necessary). There exists elements $u^k_{ij}\in\mathcal{O}$, such that for given indices $i,j$, the coefficient $u^k_{ij}$ is zero except for finitely many indices $k$, together with
\begin{equation} 
\left[ a_i,a_j\right] _{\mathcal{A}}=\sum_{k\in  I}u^k_{ij} a_k \hspace{.5cm} \forall i,j \in I.
\end{equation} 
The coefficients $u_{ij}^k$ can be chosen to satisfy the skew-symmetry condition $u^k_{ij}=-u^k_{ji}$: by skew-symmetry of the bracket $\lb_\mathcal{A}$ on has \begin{align*}
    [a_i, a_j]_\mathcal{A}&=\frac{1}{2}\left([a_i, a_j]_\mathcal{A}-[a_j, a_i]_\mathcal{A}\right)\\&=\sum_{k\in I}\frac{1}{2}(u^k_{ij}-u^k_{ji})a_k.
\end{align*}Thus, one can replace $u_{ij}^k$ by $\frac{1}{2}(u^k_{ij}-u^k_{ji})$ if necessary.\\

Now, choose $\E_{-1}$ to be the free $\mathcal O$-module generated by the symbols $(e_i)_{i\in I}$ together with the surjective map  $$\pi\colon \E_{-1}\rightarrow \mathcal{A}, \; e_i\mapsto a_i.$$\\
\noindent
We now define: 
\begin{enumerate}
\item an anchor map by $\rho(e_i):=\rho_{\mathcal{A}}(a_i$), for all $i\in I$, i.e. $\rho=\rho_\mathcal{A}\circ\pi$,
\item a skew-symmetric operation $\lb_{\E_{-1}} $ on $ \E$ as follows:
$$[e_i,e_j]_{\E_{-1}} =\sum_{k\in I}u^k_{ij}e_k\;\text{for  all $i,j\in I$}.$$
\end{enumerate}
We extend by $\mathcal{O}$-linearity and Leibniz identity. By construction $(\mathcal{E}_{-1}, {\lb_{\E}}_{-1}, \rho$) is an almost Lie algebroid over $\mathcal{O}$ whose image through the anchor map is $\mathcal{A}$. 
\end{remark}

\begin{remark}In general, this bracket does not satisfy the Jacobi identity. If $\mathcal{E}_{-1}\simeq \mathcal{A}$, this bracket is a Lie algebroid.
\end{remark}

\subsubsection{On free resolutions of length $1$}
Lie-Rinehart algebra $\mathcal{A} $ admits a free resolution of length $1$ if and only if $\mathcal{A}$ is free. In that case, the almost Lie algebroid bracket 
$[\,\cdot, \cdot]_{\E_{-1}}$ is a Lie algebroid bracket. In conclusion: free resolutions of length $1$ admit a Lie algebroid structure.

\subsubsection{Free resolutions of length $2$}
Let $(\mathcal A, \lb_\mathcal{A}, \rho_\mathcal{A})$ be a Lie-Rinehart algebra that admits a free resolution of length $2$, namely an exact sequence of the form
 \begin{equation}\label{eq:length2}
    \xymatrix{0\ar[r]&\E_{-2}\ar[r]^{\ell_1}&\E_{-1}\ar@{.>>}@/_2pc/[rr]_\rho\ar@{->>}[r]^{\pi}&\mathcal{A}\ar@{-->}[r]^{\rho_\mathcal A}&\mathrm{Der}(\mathcal{O})}
\end{equation}with $\E_{-1},\, \E_{-2}$ free modules.\\
\noindent
Since Equation \eqref{eq:length2} is a free resolution of $\mathcal{A}$, the map $\pi$ is in particular surjective. By Remark \ref{rmk:almost},  $\E_{-1}$ can be endowed with an almost Lie algebroid bracket, such that $\pi$ is a morphism of brackets. We can extend this bracket to sections of degree $-2$ to obtain an almost differential graded Lie algebroid. Let us compute it:\\

\noindent
Let  $(e'_i)_{i\in {I}'}$ and  $(e_j)_{i\in{I}}$ be a basis of  $\E_{-2}$ and $\E_{-1}$, respectively. For all $i,j,k\in I\cup I'$ we have

\begin{enumerate}
    \item 

\begin{equation*}
        \pi([\ell_1e'_i,e_j]_{\E_{-1}})=[\pi\circ\ell_1(e'_i),\pi(e_j)]=0,\; \hbox{(since   $\pi\circ\ell_1\equiv 0$).}
    \end{equation*}
In other words, $[\ell_1(e'_i),e_j]_{\E_{-1}}\in \ker \pi$. By exactness of the complex \eqref{eq:length2} there exists an element  $\nabla_{e'_i}e_j\in \E_{ -2}$ such that \begin{equation}\label{eq:nabla}
    \pi(\nabla_{e'_i}e_j)=[\ell_1(e'_i),e_j]_{\E_{-1}}.
\end{equation}
Equation \eqref{eq:nabla} allows defining a bilinear map:
     $$\begin{array}{lll} \E_{-1} \otimes \E_{-2} &\to & \E_{-2}\\ \hspace{0.8cm}(x,y) & \mapsto & \nabla_x y \end{array}$$
     
     by extending the $\nabla_{e'_i}e_j$'s by linearity and Leibniz identity with the understanding that the anchor map $\rho$ vanishes on  $\E_{-2}$ in order to have 
     
     \begin{enumerate}
         \item $\ell_1(\nabla_{x}y)=[\ell_1(x),y]_{\E_{-1}}$,\;$\forall x\ \E_{ -2},\, y\in  \E_{ -1}$,
         \item for all $f\in \mathcal{O}$:
      $ \nabla_x fy =  f \nabla_x y + \rho(x)[f] \, y$ and $ \nabla_{fx} y = f  \nabla_{x}y$, \text{for  all}\, $x \in  \E_{-1}, y \in \E_{-2}$,
\end{enumerate}

\item Remember that $$\mathrm{Jac}(e_i,e_j,e_k):=[e_i,[e_j,e_k]_2]_2 +[e_j,[e_k,e_i]_2]_2+[e_k,[e_i,e_j]_2]_2\in \ker \pi.$$ By using again exactness of the complex \eqref{eq:length2} there is an element that we denote by $[e_i,e_j,e_k]_{E_{-1}}\in\E_{-2}$ that satisfies \begin{equation}
    \ell_1\left([e_i,e_j,e_k]_{\E_{-1}}\right)=\mathrm{Jac}(e_i,e_j,e_k).
\end{equation} Thus, we can define a skew-symmetric trilinear map:
     $$[\cdot\,,\cdot\,,\cdot]_{\E_{-1}}\colon\E_{-1}\wedge\E_{-1}\wedge\E_{-1}\longrightarrow \E_{-2}$$ such that 
     $$\ell_1([x,y,z]_{\E_{-1}})= [x,[y,z]_2]_2 +[y,[z,x]_2]_2+[z,[x,y]_2]_2, \; \forall x,y, z\in\E_{-1}. $$\end{enumerate}

The following Proposition concludes the discussion above.
\begin{proposition}\label{prop:2-Lie-algebroid}
Let $(\mathcal{A}, \lb_\mathcal A, \rho_\mathcal A)$ be a Lie-Rinehart algebra that admits a free resolution of length $2$ as in \eqref{eq:length2}


\begin{enumerate}
\item $\E_{-1}$ admits an almost Lie algebroid structure $(\E_{-1},[\cdot\,,\cdot]_{E_{-1}},\rho)$.
    \item There is 
     a bilinear map:
     $$   \begin{array}{lll} \E_{-1} \otimes \E_{-2} &\to & \E_{-2}\\ \hspace{0.8cm}(x,y) & \mapsto & \nabla_x y \end{array}    $$ 
     \item[] and a skew-symmetric trilinear map:
     $$[\cdot\,,\cdot\,,\cdot]_{\E_{-1}}\colon\E_{-1}\wedge\E_{-1}\wedge\E_{-1}\longrightarrow \E_{-2}$$ 
     \item[] such that  for all $f\in\mathcal{O}$:
     \begin{enumerate}
         \item  $ \nabla_x fy =  f \nabla_x y + \rho(x)[f] \, y$ and $ \nabla_{fx} y = f  \nabla_{x}y$, \text{for  all}\, $x \in  \E_{-1}, y \in \E_{-2}$,
         \item $[fx,y,z]_{\E_{-1}} = f [x,y,z]_{\E_{-1}}$ for  all\, $x,y,z \in  \E_{-1}$,
     \end{enumerate}
     such that the \emph{$2$-ary bracket} on $\E_{-1} \oplus \E_{-2} $ defined by: 
     $$ [x,y]_2 = \left\{  \begin{array}{ll} [x,y]_{\E_{-1}}  & \hbox{ for}\quad x,y \in \E_{-1}\\ 
     \nabla_x y  & \hbox{ for}\quad x \in \E_{-1},\, y \in \E_{-2} \\ 
     \nabla_y x  & \hbox{ for}\quad x \in \E_{-2},\, y \in \E_{-1}\\ 
     0  & \hbox{ for}\quad x,\,y \in \E_{-2}\\ 
     \end{array} \right. $$
     together with the \emph{$3$-ary bracket} on $\E_{-1} \oplus \E_{-2}$ defined by $[x,y,z]_3=[x,y,z]_{\E_{-1}}  $ if $x,y,z \in \E_{-1}$ and zero otherwise,
    satisfies
    \begin{enumerate}
    \item  for all $x\in\E_{-2}, y\in\E_{-1}$, \begin{equation}
        \ell_1([x,y]_2)+[\ell_1(x),y]_2=0,
        \end{equation}
        \item for all $x,y,z \in\E_{-1}$
     $$ \ell_1([x,y,z]_3)+  [x,[y,z]_2]_2 +[y,[z,x]_2]_2+[z,[x,y]_2]_2=0
    $$ 
    \item for all $x,y\in\E_{-1}$ and $ z \in \E_{-2}$
   $$
      [x,y,\ell_1(z)]_3 + [x,[y,z]_2]_2 +[y,[z,x]_2]_2+[z,[x,y]_2]_2=0.
   $$
    \end{enumerate}
      \end{enumerate}
\end{proposition}

The structure $(\E_\bullet,\ell_1, \rho,[\cdot\,,\cdot]_2, [\cdot\,,\cdot, \cdot]_3)$ described in Proposition \ref{prop:2-Lie-algebroid} is called \emph{Lie $2$-algebroid} \cite{BaezJohnC2003HAVL,LeanMadeleineJotz2017L2am, ShengYunhe2011HEoL}, i.e. a Lie $\infty$-algebroid with $\E_{-i}=0$ for $i\geq 3$.

A generalization of this construction on free resolutions of higher (even infinite) length is the object of the next Chapter.

\vspace{2cm}

\begin{tcolorbox}[colback=gray!5!white,colframe=gray!80!black,title=Conclusion:]
We describe Lie-Rinehart algebras, give examples and several constructions. Section \ref{Lie 2-algebras} is a pedagogical section, which presents an elementary case of the general case that we will study.
\end{tcolorbox}

\chapter{Main results of Part I}\label{Chap:main}
\section{Presentation of the problem}
In order to understand the geometry of affine varieties or some similar problems related to singular foliations using the Lie algebra of their vector fields which is in fact  Lie-Rinehart algebras, have motivated us to understand Lie-Rinehart algebras in general from another point of view.\\

We have seen in the Chapter \ref{chap-Lie-Rinehrt}, Remark \ref{rmk:basic} that any Lie $\infty$-algebroid over $\mathcal{O}$ induces a Lie-Rinehart algebra which we call its basic Lie-Rinehart algebra. In this chapter, we are investigating the opposite direction, i.e., we study the following questions: given a Lie-Rinehart algebra $\mathcal{A}$ over $\mathcal O$, can we find a Lie $\infty$-algebroid over $\mathcal O$ whose basic Lie-Rinehart algebra is $\mathcal{A}$? Also, in case where it exists, do we have uniqueness?\\

Now let us formalize that in a categorical language.

\noindent
We denote by {\textbf{Lie-$\infty $-alg-oids}/$\mathcal O$} the quotient category where \begin{enumerate}
    \item the objects are Lie $\infty$-algebroids over $\mathcal O$,
    \item arrows are homotopy equivalence classes of morphisms of Lie $\infty$-algebroids over $\mathcal O$.
\end{enumerate}and by \textbf{Lie-Rhart-alg}/$\mathcal O$ the category of Lie-Rinehart algebra over $\mathcal O$.
We want to study the functor,
$$\begin{tabular}{|c|c|c|}
  {\textbf{Lie-$\infty $-alg-oids}/$\mathcal O $}   & $\overset{\mathfrak{F}}{\longrightarrow}$ &  \textbf{Lie-Rhart-alg}/$\mathcal O$\\
     & $\overset{\exists !?}{\longleftarrow}$ &
\end{tabular}$$
The question now turn out to be: when does $\mathfrak{F}$ admit a left/right inverse?\\

In the next section, we present in detail the main results related to this question.
\section{Main results}
\label{sec:main}

\subsubsection{An existence Theorem}

The results below appeared in my first article \cite{CLRL} entitled  "Lie-Rinehart algebra $\simeq$ acyclic Lie $\infty$-algebroid" co-written with my supervisor C. Laurent-Gengoux.\\

This section extends the main results of \cite{LLS} from locally real analytic finitely generated singular foliations to arbitrary Lie-Rinehart algebras.\\

Here is our first main result. It states that universal Lie $ \infty$-algebroids over a given Lie-Rinehart algebra exist. 
  It extends Theorem 2.8 in \cite{LLS}. We are convinced that it may be deduced using the methods of semi-models categories as in Theorem 4.2 in \cite{Fregier}, but does not follow from a simple homotopy transfer argument.

\begin{theorem}\label{thm:existence}
Let $(\mathcal A, \lb_\mathcal A, \rho_\mathcal{A})$  be a Lie-Rinehart algebra over $ \mathcal O$.
Any resolution  of $\mathcal A $ by free $\mathcal O $-modules 
\begin{equation}
\cdots \stackrel{\ell_1} \longrightarrow\E_{-3} \stackrel{\ell_1}{\longrightarrow} \E_{-2} \stackrel{\ell_1}{\longrightarrow} \E_{-1} \stackrel{\pi}{\longrightarrow} \mathcal A \end{equation}
  comes equipped with a Lie $\infty $-algebroid structure whose unary bracket is $\ell_1 $ and that covers $\mathcal A $  through the hook $ \pi$.
\end{theorem}
Since any module admits free resolutions (see Proposition \ref{prop:free-resol}), Theorem \ref{thm:existence} implies that:

\begin{corollary}\label{cor:existence}
Any Lie-Rinehart algebra $\mathcal A$ over $\mathcal{O}$ is the basic Lie-Rinehart algebra of an acyclic Lie $\infty$-algebroid over $\mathcal{O}$.
\end{corollary}

While proving Theorem \ref{thm:existence}, we will see that if $\E_{-1} $ can be equipped with a Lie algebroid bracket (i.e. a bracket whose Jacobiator is zero), then all $k$-ary brackets of the universal Lie $ \infty$-algebroid structure may be chosen to be zero on $\E_{-1} $:

\begin{proposition} \label{prop:lksontNuls}
Let $(\E, \ell_1, \pi)$ be a free resolution of a Lie-Rinehart algebra $\mathcal{A}$. If $\E_{-1}$ admits a Lie algebroid bracket $[\cdot, \cdot] $ such that $\pi: \mathcal E_{-1} \to \mathcal A$ is a Lie-Rinehart morphism, then there exists a structure of universal Lie $\infty$-algebroid $(\mathcal E_\bullet, \ell_\bullet, \rho)$  that covers  $\mathcal{A}$ whose $2$-ary bracket  coincides with $[\cdot, \cdot]$ on $\mathcal E_{-1} $ and such that for every $k\geq 3$ the $k$-ary bracket $ \ell_k$ vanishes on $\bigodot^k\E_{-1}$.\end{proposition}

 \subsubsection{Universality of Theorem \ref{thm:existence} and its corollaries}

Here is our second main result.  It is related to Proposition 2.1.4 in \cite{Fregier} (but morphisms are not the same), and extends Theorem 2.9 in \cite{LLS}. 

\begin{theorem} \label{th:universal}
Let $(\mathcal A, \lb_\mathcal A, \rho_\mathcal{A})$ be a Lie-Rinehart algebra over $ \mathcal O$.  Given, 
\begin{enumerate}
    \item[a)] a Lie $ \infty$-algebroid $(\mathcal E_\bullet', \ell'_\bullet, \rho')$ that covers $\mathcal A $ through a hook $\pi'$, and
    \item[b)] any acyclic Lie $\infty $-algebroid $(\mathcal E_\bullet, \ell_\bullet, \rho)$ that covers $\mathcal A $ through a hook $\pi$,
\end{enumerate} 
then
\begin{enumerate} 
\item there exists a morphism of Lie $\infty $-algebroids from $(\mathcal E'_\bullet, \ell'_\bullet, \rho')$ to $(\mathcal E_\bullet, \ell_\bullet, \rho)$ over $\mathcal A$.
\item and any two such morphisms  are homotopic.  
\end{enumerate}
\end{theorem}

Here is an immediate corollary of Theorem \ref{th:universal}.

\begin{corollary}
\label{cor:unique} 
Any two acyclic Lie $\infty $-algebroids that cover a given Lie-Rinehart algebra are homotopy equivalent. 
This homotopy equivalence, moreover, is unique up to homotopy.
\end{corollary}

We will prove that the morphism that appears in Theorem \ref{th:universal} can be made trivial upon choosing a \textquotedblleft big enough\textquotedblright\,universal Lie $\infty$-algebroid:
\begin{proposition}\label{univ:precise}
Let $(\mathcal A, \lb_\mathcal A, \rho_\mathcal{A})$ be a Lie-Rinehart algebra over $ \mathcal O$.  Given a Lie $ \infty$-algebroid structure $(\mathcal E'_\bullet, \ell_\bullet', \rho')$ that terminates in $\mathcal A $ through a hook $\pi'$, 
then there exist an acyclic Lie $ \infty$-algebroid $(\mathcal E_\bullet, \ell_\bullet, \rho)$ that covers $\mathcal A $ through a hook $\pi$ such that
\begin{enumerate}
\item $\mathcal E $ contains $\mathcal E' $ as a subcomplex,  
\item the Lie $ \infty$-algebroid morphism from $\mathcal E' $ to $\mathcal E $ announced in Theorem \ref{th:universal} can be chosen to be the inclusion map  $\mathcal E' \hookrightarrow \mathcal E $ (i.e. a Lie $\infty$-morphism  where the only non-vanishing Taylor coefficient is the inclusion $\mathcal E' \hookrightarrow \mathcal E $).
\end{enumerate}
 \end{proposition}
 
The following Corollary follows immediately from Proposition \ref{univ:precise}: 
\begin{corollary}
Let $\mathcal{A}$ be a Lie-Rinehart algebra over $\mathcal{O}$ and $\mathcal{B}$ be a Lie-Rinehart subalgebra of $\mathcal{A}$. Any universal Lie $\infty$-algebroid of $\mathcal B$ can be contained in a  universal Lie $\infty$-algebroid of $\mathcal A$. \end{corollary}

\subsubsection{Induced Lie $\infty $-algebroids structures on ${\mathrm{Tor}}_{\mathcal O}(\mathcal A, \mathcal O/\mathcal I) $.}\label{ssec:Tor}

 Let $(\mathcal A, \lb_\mathcal A, \rho_\mathcal{A})$ be a Lie-Rinehart algebra over $\mathcal O $.

 \begin{definition}
    We say that an ideal $\mathcal I \subset \mathcal O $ is a \emph{Lie-Rinehart ideal of $\mathcal{A}$} if
  $ \rho_\mathcal A (a) [\mathcal I] \subset  \mathcal I$ for all $a \in \mathcal A $.
 \end{definition}

 \begin{remark}For any Lie-Rinehart ideal $\mathcal I\mathcal{O}$ of $\mathcal{A}$,
 \begin{enumerate}
     \item  we have $[\mathcal I \mathcal A , \mathcal A]_\mathcal{A} \subset \mathcal I \mathcal A $. Therefore, the quotient space $ \mathcal A /\mathcal I \mathcal{A}$ comes equipped with a natural Lie-Rinehart algebra structure over $ \mathcal O/ \mathcal I$.
     
     \item  For $(\mathcal E_\bullet, \ell_\bullet, \rho) $ an acyclic Lie  $\infty$-algebroid that covers $\mathcal A $, the quotient space $\mathcal E_\bullet / \mathcal I \simeq \mathcal O/\mathcal I \otimes_\mathcal O \mathcal E_\bullet $ comes equipped with an induced Lie $\infty $-algebroid structure: the $n$-ary brackets for $n \neq 2$ go to quotient by linearity, while for $n =2$, the $2$-ary bracket goes to the quotient in view of the relation $\rho_{\mathcal E} (\mathcal E_{-1}) [\mathcal I] \subset \mathcal I$.
 Also, $\pi $ goes to the quotient to a Lie-Rinehart algebra morphism $ \mathcal E_{-1}/\mathcal I \to \mathcal A/\mathcal I \mathcal A $. 
 \end{enumerate}
 \end{remark}

 \begin{definition}
     Let $\mathcal A $ be a Lie-Rinehart algebra over $\mathcal O $. For every Lie-Rinehart ideal $\mathcal I \subset \mathcal O $, we call  \emph{Lie $\infty$-algebroid of $\mathcal I $} the quotient Lie $\infty$-algebroid  $\mathcal E_\bullet/\mathcal I $, with $(\mathcal E_\bullet, \ell_\bullet, \rho) $  a universal Lie $\infty $-algebroid that covers $\mathcal A $.
 \end{definition}
 
 \begin{remark}
 \label{rmk:Tor}
  \normalfont
  The complexes on which the Lie $\infty$-algebroids of the ideal $\mathcal I $ are defined compute ${\mathrm{Tor}}_\mathcal O (\mathcal A, \mathcal O/\mathcal I) $ by construction (see \ref{app:tor}). 
 \end{remark}

 Moreover,  for any two universal Lie  $\infty$-algebroids  of $\mathcal A $, defined on $\mathcal E, \mathcal E'$ the homotopy equivalences $\Phi: \mathcal E' \to \mathcal E $ and $\Psi: \mathcal E \to \mathcal E'$,  whose existence is granted by  Corollary \ref{cor:unique}, go to the quotient and induce an homotopy equivalences between $ \mathcal O /\mathcal I \otimes_\mathcal O \mathcal E_\bullet  \simeq \mathcal E_\bullet / \mathcal I$ and $\mathcal O /\mathcal I \otimes_\mathcal O  \mathcal E_\bullet' \simeq  \mathcal E_\bullet' / \mathcal I$.
 The following corollary is then an obvious consequence of Theorem \ref{th:universal}.
 
 \begin{corollary}
 \label{cor:LRideal}
  Let $\mathcal A $ be a Lie-Rinehart algebra over $\mathcal O $. Let $\mathcal I \subset \mathcal O $ be a Lie-Rinehart ideal. Then any two Lie $\infty $-algebroids of $\mathcal I $ are homotopy equivalent, and there is a distinguished class of homotopy equivalences between them.
 \end{corollary}
 
 Taking under account Remark \ref{rmk:Tor}, here is an alternative manner to restate this corollary.
  
 \begin{corollary}
 \label{cor:LRideal2}
  Let $\mathcal A $ be a Lie-Rinehart algebra over $\mathcal O $. Let $\mathcal I \subset \mathcal O $ be a Lie-Rinehart ideal. Then the complex computing $ {\mathrm{Tor}}_\mathcal O^\bullet (\mathcal A, \mathcal O/\mathcal I) $ comes equipped with a natural Lie $\infty $-algebroid structure over $\mathcal O/ \mathcal I $, and any two such structures are homotopy equivalent in a unique up to homotopy manner.
\end{corollary}
 
 When, in addition to being a Lie-Rinehart ideal, $\mathcal I $ is a maximal ideal,  then $\mathbb K:= \mathcal O/\mathcal I $ is a field and Lie $\infty $-algebroids of $\mathcal I $ are a homotopy equivalence class of Lie $\infty $-algebras.  In particular their common cohomologies, which is easily seen to be identified to ${\mathrm{Tor}}_\mathcal O^\bullet(\mathcal A, \mathbb K) $ comes equipped with a graded Lie algebra structure. In particular,  ${\mathrm{Tor}}_\mathcal O^{-1}(\mathcal A, \mathbb K) $ is a Lie algebra, ${\mathrm{Tor}}_\mathcal O^{-2}(\mathcal A, \mathbb K) $ is a representation of this algebra, and the $3$-ary bracket defines a class in the third Chevalley-Eilenberg cohomology of  ${\mathrm{Tor}}_\mathcal O^{-1}(\mathcal A, \mathbb K) $  valued in ${\mathrm{Tor}}_\mathcal O^{-2}(\mathcal A, \mathbb K) $. This class does not depend on any choice made in its construction by the previous corollaries, and trivially extends the class called NMRLA-class in \cite{LLS}. If it is not zero, then there is no Lie algebroid of rank $r$ equipped with a surjective Lie-Rinehart algebra morphism onto $\mathcal A $, where $r$ is the rank of $\mathcal A $ as a module over $\mathcal O $. All these considerations can be obtained by repeating verbatim Section 4.5.1 in \cite{LLS} (where non-trivial examples are given).
 

 \subsubsection{A categorical approach of the results}
Theorem  \ref{th:universal} means that acyclic Lie $\infty $-algebroids that covers Lie-Rinehart algebra  $(\mathcal A, \lb_\mathcal{A}, \rho_\mathcal{A})$ are terminal objects in the subcategory of Lie-$\infty $-alg-oids/$\mathcal O $ whose objects are Lie $\infty $-algebroids that terminate in $\mathcal A $. Whence, acyclic Lie $\infty $-algebroids over $ \mathcal O$ that cover $(\mathcal A, \lb_\mathcal{A}, \rho_\mathcal{A})$ are deserved to be called "\emph{universal $\infty $-algebroids of $\mathcal{A}$}". From now on, we call them by this name.\\

\noindent 
Let us re-state Corollary \ref{cor:unique} differently.
By associating to any Lie $\infty $-algebroid its basic Lie-Rinehart algebra one obtains therefore a natural functor:
 \begin{itemize}
\item  from the category {\textbf{Lie-$\infty $-alg-oids}/$\mathcal O$},
\item to the category  \textbf{Lie-Rhart-alg}/$\mathcal O$
 \end{itemize}
 
\noindent 
Theorem   \ref{thm:existence} gives a right inverse of this functor. In particular,  this functor becomes  an equivalence of categories when restricted to homotopy equivalence classes of \underline{acyclic} Lie $\infty $-algebroids over $ \mathcal O$, i.e:

\begin{corollary}
There is an equivalence of categories between:
\begin{enumerate}
    \item[(i)] Lie-Rinehart algebras over $ \mathcal O$, 
    \item[(ii)] acyclic Lie $\infty $-algebroids over $ \mathcal O$.
\end{enumerate}
\end{corollary}
This corollary justifies the title of the first part of the thesis.

\begin{remark}
 \normalfont
In the language of categories, Corollary \ref{cor:LRideal}  means that there exists a functor from Lie-Rinehart ideals of a Lie-Rinehart algebra over $ \mathcal O$, to the category of Lie $\infty $-algebroids, mapping a Lie-Rinehart ideal $\mathcal I $ to an equivalence class of Lie $\infty $-algebroids over $\mathcal O / \mathcal I $. 
\end{remark}

\section{Proof of main results}

\subsection{A crucial bi-complex: ${\mathfrak{Page}}^{(n)}(\E',\E)$}\label{interpretaion-bicomplex}

\subsubsection{Description of ${\mathfrak{Page}}^{(n)}(\E',\E)$}

\label{bi-com}
Let $ \mathcal V$ be an $ \mathcal O$-module, and let 
 $ (\E, \dd , \pi) $ and $ (\E', \dd' , \pi') $ be complexes of projective $\mathcal O $-modules that terminates at $ \mathcal V$:
\begin{equation}
  \label{eq:EE'}\cdots\xrightarrow{\dd}\E_{-2}\xrightarrow{\dd}\E_{-1}\xrightarrow{\pi} \mathcal V, \quad \cdots\xrightarrow{\dd'}\E'_{-2}\xrightarrow{\dd'}\E'_{-1}\xrightarrow{\pi'} \mathcal V.
\end{equation}

  For every $k \geq 1$, the $(k+1)$-th graded symmetric power  $ \bigodot^{k+1} \E'$ of $ \mathcal E'$ over $ \mathcal O$ is a projective $\mathcal O $-module, and comes with a natural grading induced by the grading on $\E' $. 

\begin{definition}\label{def:bi-com}
Let $ k \in \mathbb{N}_0$. We call \emph{page number $k$ of $(\E,\dd,\pi) $ and $ (\E, \dd',\pi')$}  the bicomplex of $\mathcal O$-modules on the upper left quadrant $ \mathbb Z^- \times \mathbb{N}_0$ defined by:
\begin{align}
\label{eq:jgeq0}
\text{Page}^{(k)}(\E', \E)_{j,m}&:= \text{Hom}_{\mathcal O} \left( \bigodot^{k+1}  \E' \, _{|_{-k-m-1}} \, , \, \E_{j}\right),& \hspace{0.2cm }\hbox{  for $ m \geq 0  $ and $ j \leq -1 $}
\\
\label{eq:j=1}
\text{Page}^{(k)}(\E',\E)_{0,m}&:=\text{Hom}_{\mathcal O} \left( \bigodot^{k+1}  \E' \, _{|_{-k-m-1}}\, ,\,  \mathcal V \right),  & \hspace{0.2cm }\hbox{ for $ m \geq 0  $,}
\end{align}
together with the vertical differential $D^v$ defined 
for any one  of the two $\mathcal O$-modules \eqref{eq:jgeq0} or \eqref{eq:j=1} by 
$$\begin{array}{cccllll}
     D^v\colon \text{Page}^{(k)}(\E', \E)_{j,m} &\longrightarrow& \text{Page}^{(k)}(\E', \E)_{j,m+1}\\&&&&&\\
    \Phi & \longmapsto &D^v(\Phi)\colon \bigodot^{k+1}\E'& \longrightarrow \E_j\\&&&&\\& &\phantom{ccc} x_1\odot\ldots\odot x_{k+1}& \mapsto \Phi\circ\dd'\,(x_1\odot\cdots\odot x_{k+1})
\end{array}$$
where $\dd'$ acts as an $\mathcal O$-derivation on $x_1\odot\ldots\odot x_{k+1}\in\bigodot^k\E'$ (and is 0 on  $\E_{-1}'$). The horizontal differential, $D^h\colon \text{Page}^{(k)}(\E', \E)_{j,m} \longrightarrow \text{Page}^{(k)}(\E', \E)_{j+1,m}$, is given by $$\Phi \mapsto \dd \circ \Phi\, \,\hbox{ or }\,\, \Phi \mapsto \pi\circ\Phi$$depending on whether $\Phi$ is of type \eqref{eq:jgeq0} with $j \leq -2 $ or the type \eqref{eq:jgeq0} with $j=-1$. It is zero on elements of type \eqref{eq:j=1}. We denote by $\left(\mathfrak{Page}^{(k)}_\bullet(\E',\E), D\right) $ its associated total complex. When $\E'=\E$ we shall write $\mathfrak{Page}^{(k)}_\bullet(\E)$ instead of $\mathfrak{Page}^{(k)}_\bullet(\E,\E)$.
\end{definition}

The following diagram recapitulates the whole picture of $\mathfrak{Page}^{(k)}_\bullet(\E',\E)$:

\begin{equation}\label{recap}
\scalebox{0.8}{ \hbox{$
	\begin{array}{ccccccccccc}
		& & \vdots & & \vdots & & \vdots & & 
		\\ 
		& & \uparrow & & \uparrow & & \uparrow & & 
		\\ 
		\cdots& \rightarrow & \text{Hom}_\mathcal O\left(\bigodot^{k+1} \E'\,_{|_{-k-3}},\E_{-2}\right) & \overset{D^h}{\rightarrow} & \text{Hom}_\mathcal O\left(\bigodot^{k+1} \E'\,_{|_{-k-3}},\E_{-1}\right) 
		& \overset{D^h}{\rightarrow} & \text{Hom}_\mathcal O\left(\bigodot^{k+1} \E'\,_{|_{-k-3}},\mathcal{V}\right) & \rightarrow & 0
			\\ 
		& & D^v\uparrow & & D^v\uparrow & & D^v\uparrow & & 
		\\ 
		\cdots& \rightarrow & \text{Hom}_\mathcal O\left(\bigodot^{k+1} \E'\,_{|_{-k-2}},\E_{-2}\right) & \overset{D^h}{\rightarrow} & \text{Hom}_\mathcal O\left(\bigodot^{k+1} \E'\,_{|_{-k-2}},\E_{-1}\right) 
		& \overset{D^h}{\rightarrow} & \text{Hom}_\mathcal O\left(\bigodot^{k+1}\E'\,_{|_{-k-2}},\mathcal{V}\right) & \rightarrow& 0
		\\ 
		& & D^v\uparrow & & D^v\uparrow & & D^v\uparrow & &  
		\\ 
		\cdots& \rightarrow & \text{Hom}_\mathcal O\left(\bigodot^{k+1} \E'\,_{|_{-k-1}},\E_{-2}\right) & \overset{D^h}{\rightarrow} & \text{Hom}_\mathcal O\left(\bigodot^{k+1}\E'\,_{|_{-k-1}},\E_{-1}\right) 
		& \overset{D^h}{\rightarrow}& \text{Hom}_\mathcal O\left(\bigodot^{k+1}\E'\,_{|_{-k-1}},\mathcal{V}\right) & \rightarrow & 0
	\\
	& & \uparrow & & \uparrow & & \uparrow & & 
	\\ 
	& & 0 & & 0 & & 0 & &  \\
	& & \hbox{\small{\texttt{"-$2$ column"}}} & & \hbox{\small{\texttt{"-$1$ column"}}} & & \hbox{\small{\texttt{"last column"}}} & &  	
\end{array}
$}}
\end{equation}

For later use, we spell out the meaning of being $D$-closed.

\begin{remark}
\label{lem:beingClosed}
An element $P \in \mathfrak{Page}^{(k)}_j(\E',\E)$  in $ \oplus_{ i \geq 1 }\mathrm{Hom}_\mathcal O\left(\bigodot^{k+1}\E'\,_{|_{-j-i}},\E_{-i}\right)$ is $D$-closed if and only if:
\begin{enumerate}
    \item the component $ P_{-1}\colon \bigodot^{k+1} \mathcal E'\,_{|_{-j-1}} \to \mathcal E_{-1}$ is valued in the kernel of $ \pi \colon \E_{-1} \to \mathcal V$,
    \item the following diagram commutes:
     $$ \xymatrix{\bigodot^{k+1} \mathcal E' \,_{|_{-j-i}}\ar[rr]^{(-1)^j\dd'} \ar[d]_{P_{-i}} &&\bigodot^{k+1} \mathcal E' \,_{|_{-j-i+1}} \ar[d]^{P_{-i+1}}\\ \E_{-i}  \ar[rr]_\dd && \E_{-i+1}}$$ 
     with $P_{-i} $ being the component of $P$ in $\mathrm{Hom}_\mathcal O\left(\bigodot^{k+1}\E'\,_{|_{-j-i}},\E_{-i}\right) $. 
\end{enumerate}
For $(\E, \dd) = (\E', \dd')$, the second condition above also reads $[\dd,P]_{\tiny{\mathrm{RN}}}=0$. Here is an important  technical result.
\end{remark}
\begin{proposition}
\label{bicomplex}
Let $ (\E, \dd, \pi)$ be a resolution of $ \mathcal V$ in the category of $ \mathcal O$-modules. Then,
for every $k \geq 0$,
\begin{enumerate}
    \item  the cohomology of the complex $(\mathfrak{Page}^{(k)}_\bullet(\E) , D) $  for the total differential $D_\bullet:=D^h-(-1)^\bullet D^v$ is zero in all degrees;
    \item\label{2} Moreover, a $D$-closed element whose component on the \textquotedblleft    last column\textquotedblright\,of the diagram above is zero is the image through $D$ of some element whose  two last components are also zero.
   \item\label{3}
    More generally, for all $n \geq 1 $, for a $D$-closed element $P \in \mathfrak{Page}^{(k)}_j(\E)$  of the form $$ \bigoplus_{ i \geq n } \mathrm{Hom}_\mathcal O\left(\bigodot^{k+1}\E'\,_{|_{-j-i}},\E_{-i}\right),$$ one has $P=D(R) $ and $R \in  \mathfrak{Page}^{(k)}_{j-1}(\E) $ can be chosen in $\bigoplus_{ i \geq n+1 } \mathrm{Hom}_\mathcal O\left(\bigodot^{k+1}\E'\,_{|_{-j-i+1}},\E_{-i} \right) $.
\end{enumerate}
\end{proposition}
\begin{proof} 
Since $ \bigodot^k \mathcal E' |_{j+m-k} $ is a projective $ \mathcal O$-module for all $(j,m)\in\mathbb{Z^-}\times\mathbb{N}_0$, and  $ (\E, \dd, \pi)$ is a resolution, all the lines of the above bicomplex are exact. This proves the first item. The second and the third are obtained by diagram chasing (see Appendix \ref{app:proj-res} for more details).
\end{proof}

\subsubsection{Interpretation of $\mathfrak{Page}^{(n)}(\E',\E)$ in our context}
We denote by $Q^{(0)}_\E $ and $ Q_{\E'}^{(0)}$ the differentials of polynomial-degree $0$ on $\bigodot^\bullet\mathcal E$ and $ \bigodot^\bullet\mathcal E'$ induced by $\dd $ and $\dd'$. Whence, $\left(\mathrm{Hom}_\mathcal{O}(\bigodot^\bullet\mathcal E',\bigodot^\bullet\mathcal E), \partial\right)$ with $$\partial\colon H \mapsto Q_\E^{(0)} \circ H - (-1)^{|H|}  H \circ Q_{\E'}^{(0)}$$ is a complex of $\mathcal{O}$-modules (see Lemma \ref{lemma:hom-complex}).

\begin{lemma}\label{subcomplex:coder}Let $\phi\colon (\mathcal E', \dd') \to (\E, \dd)  $ be a chain map,
and let $\Phi^{(0)}\colon\bigodot^\bullet\mathcal E' \to  \bigodot^\bullet  \mathcal E $ be its extension to a co-algebra morphism, namely:
 $$ \Phi^{(0)} (x_1 \cdot \dots \cdot x_n):= \phi (x_1) \cdot \dots \cdot \phi(x_n).$$

\begin{enumerate}

\item We have, $\Phi^{(0)}\circ Q_{\E'}^{(0)}=Q_\E^{(0)}\circ \Phi^{(0)}$.
    \item $\Phi^{(0)} $-co-derivations 
    form a sub-complex of $\left(\mathrm{Hom}_\mathcal{O}(\bigodot^\bullet\mathcal E',\bigodot^\bullet\mathcal E), \partial\right)$.
\end{enumerate}

\end{lemma}
\begin{proof}
The first item can be easily checked. The second item follows exactly the same pattern as Lemma \ref{linearity}.
\end{proof}
 The following proposition is very important.
\begin{proposition}
\label{prop:interpretation}
For every $k \in \mathbb{N}_0$, and $\phi\colon (\E', \dd') \to (\E,\dd) $ be a chain map as in Lemma \ref{subcomplex:coder}. The sub-complex of $\Phi^{(0)}$-co-derivations of polynomial-degree $k$ is isomorphic to the complex $\widehat{\mathfrak{Page}}^{(k)}(\E',\E)$ obtained from  $\mathfrak{Page}^{(k)}(\E',\E)$ by crossing its \textquotedblleft last column\textquotedblright, see diagram \eqref{recap}.
\end{proposition}
\begin{proof}
The chain isomorphism $\delta\colon \mathrm{Hom}_\mathcal{O}(\bigodot^\bullet\mathcal E',\bigodot^\bullet\mathcal E)\rightarrow\mathrm{Hom}_\mathcal{O}(\bigodot^\bullet\mathcal E',\mathcal E)$ consists in mapping a $\Phi^{(0)} $-co-derivation $H $ of polynomial-degree $k$ and degree $j$ to its unique Taylor coefficient $H_k\in{\mathfrak{Page}}_j^{(k)}(\E',\E)$, i.e. $\delta(H)=\mathrm{pr}\circ H$. Let us check that this map is indeed a chain map: for every $x=~x_1\odot\cdots\odot x_{k+1}\in \bigodot^{k+1}\E'$ one has,\begin{align*}
    \delta\circ(Q_\E^{(0)} \circ H - (-1)^{|H|}  &H \circ Q_{\E'}^{(0)})(x)\\&=\ell_1\circ H_k(x)(-1)^{|H|}  H_k \left(\sum_{i=1}^{k+1}-(-1)^{(|x_1|+\cdots+|x_{i-1}|)|x_i|} \ell_1'(x_i)\odot x_1\odot\cdots\odot x_{k+1}\right)\\&=D^h(H_k)(x)-(-1)^{|H_k|}D^v(H_k)(x),\quad \text{by definition  of $D^h$ and $D^v$}\\&=D(H_k)(x)\\&=D\circ \delta(H)(x),\quad\text{by definition of $\delta$}.
\end{align*}
\end{proof}

Here is another type of interpretation for $\widehat{\mathfrak{Page}}^{\bullet}(\E',\E)$ involving the Richardson-Nijenhuis bracket.

\begin{prop}\label{rem:bicom-rch}
\cite{zbMATH03130695}
When $\E' = \E $, $\widehat{\mathfrak{Page}}^{\bullet}(\E',\E)$ with no $0$-column is the bi-graded complex  of \emph{exterior forms on $\E $} and the differential $D$ of $\widehat{\mathfrak{Page}}^{\bullet}(\E',\E)$ is  $D(\cdot) =[\mathrm d ,\cdot\,]_{\hbox{\tiny{\emph{RN}}}}$.\end{prop}

We finish the section with the following lemma that will be important to prove Proposition \ref{univ:precise}. It uses the  consequence of the cone construction (see Appendix \ref{appendix:mod}).
\begin{lemma} \label{lem:inFactInclusion}Let $(\mathcal R,\dd^\mathcal R,\pi^\mathcal R)$ be an arbitrary complex of projective $\mathcal O$-module that terminates in a $\mathcal O$-module $\mathcal V$.
There exists a projective resolution $(\mathcal E,\dd^{\mathcal E}, \pi^{\mathcal E})$ of $\mathcal V $, which contains $(\mathcal R,\dd^\mathcal R,\pi^\mathcal R)$ as a sub-complex. Moreover, we can assume that $\mathcal R$ admits a projective submodule in $\E$ in direct sum.
\end{lemma}
\begin{proof}Resolutions of an $\mathcal O $-modules $\mathcal V $ are universal objects in the category of complexes of projective $\mathcal O $-modules. In particular, for every projective resolution $(\mathfrak F,\dd^\mathfrak F, \pi^\mathfrak F)$ of $\mathcal V $, there exist a (unique up to homotopy) chain map:
 $$  \phi \colon (\mathcal R,\dd^\mathcal R,\pi^\mathcal R) \to (\mathfrak F,\dd^\mathfrak F, \pi^\mathfrak F)  . $$
We apply the cone construction (see, e.g. \cite{Weibel}, Section 1.5) to: 
\begin{enumerate}
    \item the complex  $(\mathcal R,\dd^\mathcal R,\pi^\mathcal R)$
    \item the direct sum of the complexes $(\mathcal R,\dd^\mathcal R)$ and $(\mathfrak F,\dd^\mathfrak F)$  namely, $\left(\mathcal R\oplus\mathfrak F,\dd^\mathcal R\oplus\dd^\mathfrak F, \pi^\mathcal R\oplus\pi^\mathfrak F \right) $ 
    \item \label{chain:cone} the chain map obtained by mapping any $ x  \in \mathcal R$ to $ (x, \phi(x)) \in \mathcal R\oplus\mathfrak F$.
\end{enumerate}
The differential is given by \begin{align}
   \dd^{\E}(x, y, z)=(-\dd^\mathcal Rx , \dd^\mathcal Ry-x , \dd^\mathfrak F z-\varphi(x)) 
\end{align} for all $(x,y,z)\in\E_{-i}=\mathcal R_{-i+1}\oplus\mathcal R_{-i}\oplus\mathfrak F_{-i}$, $i\geq 2$. Since the chain given in item \ref{chain:cone} is a quasi-isomorphism, its cone is an exact complex. We truncate the latter at degree $-1$ without destroying its exactness by replacing the cone differential at degree $-1$ as follows:  $\pi^{\E}\colon\mathcal R_{-1}\oplus\mathfrak F_{-1}\rightarrow \mathcal V,\;(r,e)\mapsto\pi^{\mathfrak F}(e)-\pi^{\mathcal R}(r)$. For a visual description, see Equation \eqref{sub:resol} below: the resolution of $\mathcal V $ described in Lemma \ref{lem:inFactInclusion} is defined by:
\begin{equation}\label{sub:resol}
\xymatrix{ \cdots\ar[r] & \mathfrak F_{-3} \ar^{\dd^\mathfrak F}[rr]&&\mathfrak F_{-2} \ar^{\dd^\mathfrak F}[rr] && \mathfrak F_{-1} \ar^{\pi^\mathfrak F}[rr]&&  \mathcal V\\\cdots \ar[r]& \mathcal R_{-3} \ar^{\dd^\mathcal R}[rr]&& \mathcal R_{-2} \ar^{\dd^\mathcal R}[rr]&& \mathcal R_{-1}  \ar_{\pi^\mathcal R}[rru] && \\\cdots\ar[r] &\mathcal R_{-2} \ar_{\mathrm{id}}[urr] \ar_{\dd^\mathcal R}[rr] \ar@/_/^<<<<<{\phi}[uurr] && \mathcal R_{-1} \ar_{\mathrm{id}}[urr] \ar@/_/^<<<<<{\phi}[uurr]&& &&  \\}
\end{equation}
The proof of the exactness of this complex is left to the reader.


\noindent
The henceforth defined complex $(\E,\dd^{\E}, \pi^{\E})$ is a resolution of $\mathcal V $, and obviously contains $(\mathcal R,\dd^\mathcal R,\pi^{\mathcal R})$  as a sub-chain complex of $\mathcal O $-modules. 
\end{proof}

Let $(\mathcal E,\dd^\mathcal E, \pi^\mathcal E)$ be a free resolution of $\mathcal V$ and $(\mathcal R,\dd^\mathcal R,\pi^\mathcal R)$ a subcomplex of projective $\mathcal O$-modules, as in Lemma \ref{lem:inFactInclusion}. We say that  $P \in \mathfrak{Page}^{(k)}_j(\E)$  of the form $ \oplus_{ i \geq n } \text{Hom}_\mathcal O\left(\bigodot^{k+1}\mathcal E\,_{|_{-j-i}},\E_{-i}\right) $ \emph{preserves} $\mathcal R $ if $\bigodot^{k+1}\mathcal R\,_{|_{-j-i}}$ is mapped by $P$  to $\mathcal R_{-i}$ for all possible indices.
In such case, it defines by restriction to $\bigodot^\bullet\mathcal R$ an element $\iota^*_\mathcal R P$ in the graded $\mathcal{O}$-module $\mathfrak{Page}^{(k)}_j(\mathcal R):= \oplus_{ i \geq n } \text{Hom}_\mathcal O\left(\bigodot^{k+1}\mathcal R\,_{|_{-j-i}},\mathcal R_{-i}\right) 
$. For the sake of clarity, let us denote by $D^\E $ and $D^{\mathcal R} $ the respective differentials of the bi-complexes $ \mathfrak{Page}^{(k)}_{j}(\E)$
and $  \mathfrak{Page}^{(k)}_{j}(\mathcal R)$ and by $D^\E_h$, $D^\mathcal R_h $ and $D^{\mathcal R}_v$, $D^\mathcal R_v $ the horizontal differential resp. vertical differential, of their associated bi-complexes. Also, $\iota^*_\mathcal R P$  stands for the restriction of $ P \in \mathfrak{Page}^{(k)}_j(\E)$ to $\bigodot^\bullet\mathcal R$ (a priori it is not valued in $\mathcal R$ but in $\E$).

\begin{lemma}\label{rest:univ}
Let $(\E,\dd^\E, \pi^\E)$  be a free resolution of $\mathcal V$. Let $\mathcal R \subset \E$ be a subcomplex made of free sub-$\mathcal O$-modules such that there exists a graded free $\mathcal O $-module $\mathcal V $ such that $\E =\mathcal R \oplus \mathcal V$.
\begin{enumerate}
    \item  For every $k\geq 0$, a $D^\E$-cocycle   $P \in \mathfrak{Page}^{(k)}_j(\E)$ which preserves $\mathcal R$ is the image through $D^\E$ of some element $Q \in \mathfrak{Page}^{(k)}_{j-1}(\E) $ which preserves $\mathcal R$ if and only if its restriction $\iota^*_\mathcal R P \in \mathfrak{Page}^{(k)}_j(\mathcal R)$ is a $D^\mathcal R$-coboundary.
    \item In particular, if the restriction of $\dd^\E $ and $\pi^\mathcal E$ to $\mathcal R$  makes it a resolution of $\pi^\E(\mathcal R_{-1}) \subset \mathcal V$, then any $D^\E$-cocycle   $P \in \mathfrak{Page}^{(k)}_j(\E)$ which preserves $\mathcal R$ is the image through $D^\E$ of some element $Q \in \mathfrak{Page}^{(k)}_{j-1}(\E) $ which preserves $\mathcal R $.
\end{enumerate}
\end{lemma}
\begin{proof}
Let us decompose the element $P\in\mathfrak{Page}^{(k)}_{j}(\E)$ as $P=\sum_{i\geq 1}P_i$ with, for all $i \geq 1$, $P_i $ in $\mathrm{Hom}_\mathcal O\left(\bigodot^{k+1}\mathcal E\,_{|_{-j-i}},\E_{-i}\right)$. Assume $P\in \mathfrak{Page}^{(k)}_{j}(\E)$ is a $D^\E$-cocycle which preserves $\mathcal R$.

Let us prove one direction of item 1. If $P$ is the image through $D^\E$ of some element $Q\in\mathfrak{Page}^{(k)}_{j-1}(\E)$ which preserves $\mathcal R$, then $D^{\mathcal R} (\iota^*_\mathcal R Q)=\iota^*_\mathcal R D^\E (Q)= \iota^*_\mathcal R P$, with $\iota^*_\mathcal R Q\in\mathfrak{Page}^{(k)}_{j-1}(\mathcal R)$. Thus, the restriction $\iota^*_\mathcal R P \in \mathfrak{Page}^{(k)}_j(\mathcal R)$ of $P$ is a $D^\mathcal R$-coboundary.

Conversely, let us assume that  $\iota^*_\mathcal R P \in \mathfrak{Page}^{(k)}_j(\mathcal R)$ is a $D^{\mathcal R}$-coboundary, i.e. 
 $ \iota^*_\mathcal R P= D^{\mathcal R} Q_{\mathcal R} $ for some $Q_{\mathcal R}\in\mathfrak{Page}^{(k)}_{j-1}(\mathcal R)$.
Take $\hat{Q}\in \mathfrak{Page}^{(k)}_{j-1}(\E)$ any extension of $Q_{\mathcal R}$ (e.g. define $\hat Q$ to be $0$ as soon as one element in $\mathcal V$ is applied to it).
Then $P- D^\E (\hat Q):\bigodot^{k+1}\E\longrightarrow\E$ is zero on $ \bigodot^{k+1} \mathcal R$. We have to check that it is a $D^\E$-coboundary of a map with the same property.
\noindent
Put $\kappa=P-D^\E (\hat Q)$. By Proposition \ref{bicomplex}, item 1, there exists $\tau\in \mathfrak{Page}^{(k)}_{j-1}(\E)$ such that $D^\E(\tau)= \kappa$. The equation $D^\E(\tau)=\kappa$ is equivalent to the datum of a collection of equations\begin{align}\label{eq:tau}
    D^\E_v(\tau_i)+D^\E_h(\tau_{i+1})=\kappa_{i+1}, i\geq 1,\quad \text{and}\quad D^\E_h(\tau_1)=\kappa_1,
\end{align}
{with, $\tau_i\in\mathrm{Hom}_\mathcal O\left(\bigodot^{k+1}\mathcal E\,_{|_{-j-i+1}},\E_{-i}\right)$ and  $\kappa_i\in\mathrm{Hom}_\mathcal O\left(\bigodot^{k+1}\mathcal E\,_{|_{-j-i}},\E_{-i}\right)$ for every $i \geq 1$.
Since $\iota^*_\mathcal R\kappa_1=0$, we have that $D^\mathcal E_h(\iota^*_\mathcal R\tau_1)=\iota^*_\mathcal R \left(D^\E_h(\tau_1)\right)=0$, (with the understanding that $\iota^*_\mathcal R{\tau_1}_{|_\mathcal V}\equiv 0$). Using the exactness of the horizontal differential $D^\E_h$, there exists $C_1\in \mathfrak{Page}^{(k)}(\mathcal E)$ such that $D^\mathcal E_h(C_1)=\iota^*_\mathcal R\tau_1$. We now change $\tau_1$ to $\tau_1'$ and $\tau_2$ to $\tau'_2$ by putting $\tau_1':=\tau_1-\iota^*_\mathcal R\tau_1$ and $\tau_2':=\tau_2+D^\mathcal E_v(C_1)$. One can easily check that Equation \eqref{eq:tau} still holds  under these changes, i.e.,   $$D^\mathcal E_v( \tau'_1)+D^\mathcal E_h( \tau'_2)=\kappa_2\quad \text{and} \quad D^\mathcal E_h(\tau'_2)+D^\mathcal E_v(\tau_3)=\kappa_3.$$ We can therefore choose $\tau$ such that $\iota^*_\mathcal R\tau_1=0$.} We then iterate this procedure, which allows us to choose $\tau \in \mathfrak{Page}^{(k)}_{j-1}(\E)$ such that $\iota^*_\mathcal R \tau=0$ and $D^\E(\tau)=\kappa$.
By construction, $Q:=\tau+\hat Q$  preserves $\mathcal R$, while $ \iota_{\mathcal R}^* Q =Q_\mathcal R$, and $D^\mathcal E (Q)=P$. The second item follows from the first one.
\end{proof}

\subsection{Proof on the existence}\label{thm:existence-proof}
In this section, we prove Theorem \ref{thm:existence}.\\

Let $(\mathcal{A}, \lb_\mathcal{A}, \rho_\mathcal{A})$ be a Lie-Rinehart algebra. Consider $ (\E, \dd= \ell_1, \pi) $ a  resolution of $ \mathcal A$ by free $ \mathcal O$-modules: such resolutions always exist, see Proposition \ref{prop:free-resol}.
To start, we define a binary bracket $\ell_2$. The pair $ (\dd, \ell_2)$ will obey the axioms of the object that we now introduce.

\begin{definition}\cite{LavauSylvain}
\label{def:al-oid}
An \emph{almost differential  graded Lie algebroid of a Lie-Rinehart algebra} $ (\mathcal A, \rho_\mathcal A,  [\cdot\,, \cdot]_\mathcal A)$ is a complex 
$$\cdots\stackrel{\dd}{\longrightarrow} \mathcal E_{-3}\stackrel{\dd}{\longrightarrow} \mathcal E_{ -2}  \stackrel{\dd}{\longrightarrow}  \mathcal E_{ -1} \stackrel{\pi}{\longrightarrow}    \mathcal A$$
of projective $\mathcal O $-modules equipped a graded almost differential graded Lie algebroid $(\E_\bullet,\ell_1, \ell_2, \rho)$ over $\mathcal{O}$ such that
\begin{enumerate}
\item  $\rho= \rho_{\mathcal A} \circ \pi \colon \mathcal E_{-1} \longrightarrow {\mathrm{Der}}(\mathcal O) $, 
\item $\pi $ is a morphism, i.e. for all  $x,y \in \mathcal E_{-1}$
$$ \pi (\ell_2 (x,y)) = [\pi (x), \pi (y)]_\mathcal A .$$
 \end{enumerate}
\end{definition}

We start by proving this lemma.
\begin{lemma}\label{lem1}
Every free resolution $ (\E, \dd, \pi )$ of a Lie-Rinehart algebra $(\mathcal A , [ \cdot\,, \cdot]_\mathcal{A}, \rho_\mathcal A) $ comes equipped with a binary bracket $\ell_2$ that makes it an almost differential  graded Lie algebroid of $\mathcal A $.
\end{lemma}
\begin{proof}
For all $ k \geq 1$, let us denote by 
$(e_i^{(-k)})_{i \in I_k}$
a family of generators of the free $ \mathcal O$-module~$\E_{-k} $. 
By construction $\{a_i=\pi(e_i^{(-1)})\in \mathcal{A}\mid i\in I_{1}\}$ is a set of generators of $\mathcal{A}$. In particular, there exists elements $u^k_{ij}\in\mathcal{O}$, such that for given indices $i,j$, the coefficient $u^k_{ij}$ is zero except for finitely many indices $k$, 
and satisfying the skew-symmetry condition $u^k_{ij}=-u^k_{ji}$ together with
\begin{equation} 
\left[ a_i,a_j\right] _{\mathcal{A}}=\sum_{k\in I}u^k_{ij} a_k \hspace{.5cm} \forall i,j \in I_{1} 
\end{equation}
We now define: 
\begin{enumerate}
\item an anchor map by $\rho(e_i^{(-1)})=\rho_{\mathcal{A}}(a_i$) for all $i\in I$,
\item a degree $ +1$ graded symmetric operation $\tilde{\ell}_2 $ on $ \E$ as follows:
\begin{enumerate}
    \item $\tilde{\ell}_2\left(e_i^{(-1)},e_j^{(-1)} \right) =\sum_{k\in I}u^k_{ij}e_k^{(-1)}$ for all $i,j\in I_{-1}$.
    \item 
    $\tilde{\ell}_2 \left(e_i^{(-k)},e_j^{(-l)} \right) =0$  for all $ i \in I_{k}, j \in I_l$ with $ k \geq 2$ or  $ l \geq 2$.  
\item we extend $\tilde{\ell}_2$ to $\E $ using $\mathcal O$-bilinearity and Leibniz identity with respect to the anchor $\rho $.
\end{enumerate} 
\end{enumerate} 

By construction, $\tilde{\ell}_2$ satisfies the Leibniz identity with respect to the anchor $ \rho_\mathcal E$. Also, $\rho\circ\dd=0$ on $\E_{-2}$. The map defined for all homogeneous $x,y\in \mathcal{E}$ by$$
 [\dd, \tilde{\ell}_2]_{\hbox{\tiny{RN}}}(x,y) = \dd \circ \tilde \ell_2 \left( x,y\right)+ \tilde \ell_2 \left(\dd x,y \right) +(-1)^{\lvert x\rvert} \tilde \ell_2 \left(x,\dd y \right),
$$
is a graded symmetric degree $ +2$ operation $ (\E \otimes \E)_\bullet \longrightarrow \E_{\bullet +2} $, and $[\dd, \tilde{\ell}_2]_{{\hbox{\tiny{RN}}}_{|_{\E_{-1}}}}=0$.
Let us check that it is $\mathcal O$-bilinear, i.e. for all $f \in \mathcal O, x,y \in \E$:
 $$  [\dd, \tilde{\ell}_2]_{\hbox{\tiny{RN}}}(x,f y )- f  [\dd, \tilde{\ell}_2]_{\hbox{\tiny{\text{RN}}}}(x,y )=0.$$
 
\begin{enumerate}
    \item if $x \in \E_{-1}$, this quantity is zero in view of
     \begin{align*}
     [ \dd, \tilde{\ell}_2]_{\hbox{\tiny{RN}}}(x,f y )&=f[\dd,\tilde{\ell}_2](x,y)+\underbrace{\dd\rho(x)[f]\, y-\rho(x)[f] \, \dd y}_{=0} 
\end{align*}
    \item if $x \in \E_{-2}$, one has
     \begin{align*}
      [ \dd, \tilde{\ell}_2]_{\hbox{\tiny{RN}}}(x,f y ) - f [ \dd, \tilde{\ell}_2](x, y )  = \tilde \ell_2 (\dd x , fy) - f \tilde \ell_2 (\dd x , y) 
      = \rho  (\dd x)(f) \, y = 0 
     \end{align*}
     since $ \rho \circ \dd =\rho_{\mathcal A} \circ \pi \circ \dd =0 $,
    \item if $ x \in \E_{-i} $ with $ i \geq 3$, it is obvious by $\mathcal O $-linearity of $\tilde \ell_2 $ on the involved spaces.
\end{enumerate}
\noindent
As a consequence, $ [\dd, \tilde{\ell}_2]_{\hbox{\tiny{RN}}}$ is a degree $+2$ element in the total complex $\mathfrak{Page}^{(1)}(\E)$. By construction $ [\dd, \tilde{\ell}_2]_{\hbox{\tiny{RN}}}$ has no component on the last column. Since $ \pi ( [\dd, \tilde{\ell}_2]_{{\hbox{\tiny{RN}}} _{|_{\E_{-1}}}}) =0$ and also $[ \dd, [\dd, \tilde{\ell}_2]_{\hbox{\tiny{RN}}}]_{{\hbox{\tiny{RN}}}_{|_{\E_{\leq -2}}}}=0$, the $\mathcal O$-bilinear operator $ [\dd, \tilde{\ell}_2]_{\hbox{\tiny{RN}}}$ is $D$-closed in  $\mathfrak{Page}^{(1)}(\E)$.

By virtue of the first item of Proposition \ref{bicomplex}, the operator $ [\dd, \tilde{\ell}_2]_{\hbox{\tiny{RN}}}$ is then a $D$-coboundary, so there exists  $\tau_2 \in \oplus_{j\geq 2} \text{Hom}_\mathcal O\left(\bigodot^2\E_{|_{-j-1}},\E_{-j}\right)$ such as $D(\tau_2)= -[\dd, \tilde{\ell}_2]_{\hbox{\tiny{RN}}}.$  Upon replacing  $ \tilde{\ell}_2$  by $\tilde{\ell}_2+\tau_2$ we obtain a 2-ary bracket $\ell_2$ of degree +1 which satisfies all items of Definition \ref{def:al-oid}.
\end{proof} 
\begin{proof}[Proof (of Theorem \ref{thm:existence})]
Lemma \ref{lem1} gives the existence of an almost differential graded Lie algebroid with differential $\ell_1=\dd$ and binary bracket $\ell_2$. We have to construct now the higher brackets $\ell_k$ for $k\geq 3$. 

\noindent
{\textbf{Step 1:
Construction of the 3-ary bracket $\ell_3 $.}} (Its construction being different from the one of the higher brackets, we put it apart).  We first notice that the graded Jacobiator defined for all $x,y,z \in \E$ by 
$$ \text{Jac}(x,y,z):= \ell_2(\ell_2(x,y),z)+(-1)^{\lvert y\rvert\lvert z\rvert}\ell_2(\ell_2(x,z),y)+(-1)^{\lvert x\rvert\lvert y\rvert+\lvert x\rvert\lvert z\rvert}\ell_2(\ell_2(y,z),x)  $$is $ \mathcal O$-linear in each variable, hence is a degree $ +2$ element in $\bigoplus_{j \geq 1} \text{Hom}_\mathcal O(\bigodot^3\E\, _{|_{-j-2}},\E_{-j}) \subset  \mathfrak{Page}^{(2)}(\E)$. For degree reason, its component on the last column of diagram \eqref{recap} is zero, i.e. it belongs to $\widehat{\mathfrak{Page}}^{(1)}(\E)$.

Let us check that it is $D$-closed: for this purpose we have to check that both conditions in Lemma \ref{lem:beingClosed} hold:
\begin{enumerate}
\item Since $\pi $ is a morphism from $(\E_{-1}, \ell_2)  $ to $(\mathcal A , [\cdot, \cdot]_{\mathcal A},\rho_\mathcal{A}) $, and since $ [\cdot, \cdot]_{\mathcal A} $ satisfies the Jacobi identity, one has for all $x,y,z\in\E_{-1}$: $$\text{Jac}(x,y,z)\in \ker\pi .$$ 
\item Furthermore, a direct computation of $[\text{Jac}, \dd]_{\hbox{\tiny{RN}}}$ gives in view item 2 of Definition \ref{def:almost}: \begin{equation*}
\dd \text{Jac}(x,y,z)=\text{Jac}(\dd x,y,z)+(-1)^{\lvert x\rvert}\text{Jac}(x,\dd y,z)+(-1)^{\lvert x\rvert+\lvert y\rvert}\text{Jac}(x,y,\dd z)
\end{equation*}\text{for all}\;$x,y,z \in \E$. 
\end{enumerate}

\noindent
Thus, $D(\text{Jac})=0$. By Proposition \ref{bicomplex}, item 2, $\text{Jac}$ is a $D$-coboundary, and, more precisely, there exists an element $\ell_3=\sum_{j\geqslant 2} \ell_3^{j}\in\widehat{\mathfrak{Page}}_1^{(2)}(\E)$ with $\ell_3^{j}\in\text{Hom}(\bigodot^3\E\, _{|_{-j-1}},\E_{-j})$ such that \begin{equation}
D(\ell_3) =-\text{Jac}\quad\text{i.e.}\quad\left[\dd,\ell_3 \right]_{\hbox{\tiny{RN}}} =-\text{Jac}.
\end{equation}
We choose the $3$-ary bracket to be $\ell_3$.

\noindent
{\textbf{Step 2:
Recursive construction of the $k$-ary brackets $\ell_k $ for $k \geq 4$.}}
Let us recapitulate: $ \ell_1=\dd$
, $\ell_2 $ and $\ell_3 $ are constructed and
the lowest polynomial-degree terms of 
 $ [\ell_1 + \ell_2 +\ell_3, \ell_1 + \ell_2 +\ell_3]_{\hbox{\tiny{RN}}}$
satisfy
\begin{enumerate}
    \item  $[\ell_1, \ell_1]_{\hbox{\tiny{RN}}} =0 $ (since $\dd^2=0$),
    \item  $[\ell_1, \ell_2]_{\hbox{\tiny{RN}}} =0 $ (since $\dd=\ell_1$ and $\ell_2$ define an almost Lie algebroid structure).
    \item $ [\ell_2, \ell_2]_{\hbox{\tiny{RN}}} + 2 [\ell_3, \ell_1]_{\hbox{\tiny{RN}}} = 2( \mathrm{Jac} +  [\ell_3 ,\ell_1]_{\hbox{\tiny{RN}}}) =0$ 
    by definition of $ \ell_3$, and because $[\ell_2, \ell_2]_{\hbox{\tiny{RN}}} = 2 \mathrm{Jac}$.
\end{enumerate}
However, the following term of degree $+2$ and polynomial-degree $3 $ may not be equal to zero:
\begin{equation}
[\ell_3 , \ell_2]_{\hbox{\tiny{RN}}}\in\bigoplus\text{Hom}_\mathcal O\left(\bigodot^4\E_{j+1},\E_{-j}\right)=\widehat{\mathfrak{Page}}_1^{(3)}(\E).\end{equation}
Let us check that this term is indeed a $\mathcal{O}$-multilinear map: For $x_1\in \E_{-1},x_2,x_3,x_4\in\E_{\leq -2}$ and $f\in\mathcal O$, the only terms of $(\ell_3\circ \ell_2+\ell_2\circ \ell_3)(x_1,fx_2,x_3,x_4)$ where the anchor shows up are: 
$$\begin{cases}
\ell_3(\ell_2(x_1,fx_2),x_3,x_4)&=\rho(x_1)[f]\ell_3(x_2,x_3,x_4)+f(\ell_3(\ell_2(x_1,x_2),x_3,x_4))\\ (-1)^{|x_2|+|x_3|+|x_4|}\ell_2(f\ell_3(x_2,x_3,x_4),x_1)&=-\rho(x_1)[f]\ell_3(x_2,x_3,x_4) \\ &+f((-1)^{|x_2|+|x_3|+|x_4|}\ell_2(f\ell_3(x_2,x_3,x_4),x_1))
\end{cases}$$
The terms containing the anchor map add up to zero. When there are  more elements in $\E_{-1}$, the computation follows the same line. Moreover, by graded Jacobi identity of the Richardson-Nijenhuis bracket:
 $$ [[\ell_1 + \ell_2 +\ell_3, \ell_1 + \ell_2 +\ell_3]_{\hbox{\tiny{RN}}}, \ell_1 + \ell_2 +\ell_3]_{\hbox{\tiny{RN}}}  =0 $$ 
The term of polynomial-degree $4$ in the previous expression gives
$[[\ell_3, \ell_2]_{\hbox{\tiny{RN}}} , \ell_1]_{\hbox{\tiny{RN}}}=0$. Hence, by Proposition \ref{rem:bicom-rch}, $[\ell_3, \ell_2]_{\hbox{\tiny{RN}}}$ is a $D$-cocycle in the complex $\mathfrak{Page}^{(3)}(\E)$, whose components on the last column and the column $-1$ are zero. 
It is therefore a coboundary by Proposition  \ref{bicomplex} item 3: we can continue a step further and define $\ell_4\in \oplus_{j \geq 3} \text{Hom}\left(\bigodot^4\E _{|_{-j-1}},\E_{-j}\right)$ such  that: \begin{equation}
-\left[\ell_2,\ell_3 \right]_{\hbox{\tiny{RN}}}
=
\left[\ell_1,\ell_4 \right]_{\hbox{\tiny{RN}}}
=
\left[\dd,\ell_4 \right]_{\hbox{\tiny{RN}}} .
\end{equation}

We choose the $4$-ary bracket to be $\ell_4$. We now proceed by recursion. We assume that we have constructed all the $k$-ary brackets, $\ell_k$ such as : \begin{equation}\label{Maurer}
\left[\dd,\ell_k \right]_{\hbox{\tiny{RN}}} =-\sum_{\overset{i+j=k+1}{i\leq j}}\left[\ell_i,\ell_j \right]_{\hbox{\tiny{RN}}}=-\frac{1}{2}\sum_{\overset{i+j=k+1}{i,j\geq 1}}\left[\ell_i,\ell_j \right]_{\hbox{\tiny{RN}}}
\end{equation} for every $k=1,\ldots,n$ with $n\geq 4$. The ($n+1$)-ary bracket is constructed as follows. First, the operator $\sum_{\overset{i+j=k+1}{i,j\geq 1}}\left[\ell_i,\ell_j \right]_{\hbox{\tiny{RN}}}$ is checked to be $\mathcal O$-linear as before. Now, we have \begin{align*}
\sum_{\overset{i+j=n+2}{i,j\geq 1}}\left[\dd,\left[ \ell_i,\ell_j \right]_{\hbox{\tiny{RN}}}\right]_{\hbox{\tiny{RN}}} 
&= -2\sum_{\overset{i+j=n+2}{i,j\geq 1}}\left[\ell_i,\left[ \dd,\ell_j \right]_{\hbox{\tiny{RN}}}\right]_{\hbox{\tiny{RN}}}\quad\text{(by graded Jacobi identity)}
.\end{align*}
Since $\ell_j$ satisfies Equation \eqref{Maurer} up to order $n$, we obtain\begin{equation*}
\sum_{\overset{i+j=n+2}{i,j\geq 1}}\left[\dd,\left[ \ell_i,\ell_j \right]_{\hbox{\tiny{RN}}}\right]_{\hbox{\tiny{RN}}}=\sum_{\overset{i+j+k=n+3}{i,j,k\geq 1}}\left[  \ell_i,\left[ \ell_j,\ell_k\right]_{\hbox{\tiny{RN}}} \right]_{\hbox{\tiny{RN}}} =0,
\end{equation*}
\noindent
where we used the graded Jacobi identity of the  Nijenhuis-Richardson bracket in the last step. Therefore, $\sum_{\overset{i+j=n+2}{i,j\geq 1}}\left[\ell_i,\ell_j \right]_{\hbox{\tiny{RN}}}$, seen as an element in $\mathfrak{Page}^{(i+j-2)}(\E)$ by Remark \ref{rem:bicom-rch}, is a cocycle  and for degree reason it has no element on the last column, and the columns $-1,\ldots,3-n$ in \ref{recap}. The third item of Proposition \ref{bicomplex} gives the existence of an $(n+1)$-ary bracket $\ell_{n+1}$ such as\begin{equation*}
\left[\dd,\ell_{n+1} \right]_{\hbox{\tiny{RN}}} =-\sum_{\overset{i+j=n+2}{i\leq j}}\left[\ell_i,\ell_j \right]_{\hbox{\tiny{RN}}}.
\end{equation*}
{This completes the proof.}
\end{proof}

\subsubsection{Proof of Proposition \ref{prop:lksontNuls} and Proposition \ref{univ:precise}}

\begin{proof}[Proof (of Proposition \ref{prop:lksontNuls})]
This is a consequence of Proposition \ref{thm:existence} and the third item of the Proposition \ref{bicomplex}: If the component of $\text{Jac}$ on the column $-1$ is zero, we can choose $\ell_3$ with no component on the last column and in column $-1$ (see Proposition \ref{bicomplex}), i.e. the restriction of $\ell_3$ to $\bigodot^3\E_{-1}$ is zero. Then $\ell_3$ has no component on the last column, the column $-1$ and the column $-2$. so $[\ell_2,\ell_3]_{\hbox{\tiny{RN}}}$ has no component in the last column, $-1$ and $-2$ columns as well. Hence, $\ell_4$ can be chosen with no component on column $-1$, $-2$ and $-3$ by the third item of Proposition \ref{bicomplex}. The proof continues by recursion.
\end{proof}
We finish this section with a proof of Proposition \ref{univ:precise}.
\begin{proof}[Proof (of Proposition \ref{univ:precise})]We prove this Proposition in two steps.
\begin{enumerate}
    \item Lemma \ref{lem:inFactInclusion} guarantees the existence a free resolution $(\E,\dd,\pi)$ of the Lie-Rinehart algebra $\mathcal A$ such that $\mathcal E$ contains $\mathcal E'$ and such that there exists a graded free module $\mathcal V$ with $\E ' \oplus \mathcal V = \E $.
    \item Let $D^\E$ and $D^{\E'}$ be as in the proof of Lemma \ref{rest:univ}. We construct the $n$-ary brackets on $\E$ by extending the ones of $(\mathcal E', (\ell_k')_{k\geq 1}, \rho_{\mathcal E'}, \pi')$ in the following way:
    \begin{enumerate}
        \item We first construct an almost Lie algebroid bracket $\tilde{\ell_2}$ on $ \E_{-1}$ that extends the $2$-ary bracket of $ \E_{-1}'$. Since the $2$-ary bracket is determined by its value on a basis, the existence of a free module $\mathcal V_{-1}$ such that $\mathcal E_{-1} ' \oplus \mathcal V_{-1} = \mathcal E_{-1} $ allows to construct $\tilde{\ell_2} $ on $\E $ such that its restriction to $ \mathcal E' $ is $\ell_2' $ and such that it satisfies the Leibniz identity.
        
        As in the proof of Theorem \ref{thm:existence} (to be more precise: Lemma \ref{lem1}), we see that $[ \tilde \ell_2 , \dd^\E ]_{\hbox{\tiny{RN}}}$  is $\mathcal O $-linear, hence belongs to ${\mathfrak{Page}}^{(2)}_{2} (\E) $ and is a $D^\E$-cocycle. Since $\E' $ is a Lie $ \infty$-algebroid, its restriction to $\bigodot^2\mathcal E' $ is zero. Lemma  \ref{rest:univ} allows to change $\tilde{\ell}_2 $ to an $2$-ary bracket $\ell_2:=\tilde{\ell}_2+\tau_2$ with $\tau_2=0$ on $\bigodot^2 \E'$. Hence, $\ell_2$ defines a graded almost Lie algebroid bracket, whose restriction to $\E'$ is still $\ell_2' $.  
        
        \item Since $\ell_2$  is an extension of  $\ell_2'$, its Jacobiator $\text{Jac}\in \mathfrak{Page}^{(2)}_2(\E)$ of the $2$-ary bracket $\ell_2$ preserves $\E'$. Also, its restriction $\iota^*_{\E'}\text{Jac}\in \mathfrak{Page}^{(2)}_2({\E'})$ is the Jacobiator of $\ell_2' $, and the latter is the $D^{\E'}$-coboundary of $\ell_3'$ in view of the higher Jacobi identity of $\E'$. Since $\text{Jac}\in \mathfrak{Page}^{(2)}_2(\E)$ is a $D^\E$-cocycle, Lemma \ref{rest:univ} assures that $\text{Jac}$ is the image through $D^\E$ of some element $\ell_3\in\mathfrak{Page}^{(2)}_1(\E)$  which preserves $\E'$ and whose restriction to $ \bigodot^3\mathcal E'$ is $ \ell_3'$. The proof continues by recursion: at the $n$-th step, we use Lemma  \ref{rest:univ} to construct an $n$-ary bracket for $\E $ that extends the $n$-ary bracket of $\E' $.
    \end{enumerate}
\end{enumerate}
By construction, the inclusion map $\iota\colon \E'\hookrightarrow \E$ is a morphism for the $n$-ary brackets for all $n \geq 1$.
\end{proof}

\subsection{Proof of universality}
Before proving the universal character of the construction, we need to do some preparations.

Let us prove the following lemma,
\begin{lemma} \label{imp-lemma2}
Let $\Psi,\Xi: S^\bullet_\mathbb{K} (\E') \rightarrow S^\bullet_\mathbb{K}(\E)$ be $\mathcal O $-linear Lie $\infty$-algebroid morphisms. Let $n \in \mathbb{N}_0$.
If $\Xi^{(i)}= \Psi^{(i)}$ for every $0 \leq i \leq n$,
 there exists 
 \begin{enumerate}
     \item a Lie $\infty$-morphism of algebroids $J_1\colon  S_{\mathbb K}(\E')\rightarrow S_{\mathbb K}(\E)$ 
     \item and a homotopy $(J_t,H_t)_{[0,1]}$ joining  $\Psi$ and $J_1$, 
 \end{enumerate}
 such that 
 \begin{enumerate}
     \item the components of polynomial-degree less or equal to $n$ of $H_t$ vanish, 
     \item $J_1^{(i)}=\Xi^{(i)}$ for every $0 \leq i \leq n+1$.
 \end{enumerate}
 \end{lemma}
\begin{proof} Let us consider $(\Xi-\Psi)^{(n+1)}: \bigodot^\bullet\mathcal E' \longrightarrow  \bigodot^\bullet\mathcal E $. By assumption, $ (\Xi-\Psi)^{(i)}=0 $ for all $i \leq n$, so that in view of Lemma \ref{imp-lemma}
\begin{enumerate}
    \item 
$(\Xi-\Psi)^{(n+1)}$ is a $\Psi^{(0)} $-co-derivation. 
    \item Proposition \ref{prop:interpretation} means that the map the restriction of the map
$(\Xi-\Psi)^{(n+1)}$
to $\bigodot^{n+2}\mathcal E'$ corresponds to a closed element of degree $0$ in ${\mathfrak{Page}}^{(n+1)} (\E',\E)$ equipped with differentials $\ell_1,\ell_1'$.
\end{enumerate}
 Proposition \ref{bicomplex} implies that there exists  $\mathcal O$-linear map $H_{n+1}\colon \bigodot^{n+2}(\E')\longrightarrow \E$, a degree $ -1$ and of polynomial-degree $n+1$, i.e. an element in $\mathfrak{Page}^{(n+1)}(\E',\E)$, such that, 
\begin{equation}
(\Xi-\Psi)^{(n+1)}=Q^{(0)}_{\E}\circ H_{n+1}+H_{n+1}\circ Q_{\E'}^{(0)}.
\end{equation}
We denote its extension
to a $\Psi^{(0)}$-co-derivation of degree $-1$ by $H^{(n+1)}$. We now consider the following differential equation for $t \in [0,1]$: 
\begin{equation}\label{homotpy}
\frac{\dd J_t}{\dd t}=Q_\E\circ H_t+H_t\circ Q_{\E'},\quad \text{and}\quad J_0=\Psi,
\end{equation}where $H_t$ is the unique $J_t$-co-derivation of degree $-1$  whose unique non-zero Taylor coefficient is $H_{n+1}$. The existence of a solution for the differential equation \eqref{homotpy} is granted by Proposition \ref{prop:justify}. By considering the component of polynomial-degree $1, \dots, n, n+1$ in Equation \eqref{homotpy}, we find
 $$ \left\{ \begin{array}{rcl}
 \frac{\dd J_t^{(i)}}{\dd t} &  = &0 \, \, \, \hbox{ for $i=0, \dots, n$ ,} \\
 \frac{\dd J_t^{(n+1)}}{\dd t}&=&Q_\E^{(0)}\circ H^{(n+1)}+H^{(n+1)}\circ Q_{\E'}^{(0)}  =   (\Xi-\Psi)^{(n+1)} 
 \end{array}\right.$$
Hence:
\begin{align*}
 \left\{ \begin{array}{rcl}
  {J_t^{(i)}} &  = & \Psi^{(i)} \, \, \, \hbox{ for $i=0, \dots, n$ ,} \\
J_t^{(n+1)}&=&
\Phi^{(n+1)}+t(\Xi-\Psi)^{(n+1)}. 
\end{array} \right.
\end{align*}
Therefore, applying $t=1$ to the previous relation, one finds $$
 \left\{ \begin{array}{rcl}
{J_t^{(i)}} &  = & \Psi^{(i)} \, \, \, \hbox{ for $i=0, \dots, n$ ,} \\
 J^{(n+1)}_1&=&\Psi^{(n+1)}+(\Xi-\Psi)^{(n+1)}= \Xi^{(n+1)} 
\end{array}\right. 
$$
This completes the proof.
\end{proof}

\subsubsection{Construction of the Lie $\infty$-morphism}
\begin{proof}[Proof (of Theorem \ref{th:universal})]	
Let us prove item 1. 
We construct  the Taylor coefficients of the Lie $\infty $-algebroid $ \Phi$ by recursion.

The Taylor coefficient of polynomial-degree $0$ is obtained out of classical properties of projective resolutions of $\mathcal O $-modules.
Given any complex $(\E', \rho', \ell'_{1},\pi')$ which terminates in $\mathcal A$ through $\pi $, for every free resolution $(\E, \rho, \ell_{1},\pi)$ of $\mathcal A$, there exists   a chain map  $\Phi^{(0)}\colon (\E',\ell_1') \to (\E, \ell_1)$ as in Equation (\ref{eq:EE'}), and any two such chain maps are homotopic. We still denote by $ \Phi^{(0)} $ its extension to an polynomial-degree $ 0$ co-morphism $\bigodot^\bullet \E' \to \bigodot^\bullet \E $.

To construct the second Taylor coefficient,  let us consider the map:
\begin{equation} \label{degre2Obs} \begin{array}{rcl}   
  S^2_\mathbb K (\E')& \to  & \E \\
 (x,y)&\mapsto&\Phi^{(0)} \circ \ell_2'(x,y)-\ell_2(\Phi^{(0)}(x),\Phi^{(0)}(y)).\end{array}\end{equation}
 This map is in fact $\mathcal O $-bililinear, i.e. belongs to $\text{Hom}_\mathcal{O}(\bigodot^2\E',\E)$, hence to ${\mathfrak{Page}}^{(1)}(\E',\E)$, see Equation \eqref{recap}. Let us check that it is a $D$-cocycle:
 \begin{enumerate}
     \item[A.]  If either one of the homogeneous elements $x \in \E'$ or $y \in \E'$ is not of degree $-1$, a straightforward computation  gives:
     \begin{align*}
          \ell_1\circ\left( \Phi^{(0)}\circ\ell'_2(x,y)-\ell_2\left(\Phi^{(0)}(x),\Phi^{(0)}(y)\right)\right)&=\Phi^{(0)}\circ\ell_1'\circ\ell'_2( x,y)+\ell_2\left(\Phi^{(0)}\circ\ell_1'(x),\Phi^{(0)}(y)\right)\\&\hspace{2.1cm}+(-1)^{\lvert x\rvert}\ell_2\left(\Phi^{(0)}(x),\Phi^{(0)}\circ\ell_1'(y)\right)\\&=\left(\Phi^{(0)}\circ\ell'_2-\ell_2\left(\Phi^{(0)},\Phi^{(0)}\right)\right)\circ\ell_1'(x\odot y).
     \end{align*}
     \item[B.] If both $x,y \in \E'$ are of degree $-1 $:
     \begin{eqnarray*}\pi \left( \Phi^{(0)}\ell'_2(x,y)-\ell_2(\Phi^{(0)}x,\Phi^{(0)}y)\right) &=&  \pi'  \circ \ell_2'(x,y)- \pi \circ \ell_2\left(\Phi^{(0)}x,\Phi^{(0)}y\right)\\
     & =&   [\pi'(x), \pi'(y)]- \left[\pi \left(\Phi^{(0)}x\right), \pi\left(\Phi^{(0)}x\right) \right] \\&=& [\pi'(x), \pi'(y)]-[\pi'(x), \pi'(y)]\\ &=& 0.\end{eqnarray*}
 \end{enumerate}
By Proposition \ref{bicomplex} item 2), there exists $\Phi^{(1)}\in \text{Hom}_\mathcal{O}\left(\bigodot^2 \E',\E\right)$, of degree $0$, so that \begin{equation}\label{phi1}
\Phi^{(0)}\circ\ell'_2(x,y)+\Phi^{(1)}\circ\ell_1'(x\odot y)=\ell_1\circ \Phi^{(1)}(x,y)+\ell_2(\Phi^{(0)}(x),\Phi^{(0)}(y)) \hspace{1cm} \hbox{for all $x,y \in \E'$.}
\end{equation}	
 $\Phi^{(0)}$ is a chain map and Relation \eqref{phi1} can be rewritten in terms of $Q_\E$ and $Q_{\E'}$ as follows\begin{align}  \left\{ \begin{array}{rcl} Q^{(0)}_{\E}\circ\Phi^{(0)}&=&\Phi^{(0)}\circ Q^{(0)}_{\E'}\\Q^{(0)}_{\E}\circ\Phi^{(1)}-\Phi^{(1)}\circ Q^{(0)}_{\E'}&=&\Phi^{(0)}\circ Q^{(1)}_{\E'}-Q^{(1)}_{\E}\circ\Phi^{(0)} \end{array} \right.\end{align}

The construction of the morphism $\Phi$ announced in Theorem \ref{th:universal} is then done by recursion. The recursion assumption is that we have already defined a $\mathcal{O}$-multilinear co-morphism  $\Phi:S_\mathbb{K}^\bullet(\E')\rightarrow S^\bullet_\mathbb{K}(\E)$ with $$\left( \Phi\circ Q_{\E'}-Q_{\E}\circ\Phi\right)^{(k)}=0 \quad\text{for all}\quad 0\leq k\leq n.$$
The co-morphism $\Phi:S_\mathbb{K}^\bullet(\E')\rightarrow S^\bullet_\mathbb{K}(\E)$ with Taylor coefficients $\Phi^{(0)} $ and $\Phi^{(1)} $ satisfies the recursion assumption for $n=1$.

Assume now that we have a co-morphism $\Phi $ that satisfies this assumption for some $n \in \mathbb{N}$, and
 consider the map $T_\Phi:= \Phi\circ Q_{\E'}-Q_{\E}\circ\Phi$.  Lemma \ref{O-linearity-lemma}  implies that $T_\Phi^{(n+1)}$ is a $\mathcal{O}$-multilinear $ \Phi^{(0)}$-co-derivation, and that it corresponds to a $D$-closed element\footnote{The following remark is crucial. 
Under the assumptions of Lemma \ref{O-linearity-lemma}, $(\Phi\circ Q_{\E'}-Q_{\E}\circ\Phi)^{(n+1)} $ corresponds to a $D$-closed element of degree $+1$ in  the bi-complex $\mathfrak{Page}^{(n+1)}(\E',\E)$ through the chain isomorphism described in Proposition \ref{prop:interpretation}.
Here, $\E,\E' $ are equipped with the differentials $\ell_1, \ell_1' $ which are the restriction of the components $Q^{(0)}_\E,Q^{(0)}_{\E'}$.} in $\mathfrak{Page}^{(n+1)}(\E',\E)$. Since it has no component on the last column for degree reason, Proposition \ref{bicomplex} implies that  $T_\Phi^{(n+1)}$ is a coboundary: That is to say that there is a $\Phi^{(0)} $-co-derivation $\Theta\in \mathfrak{Page}^{(n+1)}(\E',\E)$ (of polynomial-degree $n+1$ and degree $0$) which can be seen as a map $\Theta:\bigodot^{n+2} (\E')\rightarrow\E$ such that: $$T_\Phi^{(n+1)}=Q^{(0)}_{\E}\circ\Theta-\Theta\circ Q_{\E'}^{(0)}.$$
 Consider now the co-morphism $\tilde{\Phi}$ whose Taylor coefficients are those of $\Phi $ in polynomial-degree $0, \dots,n$ and  $ \Phi^{(n+1)}+\Theta$ in polynomial-degree $n+1$:  \begin{equation}\tilde{\Phi}^{(i)}:=\begin{cases} \Phi^{(i)}&\text{if $0\leq i\leq n$,} \\\Phi^{(n+1)}+\Theta &\text{if $i=n+1$}\end{cases}
\end{equation}
This is easily seen to satisfy the recursion relation for $n+1$. This concludes the recursion. The Taylor coefficients obtained by recursion define a Lie $\infty$-algebroid $\Phi\colon S^\bullet_{\mathbb K}(\E')\longrightarrow S^\bullet_{\mathbb K}(\E)$ which is compatible by construction with the hooks $\pi,\pi'$.

By continuing this procedure, we construct a Lie $\infty$-morphism from $S^\bullet_{\mathbb K}(\E')$ to $S^\bullet_{\mathbb K}(\E)$. This proves the  first item of Theorem \ref{th:universal}.

\subsubsection{Construction of a  homotopy that joins two such morphisms}
Let us prove the second item in Theorem \ref{th:universal}.
Notice that in the proof of the existence of the Lie $\infty$-morphism between $S^\bullet_{\mathbb K}(\E')$ and $S^\bullet_{\mathbb K}(\E)$ obtained in the first item, we made many choices, since we have chosen a coboundary at each step of the recursion.

Let $\Phi,\Psi$ be two such Lie $\infty$-morphisms between $S^\bullet_{\mathbb K}(\E')$ and $S^\bullet_{\mathbb K}(\E)$. The polynomial-degree $0$ component of the co-morphisms $\Phi$ and $\Psi$ restricted to $\E'$ are chain maps:
$$
\xymatrix{ \cdots\ar[rr]& & \ar[rr]  \mathcal E_{-2}'\ar@{-->}[lldd]_h\ar@<3pt>[dd]^{\Phi^{(0)}} \ar@<-3pt>[dd]_{\Psi^{(0)}}  & & \mathcal E_{-1} '\ar@<3pt>[dd]^{\Phi^{(0)}} \ar@<-3pt>[dd]_{\Psi^{(0)}}  \ar@{-->}[lldd]_h \ar[dr]^{\pi'} \\ & & & & & \mathcal A \\ \cdots\ar[rr] & & \mathcal E_{-2}  \ar[rr] & &\mathcal E_{-1} \ar@{->>}[ru]_\pi & } 
$$
which are homotopy equivalent in the usual sense because $(\E,\ell_1) $ is a projective resolution of $\mathcal A $:
said differently, there exists a degree $ -1$ $\mathcal O$-linear map $h\colon \E' \to \E$ such that \begin{equation}\Psi^{(0)} -\Phi^{(0)} =\ell_1 \circ h+h\circ\ell_1'\quad\text{on}\;\,\E'.\end{equation}
Let us consider the following differential equation:\begin{equation}\label{eq-diff1}
    \begin{cases}
    \frac{\dd J_t}{\dd t}=Q_\E\circ H_t(J_t)+H_t(J_t)\circ Q_{\E'},&\text{for $t\in[0,1]$}\\J_0=\Phi.
\end{cases}\end{equation}
with $H_t( J_t)$ being a $J_t$-co-derivation of degree $-1$ whose Taylor coefficient of polynomial-degree $0$ is $h$. This equation does admit solutions in view of Proposition~\ref{prop:justify}.

By looking at the component polynomial-degree $0$ of Equation \eqref{eq-diff1} on $\E'$, one has: \begin{align*}
\frac{\dd J^{(0)}_t}{\dd t}&=\ell_1 \circ h+h\circ \ell_1'\\&=\Psi^{(0)}-\Phi^{(0)}.
\end{align*}
Hence, $J^{(0)}_t=\Phi^{(0)}+t\left( \Psi^{(0)}-\Phi^{(0)}\right)$ is a solution such that $J^{(0)}_1=\Psi^{(0)}$. By construction, $J_1$ is homotopic to $\Phi$ via the pair $\left(J_t,H_t \right) $ over $[0,1]$, and its polynomial-degree $0$ Taylor coefficient coincides with the Taylor coefficient of $ \Psi$. 

From there, the construction goes by recursion using Lemma \ref{imp-lemma2}. 
Indeed, this lemma allows constructing
recursively a sequence of Lie $\infty $-algebroids morphism $ (\Psi_n)_{n \geq 0} $ and homotopies  $(J_{n,t},H_{n,t})$
(with $t \in [n,n+1] $) between $\Psi_n $ and $\Psi_{n+1} $ such that:  $H_{n,t}^{(i)}$ is zero for $t\geq n$ and $i\neq n+1$. By Lemma \ref{imp-lemma2}, all these homotopies are compatible with the hooks. These homotopies are glued in a homotopy $(J_t,H_t)_{[0,+\infty[}$ such that for every $n\in\mathbb{N}_0$, the components of polynomial-degree $n$ of the Lie $\infty$-algebroids morphism $J_t^{(n)}$ are constant and equal to $\Psi^{(n)}$ for $t\geq n$. By Lemma \ref{gluing-lemma}, these homotopies can be glued to a homotopy on $[0,1]$. Explicitly, since $t\mapsto\frac{t}{1-t}$ maps $[0,1[$ to $[0,+\infty[$ and by Lemma \ref{gluing-lemma}, the pair $\left(J_{\frac{t}{1-t}},\frac{1}{(1-t)^2}H_{k,\frac{t}{1-t}}\right)$ is a homotopy between $\Phi$ and $\Psi$.
This proves the second item of the Theorem \ref{th:universal}.

\end{proof}

\section{Examples of universal Lie $\infty$-algebroids of  Lie-Rinehart algebras}

\label{sec:examples}
\subsection{New constructions from old ones}

In this section, we explain how to construct universal Lie $\infty$-algebroids of some Lie-Rinehart algebra which is derived from a second one through one of natural constructions as in Section \ref{LR-construction} (localization, germification, restriction), when a universal Lie $\infty$-algebroid of the latter is already known.
\label{constructions}

\subsubsection{Localization}\label{ssec:loc}

Localization is an useful algebraic tool, specially in algebraic geometry. When $\mathcal O $ is an algebra of functions, it corresponds to study local properties of a space, or germs of functions. \\

Let $(\mathcal A,\lb_\mathcal A,\rho_\mathcal A)$ be a Lie-Rinehart algebra over $\mathcal O$. Let $S\subset\mathcal O$ be a multiplicative-closed subset containing no zero divisor. We recall from item 2, Example \ref{loc:res} that the localization $S^{-1}\mathcal A\cong\mathcal A\otimes_\mathcal O S^{-1}\mathcal O$ of $\mathcal A$ at $S$ comes equipped with  a natural structure of Lie-Rinehart algebra over the localization algebra $S^{-1}\mathcal O$. Recall that for $\varphi\colon \E \longrightarrow \mathcal T$ a homomorphism of $\mathcal O$-modules, there is a well-defined homomorphism of $\mathcal O$-modules, $$\varphi\otimes\text{id}\colon \mathcal E\otimes_\mathcal O S^{-1}\mathcal O\longrightarrow \mathcal T\otimes_\mathcal O S^{-1}\mathcal O,\hspace{0.2cm}\varphi\otimes\text{id} \, (x\otimes\frac{f}{s}):=\varphi(x)\otimes\frac{f}{s}$$ that can be considered as a $S^{-1}\mathcal O$-module homomorphism $$S^{-1}\varphi\colon S^{-1}\E\longrightarrow S^{-1}\mathcal T\hspace{0.2cm}\hbox{with}\hspace{0.2cm}S^{-1}\varphi\left(\frac{x}{s}\right):=\frac{\varphi(x)}{s},\,\;x\in\E,\,(f,s)\in\mathcal O\times S,$$ called the \emph{localization of $\varphi$}.
	
\noindent	
Given a Lie $\infty$-algebroid structure $(\E_\bullet,\ell_\bullet,\rho)$ that covers $\mathcal A$ through $\pi$. The triplet $(\E_\bullet',\ell_\bullet',\rho')$ is a Lie $\infty$-algebroid structure that covers $S^{-1}\mathcal A$ through the hook $\pi'$ where
\begin{enumerate}
\item $\E'=S^{-1}\E$;
\item  The anchor map $\rho'$ is defined by
\begin{align*}
  \begin{array}{rcrrcl} \rho' \colon S^{-1}\E_{-1}&\longrightarrow &  \text{Der}(S^{-1}\mathcal O) & & & \\\label{der_ext} 
    \frac{x}{s}&\longmapsto & \rho'\left(\frac{x}{s}\right): &S^{-1}\mathcal O&\longrightarrow &S^{-1}\mathcal O\\
&  & & \frac{f}{u}&\longmapsto&\frac{1}{s}\cdot\left(\frac{\rho(x)[f]u-f\rho(x)[u]}{u^2}\right) \end{array}
\end{align*}for $x\in\E,f\in\mathcal O,(s,u)\in S\times S$;
\item $\ell'_k=S^{-1}\ell_k$, for all $k\in\mathbb N\setminus\{2\}$;
\item The binary bracket is more complicated because of the anchor map: we
set $$ \ell_2'\left(\frac{1}{s} x,\frac{1}{u} y\right) =\frac{1}{su}\ell_2(x, y)- \frac{\rho(x) [u]}{su^2} \, y + \frac{\rho(y) [s]}{s^2u} \, x$$
for $x,y\in\E, (s,u)\in S^2$
(with the understanding that $\rho\equiv0 $ on $\E_{-i}$ with $ i \geq 2 $);
\item  $\pi'=S^{-1}\pi$.
\end{enumerate}
 One can check that these operations above are well-defined and for all $z\in S^{-1}\E_{-1}$ the map $\rho'(z)$ is  indeed a derivation on $S^{-1}\mathcal O$. The previously defined structure is also a Lie $ \infty$-algebroid that we call \emph{localization of the  Lie $\infty$-algebroid  $(\E_\bullet,\ell_\bullet,\rho)$ with respect to $S$}. 

\begin{proposition}
\label{prop:localisation}
Let $S \subset \mathcal O$ be a multiplicative subset containing no zero divisor. 
The localization of a universal Lie $\infty $-algebroid of a Lie-Rinehart algebra $ \mathcal A$ is a  universal Lie $\infty $-algebroid of $ S^{-1}\mathcal A$.
\end{proposition}
\begin{proof}
The object $((\E_\bullet',\ell_\bullet',\rho'))$ described above is also a Lie $\infty$-algebroid terminating in $S^{-1}\mathcal A$ through $\pi'$. It is universal because localization preserves exact sequences \cite{A.Gathmann}.
\end{proof}

\subsubsection{Restriction}\label{ssec:res}
When $\mathcal O_Y$ is the ring of functions of an affine variety $Y$ (see Section \ref{sec:reminder-affine-variety}), to every subvariety $X\subset Y$ corresponds its zero locus, which is an ideal $\mathcal I_X \subset \mathcal O_Y $. A Lie $\infty$-algebroid or a Lie-Rinehart algebra over $\mathcal O_Y$ may not restrict to a Lie-Rinehart algebra over $\mathcal O_X$: it only does so when one can quotient all brackets by $\mathcal I_X$, which geometrically means that the anchor map takes values in vector fields tangent to $X$. We can then \textquotedblleft restrict\textquotedblright, i.e. replace $\mathcal O_Y $ by $\mathcal O_Y / \mathcal I_X $. This operation has already been defined in Section \ref{ssec:Tor},
and here is an immediate consequence of Corollary \ref{cor:LRideal2}:

\begin{proposition}\label{prop:restriction}
Let $ \mathcal I \subset \mathcal O$ be a Lie-Rinehart ideal, i.e. an ideal such that $ \rho_{\mathcal A}(\mathcal A)[\mathcal I] \subset \mathcal I$.
The quotient of a universal Lie $\infty$-algebroid of $ \mathcal A$ with respect to an ideal $ \mathcal I$ is a   Lie $\infty $-algebroid that terminates in $ \mathcal A/ \mathcal I \mathcal A$. It is universal if and only if
$((\frac{\E_{-i}}{\mathcal I\E_{-i}})_{i\geq 1},\Bar{\ell_1},\overline{\pi})$ is exact, i.e. if
$ {\mathrm{Tor}}_{\mathcal O}^\bullet (\mathcal A , \mathcal O/ \mathcal I)=0$.
\end{proposition}

\begin{remark}
Note that the anchor map $\rho\colon\mathcal E_{-1}\rightarrow \mathrm{Der}(\mathcal O)$ goes to quotient to $\frac{\E_{-1}}{\mathcal I \mathcal  E_{-1}}\rightarrow \mathrm{Der}(\mathcal O)$ as  an $\mathcal O$-linear map,  but needs the extra condition $\rho(\mathcal{E}_{-1})[\mathcal{I}]\subset \mathcal I$ to induce an $\mathcal O/\mathcal I$- linear map $\frac{\mathcal E_{-1}}{\mathcal I \mathcal E_{-1}}\rightarrow \mathrm{Der}(\mathcal O/\mathcal I)$.
\end{remark}



\subsubsection{Germification}\label{sec:germification}
Let $W\subseteq\mathbb C^{N} $ be an affine variety and $\mathcal O_W$ its coordinates ring (see Section \ref{sec:reminder-affine-variety}). For $a\in W$,  consider $\mathcal O_{W,a}$ the local ring at $a$. Note that $\mathcal O_{W,a}\simeq(\mathcal O_W)_{\mathfrak m_{a}}$ \cite{Hartshorne},  where $\mathfrak m_{a}=\{f\in \mathcal O_W\mid f(a)=0 \}$ and $(\mathcal O_W)_{\mathfrak m_{a}}$ is the localization w.r.t the complement of $\mathfrak m_{a}$,	Proposition \ref{prop:localisation} implies the following statement:

		\begin{prop}
		\label{prop:germ}
		Let $W$ be an affine variety with functions $ \mathcal O_W$. For every point $a\in W$ and any Lie-Rinehart $\mathcal A $ over $\mathcal O_W $, the germ at $a $ of the universal Lie $\infty$-algebroid of $\mathcal A $ is the universal Lie $\infty$-algebroid of the germ of $\mathcal A $ at $a$. 
	\end{prop}
	
	\noindent
	Therefore, the germ at $a$ of a Lie-Rinehart algebra or a Lie $\infty$-algebroid is simply its  localization w.r.t the complement of $\mathfrak m_{a}$.

\subsection{Sections vanishing on a codimension $1$ subvariety} 

\label{sec:codim1} 

Let $(\mathcal{A},\lb,\rho_\mathcal{A})$ be an arbitrary Lie-Rinehart algebra over $ \mathcal O$. For any ideal $\mathcal I  \subset \mathcal O$, $\mathcal I \mathcal A $ is also a Lie-Rinehart algebra (see Example \ref{ex:vanishing}). When $ \mathcal O$ are functions on a variety $M$, $\mathcal I $ are functions vanishing on a subvariety $X$ and $\mathcal A $ is a $\mathcal O $-module of sections over $M$, $\mathcal I \mathcal A$ corresponds geometrically to sections vanishing along $X$. It is not an easy task.
In codimension $1$, i.e. when $\mathcal I $ is generated by one element, the construction can be done by hand.

\begin{prop}Let $(\mathcal A,\lb_\mathcal A,\rho_\mathcal A)$ be a Lie-Rinehart algebra over a commutative algebra $\mathcal O$. Let $(\E,\ell_k=\{\cdots\}_{k\geq 1},\rho)$ be a Lie $\infty$-algebroid that terminates in $\mathcal A$ through a hook $\pi$. For any element $\chi\in \mathcal O$, the $\mathcal O$-module $\mathcal A'=\chi\mathcal A\subseteq\mathcal A$ is closed under the Lie bracket, so the triple  $(\chi \mathcal A, [\cdot, \cdot]_\mathcal A, \rho_\mathcal A) $ is a Lie-Rinehart algebra over $\mathcal O$. A Lie $\infty$-algebroid $(\mathcal E'=\E,\ell_k'=\{\cdots\}_{k \geq 1}', \rho') $ hooked in $\chi \mathcal A $ through $\pi'$ can be defined as follows:
	\begin{enumerate}
	   \item The brackets are given by
	 \begin{enumerate}
	    \item $\{\cdot\}_1' = \{\cdot\}_1 $,
	    \item the $2$-ary bracket: \begin{equation}
	\{x,y\}'_{2}:=\chi\{x,y\}_{2}+\rho(x)[\chi]\,y+(-1)^{|x||y|}\rho(y)[\chi]\, x,
	\end{equation}
	for all $x,y \in \mathcal E_\bullet$, with the understanding that $\rho=0 $ on $\mathcal E_{\leq -2} $, 
	    \item $\{\cdots\}_k' = \chi^{k-1}\{\cdots\}_k$ for all $k \geq 3$,
	    \end{enumerate}
        \item $\rho_{\E '}= \chi \rho $,
	    \item $\pi' = \chi \pi $.
	    \end{enumerate}
\end{prop}	
\begin{proof}
We leave it to the reader.
 \end{proof}
 	\begin{prop}
 	If $\chi $ is not a zero-divisor in $\mathcal O$, and $(\E,\{\cdots\}_{k\geq 1},\rho,\pi)$ is a universal Lie $\infty$-algebroid of $\mathcal A$, then the Lie $\infty $-structure described in the four items of \ref{constructions} is the universal Lie $\infty $-algebroid of $\chi \mathcal A $. 
 \end{prop}
\begin{proof}
		If $\chi $ is not a zero-divisor in $\mathcal A $ (i.e. if $a \mapsto \chi a $ is an injective endomorphism of $\mathcal A $), then the kernel of $\pi' $ coincides with the kernel of $\pi $, i.e. with the image of $\{\cdot\}_1=\{\cdot\}_1$, so that $(\E,\ell_1,\chi\pi)$ is a resolution of $\chi\mathcal A$.
\end{proof}

\subsection{Algebra extension}\label{blow-up-algebra}

 Recall that for $\mathcal O $ a unital algebra with no zero divisor,  
derivations of $\mathcal O$ induce derivations of its field of fractions $\mathbb O $.

 \begin{prop}\label{prop:blwoup}
 Let $\mathcal O $ be an unital algebra with no zero divisor, $\mathbb O $ its field of fractions, and  $\tilde{\mathcal O} $ an algebra with $\mathcal O \subset \tilde{\mathcal O} \subset \mathbb O$.   
 For every Lie-Rinehart algebra $\mathcal A $ over $\mathcal O$ whose anchor map takes values in derivations of $\mathcal O$ preserving $\tilde{\mathcal O} $, then
 \begin{enumerate}
     \item  any Lie $ \infty$-algebroid structure $ (\E, (\ell_k)_{k \geq 1}, \rho, \pi)$ that terminates at $\mathcal A $ extends for all $i =0, \dots, n $ to a Lie $ \infty$-algebroid structure on $\tilde{\mathcal O} \otimes_{ {\mathcal O}} \mathcal E  $, 
     \item and this extension $\tilde{\mathcal O} \otimes_\mathcal O \mathcal E  $ is a Lie $\infty $-algebroid that terminates at the Lie-Rinehart algebra $\tilde{\mathcal O} \otimes_\mathcal O \mathcal A $.
 \end{enumerate}
 \end{prop}
 \begin{proof}
Since they are $ \mathcal O$-linear, the  hook $\pi$, the anchor $ \rho$, and the brackets $\ell_k $  for $k \neq 2 $ are extended to $\tilde{\mathbb O} $-linear maps. Since the image of $ \rho$  is the image of $\rho_\mathcal A $,  it is made of derivations preserving $\mathbb O $, which is easily seen to allow an extension of $\ell_2 $ to $ \tilde{\mathcal O} \otimes_\mathcal O \mathcal E $ using the Leibniz identity. 
 \end{proof}

\begin{remark}
 Of course, the Lie $\infty $-algebroid structure obtained on $ \tilde{\mathcal O} \otimes_\mathcal O \mathcal E $ is not in general the universal Lie $\infty $-algebroid of $ \tilde{\mathcal O} \otimes_\mathcal O \mathcal A $, because the complex $ (\tilde{\mathcal O} \otimes_\mathcal O \mathcal E , \ell_1, \pi)$
 may not be a resolution of $ \tilde{\mathcal O} \otimes_\mathcal O \mathcal A $ (see  Example \ref{ex:counter}). 
\end{remark}

\begin{remark}
\normalfont
Since any module over a field is projective, any Lie-Rinehart algebra over a field is a Lie algebroid. If we choose $\tilde{\mathcal O}  = \mathbb O $ therefore,
the Lie-Rinehart algebra  $ {\mathbb O} \otimes_\mathcal O \mathcal A $  is a Lie algebroid, so is homotopy equivalent to any of its universal Lie $\infty $-algebroid. 
 Unless $\mathcal A $ is a Lie algebroid itself, the Lie $\infty $-algebroid in Proposition \ref{prop:blwoup} will not be homotopy equivalent to a Lie $\infty $-algebroid whose underlying complex is of length one, and is therefore not a universal Lie $ \infty$-algebroid of $ {\mathbb O} \otimes_\mathcal O \mathcal A $.
 \end{remark}

\subsection{Blow-up}
Let us investigate the behavior of the universal Lie $\infty$-algebroids under blow-up.\\

\noindent
We recall that the blowup of $\mathbb C^{N+1}$ at the origin is given by, $B_0 (\mathbb C^{N+1} ) =\{(x,\ell)\in\mathbb C^{N+1} \times \mathbb P^{N} \mid x \in\ell\}$
together with the map
\begin{align*}
\pi\colon B_0(\mathbb C^{N+1})&\longrightarrow \mathbb C^{N+1}\\(x,\ell)&\mapsto x. 
\end{align*}
 

For $\mathcal O=\mathbb{C}[z_0,\ldots,z_N]$ the coordinate ring of $\mathbb{C}^{N+1}$, the blow-up of $\mathbb{C}^{N+1}$ at the origin is covered by affine charts: in the $i$-th affine chart $U_i$, the coordinate ring is $$\mathcal O_{U_i}=\mathbb{C}[z_0/z_{i}, \dots, z_i , \dots, z_{N}/z_{i} ].$$ 
 By Remark \ref{loc:res}, 
and Proposition \ref{prop:blwoup} 
for any
 Lie-Rinehart algebra $\mathcal A $  whose anchor map takes values in vector fields vanishing at $0 \in \mathbb C^{N+1}$,
 we obtain a Lie $\infty $-algebroid of $\mathcal O_{U_i} \otimes_\mathcal O \mathcal A $ that we call \emph{blow-up of at $0$ in the chart $ U_i$}. Proposition \ref{prop:blwoup} then says that the blow-up at $0$ of the universal Lie $\infty $-algebroid of $\mathcal A $, in each chart, is a Lie $\infty $-algebroid that terminates in the blow-up of $\mathcal A $  (as defined in remark \ref{loc:res}). It may not be the universal one, see Example \ref{ex:counter}.

\begin{example}[Universal Lie-$ \infty$-algebroids and blow-up: a counter example]\label{ex:counter}
Consider
the polynomial function in $N+1$ variables $\varphi=\sum_{i=0}^{N}z_i^3$. 
Let us consider  the singular foliation $\mathfrak F_\varphi\subset{\mathfrak{X}(\mathbb{C}^N)}$  as in Example \ref{koszul}. Its generators are $\Delta_{ij}:= z_i^2 \tfrac{\partial}{\partial z_j}-z_j^2 \tfrac{\partial}{\partial z_i} $, for $0\leq i < j\leq N$. 

Let us consider its blow-up in the chart $U_N$.
Geometrically speaking, $\widetilde{\mathfrak F_\varphi}= \mathcal O_{U_N} \otimes_{\mathcal O}\mathfrak F_\varphi$ is the $\mathcal{O}_{ U_N}$-module generated  by the blown-up vector fields $\widetilde{\Delta }_{ij}=z_N\left( z^2_i\frac{\partial}{\partial z_j}-z_j^2\frac{\partial}{\partial z_i}\right),\,j\neq N$, and $\widetilde{\Delta}_{iN}=z_N\left( z_Nz^2_i\frac{\partial}{\partial z_N}-\frac{\partial}{\partial z_i} -z^2_i\sum_{j=0}^{N-1}z_j\frac{\partial}{\partial z_j}\right)$ of the vector fields $\Delta_{ij}$, for $i,j\in\{0,\ldots,N\}$, w.r.t affine chart $U_N$. 
The vector fields $\widetilde{\Delta}_{ij},\,j\neq N$, belong to the $\mathcal{O}_{ U_N}$-module generated by the vector fields $\widetilde{\Delta}_{iN}$, (explicitly $\widetilde{\Delta}_{ij} =z_j^2 \widetilde{\Delta}_{iN} - z_i^2  \widetilde{\Delta}_{jN}$). Said differently, the vector fields $\widetilde{\Delta}_{iN},\;i=0,\ldots N-1$  are generators of $\widetilde{\mathfrak F_\varphi}$.
Since they are independent, the singular foliation $\widetilde{\mathfrak F_\varphi}$ is a free $\mathcal O_{ U_N} $-module. A universal Lie $\infty $-algebroid for it is therefore concentrated in degree $-1$ and is given by the $\mathcal{O}_{U_N}$-module generated by some set $\left( e_i,\,i=0\ldots N-1\right) $, equipped with the Lie bracket:\begin{equation}
\left[e_i,e_j \right]:= 2z_N\left(z^2_ie_j-z^2_je_i \right)
\end{equation}  together with the anchor map  $\rho$ which assigns $e_i$ to $\widetilde{\Delta}_{iN}$, for $i=1\ldots N-1$. 

On the other hand, the blow-up of the universal Lie $\infty $-algebroid of $\mathfrak F_\varphi $ is not homotopy equivalent to a Lie algebroid. Let us show this point. In this case, the Lie $\infty $-algebroid is given in Section  \ref{koszul}: notice that $ E_{-i} \neq 0$ for $i=1, \dots, N $ and that $\ell_1|_0 = 0$. The pull-back Lie $\infty $-algebroid $(\tilde{E}, (\tilde{\ell}_k)_{k \geq 1}, \tilde{\rho}, \tilde{\pi}) $ verifies by construction that  $ \tilde{E}_{-i} \neq 0$ for $i=1, \dots, N $ and that $\tilde{\ell}_1|_x=0 $ for every $x$ in the inverse image of zero, and such a complex can not be homotopy equivalent to a complex of length $1$.
In other words,  $\mathcal O_{U_i} \otimes_\mathcal O \cdot$ is not an exact functor in this case (which is a classical fact in algebraic geometry).\\

This example tells us that the blow-up of the universal Lie $\infty $-algebroid of an affine variety $W$ may not be the universal Lie $\infty$-algebroid of its blow-up, (even locally).

\end{example}

\subsection{Universal Lie $\infty$-algebroids  of some singular foliations}Singular foliation is defined in Section \ref{sec:Alg-Geo-ex}.
The coming example uses the notion "multi-derivation", it is useful to recall the definition and write down some basic operations on them. We refer the reader to Section 3.1 of \cite{CPA} for more details on this notion.\\

\begin{definition}
   A skew symmetric $k$-multilinear map $P\in \mathrm{Hom}_\mathbb{K}(\mathcal O^k, \mathcal{O})$ is a said to a \emph{$k$-multi-derivation of $\mathcal{O}$} if $P$ is a derivation in each of its argument, i.e., for every $i\in \{1,\ldots,k\}$ and for all $f,g\in\mathcal{O}$ we have 
   
   \begin{equation}\label{eq:multi-der}
       P[f_1, \ldots,\underbrace{fg}_{i-th\; \text{slot}}\ldots, f_k]=P[f_1, \ldots,\underbrace{f}_{i-th\; \text{slot}}\ldots, f_k]\,g+ f\,P[f_1, \ldots,\underbrace{g}_{i-th\; \text{slot}}\ldots, f_k]
   \end{equation}

   When $\mathcal{O}$ happen to be the coordinate ring of some affine variety $W$, $k$-multi-derivations of $\mathcal{O}$ are called \emph{$k$-multi-vector fields on $W$}, and denote by $\mathfrak{X}^k(W)$ the module of $k$-multi-vector fields on $W$. In general, it is denoted by $\mathfrak{X}^k(\mathcal O)$.  
\end{definition}
\begin{remark}
By skew-symmetry argument, it is enough that the relation \eqref{eq:multi-der} holds for the first slot, i.e.  for $i=1$.
\end{remark}
\begin{lemma}
The graded module $\mathfrak{X}^\bullet(\mathcal{O}):=\bigoplus_{k\geq 0}\mathfrak{X}^k(\mathcal{O})$ comes equipped with graded algebra structure whose product $\wedge$ is defined as follows,

$$(P\wedge R)[f_1,\ldots,f_{k+l}]:=
\sum_{\sigma \in\mathfrak S(k,l)}\epsilon(\sigma)P[f_{\sigma(1)},\ldots,f_{\sigma(k)}]R[f_{\sigma(k+1)},\ldots,f_{\sigma(k+l)}],$$
for $P\in\mathfrak{X}^k(\mathcal{O}),\;R\in \mathfrak{X}^l(\mathcal{O})$ and $f_1, \ldots, f_{k+l}\in \mathcal{O}$. By convention $\mathfrak{X}^0(\mathcal{O})=\mathcal{O}$. Also, $f\wedge X:=fX$ for all $f\in \mathcal O.$
\end{lemma}
\begin{lemma}
For every $f\in\mathcal{O}$, the contraction by $f$, $\iota_f\colon \mathfrak{X}^k(\mathcal{O})\rightarrow \mathfrak{X}^{k-1}(\mathcal{O}), \forall\; k\geq 1$,   which is defined as follows\footnote{It should be understood that $\iota_f(\mathcal{O}):=0$.}\begin{align}
    P\mapsto\iota_f(P)[f_1,\ldots,f_{k-1}]:= P[f,f_1,\ldots,f_{k-1}],\quad \text{for all }\;\; f_1,\ldots,f_{k-1}\in \mathcal{O}
\end{align}
satisfies 
\begin{itemize}
    \item $\iota_f\circ \iota_f=0$,
    \item for all $P\in\mathfrak{X}^k(\mathcal{O})$ and $R\in\mathfrak{X}^l(\mathcal{O})$, \begin{equation}
        \iota_f(P\wedge R)= \iota_f(P)\wedge R +(-1)^k P\wedge \iota_f(R).
    \end{equation}
    \end{itemize}
\end{lemma}
\begin{proof}
The first item is true by skew-symmetry of multi-derivations. Now let us prove the second item. For $f_{1},\ldots ,f_{k+l-1}\in \mathcal{O}$, we have \begin{align*}
\iota_{f}(P \wedge R )[f_{1},\ldots ,f_{k+l-1}]&=(P \wedge R) [f,f_{1},\cdots ,f_{k+l-1}]\\&=\sum _{\sigma\in \mathfrak S(k-1,l)}\epsilon(\sigma)P [f,f_{\sigma(1)},\ldots ,f_{\sigma(k-1})]R[f_{\sigma(k)},\ldots ,f_{\sigma(k+l-1)}]+\\&\quad(-1)^k\sum _{\sigma\in \mathfrak S(k,l-1)}\epsilon(\sigma)\underbrace{P[f_{\sigma(1)},\ldots ,f_{\sigma(k)}]R[f,f_{\sigma(k+1)},\ldots ,f_{\sigma(k+l-1)})}_{\text{here, the inversion number with}\;f\;\text{is $k$}}\\&=\iota_f(P)\wedge R +(-1)^k P\wedge \iota_f(R).\end{align*}
We have used the fact that for every $\sigma\in \mathfrak S(k,l)$, one has that either $\sigma(1) = 1$ or $\sigma(k+1) = 1$.
\end{proof}

\begin{remark}Notice the following:
\begin{itemize}
\item For all $f\in\mathcal{O}$ and $X\in \mathrm{Der}(\mathcal
O)$, one has $\iota_f(X)=X[f]$.
    \item For any two derivation $X,Y\in\mathrm{Der}(\mathcal O)$ we have $$(X\wedge Y)[f,g]=X[f] Y[g]- Y[g]X[f].$$ \item For a family of derivations $X_1, \ldots X_k\in \mathrm{Der}(\mathcal O)$, the $k$-multi-derivation $X_1\wedge\cdots\wedge X_k$ applied to $f_1, \cdots, f_k\in \mathcal
    O$ is equal to  the determinant  $$ \begin{vmatrix} X_1[f_1]& \cdots& X_1[f_k]\\ \vdots && \vdots\\X_k[f_1] & \cdots &X_k[f_k]\end{vmatrix}.$$
    
    \item When $\mathcal{O}$ is the algebra of polynomial functions in $d$ variables $\mathbb K[x_1, \ldots,x_d]$, a $k$-multi-derivation $P$  admits the coordinate expression
$$P =\sum_{1\leq i_1<\cdots<i_k\leq d} P[x_{i_1},\ldots,x_{i_k}]\frac{\partial}{\partial x_{i_1}}\wedge \cdots \wedge \frac{\partial}{\partial x_{i_k}}.$$

\end{itemize}
\end{remark}

\subsubsection{Lie derivative}
For $P\in \mathfrak{X}^k(\mathcal{O})$ and $X\in\mathrm{Der}(\mathcal O)$, the \emph{Lie derivative} $\mathcal{L}_XP\in \mathfrak{X}^k(\mathcal O)$ of $P$ along $X$ is defined as \begin{equation}\label{formula:Lie-der}(\mathcal{L}_X P)[f_1, \cdots , f_k]:=X\left[P[f_1, \cdots, f_k]\right]-\sum_{j=1}^kP[f_1, \ldots, X [f_j],\ldots, f_k].\end{equation}

There is another important operation on multi-derivations, the so-called "Schouten bracket" which is a generalization of the commutator of derivations, also Lie derivative of multi-derivation along a vector field. We will not recall how it is defined here, since we do not really  make use of it. We refer the reader to, e.g \cite{CPA} for this notion.
\subsection{Vector fields annihilating a Koszul function $\varphi$ }\label{koszul}

This universal Lie $\infty $-algebroid was already described in Section 3.7 of \cite{LLS}, where the brackets were simply checked to satisfy the higher Jacobi identities - with many computations left to the reader. Here, we give a theoretical explanation of the construction presented in \cite{LLS}.

Let $\mathcal O$ be the algebra of all polynomials on $V:=\mathbb{C}^{d}$. A function $\varphi \in \mathcal O$ 
is said to be a \emph{Koszul polynomial}, if the \emph{Koszul complex} 
\begin{equation}
    \label{eq:KoszulComplex}
\ldots\xrightarrow{\iota_{\varphi}}\X^3 (V) \xrightarrow{\iota_{\varphi}}\X^2 (V)\xrightarrow{\iota_{\varphi}}\X(V) \xrightarrow{\iota_{\varphi}}\mathcal O \longrightarrow 0\end{equation}
is exact in all degree, except in degree $0$.
By virtue of a theorem of Koszul \cite{zbMATH00704831}, see \cite{Matsumura} Theorem 16.5 $(i)$, $\varphi$ is \emph{Koszul} if $\left(\frac{\partial\varphi}{\partial x_1} ,\cdots,\frac{\partial\varphi}{\partial x_d}\right) $ is a regular sequence. 

From now on, we choose $\varphi $ a Koszul function, and consider the Lie-Rinehart algebra (which is a singular foliation) \begin{equation}\label{varphi}
\mathfrak{F}_\varphi:=\{X\in\mathfrak{X}(V):X[\varphi]=0\} = {\mathrm{Ker}}( \iota_{\varphi} )\colon \X(V) \xrightarrow{\iota_{\varphi}}\mathcal O 
.\end{equation} 
 The Koszul complex \eqref{eq:KoszulComplex} truncated of its degree $0$ term  gives  a free resolution $(\E,\dif,\rho)$ of $\mathfrak{F}_\varphi$, with $\E_{-i}:=\X^{i+1}(V)$, $\dd:=\iota_{\varphi}$, and $\rho:=-\iota_{\varphi}$.
 
\begin{remark}
Exactness of the Koszul complex implies in particular that $\mathfrak{F}_\varphi$ is generated by the vector fields
: \begin{equation}\left\lbrace  \frac{\partial\varphi}{\partial x_i}\frac{\partial}{\partial x_j}-\frac{\partial\varphi}{\partial x_j}\frac{\partial\varphi}{\partial x_i},\mid 1\leq i<j\leq d\right\rbrace .
\end{equation}
\end{remark} 

In \cite{LLS}, this resolution is equipped with a Lie $\infty $-algebroid structure, whose brackets we now recall. 
 
\begin{prop}\label{prop-koszul}
A universal Lie $\infty $-algebroid of $\mathfrak{F}_\varphi\subset\mathfrak{X}(V)$ is given on the free resolution $\left(\E_{-\bullet} = \X^{\bullet+1}(V),\text{\emph{d}}=\iota_{\varphi},\rho= -\iota_{\varphi} \right)$ 
by defining the following $n$-ary brackets:
\begin{equation}
\label{eq:nary}
\left\lbrace\partial_{I_{1}},\cdots, \partial_{I_{n}}\right\rbrace_{n}:=\sum_{i_{1}\in I_{1},\ldots,i_{n}\in I_{n}} \epsilon(i_{1},\ldots,i_{n})\varphi_{i_{1}\cdots i_{n}}\partial_{I_{1}^{i_{1}}\bullet\cdots\bullet I_{n}^{i_{n}}};
\end{equation}and the anchor map given for all $ i,j \in \{1, \dots,n\}$ by\begin{equation}
\rho\left(\frac{\partial}{\partial x_i}\wedge\frac{\partial}{\partial x_j}\right) :=\frac{\partial\varphi}{\partial x_j}\frac{\partial}{\partial x_i}-\frac{\partial\varphi}{\partial x_i}\frac{\partial}{\partial x_j}.
\end{equation}Above, for every multi-index $J=\left\lbrace j_1,\ldots ,j_n\right\rbrace\subseteq\left\lbrace 1,\ldots,d\right\rbrace$ of length $n$, $\partial_J$ stands for the $n$-vector field $\frac{\partial}{\partial x_{j_1}}\wedge\cdots\wedge\frac{\partial}{\partial x_{j_n}}$ and $\varphi_{j_{1}\cdots j_{n}}:=\frac{\partial^{n}\varphi}{\partial x_{j_1}\cdots\partial x_{j_n}}$ . Also, $I_{1}\bullet\cdots\bullet I_{n}$ is a multi-index obtained by concatenation of $n$ multi-indices $I_{1},\ldots,I_{n}$. For every $i_1\in I_1,\ldots,i_n\in I_n$,\;$\epsilon(i_1,\ldots,i_n)$ is the signature of the permutation which brings $i_1,\ldots,i_n$ to the first $n$ slots of $I_{1}\bullet\cdots\bullet I_{n}$. Last, for $i_s\in I_s$, we define $I_{s}^{i_s}:=I_s\backslash i_s$.\end{prop}

\noindent
To understand this structure, let us first define a sequence of degree $ +1$ graded symmetric poly-derivations on $\X^\bullet (V) $ (by convention, $i$-vector fields are of degree $-i+1 $) by:
\begin{equation}
\left\lbrace \partial_{i_1},\ldots,\partial_{i_k}\right\rbrace _k':=\frac{\partial^{k}\varphi}{\partial x_{i_1}\cdots\partial x_{i_k}}.
\end{equation}
We extend them to a graded poly-derivation of $\X^\bullet (V)$.

\begin{lemma}
\label{lem:Poisson}
The poly-derivations $(\left\lbrace \cdots \right\rbrace_k' )_{k \geq 1} $ are $\mathcal O $-multilinear and equip $\X^\bullet (V) $ with a (graded symmetric) Poisson $\infty$-algebra structure. 
Also, $\{\cdot \}_1'=\iota_{\varphi}$.
\end{lemma}   
\begin{proof} 
For degree reason, $\left\lbrace F,  X_1, \dots, X_{k-1}  \right\rbrace_k'= 0$ for all $X_1, \dots, X_{k-1}\in \X^\bullet (V)$ and all $ F \in \X^0(V) = \mathcal O$. This implies the required $\mathcal O $-multilinearity. It is clear that the higher Jacobi identities hold since brackets of generators $\{\delta_{i_1}, \cdots, \delta_{i_n}\}' $ are elements in $ \mathcal O $, and all brackets are zero when applied an element in $ \mathcal O$.
\end{proof}

\begin{proof} [Proof (of Proposition \ref{prop-koszul})]
The brackets introduced in Proposition \ref{prop-koszul} are modifications of the Poisson ${\infty}$-algebra described in Lemma \ref{lem:Poisson}. By construction, $\left\lbrace \cdots\right\rbrace^{'}_{n}=\left\lbrace
\cdots\right\rbrace_{n}$ when all arguments are generators of the form $ \partial_{I}$ for some $I \subset \{1 \dots, n\} $ of cardinal $\geq 2 $. By $\mathcal O $-multilinearity, this implies  $\left\lbrace \cdots\right\rbrace^{'}_{n}=\left\lbrace
\cdots\right\rbrace_{n}$ when $n \geq 3$, or when $n=2$ and no argument is a bivector-field, or when  $n=1$ and the argument is not a bivector field. As a consequence, all higher Jacobi identities hold when applied to $n$-vector fields with $n \geq 3$.

Let us see what happens when one of the arguments is a bivector field, i.e. in the case where we deal with at least an element of degree $-1$. Let us assume that there is one such element, i.e. $Q_1=\partial_i\wedge\partial_j,\,Q_2=\partial_{I_{2}},\,\ldots,Q_n=\partial_{I_{n}}$ with $\lvert I_{j}\rvert\geq 3,\; j=2,\ldots,n$. Then, in view of the higher Jacobi identity for the Poisson $\infty$-brackets $ (\{\cdots\}_k')_{k\geq 1}$ gives:
\begin{align}
0=\sum_{2 \leq k  \leq n-2}&\label{a}\sum_{\sigma\in S_{k,n-k}}\epsilon(\sigma)
\left\lbrace \left\lbrace  Q_{\sigma(1)},\ldots,Q_{\sigma(k)}\right\rbrace_{k}',Q_{\sigma(k+1)},\ldots,Q_{\sigma(n)}\right\rbrace_{n-k+1}'\\&\label{b}+\sum_{\sigma\in S_{n-1,1}, \sigma(n) \neq 1}\epsilon(\sigma)\left\lbrace \left\lbrace  Q_{\sigma(1)},\ldots,Q_{\sigma(n-1)}\right\rbrace_{n-1}',Q_{\sigma(n)}\right\rbrace _{2}' \\&
\label{bb}+\sum_{\sigma\in S_{1,n-1}, \sigma(1) \neq 1 }\epsilon(\sigma)  \left\lbrace \left\lbrace  Q_{\sigma(1)}\right\rbrace_{1}',Q_{\sigma(2)}\ldots,Q_{\sigma(n)}\right\rbrace _{n}'
\\&\label{c}+(-1)^{\sum_{k=2}^n\lvert\partial_{I_k}\rvert}\left\lbrace \left\lbrace  Q_{2},\ldots,Q_{n}\right\rbrace_{n-1}',Q_{1}\right\rbrace _{2}'
\\&\label{d}+  \left\lbrace \left\lbrace  Q_{1}\right\rbrace_{1}',Q_{\sigma(2)}\ldots,Q_{\sigma(n)}\right\rbrace _{n}'.
\end{align}
In lines \eqref{a}-\eqref{b}-\eqref{bb} above, we have $\{\cdots \}'=\{\cdots \} $ for all the terms involved. This is not the case for \eqref{c}-\eqref{d}.
Indeed: \begin{align*}
\left\lbrace\left\lbrace \partial_{I_{2}},\cdots,\partial_{I_{n}}\right\rbrace^{'}_{n-1} ,\partial_i\wedge\partial_j\right\rbrace^{'}_2
&=
\left\lbrace\left\lbrace \partial_{I_{2}},\cdots,\partial_{I_{n}}\right\rbrace_{n-1} ,\partial_i\wedge\partial_j\right\rbrace_2
\\&-  \sum_{i_{2}\in I_{2},\ldots,i_{n}\in I_{n}}\epsilon(i_{2},\ldots,i_{n})\rho(\partial_i\wedge\partial_j)[\varphi_{i_{2}\cdots i_{n}} ] \, \partial_{I_{2}^{i_{2}}\bullet\cdots\bullet I_{n}^{i_{n}}}
\end{align*}
and
\begin{align*}
\left\lbrace\left\lbrace \partial_i\wedge\partial_j\right\rbrace _{1}',\partial_{I_{2}}\ldots,\partial_{I_{n}} \right\rbrace _{n}'&=(-1)^{\sum_{k=2}^n\lvert\partial_{I_k}\rvert+1}\left( \varphi_i\left\lbrace\partial_j, \partial_{I_{2}}\ldots,\partial_{I_{n}}\right\rbrace_{n}' -\varphi_j\left\lbrace\partial_i, \partial_{I_{2}}\ldots,\partial_{I_{n}}\right\rbrace_{n}'\right) \\&=-(-1)^{\sum_{k=2}^n\lvert\partial_{I_k}\rvert}\sum_{i_{2}\in I_{2},\ldots,i_{n}\in I_{n}}\epsilon(i_{2},\ldots,i_{n})\iota_{\varphi}(\partial_i\wedge\partial_j)[\varphi_{i_{2}\cdots i_{n}} ]\,\partial_{I_{2}^{i_{2}}\bullet\cdots\bullet I_{n}^{i_{n}}} \\ 
& = (-1)^{\sum_{k=2}^n\lvert\partial_{I_k}\rvert}\sum_{i_{2}\in I_{2},\ldots,i_{n}\in I_{n}}\epsilon(i_{2},\ldots,i_{n})
\rho(\partial_i\wedge\partial_j)[\varphi_{i_{2}\cdots i_{n}} ] \,\partial_{I_{2}^{i_{2}}\bullet\cdots\bullet I_{n}^{i_{n}}}
\end{align*} since $\rho=-\iota_{\varphi}$.
Hence, the quantities in lines \eqref{d} and  \eqref{c} add up, when we re-write them in terms of the new brackets $\left\lbrace \cdots\right\rbrace^{} _{k}$,  to yield precisely  the higher Jacobi identity for this new bracket. It is then not difficult to see this is still the case if there is more than one bivector field, by using many times the same computations.
\end{proof}
\subsection{Restriction to $\varphi=0 $ of vector fields annihilating $\varphi $}

We keep the convention and notations of the previous section.
Let us consider  the restriction $\mathfrak i_W^* \mathfrak{F}_\varphi$ of the Lie-Rinehart algebra $\mathfrak{F}_\varphi$ to the zero-locus $W$ of a Koszul polynomial $\varphi $. 
Since all vector fields in $\mathfrak F_\varphi$ are tangent to $W$, this restriction is now a Lie-Rinehart algebra over $\mathcal{O}_W = \tfrac{\mathcal O}{\mathcal O \varphi}$,
see Section \ref{loc:res} (1).

\begin{proposition}
\label{prop:univphiIsZero}
Let $\varphi$ be a Koszul Polynomial. The restriction of the universal Lie $\infty$-algebroid of Proposition \ref{prop-koszul}
to the zero-locus $W$ of $\varphi $ is
a universal Lie $\infty$-algebroid of the Lie-Rinehart algebra $\mathfrak i_W^* \mathfrak{F}_\varphi$. 
\end{proposition}

Since the image of its anchor map are vector fields tangent to $W $,
it is clear that the universal Lie $\infty$-algebroid of Proposition \ref{prop-koszul} restricts to $W$. To prove Proposition \ref{prop:univphiIsZero}, it suffices to check that the restriction $\mathfrak i_W^* \X (V) $ to $W $ of the Koszul complex is still exact, except in degree $0$. 

\begin{lemma}The restriction to the zero locus $W$ of $ \varphi $ of the Koszul complex \eqref{eq:KoszulComplex}, namely the complex,
$$\ldots\xrightarrow{\iota_{\varphi}}\mathfrak i_W^* \X^{3}(V)\xrightarrow{\iota_{\varphi}}\mathfrak i_W^* \X^{2}(V)\xrightarrow{\iota_{\varphi}}\mathfrak i_W^* \mathfrak F  $$
is a free resolution of $\mathfrak i_W^* \mathfrak F$ in the category of $\mathcal O_W $-modules.
\end{lemma}
\begin{proof}
Let $k \geq 2$ be an integer.
For a given $P \in \mathfrak i_W^* \X^{k}(V) $, the relation $\iota_{\varphi} P=0$ means that for any $ \tilde{P} \in \X^{k}(V)$ extending $P$,  there exists $U\in\X^{k-1} (V)$ such that $\iota_{\varphi}\tilde{P}=\varphi U$. By exactness of the Koszul complex (\ref{eq:KoszulComplex}), one has, $U=\iota_{\varphi} \tilde Q$ for some $\tilde Q\in\X^{k}(V)$. Hence, $\tilde{P}-\varphi \tilde Q$ is an extension of the bivector field $P$ such that $\iota_{\varphi}(\tilde{P}-\varphi \tilde  Q)=0$. Using one more time exactness of the Koszul complex (\ref{eq:KoszulComplex}), we construct $\tilde R \in \X^{k+1}(V)$ such that $\tilde{P}=\varphi Q+\iota_{\varphi} \tilde R$ . Thus, $ P = \mathfrak i_W^* \tilde{P}=\iota_{\varphi} \mathfrak i_W^* \tilde  R = \iota_{\varphi} R $.
\end{proof}

\subsection{Vector fields vanishing on subsets of a vector space} 
\label{ideal}
Let $ \mathcal O$ be the algebra of smooth or holomorphic or polynomial or formal functions on $\mathbb K^d $, and $\mathcal I \subset \mathcal O$ be an ideal.
Then $ \mathcal I {\mathrm{Der}} (\mathcal O) $, i.e. vector fields of the form: 
 $  \sum_{i=1}^d f_i  \frac{\partial}{\partial x_i} $, with $f_1, \dots, f_d \in \mathcal I$, is a Lie-Rinehart algebra (It is also a singular foliation). 

\begin{remark}
 Geometrically, when $\mathcal I $ corresponds to functions vanishing on a sub-variety $N \subset \mathbb K^n$, 
  $ \mathcal I {\mathrm{Der}} (\mathcal O) $ must be interpreted as vector fields vanishing along $N$.
\end{remark}

Let us describe a Lie $\infty $-algebroid that terminates at $ \mathcal I {\mathrm{Der}} (\mathcal O)$, then discuss when it is universal. 
Let $(\varphi_i)_{i \in I}$  
be generators of $\mathcal I $. Consider the free graded algebra $ \mathcal K = \mathcal O [(\mu_i)_{i \in I} ] $
generated by variables $(\mu_i )_{i \in I} $ of degree $-1 $. The degree $-1 $ derivation 
$\partial:=\sum_{i\in I} \varphi_i\frac{\partial}{\partial\mu_i}$ squares to zero. 
The $\mathcal O$-module $\mathcal K_{-j}$ of elements degree $j$ in $\mathcal K_\bullet $ is made of all sums
$\sum_{i_1 , \dots , i_j \in I }   f_{i_1 \dots i_j}   \mu_{i_1}\cdots\mu_{i_j} $
with $ f_{i_1 \dots i_j} \in \mathcal O$. Consider the complex of free $\mathcal O $-modules
\begin{equation}\label{eq:resolIX} \cdots\stackrel{\partial\otimes_\mathcal O \text{id}}{\longrightarrow} \mathcal K_{-2} \otimes_\mathcal O {\mathrm{Der}}(\mathcal O)   \stackrel{\partial\otimes_\mathcal O \text{id}}{\longrightarrow} \mathcal K_{-1}\otimes_\mathcal O{\mathrm{Der}}(\mathcal O)\end{equation}

\begin{proposition} \label{prop:existsLieInfty}
The complex \eqref{eq:resolIX} comes equipped with a Lie $\infty $-algebroid structure that terminates in $\mathcal I {\mathrm{Der}} (\mathcal O)$ through the anchor map  given by $\mu_i\frac{\partial}{\partial x_j}   \mapsto \varphi_i \frac{\partial}{\partial x_j} $ for all $i \in I$, and $ j \in 1, \dots,d$. 
\end{proposition} 
\begin{proof}
 First, one defines a $\mathcal O$-linear Poisson-$\infty $-algebra structure on the free algebra generated by $(\mu_i)_{i\in I} $ (in degree $-1$) and $  \left( \frac{\partial}{\partial x_j} \right)_{j=1}^d $ (in degree $0$) and $1$ by:
     \begin{equation}
         \label{PoissonInfty}
\left\lbrace\mu_{i},{  \frac{\partial}{\partial x_{j_1}}},\ldots,\frac{\partial}{\partial x_{j_r}} \right\rbrace_{r+1}':=   \frac{\partial^r \varphi_i}{\partial x_{i_1} \dots \partial x_{i_r} }   
 \end{equation}
all other brackets of generators being equal to $0$. Since the brackets of generators take values in $\mathcal O$, and since an $n$-ary bracket where an element of $\mathcal O $ appears is zero, this is easily seen to be a Poisson $\infty $-structure. The general formula is
\begin{equation}
\label{eq:onGenMinusALl}
\left\lbrace\mu_{I_1}\otimes_\mathcal O\partial_{x_{a_1}},\ldots,\mu_{I_n}\otimes_\mathcal O\partial_{x_{a_n}} \right\rbrace_n':= \sum_{\hbox{\scalebox{0.5}{$
\begin{array}{c} j=1 , \dots, n \\ 
i_j \in I_j \end{array}$}}
} \epsilon
\frac{\partial^{n-1}\varphi_{i_j}}{\partial x_{a_1}\cdots\widehat{\partial x}_{a_j}\cdots\partial x_{a_n}}  \mu_{I_1} \cdots \mu_{I_j}^{i_j} \cdots \mu_{I_n\otimes_{\mathcal O} \partial_{x_{a_j}}},
\end{equation}    
 where $\mu_J = \mu_{j_1} \dots \mu_{j_s}$ for every list $J= \{j_1, \dots, j_s\} $, 
where $ \epsilon$ is the Koszul sign, and where for a list $J$ containing $j$, $J^j $ stands for the list $J$ from which the element $j$ is crossed out, as in Equation \eqref{eq:nary}.

The $\mathcal O $-module generated by $\mu_{i_1} \cdots \mu_{i_k } \otimes_{\mathcal O}  \partial_{ x_{a}}  $ , i.e.
the complex \eqref{eq:resolIX} is easily seen to be stable under the brackets $\{\cdots \}_k' $
for all $k \geq 1$, so that we can define on $\mathcal K \otimes_{\mathcal O}  {\mathrm{Der}} (\mathcal O)$ a sequence of brackets
$(\ell_k = \{\cdots\}_k)_{k \geq 1} $ by letting them coincide with the previous brackets on the generators, i.e. $\{\cdots\}_n$ is given by Equation \eqref{eq:onGenMinusALl} for all $n \geq 1$. The brackets are then extended by derivation, $\mathcal O$-linearity or Leibniz identity with respect to the given anchor map, depending on the degree.
\noindent 
In particular, $\{\cdots \}'_k= \{\cdots \}_k $ 
for $k\geq 2$. For $k=1$,  $\{\cdots \}'_1= \{\cdots \}_1 $ on  $\oplus_{i \geq 2} \mathcal K_{-i} \otimes_\mathcal O  {\mathrm{Der}} (\mathcal O) $.
For $k=2$, we still have $\{\cdots \}'_2= \{\cdots \}_2 $ on  $\oplus_{i,j \geq 2} \mathcal K_{-i} \odot \mathcal K_{-j} \otimes_\mathcal O  {\mathrm{Der}} (\mathcal O) $.\\

Let us verify that all required axioms are satisfied.  
For $n=2$, Equation
\eqref{eq:onGenMinusALl} specializes to:
 \begin{align*} 
 \ell_2 ( \mu_i \otimes_\mathcal O \partial_{x_a} , \mu_j \otimes_\mathcal O \partial_{x_b} ) = \frac{\partial \varphi_j}{\partial x_a} \,  \mu_i \otimes_\mathcal O \partial_{x_b} - \frac{\partial \varphi_i}{\partial x_b} \,  \mu_j \otimes_\mathcal O \partial_{x_a} 
 \end{align*}
 which proves that the anchor map is a morphism when compared with the relation:
  $$  \left[ \varphi_i \partial_{x_a}  \, , \, \varphi_j \partial_{x_b}  \right] = 
  \frac{\partial \varphi_j}{\partial x_a} \,  \varphi_i  \, \partial_{x_b} - \frac{\partial \varphi_i}{\partial x_b} \,  \varphi_j  \, \partial_{x_a}.
  $$
  The higher Jacobi identities are checked on generators as follows:
 \begin{enumerate}
     \item When there are no degree $-1$ generators, it follows from the higher Jacobi identities of the Poisson $\infty$-structure \eqref{PoissonInfty} and the $ \mathcal O$-multilinearity of all Lie $\infty $-algebroid brackets involved.
     \item When generators of degree $-1 $ are involved, the higher Jacobi identities are obtained by doing the same procedure as in the proof of Proposition \ref{lem:Poisson}, that is, we first consider the higher Jacobi identities for the Poisson $\infty $-structure \eqref{PoissonInfty}, and we put aside the terms where $\{\cdot\}_1 $ is applied to these degree $-1 $ generators. We then check that the latter terms are exactly the terms coming from an anchor map when the $2$-ary bracket is applied to generators of degree $-1$ and the $(n-1)$-ary brackets of the remaining generators. \end{enumerate}\end{proof}

\subsection{Vector fields vanishing on a complete intersection}
\begin{proposition}\label{complete} 
Let $W \subset   \mathbb C^n$ be an affine variety defined by a regular sequence $\varphi_1, \dots, \varphi_k \in \mathcal{O}$. Then, the Lie $ \infty$-algebroid described in Proposition \ref{prop:existsLieInfty} is the universal Lie $\infty$-algebroid of the singular foliation of vector fields vanishing along $W$.
\end{proposition}
\begin{proof}
In the notation of the proof of Proposition \ref{prop:existsLieInfty}, $ \mathcal K_\bullet$ equipped with the derivation $ \partial=\sum_{i=1}^k \varphi_i \frac{\partial}{\partial \mu_k} $ is a free $ \mathcal O$-resolution of the ideal $\mathcal I_W  $ of functions vanishing along $W$, since $\varphi_1, \dots, \varphi_k $ is  a regular sequence. Since $\mathfrak X(\mathbb{C}^d)$ is a flat $\mathcal O $-module,  the sequence \begin{equation}\label{eq:resolIXX}
\xymatrix{\cdots\ar[rr]^-{\partial\otimes_\mathcal O \text{id}} & & \mathcal K_{-2} \otimes_\mathcal O \mathfrak X(\mathbb{C}^d) \ar^{\partial\otimes_\mathcal O \text{id}}[rr]& & \mathcal K_{-1}\otimes_\mathcal O \mathfrak X(\mathbb{C}^d) \ar^ {\partial\otimes _\mathcal O\text{id}}[rr] & & \mathcal I_W \mathfrak X(\mathbb{C}^d).
}
\end{equation} is a free $\mathcal O $-resolution of the singular foliation $\mathcal I_W \mathfrak X(\mathbb{C}^d)$.
The Lie $\infty $-algebroid structure of Proposition \ref{prop:existsLieInfty} is therefore universal.
\end{proof}

\begin{example}
As a special case of the Proposition \ref{complete},  let us consider a complete intersection defined by \emph{one} function, i.e. an affine variety $W$ whose ideal $\langle\varphi\rangle$ is generated by a regular polynomial $\varphi \in \mathbb{C}[X_1,\ldots,X_d]$. One has a free resolution of the space of vector fields vanishing on $W$ given as follows: 
$$ 
\xymatrix{0\ar[r] & \mathcal{O}\mu \otimes_\mathcal O \X(\mathbb C^d) \ar^{\varphi\frac{\partial}{\partial\mu}\otimes_\mathcal O \text{id}}[rr] & & I_W \mathfrak X({\mathbb C^d}}),
$$
where $\mu$ is a degree $ -1$ variable, so that $\mu^2=0$. The universal Lie $\infty$-algebroid structure over that resolution is given on the set of generators by :$$\{\mu\otimes_\mathcal O \partial_{x_a},\mu\otimes_\mathcal O \partial_{x_b}\}_2:=\frac{\partial\varphi}{\partial x_a}\mu\otimes_\mathcal O \partial_{x_b}-\frac{\partial\varphi}{\partial x_b}\mu\otimes_\mathcal O \partial_{x_a}$$ and $\{\cdots\}_k:=0$ for every $k\geq 3$. It is a Lie algebroid structure. Notice that this construction could be also be recovered using Section \ref{sec:codim1}.
\end{example}

    \vspace{2cm}

\begin{tcolorbox}[colback=gray!5!white,colframe=gray!80!black,title=Conclusion:]
This chapter described the $1$-$1$ correspondence "Lie-Rinehart algebras $\longleftrightarrow$ Lie $\infty$-algebroids on acyclic complexes". It extends greatly \cite{LLS} for singular foliations. The functor "$\longleftarrow$" consists in the $1$-truncation of the  Lie $\infty$-agebroid structure. The converse functor consists in taking any free resolution, and constructing the brackets by recursion.\\

\phantom{cc}We prove that it is unique by proving it is universal. Notice that we need the "complicated" notion of homotopy given in Definition \ref{def:homotopy}.\\

\phantom{cc}Last, some examples of \cite{LLS} are conceptually understood, and new examples are given. Some algebraic constructions (blow-up, localization, germs, quotient) are also given.\\

\phantom{cc}Obstruction classes to the existence of a Lie algebroid with a surjective
morphism onto the Lie-Rinehart algebra are also described.\\

\begin{center}
    We never assume finite rank here!
\end{center}

\end{tcolorbox}

\part{Geometric Applications}
\chapter{Universal Lie $\infty$-algebroids of affine varieties}
\label{chap:5}
In this chapter, we apply the results of Chapter \ref{Chap:main} to answer some elementary but open questions that have to do with algebraic geometry, such as the interaction between the singularities of an affine variety and its Lie algebra of vector fields.

Notice that the Theorem 2.1 of Chapter \ref{sec:main} only says that it is possible to associate a  Lie $\infty$-algebroid structure to an affine variety by considering the Lie-Rinehart algebra made of its vector fields. But the construction can be extremely complicated, see e.g. Section \ref{example-oid}. We only have an existence theorem. This leads to the natural question:\\

\noindent
{\textbf{Question}.\emph{ How is the geometry of an affine variety related to its universal Lie $\infty$-algebroid?}}\\

For instance, in view of the construction of Section 3.10. It is also relevant to ask about the effect of blow-ups on this construction. We will see in Example \ref{ex:counter} that blow-ups may change a universal Lie
$\infty$-algebroid to a Lie algebroid. But this example does not tell us really how the higher brackets
disappear under the effect of blow-ups. One of the important question we may ask is about
the description of “the big theorem” of Hironaka \cite{KollarJanos} in terms of the universal Lie $\infty$-algebroids
obtained at each step while resolving singularities. This question remains open, but there are several other problems about which we are able to make some progresses. Those are quite modest, but, at least, we want to have the question clarified. Application to "blowup-up" will be discusses in Section \ref{blow-up-procedure}.

\section{Background on  affine varieties and some constructions}\label{sec:reminder-affine-variety}
We recall definitions and some main properties of the notion of affine variety in order to fix notations. Our main references for this chapter are \cite{Hartshorne,zbMATH00704831,lakshmibai2015grassmannian}.

In this chapter we will sometimes see $\mathbb{C}^d$ as the $d$-dimensional affine space which is commonly denoted by $\mathbb{A}^d_\mathbb{C}$, forget about its vector space structure, but here we will not make any notational distinctions. The latter is equipped  with the  Zariski topology  that is, the topology whose closed subsets are  the zero set of some ideal $\mathcal I\subseteq{\mathbb C[x_1,\ldots, x_d]=:\mathcal O}$, i.e. which are of the form $\{a=(a_1,\ldots,a_d)\in \mathbb{C}^d\mid f(a)=0, \forall f\in \mathcal{I} \}.$  One can check that these subsets indeed define a topology on $\mathbb{C}^d$ \cite{Smith-Karen-E,Hartshorne,lakshmibai2015grassmannian}.

\begin{definition}\label{def:affine-v}
   An affine variety $W\subseteq{\mathbb{C}^d}$ is a  the zero locus of an ideal $\mathcal I\subseteq{ \mathcal O}$, i.e., $W:=Z(\mathcal I):=\{a=(a_1,\ldots,a_d)\in \mathbb{C}^d\mid f(a)=0, \forall f\in \mathcal{I}\}.$  It admits a topology, induced by the Zariski topology in $\mathbb C^d$.
\begin{remark}
   In  Definition \ref{def:affine-v}, we do not exclude irreducible varieties, e.g.  $W=\{(x,y)\in \mathbb C^2\mid xy=0\}$ is an affine variety.
\end{remark}
 \end{definition}
The following facts and remarks are important. For $W\subseteq \mathbb{C}^d$ an affine variety in the notation of Definition \ref{def:affine-v},
\begin{enumerate}
    \item we denote by $\mathcal{I}_W$ the vanishing ideal of $W$, namely, $$\mathcal{I}_W=\{ f\in \mathcal{O} \mid f(x)=0,\;\forall \,x\in W\}.$$In general, we have $\mathcal{I}_W\neq \mathcal{I}$ because the vanishing ideal of the affine variety ${W}=\{x^2=0\}$ is the ideal $\mathcal{I}_W=\langle x\rangle\neq \langle x^2\rangle$. Notice that $\mathcal{I}_W$ can be defined for any arbitrary subset $W\subset \mathbb{C}^d$.

 It is easy to check that
 \begin{enumerate}
     \item for  $S\subseteq T\subseteq{\mathcal O}$, one has $Z(S)\supseteq Z(T)$,
     \item for $U\subseteq V\subseteq{\mathbb{C}^d}$, one has $\mathcal{I}_U\supseteq \mathcal{I}_V$.
 \end{enumerate}
 
 \item The ideal $\mathcal{I}_W$ is larger than $\mathcal I$. Hilbert's Nullstellensatz theorem (e.g. see Theorem 1.6  of \cite{zbMATH00704831}) claims that $\mathcal{I}_W=\sqrt{\mathcal{I}}:=\{f\in \mathcal{O}\mid f^N\in \mathcal{I},\;\text{for some}\; N\in \mathbb N\}$.

 \item We have, $W=Z(\mathcal I_W)$
: if $x\in W$, then by definition of $\mathcal{I}_W$, one has  $f(x)=0$ for all $f\in\mathcal{I}_W$. Whence, $W\subseteq Z(\mathcal I_W)$. 
Conversely, it is clear that $\mathcal I\subseteq\mathcal{I}_W$. This fact proves the other inclusion.
 

 \item Also, by Noetheriality of the polynomial ring $\mathcal{O}$,  the ideal $\mathcal{I}_W\subset\mathcal{O}$ is generated by a finite number of generators. In the sequel, we shall define an affine variety $W$ as the zero locus of an ideal generated by a finite set of polynomials $\varphi_1,\ldots,\varphi_r\in \mathcal O$.
 
 \item The Zariski closure $\overline{V}$ of a subset $V\subseteq\mathbb{C}^d$ is equal to $Z(\mathcal{I}_V)$: by definition of $\mathcal{I}_V$, one has $V\subseteq Z(\mathcal{I}_V)$. Conversely, $V\subseteq \overline{V}=Z(\mathcal{I})$ for some ideal $\mathcal{I}\subseteq\mathcal{O}$. By item 1.(b), $\mathcal{I}_{\overline{V}}\subseteq \mathcal{I}_V$. In particular, $\mathcal{I}\subseteq \mathcal{I}_V$. By item 1.(a), this implies $Z(\mathcal{I}_V)\subseteq Z(\mathcal{I})=\overline{V}$.
\end{enumerate}

\begin{definition}Let $W\subseteq \mathbb{C}^d$ be an affine variety. 
   \begin{enumerate}
    \item A function $F\colon W \rightarrow \mathbb{C}$ is said to be a \emph{polynomial} if  $F (x) =f (x),\, \forall x \in W$, for some element $f\in  \mathcal O$. The set $\mathbb{C}[W]$ of polynomial functions on $W$ is, under the restriction map $f\in \mathcal O\mapsto f_{|_W}$,  isomorphic the quotient $ \mathcal O/\mathcal{ I}_W=:\mathcal{O}_W$, called the \emph{coordinate ring} of $W$.
    
    \item  Elements of the Lie algebra of $\mathbb{C}$-linear derivations, $\mathrm{Der}(\mathcal{O}_W)=:\mathfrak{X}(W)$, of  $\mathcal{O}_W$ are called \emph{vector fields on $W$}.
\end{enumerate}
\end{definition}

\begin{remark}Notice that

\begin{enumerate}
\item the coordinate ring $\mathcal{O}_W$ of an affine variety besides being a ring is also a vector space over $\mathbb C$, hence it is a $\mathbb{C}$-algebra. This algebra is generated by the images $\bar{x}_1,\ldots, \bar{x}_d$ in $\mathcal{O}_W$ through the projection map of the coordinate functions $x_1,\ldots, x_d \in\mathcal O$.

For each $a\in W$, the kernel of  the \emph{evaluation map} $\mathrm{ev}_a\colon \mathcal{O}_W \rightarrow \mathbb{C},\, F\mapsto F(a)$, is the maximal ideal $\mathfrak{m}_a:=\ker(\mathrm{ev}_a)$ made of all polynomial functions on $W$ that vanish at $a$.
   \item $W = \mathbb{C}^d$ is an affine variety with  $\mathcal{I}_W= \{0\}$, and $\mathcal{O}_W= \mathcal O$.
\item $W=\{a\}\subseteq{\mathbb{C}^d}$ is an affine variety with $\mathcal{I}_W= (x_1- x_1(a),\ldots,x_d- x_d(a))$ the maximal ideal of $a$, and $\mathcal{O}_W = \mathbb{C}$. 
\item For an affine variety $W\subset{\mathbb{C}^d}$ with corresponding ideal $\mathcal{I}_W$, we have $$\mathrm{Der}(\mathcal{O}_W)\simeq\frac{\{X\in \mathfrak{X}(\mathbb C^d)\mid  X[\mathcal{I}_W]\subset{\mathcal I_W}\}}{ \mathcal I _W \mathfrak{X}(\mathbb{C}^d)}.$$ 
\end{enumerate}
\end{remark}
\begin{lemma}
For every affine variety $W\subseteq{\mathbb{C}^d}$, the ring  $\mathcal{O}_W$ is Noetherian.
\end{lemma}

\begin{proof}
Let $\mathcal I$ be an ideal of $\mathcal O_W$. Denote by  $p\colon \mathcal{O} \rightarrow \mathcal{O}/\mathcal I_W$ be the quotient map. Then $p^{-1}(\mathcal I)$ is also an ideal of $\mathcal{O}$. By Noetheriality of $\mathcal O$,  $p^{-1}(\mathcal I)$ is finitely generated. In particular, $\mathcal{I}= p(p^{-1}(\mathcal I))$ is finitely generated. This shows that $\mathcal{O}_W$ is Noetherian. 
\end{proof}

\subsubsection[]{Germs} 	
Here we mention the notion of local rings. We refer the reader to \cite{Weibel,Matsumura,zbMATH00704831} for more details.\\

Let $W\subseteq\mathbb C^{N} $ be an affine variety and $\mathcal O_W$ its coordinates ring. We recall for $U\subseteq W$ an open subset, a function $f\colon U\longrightarrow \mathbb C$ is said to be \emph{regular} at $a\in U$ if there exists $g,h\in \mathcal O_W$ with $h(a)\neq 0$ such that $f=\frac{g}{h}$ in a neighborhood of $a$, namely there exists an open set $ V\subset U$ that contains $a$ such that $f|_V=\frac{g}{h}|_V$. A \emph{function germ} at a point $a\in W$ is an equivalence class $(f)_a$ of pairs $(U,f)$ with $a\in U\subset W$ an open subset containing $a$, and $f\colon U\longrightarrow \mathbb C$ is regular at $a$, under the relation equivalence: $(U,f)\sim(V,g)$ if $f|_\mathcal{W}=g|_\mathcal{W}$ on an open subset $\mathcal W\subseteq U\cap V$. The set of equivalence classes of the above equivalence relation inherits naturally an associative $\mathbb{C}$-algebra, that is called \emph{germs of regular functions at $a$} and is denoted by $\mathcal O_{W,a}$. Also, a function germ $(f)_a$ at $a\in W$ has a well-defined value at $a$, given by the image of any representative $(U,f)$  at $a$, namely  $(f)_a(a):=f(a)$. Since the map $$\mathcal O_{W,a}\longrightarrow (\mathcal O_W)_{\mathfrak m_{W,a}},\;(U,f)\mapsto f|_U=\frac{g}{h}$$ with $g,h\in \mathcal O_W$ and $h$ does not vanish on $U$, is a bijection. One has, $\mathcal O_{W,a}\simeq(\mathcal O_W)_{\mathfrak m_{W,a}}$ \cite{Hartshorne}.  Here $\mathfrak m_{W,a}=\{f\in \mathcal O_W\mid f(a)=0 \}$ and $(\mathcal O_W)_{\mathfrak m_{W,a}}$ is the localization w.r.t the complement of $\mathfrak m_{W,a}$. It is important to notice that $\mathcal O_{W,a}$ is a local ring. We denote the unique maximal ideal
of $\mathcal O_{W,a}$ again by $\mathfrak m_{W,a}$. Also, It is worth it to notice that $\mathcal O_{W,a}$ isomorphic to the quotient $\mathcal{O}_{a}/\mathcal{I}_a$ of the local ring $\mathcal{O}_{a}$ of $\mathbb{C}^d$ at $a$ by the
ideal $\mathcal I_a$ which is spanned by the ideal $\mathcal{I}_W$ in $\mathcal{O}_{a}$.


\subsubsection{The Zariski tangent space}
Let $a\in W\subseteq \mathbb{C}^d$ a point of an affine variety $W$. There are several equivalent descriptions of the tangent  space of $W$ variety at the point $a$. Here we define it as pointwise derivations of the local ring $\mathcal{O}_{W,a}$, see \cite{Hartshorne,Igor} or Appendix B.2 of \cite{CPA}, for more details on this topic.\\

A \emph{pointwise derivation} of $\mathcal{O}_{W}$ at $a$ is a $\mathbb{C}$-linear map
$$\delta_a\colon \mathcal{O}_{W,a}\rightarrow \mathbb{C}$$
satisfying the following Leibniz identity, $\delta_a(FG)=\delta_a(F)G(a) +F(a)\delta_a(G)$. In particular, if a regular function $f$ is constant in a neighborhood of $a$  then its germ $(f)_a$ at $a$ satisfies $\delta_a((f)_a)=0$, since \begin{align*}
    \delta_a(1\cdot 1)&=\delta_a(1)\cdot 1 +1\cdot \delta_a(1)=2\delta_a(1)\\\delta_a(1)&=0.
\end{align*}

It is not hard to check that the set of all pointwise derivations of $\mathcal{O}_W$ at $a$ is a $\mathbb{C}$-vector space.\\
\begin{remark}\label{rmk:pointwise-derivation}
Assume now that $W=\mathbb{C}^d$. Denote by  $(e_1, \ldots,e_d)$ the canonical basis of $\mathbb{C}^d$. For every $i\in\{1,\ldots,d\}$ $$(e_i)_a\colon \mathcal{O}_{W,a}\rightarrow \mathbb{C}, \,(f)_a\mapsto\lim_{t\to 0}\frac{f(a+te_i)-f(a)}{t}=:\frac{\partial f}{\partial x_i}(a),$$
is a well-defined pointwise derivation of $\mathcal{O}$ at $a$. One can show that (see e.g \cite{CPA}, Appendix B.2) any pointwise derivation $\delta_a$ of $\mathcal{O}$ at $a\in \mathbb{C}^d$ has the form \begin{equation}\label{eq:pointwise-derivation}\delta_a=\sum_{i=1}^d\delta_a((x_i)_a)\,(e_i)_a,\end{equation}i.e., \begin{equation}\delta_a(f)_a=\sum_{i=1}^d\frac{\partial f}{\partial x_i}(a)\,\delta_a(x_i)_a.\end{equation} Hence, pointwise derivations $(e_i)_a$, $i=1, \ldots,d$ form a basis for the vector space of pointwise derivation of $\mathcal O$ at $a$. 
\end{remark}

\begin{definition}
   The \emph{Zariski tangent space} $T_aW$ of $W\subseteq \mathbb{C}^d$ at $a\in W$ is the vector space of all pointwise derivations of $\mathcal{O}_W$ at $a$.
\end{definition}

\begin{proposition}\cite{Igor}
For $a\in W$, one has $T_aW\simeq\left(\mathfrak m_{W,a} /\mathfrak{m}^2_{W,a}\right)^* $. 
\end{proposition}
The vector space $\mathfrak m_{W,a} /\mathfrak{m}^2_{W,a}$ is called the \emph{cotangent space} to $W$ at $a$.

\begin{remark}Notice that,
\begin{enumerate}
    \item by Remark \ref{rmk:pointwise-derivation}, we have $T_a\mathbb{C}^d\simeq \mathbb{C}^d$.

\item the tangent space $T_aW$ of $W\subseteq \mathbb{C}^d$ at $a\in W$  can be seen as  pointwise derivations $\delta_a$ of $\mathcal{O}$ at $a$ of the form \eqref{eq:pointwise-derivation} such that $\delta_a((f)_a)=0$ for all $f\in \mathcal{I}_W$. From this point of view, one has

$$T_aW\simeq \left\lbrace(v_1,\ldots,v_d)\in \mathbb{C}^d\, \middle|\,   \displaystyle{\sum_{i=1}^dv_i\frac{\partial f}{\partial x_i}(a)=0},\, \forall f\in \mathcal{I}_W\right\rbrace.$$
\end{enumerate}
\end{remark}

\subsubsection{Singularities of an affine variety $W$}
 In this section we recall some definitions and some facts on singularities on affine varieties and fix some notations. We refer the reader to \cite{Hartshorne,Igor} for the full theory.\\
	
There are various equivalent ways to define the dimension of an affine variety $W$, we refer the reader to  Page 4 of \cite{Hartshorne} also to the Chapter 11 of \cite{Harris-Joe}  for more details. The \emph{dimension} $\dim W$ of $W$ is defined to be the  maximal length $d$ of the chains ${\displaystyle W_{0}\subset W_{1}\subset \cdots \subset W_{d}}$ of distinct nonempty irreducible sub-varieties of $W$.  Notice that a chain of  irreducible sub-varieties corresponds to a chain of prime ideals in $\mathcal{O}_W$, by Noetheriality it must be of  finite length.
\begin{definition}
A point $a\in W$ is said to be  \emph{regular} if $\dim T_aW=
\dim W.$  Otherwise,  we say that $a$ is \emph{singular}.  The set of regular  points of $W$ is denoted by $W_{reg}$, and the singular ones by $W_{sing}$.\\

\noindent
We say that $W$ is \emph{regular} if $W_{sing}=\emptyset$. 
\end{definition}

\begin{proposition}\cite{Hartshorne, Igor}We have the following

\begin{enumerate}
    \item For every $a\in W$, $\dim T_aW\geq \dim W$ 
    \item There is an open dense  open subset of $W$ such that the map $a\mapsto \dim T_aW$ is constant. In particular, 
 \begin{itemize}
     \item regular points of $W$ form an open dense subset of $W$,
     \item and singular points a (closed) proper sub-variety of $W$.
 \end{itemize}
\end{enumerate}
 
\end{proposition}
\begin{remark}
In particular,  $a\in W$ is a singular point of $W$ if only if $\dim T_aW>\dim W$. That is, $$W_{sing}=\{a \in W\,\mid \, \dim T_aW  > \dim W\}.$$
\end{remark}

\subsubsection{Local coordinates at a point}	

Let $a\in W\subseteq \mathbb{C}^d$. We recall that (see e.g \cite{Igor}) that a family of elements $t_1,\ldots, t_r\in \mathcal{O}_{W,a}$ are called \emph{local coordinates} of $W$ at $a$, if they vanish at $a$ (i.e.  $t_i\in \mathfrak{m}_{W,a}$\, for $i=1,\ldots,r$), and if the classes of $\overline{t}_1,\ldots,\overline{t}_r\in \mathfrak m_{W,a} /\mathfrak m^2_{W,a}$ form a basis.\\
\begin{example}
If $a = (a_1,\ldots, a_d) \in\mathbb{C}^d$, then $x_1-a_1,\ldots,x_d- a_d$ are local coordinates of $\mathbb{C}^d$ at $a$.
\end{example}
\begin{remark}Notice that:
\begin{enumerate}
    \item[] Local coordinates at $a\in \mathbb C^d$ generate the maximal ideal $\mathfrak{m}_a$ of $\mathcal{O}_{a}$. Indeed, let $t_1,\ldots, t_d\in \mathcal{O}_{a}$ be local coordinates at $a$. By applying Nakayama Lemma (\ref{Nakayama})  to $R=\mathcal{O}_{a}\supseteq \mathfrak{m}_{a}$ and $\mathcal{V}=\mathfrak{m}_{a}$: the  basis $\overline{t}_1, \ldots , \overline{t}_d\in{\mathfrak{m}_{a}}/{\mathfrak{m}^2_{a}}$ lifts
to a (minimal) generating set $t_1,\ldots, t_d$ for $\mathfrak{m}_{a}$.
    
\end{enumerate}
\end{remark}

The following proposition explains how an affine variety looks around a regular point (see \cite{HauserHerwig} Proposition 3.5,  also \cite{de-Jong-T-Pfister}).

\begin{proposition}\label{local-form-of-variety}
Let  $W\subseteq \mathbb{C}^d$ be affine variety of codimension $k$, (i.e., $k=d-\dim W$). Then, $W$ is regular at a point $a\in W$ if and
only if there exist local coordinates $y_1,\ldots, y_d$ of $\mathbb{C}^d$ at $a$ such that $W$ is locally of the form
 $$y_1 = \cdots = y_k = 0.$$
\end{proposition}

\subsection{Three main constructions}\label{three-main-constructions}

Consider an affine variety $W$ which is given by an ideal $\mathcal I_W \subset \mathcal O $,  with $ \mathcal O$ the algebra of polynomials in $d$-variables, and $\mathcal{O}_W = \mathcal O/\mathcal I_W$ the algebra of functions on $W$. 
There are three natural Lie-Rinehart algebras associated to $W$: 
\begin{enumerate}
    \item  The  $\mathcal O_W$-module $\mathfrak X(W) $ of vector fields on $W$ (i.e. derivations of $\mathcal O_W$) is  a Lie-Rinehart algebra over $\mathcal O_W$; its anchor map is the identity.
    \item The  $\mathcal O$-module $\mathfrak X_W (\mathbb C^d) $ of vector fields on $\mathbb C^d $ tangent to $W$ (i.e.  derivations of $\mathcal O$ preserving $ \mathcal I_W$) is  a Lie-Rinehart algebra with respect to the $\mathcal O$-module structure; its anchor map is the inclusion $\mathfrak X_W (\mathbb C^d) \hookrightarrow  \mathfrak X (\mathbb C^d) $.
    \item The  $\mathcal O$-module $\mathcal I_W \mathfrak X (\mathbb C^d) $ of vector fields on $\mathbb C^d $ vanishing at every point of $W$ (i.e. $ \mathcal I_W$-valued  derivations of $\mathcal O$) is  a Lie-Rinehart algebra with respect to the $\mathcal O$-module structure; its anchor map is again the inclusion $\mathcal I_W \mathfrak X(\mathbb C^d) \hookrightarrow  \mathfrak X (\mathbb C^d) $.
\end{enumerate}

These three Lie-Rinehart algebras are related:
\begin{enumerate}
    \item There is an inclusion $\mathcal I_W \mathfrak X (\mathbb C^d) \subset \mathfrak X_W (\mathbb C^d)$
    \item  the restriction of 
$\mathfrak X_W (\mathbb C^d)$ to $W$  coincides with  $\mathfrak X(W)$. Let us justifies this. Every vector field on $W$ extends to $\mathbb{C}^d$: to see that, let $\delta\in \mathfrak{X}(W)$, we have  $\delta(x_i + \mathcal{I}_W) = f_i + \mathcal{I}_W$ for some $f_i\in  \mathbb{C}[x_1, \ldots, x_d], i = 1, \ldots, d$. We define the vector field $$\widetilde \delta :=\sum_{i=1}^d f_i\frac{\partial}{\partial x_i}$$ 
on $\mathbb{C}^d$. The vector field $\widetilde \delta$ restricts to $\delta$ on $W$, since  for every  $f \in  \mathbb{C}[x_1,\ldots, x_n]$, $$\widetilde \delta(f)+ \mathcal{I}_W = \delta(f+\mathcal I_W).$$ In particular,  $\widetilde \delta(\mathcal{I}_W)\subset \mathcal{I}_W$.

\end{enumerate}
Note that the Lie-Rinehart algebras $\mathfrak{X}_W(\mathbb{C}^d), \mathcal{I}_W\mathfrak X(\mathbb C^d)\subset \mathfrak X(\mathbb C^d)$ are finitely generated as $\mathcal{O}$-modules, since $\mathcal{O}$ is Noetherian (Proposition  \ref{prop:noetherian}). Whence, these Lie-Rinehart algebras are singular foliations on the complex manifold $\mathbb C^d$ in the sense of Example \ref{ex:singfoliation}. 


\begin{remark}\label{rmk:tangent-distribution}
What happens if we take a look at the evaluation map at some point $a\in \mathbb{C}^d$? Any vector field $X=\displaystyle{\sum_{i=1}^dX[x_i]\frac{\partial}{\partial x_i}}\in\mathfrak X_W (\mathbb C^d)$ tangent to $W$ induces a pointwise derivation of $\mathcal{O}_{W}$ at $a\in W$ as follows: \begin{enumerate}
    \item We extend $X$ by localization at the maximal ideal $\mathfrak{m}_a$ to a derivation $(X)\in \mathrm{Der}(\mathcal{O}_{W,a})$.
    
    \item We define $X|_a((f)_a):=\displaystyle{\sum_{i=1}^dX[x_i](a)(e_i)_a((f)_a)}\in T_aW$.
\end{enumerate}
$X|_a$ is well-defined, since $X[\mathcal{I}_W]\subset \mathcal{I}_W$ for all $f\in\mathcal{I}_W$, in particular $X|_a((f)_a)=0$ for all $f\in\mathcal{I}_W$.\\

In fact, we do not need to localize $X$ to define $X|_a:=(X[x_1](a),\ldots,X[x_d](a))\in T_aW\hookrightarrow \mathbb C^d$.
\end{remark}

\begin{lemma}\label{lemma:Distribution-aff-v}
The image of the map
\begin{equation}
    \mathrm{Ev}_a\colon\mathfrak X_W(\mathbb{C}^d)\rightarrow \mathbb{C}^d,\;X\mapsto X|_a
\end{equation}
is denoted by  $T_a\mathfrak{X}_W(\mathbb{C}^d)$

\begin{enumerate}
\item if $a\in W$, then $T_a\mathfrak{X}_W(\mathbb{C}^d)\subseteq T_aW$.  \item If $a\not\in W$, then $T_a\mathfrak{X}_W(\mathbb{C}^d)=\mathbb{C}^d$. \item If $W$ is a complete intersection (i.e. $\mathcal{I}_W=(\varphi_1,\ldots,\varphi_r)$ and $\dim W= d-r$), then $T_a\mathfrak{X}_W(\mathbb{C}^d)=T_aW$ for all $a\in W_{reg}$.
\item More generally, for an arbitrary affine variety $W$, $T_a\mathfrak{X}_W(\mathbb{C}^d)=T_aW$ for all $a\in W_{reg}$.
\end{enumerate}
\end{lemma}
\begin{proof}
Item 1. is given by the construction in item 2. of Remark \ref{rmk:tangent-distribution}. Let us show item 2: let $(v_1,\ldots,v_d)\in \mathbb{C}^d$. For $a\not\in W$, there exists $\varphi\in\mathcal{I}_W$ such that $\varphi(a)\neq 0$. The vector field $$X=\frac{\varphi}{\varphi(a)}\sum_{i=1}^dv_i\frac{\partial}{\partial x_i}$$ belongs to $\mathcal{I}_W\mathfrak{X}(\mathbb C^d)\subset\mathfrak{X}_W(\mathbb{C}^d)=\mathbb{C}^d$ and $X(a)=(v_1,\ldots,v_d)$.\\

\noindent
Now we prove item 3. Let $(v_1,\ldots,v_d)\in T_aW=\ker(J(a))$, with $ J:=\left( \frac{\partial\varphi_i}{\partial x_j}\right)_{i,j}$. By assumption, we have $\mathrm{rk}(J(a))=r$. Thus, $J$ admits  $(r,r)$-minor $\mu$ such that $\mu(a)\neq 0$. We can assume that $\mu$ is the determinant of the first $r$-columns of $J$. Consider the vector fields

$$H_j:=\begin{vmatrix} \frac{\partial}{\partial x_1} & \cdots&\frac{\partial}{\partial x_r} &\frac{\partial}{\partial x_j}\\\frac{\partial \varphi_1}{\partial x_1} & \cdots&\frac{\partial \varphi_1}{\partial x_r} &\frac{\partial \varphi_1}{\partial x_j}\\ \vdots & &\vdots&\vdots\\ \frac{\partial \varphi_r}{\partial x_1} & \cdots&\frac{\partial \varphi_r}{\partial x_r} &\frac{\partial \varphi_r}{\partial x_j}\end{vmatrix},\;\text{ for $j\in\{r+1,\ldots,d\}$},$$
understood as the cofactor expansion along the first row. Since for each $i\in \{1,\ldots,r\}$, $H_j[\varphi_i]$ has two repetitive lines, therefore it vanishes. Therefore $H_j$'s are tangent to $W$. We claim that the following vector fields does the job, namely
\begin{equation}
    X= (-1)^r\sum_{j=r+1}^d\frac{v_j}{\mu(a)}H_j.
\end{equation}
Indeed, if we denote by $\mu_1, \ldots,\mu_r$ the minors associated to the partials $ \frac{\partial}{\partial x_1},\ldots,\frac{\partial}{\partial x_r}$, respectively, then for every $j$ the decomposition of $H_j$ reads  $$H_j=\displaystyle{\sum_{i=1}^r(-1)^{i+1}\mu_i\frac{\partial}{\partial x_i}+(-1)^r\mu\frac{\partial}{\partial x_j}}.$$ Hence,
$$
    X=\sum_{j=r+1}^d\frac{v_j}{\mu(a)}\mu\frac{\partial}{\partial x_j} + \frac{(-1)^{r}}{\mu(a)}\left(\sum_{i=1}^{r}\sum_{j=1}^d(-1)^{i+1}v_j\mu_i\frac{\partial}{\partial x_i}\right)$$
But, by developing the minors $\mu_i(a)$'s along their last columns, it takes the form, $$\mu_i(a)=\pm\frac{\partial\varphi_1}{\partial x_j}(a)C_1\pm \cdots \pm \frac{\partial  \varphi_{r-1}}{\partial x_j}(a)C_{r-1}.$$ Notice that  the determinants $C_1,\ldots,C_{r-1}$ are the same for each $H_j$. 
Thus, 
\begin{align*}
    \sum_{j=1}^dv_j\mu_i(a)&= \sum_{s=1}^{r-1}\pm\underbrace{\left(\sum_{j=1}^d v_j\frac{\partial\varphi_s}{\partial x_j}(a)\right)}_{=0}C_s\\&=0.
\end{align*}

Item 4. is obtained as follows: the local ring at $a$ is by definition the localization $\mathcal O_a$ of $\mathbb C[x_1 \dots, x_d] $ with respect to the multiplicative set of all polynomials that do not vanish at $a$. By Proposition \ref{local-form-of-variety} $a\in W$ is a regular point if and only if there exists  "local coordinates" $y_1,\ldots, y_d\in \mathcal O_a $ such that $W$ is locally of the form
 $$y_1 = \cdots = y_k = 0,$$
 i.e. the localization of $\mathcal I_W $ is generated by these variables. 
 Hence, the tangent space at $m$ is the vector space, $span\{{\frac{\partial}{\partial y_i}}_{|_m},\, i\geq k+1\}$. Therefore, for $v\in T_aW$ the local vector field
 $$X=\sum_{i=1}^{\dim W}v_i\frac{\partial }{\partial y_{k+i}}$$ maps $ \mathcal{O}_{a}$ to $\mathcal{O}_{a}$, in particular it maps $\mathcal{O}$ to $\mathcal{O}_{a}$ and we have ${X}[\mathcal I_W]\subset ({\mathcal{I}_W})_{\mathfrak m_a}$. Therefore, for every, $i\in \{1,\ldots,d\}$ there exists  a  polynomial function $g_i$ that does not vanish at $a$ such that $g_iY[x_i]\in\mathbb{C}[x_1,\ldots, x_d]$. Hence, the vector field $\hat{X}=\frac{g_1\cdots g_{r}}{g_1(a)\cdots g_{r}(a)}X$ is tangent to $W$ satisfies $\hat{X}(a)=v$ and $\hat{X}[\mathcal I_W]\subset \mathcal{I}_W$.\\
  
This concludes the proof.
\end{proof}

\begin{remark}
We may not have equality in item 1. in Lemma \ref{lemma:Distribution-aff-v}. To see this, consider the cups, $$W=\{(x,y)\mid \mathbb{C}^2\mid x^3-y^2=0\}.$$ It is clear that the tangent space $T_0W$ of $W$ at $0\in W$ is the whole space $\mathbb{C}^2$. But the vector fields in $\mathfrak X_W(\mathbb{C}^2)$ vanish at zero, since it is spanned as a $\mathbb C[x,y]$-module by the Hamiltonian $2y\frac{\partial}{\partial x}+3x^2\frac{\partial}{\partial y}$ and the weighted Euler vector field $2x\frac{\partial}{\partial x} +3y\frac{\partial}{\partial y}$ (see Proposition \ref{prop:vectorfieldsonW}).
\end{remark}
The following lemma shows that the vector fields that are tangent to $W$ are also tangent to every strata of  the stratification that consists of by taking the singular locus $W_{sing}$ of the singular locus  of $W$ then the singular locus $({W_{sing}})_{sing}$ of the singular locus $W_{sing}$ and so on.... We obtain a sequence of inclusions of the form

\begin{equation}\label{eq:stratification}
    W\supset \underbrace{({W_{sing}})_{sing}}_{=:W_1}\supset \underbrace{(({W_{sing}})_{sing})_{sing}}_{=:W_2}\supset\cdots.
\end{equation}

\begin{lemma}\label{lemma:stratification}
We have the following inclusions

\begin{equation}\label{stratas-inclusions}
    \mathfrak X_W(\mathbb C^d)\subseteq \mathfrak X_{W_1}(\mathbb C^d)\subseteq\cdots\subseteq\mathfrak X_{W_i}(\mathbb C^d)\subseteq\cdots
\end{equation}
\end{lemma}

\begin{proof}
Let us prove that if $X\in \mathfrak X(\mathbb C^d)$ is such that $X[\mathcal I_W]  \subset \mathcal I_W$
    then $X[\mathcal I_{W_{sing}}]  \subset \mathcal I_{W_{sing}}$, where $\mathcal I_{W_{sing}}$ is the ideal of functions on the singular part of $ W $. Since $\mathcal I_{W_{sing}}$ is obtained by considering the minors of order $ k = d-\dim W $ of $k$ elements chosen into the generators $\varphi_1, \dots, \varphi_r $. That is, $W_{sing}$ is given the ideal \begin{equation}\label{eq:sing-iadel}\left\langle\varphi_1,\cdots \varphi_r, P[\varphi_{i_1}, \cdots, \varphi_{i_k}],\; P\in\mathfrak X^k(\mathbb{C}^d),\;\text{for integers}\; \begin{array}{c}
         1\leq i_1<\cdots<i_k\leq r
    \end{array}
    \right\rangle\end{equation}
    
    Let us explain why the vector fields that tangent to $W$ are also tangent to its singular locus. 

For a vector field $X\in\mathfrak X_W(\mathbb C^d)$ one has by Formula \eqref{formula:Lie-der} that,
\begin{equation}\label{eq:Lie-derivative}X\left[P[\varphi_{i_1}, \cdots, \varphi_{i_k}]\right]=(\mathcal{L}_X P)[\varphi_{i_1}, \cdots , \varphi_{i_k}]+\sum_{j=1}^kP[\varphi_{i_1}, \ldots, X [\varphi_{i_j}],\ldots, \varphi_{i_k}].\end{equation}
Notice that  $(\mathcal{L}_X P)[\varphi_{i_1}, \cdots ,\varphi_{i_k}]\in \mathcal{I}_{sing}$ since $(\mathcal{L}_X P)\in \mathfrak X^k(\mathbb{C}^d)$. On the other hand, for every $j$ there exists polynomial functions $f_1,\ldots, f_r$  such that $X[\varphi_{i_j}]=\sum_{i=1}^rf_l\varphi_l$. Since $P$ is a multi-derivation, one has, \begin{align*}
    P[\varphi_{i_1}, \ldots, X [\varphi_{i_j}], \ldots, \varphi_{i_k}]&=\sum_{l=1}^r\varphi_lP[\varphi_{i_1}, \ldots, f_l, \ldots,\varphi_{i_k}]+\\&\qquad \sum_{i=1}^rf_lP[\varphi_{i_1}, \ldots, \varphi_l, \ldots,\varphi_{i_k}]
\end{align*}
It is now clear that the RHS of the equation \eqref{eq:Lie-derivative} is in the ideal $\mathcal I_{sing}$.  The proof goes by recursion.
\end{proof}

Here is a direct consequence of Lemma \ref{lemma:stratification}.
\begin{theorem}
Every vector field $X\in \mathfrak X(W)$ is tangent to the stratification \eqref{eq:stratification}, i.e. $X\in\mathfrak{X}(W_i)$ for each $i\geq 1$.
\end{theorem}

The coming example shows that the inclusions \eqref{stratas-inclusions} may be strict. This Example can also be found in the problem list \cite{CLRL-TA} of the lecture on singular foliations, Poisson 2022 \cite{LLL}.
\begin{example}
Let $ W = \{ (x,y,z) \in \mathbb C^3 \, | \, xy(x+y)(x+yz)=0\}\subset \mathbb{C}^3$. 

\begin{center}
        \includegraphics[height=7.5cm]{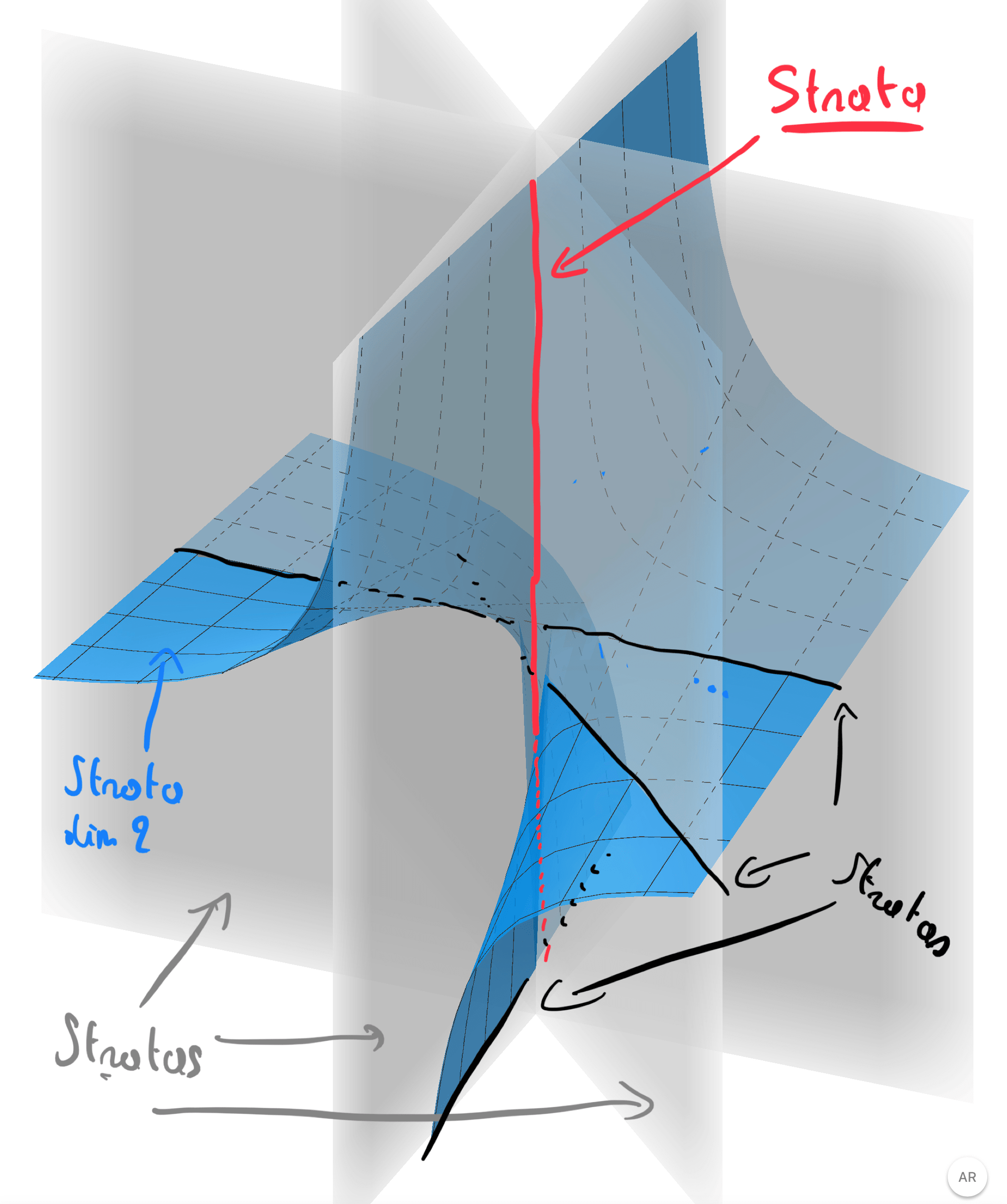}
\end{center}

The straight line $x=y=0$ is a strata of the previous affine variety $W$. Any vector field tangent to $W$ is tangent to this straight line. Let us show that it has to vanish at every point of this straight line. If not, its flow at time $t$ would map a point $(0,0,z_0)$ to a point $ (0,0,z_1)$ with $z_1 \neq z_0$.  Its differential then induce a linear automorphism of the normal bundle of that straight line that has to preserve the straight lines $x=0,y=0, x+y=0$. Since a linear endomorphism of $\mathbb C^2 $ preserving three straight lines has to be a multiple of the identity map, this differential cannot map the straight line $x+z_0y$ to the straight line $x+z_1y$.
\end{example}


\subsubsection{On their universal Lie $\infty$-algebroids}

Let $W$ be an affine variety. The following is  a direct consequence of the results of Chapter \ref{Chap:main}.
\begin{proposition}
Let $W$ be an affine variety as above.
\begin{enumerate}
    \item   $\mathfrak{X}(W)$ admits a universal Lie $ \infty$-algebroid 
    made of free $\mathcal O_W $-modules of finite rank.
    \item  $\mathcal I_W \mathfrak{X} (\mathbb C^d)$ and $\mathfrak{X}_W (\mathbb C^d)$ admit universal Lie $ \infty$-algebroids made of finitely many free $\mathcal O $-modules of finite rank.
\end{enumerate}
\end{proposition}
\begin{proof}
Classical theorems of commutative algebras allows equipping the three Lie-Rinehart algebras above with resolutions of a certain type:
\begin{enumerate}
  
    \item Since $\mathcal O $ is a Noetherian regular ring,  $\mathfrak{X}_W (\mathbb C^d)$ and $\mathcal I_W \mathfrak{X} (\mathbb C^d)$  admit free resolutions by finitely generated $\mathcal{O}$-modules. By Hilbert Syzygy Theorem \ref{app:Syzygies}, those can be chosen to be of finite length.
     \item Since $\mathcal O_W $ is a Noetherian ring,  $\mathfrak{X}(W)$ admits a free resolution by finitely generated $\mathcal{O}_W$-modules (by Proposition \ref{prop:noetherian}).
\end{enumerate}
In view of Theorem \ref{thm:existence}, these three resolutions do admit Lie $\infty $-algebroid structures over their respective algebras, and those are universal Lie $\infty $-algebroids.  We denote by $\mathbb U_{\mathcal O_W} (\mathfrak X(W))$, 
  $\mathbb U_\mathcal O (\mathfrak X_W(\mathbb C^d) ) $ and $\mathbb U_\mathcal O (\mathcal I_W \mathfrak X(\mathbb C^d) ) $ the universal Lie $ \infty$-algebroids associated to the three Lie-Rinehart algebras above.
\end{proof}
\begin{remark}
There exist natural Lie $\infty $-algebroid morphisms between these structures:
\begin{enumerate}
    \item Since $ \mathcal I_W \mathfrak X(\mathbb C^d) \subset \mathfrak X_W(\mathbb C^d)$,   Theorem \ref{th:universal} implies the existence of a unique up to homotopy Lie $\infty $-algebroid morphism  $ \Psi \colon \mathbb  U_\mathcal O (\mathcal I_W \mathfrak X(\mathbb C^d)) \longrightarrow \mathbb U_\mathcal O (\mathfrak X_W(\mathbb C^d))$. In view of Proposition \ref{univ:precise}, the morphism $\Psi$ may be represented by the inclusion map for well-chosen representation of $U_\mathcal O (\mathcal I_W \mathfrak X(\mathbb C^d))$ and $ \mathbb U_\mathcal O (\mathfrak X_W(\mathbb C^d))$.
    \item The anchor map of $\mathbb U_\mathcal O (\mathfrak X_W(\mathbb C^d) ) $ is tangent to $W$, hence the restriction $ \mathfrak i_W^* \mathbb U_\mathcal O (\mathfrak X_W(\mathbb C^d) ) $ to $W $ of $\mathbb U_\mathcal O (\mathfrak X_W(\mathbb C^d) )$ exists, and is a Lie $\infty $-algebroid over $\mathfrak X(W) $. In general, it does not need to be a universal one, but Theorem \ref{th:universal} implies the existence of a unique up to homotopy Lie $\infty $-algebroid morphism:
$$ \Phi \colon \mathfrak i_W^*  \mathbb U_\mathcal O (\mathfrak X_W(\mathbb C^d) )  \longrightarrow  \mathbb U_{\mathcal O_W} (\mathfrak X (W) ).$$
Notice that this morphism is in general not a Lie $\infty$-quasi-isomorphism.
\end{enumerate}
\end{remark}

\section{Universal Lie $\infty$-algebroid of an affine variety}
We give the following definition,
\begin{definition}
A \emph{Lie $\infty$-algebroid of an affine variety $W$} is the homotopy class of Lie-$\infty$-algebroid associated to the  Lie-Rinehart algebra $\mathfrak{X}(W)$ over $\mathcal{O}_W$.
\end{definition}

\begin{example}[Vector fields on hyperelliptic curves]
We follow the notations of \cite{gYulyLaov}.
 Let us consider on $\mathbb{C}^2$ the hyperelliptic curve $\mathcal{H}$ given by the equation $y^2= 2h(x)$, where $h$ is a monic polynomial of odd degree $2\nu+ 1\geq 3$. Let $\mathcal{O}_\mathcal{H}=\frac{\mathbb{C}\left[ x,y\right] }{ \langle y^2-2h(x) \rangle}$ be the coordinate ring of  $\mathcal{H}$.
Let $\X(\mathcal H)=\text{Der}(\mathcal{O}_\mathcal{H})$ be the Lie-Rinehart algebra of derivations  of $\mathcal{O}_\mathcal{H}$, i.e. of vector fields on $\mathcal H $. As an $\mathcal{O}_\mathcal{H}$-module, $\X(\mathcal H)$ is the submodule
$$\left\lbrace f\frac{\partial}{\partial x}+g\frac{\partial}{\partial y}\mid f,g\in \mathcal{O}_\mathcal{H}, yg-h'(x)f= 0\right\rbrace\subset \mathcal{O}_\mathcal{H}\frac{\partial}{\partial x}\oplus \mathcal{O}_\mathcal{H}\frac{\partial}{\partial y}.$$
\noindent
The curve $\mathcal H$ is non-singular if and only if $\text{gcd}(h(x),h'(x))$ = 1. In that case, the Lie algebra of vector fields $\X(\mathcal H)$ is a free $\mathcal{O}_\mathcal{H}$-module of rank $1$ generated by the vector field $X=y\frac{\partial}{\partial x}+h'(x)\frac{\partial}{\partial y}$. A universal Lie $\infty$-algebroid is given by the Lie-Rinehart algebra $\E_{-1}= \mathcal{O}_\mathcal{H}X\subset \mathcal{O}_\mathcal{H}\frac{\partial}{\partial x}\oplus \mathcal{O}_\mathcal{H}\frac{\partial}{\partial y}$.  
In the singular case, i.e., $\text{gcd}(h(x),h'(x))=d(x)\neq 1$, the $\mathcal{O}_\mathcal{H}$ -module $\X(\mathcal H)$ is not free and has two generators, $X=y\frac{\partial}{\partial x}+h'(x)\frac{\partial}{\partial y}$ and $Y= 2\frac{h(x)}{d(x)}\frac{\partial}{\partial x}+y\frac{h'(x)}{d(x)}\frac{\partial}{\partial y}$ with a relation $yX=d(x)Y$ (see \cite{gYulyLaov} for more details). A free resolution is described as follows: $\E_{-1}$ is the $\mathcal{O}_\mathcal{H}$-module, generated by two elements that we denote by $\tau,\mu$. The anchor is defined then by \begin{equation*}
\rho(\tau)=X,\quad\rho(\mu) =Y.
\end{equation*}
Then we choose $\E_{-2}$ to be the $\mathcal{O}_\mathcal{H}$-module given by the generator $\eta$, and we set $\E_{-i}=0$ for $i\geq 3$. The differential map $ \ell_1= \mathrm d$ is chosen to be zero, except on degree $-2$ where it is the $\mathcal{O}_\mathcal{H}$-linear map $\dd\colon\E_{-2}\longrightarrow\E_{-1}$  given by \begin{equation*}
\dd(\eta)=y\tau-d(x)\mu.
\end{equation*}
Let us now describe a universal Lie $\infty $-algebroid structure: The $2$-ary bracket is defined on generators of $\E_{-1}$ by $$\left\lbrace \tau,\mu \right\rbrace _2=\frac{h'(x)}{d(x)}\tau -y\frac{d'(x)}{d(x)}\mu.$$ Then we extend this bracket to the whole space $\E_{-1}$ by $\mathcal{O}_\mathcal{H}$-linearity, skew-symmetric  and Leibniz identity. Notice that $\left\lbrace \dd\eta,\tau \right\rbrace_2=\left\lbrace \dd\eta,\mu \right\rbrace_2=0$, thus, one can define the $2$-bracket by  $\left\lbrace \eta,\tau \right\rbrace_2=\left\lbrace \eta,\mu \right\rbrace_2=0$. We extend all brackets using Leibniz identity. All $k$-ary brackets are zero for $k\geq 3$.
%
\end{example}

\noindent
\textbf{Lie $\infty$-algebra of an affine variety at a point}:
Let $\mathbb U_{\mathcal O_W} (\mathfrak X (W) )= (\E, \ell_\bullet, \rho)$ be a universal Lie algebroid of $W$. Let us choose $a \in W$. As stated in Section \ref{ssec:res}, the Lie $\infty$-algebroid structure of $W$ restricts at $a$ (i.e. goes to the quotient with respect to the maximal ideal  $ \mathfrak m_a $) if and only if $\rho[\mathfrak m_{a}]\subseteq \mathfrak m_{a}$, (i.e. $\mathfrak m_{a}$ is a Lie-Rinehart ideal).\\

\noindent
Define the $\mathcal{O}_W$-submodule $$\mathrm{Sker}_a(\rho):=\{e\in \E_{-1} \mid  \rho(e)[\mathcal{O}]\subseteq \mathfrak m_a\}\subseteq \E_{-1}.$$ 

The $k$-ary bracket $\ell_k, k\geq  1$ and  $\rho$ 
restrict to the exact complex

\begin{equation}\label{eq:strong-kernel}
    \xymatrix{\cdots \ar[r]&\E_{-3}\ar[r]^{\ell_1}&\E_{-2}\ar[r]^{\ell_1} &\mathrm{Sker}_a(\rho)\ar[r]^{\rho}& \mathrm{Der}(\mathcal O_W)}.
\end{equation}

Let us check that $\ell_2$ is well-defined: for all $e, e'\in \mathrm{Sker}_a(\rho)$, \begin{align*}
    \rho(\ell_2(e,e'))[\mathcal{O}]&=[\rho(e), \rho(e')](\mathcal{O}),\qquad\text{(since $\rho$ is morphism of brackets)}\\&\subset \mathfrak{m}_a,\;\;\,\quad\qquad\qquad\qquad\text{(by definition of the commutator $\lb$).}
\end{align*}

The latter Lie $\infty $-algebroid goes to quotient to a Lie $\infty $-algebroid \begin{equation}\label{eq:oid-at-a-point}
 \left((\oplus_{i\geq 2}\frac{\E_{-i}}{\mathcal \mathfrak{m}_a\E_{-i}})\oplus  \mathrm{ker}_a(\rho) ,\Bar{\ell}_\bullet,\overline{\rho}\right)   
\end{equation} over $\mathcal O_W/\mathfrak m_{a} $, where ${\ker}_a(\rho):=\frac{ \mathrm{Sker}_a(\rho)}{\mathfrak m_a  \mathrm{Sker}_a(\rho)}$. Since this quotient is the base field $ \mathbb{C}$, we obtain in fact a Lie $\infty $-algebra. 
\begin{remark}
In particular, if $a\in W$ is an isolated singular point, then $\E_{-1}=\mathrm{Sker}_a(\rho)$. In that case, $\mathbb U_{\mathcal O_W} (\mathfrak X (W) )$ is a universal of the Lie-Rinehart algebra $\mathcal{A}_a=\{\delta\in \mathfrak{X}(W)\mid \delta[\mathfrak{m}_a]\subset \mathfrak m_a\}={\mathfrak{X}(W)}$. By Section \ref{ssec:res} again, \eqref{eq:oid-at-a-point} is a Lie $\infty $-algebra on ${\mathrm{Tor}}_{\mathcal O_W} (\mathfrak X(W), \mathbb C)$. 
\end{remark}


\subsubsection{Lie $\infty$-algebroids on minimal resolutions}

The germ at $a$ of the Lie-Rinehart algebra of vector fields on $W$ is easily checked to coincide with the Lie-Rinehart algebra of derivations of $\mathcal O_{W,a}$.\\




Here is an immediate consequence of Proposition \ref{prop:germ}.

\begin{prop}
Let $W$ be an affine variety.
For every $a \in W$, the germ at $a$ of the universal Lie $\infty$-algebroid of $W$ is the universal Lie $\infty$-algebroid of
${\mathrm{Der}}(\mathcal O_{W,a})$.
\end{prop}

 To describe this structure, let us start with the following Lemma.

	\begin{lemma}\label{lem:finite}
		Let $a$ be a point of  an affine variety $W$.  
		The universal Lie $\infty$-algebroid of
${\mathrm{Der}}(\mathcal O_{W,a})$
		can be constructed on a resolution $((\E_{-i}^{a})_{i\geq 1}, \ell_1, \pi) $, with $\mathcal E_{-i}^{a} $  free $\mathcal O_{W,a}$-modules of finite rank for all $i \geq 1$, which is minimal in the sense that $\ell_1 (\E_{-i-1}) \subset \mathfrak m_{a}\E_{-i} $ for all $i \geq 1$. 
	\end{lemma}
	\begin{proof}
		Since Noetherian property is stable by localization, the ring $\mathcal O_{W,a}$ is a Noetherian local ring. Proposition 8.2 in \cite{CohenMacauley}  assures that $\mathcal O_{W,a}\otimes_{\mathcal O_W}\text{Der}(\mathcal O_{W})$ admits a free minimal resolution  by free finitely generated $\mathcal O_{W,a}$-modules. 
		Since $\mathcal O_{W,a}$ is a local ring with maximal idea  $ \mathfrak m_{a}$, we can assume that this resolution is minimal.
		In view of Theorem \ref{thm:existence}, there exists a Lie $\infty$-algebroid structure over this resolution, and the latter is an universal of ${\mathrm{Der}}(\mathcal O_{W,a})$.
	\end{proof}
	
By Theorem \ref{thm:existence}, a resolution of  ${\mathrm{Der}}(\mathcal O_{W,a})$ as in Lemma \ref{lem:finite} comes equipped with a universal Lie $\infty $-algebroid structure for  ${\mathrm{Der}}(\mathcal O_{W,a})$. The quotient with respect to $\mathfrak m_{a} $ is a
	Lie $\infty $-algebra of the isolated singular point $a$
	with trivial $1$-ary bracket.
Using Corollary \ref{cor:LRideal2} and its subsequent discussion, we can prove the next statement.
\begin{prop}
For any universal Lie $\infty$-algebroid structure on a resolution of
${\mathrm{Der}}(\mathcal O_{W,a})$ as in Lemma \ref{lem:finite},
the quotient with respect to the ideal $ \mathfrak m_{a}$ is a representative of the Lie $\infty $-algebra of the isolated singular point $a$, with trivial $1$-ary bracket, on a graded vector space canonically isomorphic to ${\mathrm{Tor}}_{\mathcal O_W} (\mathfrak X_W, \mathbb C) $
($\mathbb C$ being a $\mathcal O_{W}$-module through evaluation at $a$).

In particular, its $2$-ary bracket is a graded Lie bracket on ${\mathrm{Tor}}_{\mathcal O_W} (\mathfrak X_W, \mathbb C) $ which does not depend on any choice made in the construction, and its $3$-ary bracket is a Chevalley-Eilenberg cocycle whose class is also canonical. 
\end{prop}

\section{Some examples of universal Lie-algebroids over an affine variety}\label{example-oid}

\subsection{Vector fields tangent to $W$: a codimension one example}

The zero-set in $W\subset V = \mathbb C^d $ of a weight homogeneous polynomial function $\varphi \in \mathbb{C}[X_1,\cdots,X_d]$ admitting only an isolated singularity at the origin is one of the simplest possible example of a non-smooth affine variety $W\subset V =\mathbb{C}^d$. 
We give in this section a description of $\mathbb U_\mathcal O (\X (W)) $. Although our description is not complete, it will show how complex the universal Lie $\infty$-algebroids may be, even for simple objects.

First, let us describe $\X (W) = {\mathrm{Der}}(\mathcal O_W)$.
We denote by $\omega_1, \dots, \omega_d $ the weights of the variables $x_1, \dots, x_d $ and by  $\lvert\varphi\rvert$ the weighted degree of $\varphi $. 
Recall  that, by definition, $\mathcal{O}_W:=\frac{\mathbb{C}[X_1,\cdots,X_d]}{\langle\varphi\rangle}$.

\begin{proposition}
\label{prop:vectorfieldsonW}
As a $\mathcal O_W $-module, $ \mathfrak X(W):= {\mathrm{Der}}(\mathcal O_W)$ is generated by the restrictions to $W$ of the following vector fields in $\mathbb C^d$:
\begin{enumerate}
    \item the weighted Euler vector field $\overrightarrow{E} := \sum_{i=1}^d \omega_i x_i \partial_{x_i}  $
    \item the $d(d-1)/2$ vector fields given by:
     $$ X_{ij} :=  \frac{\partial \varphi}{\partial x_i}\partial_{x_j} -\frac{\partial \varphi}{\partial x_j} \partial_{x_i}    \hbox{ with $ 1 \leq i < j  \leq d.$}   $$
\end{enumerate}
\end{proposition}


We start with a lemma. 
Recall that a homogenous function $\varphi $ with an isolated singularity at $0$ is a Koszul function (see Example \ref{koszul}), so that the Koszul complex, i.e. the complex of poly-vector fields on $ V = \mathbb C^d$, equipped with  $\iota_{\varphi} $: 
  $$   \cdots \stackrel{ \iota_{\varphi}}{\longrightarrow} \mathfrak X^2(\mathbb C^d)  \stackrel{ \iota_{\varphi}}{\longrightarrow} \mathfrak X^1(\mathbb C^d) \stackrel{\iota_{\varphi}}{\longrightarrow} \mathcal O $$ 
 has no cohomology except in degree $0$. We denote it by $ (\X^\bullet, \iota_{\varphi})$.

\begin{lemma}\label{Euler-lemma}
 If $P \in \mathfrak X^{i+1}(\mathbb C^d)  , Q \in \mathfrak X^i(\mathbb C^d)$ satisfy $\iota_{\varphi} (P) = \varphi Q $, then there exists $R \in  \X^{i+2}(\mathbb C^d) $ such that
  $$ P = \frac{1}{|\varphi|} \overrightarrow{E} \wedge Q  +\iota_{\varphi} (R) $$
\end{lemma}
\begin{proof}
This follows from the easily checked fact that $\iota_{\varphi} (Q)=0 $, so that
 $$\iota_{\varphi}\left( \frac{1}{|\varphi|} \overrightarrow{E} \wedge Q \right) = \varphi Q .$$
This implies that
  $$\iota_{\varphi}\left( P -  \frac{1}{|\varphi|} \overrightarrow{E} \wedge Q \right) = 0 $$
  and the existence of $R$ now follows from the exactness of the Koszul complex.
\end{proof}

\begin{proof}[Proof (of Proposition \ref{prop:vectorfieldsonW})] 
Any $X \in \X (W)$ is the restriction to $W$ of a vector field $\tilde{X} $ in $V$ tangent to $W$, i.e. that satisfies $\tilde X [\varphi] = f \varphi $ 
for some $f \in \mathcal O $. 
By Lemma \ref{Euler-lemma} the vector field $\tilde X$ can be written as $\tilde X =\frac{f}{|\varphi|} \overrightarrow{E} +\iota_{\varphi} (P)$,
for some bivector field $P \in \mathfrak X^2 (\mathbb {C}^d) $. The restriction  to $W$ of the first (resp the second) term is tangent to $W$ and is of the first (resp. second) type described in Proposition \ref{prop:vectorfieldsonW}. This completes the proof.
\end{proof}

We now intend to construct a free resolution of $\X (W) $ in the category of $\mathcal O_W$.
\noindent
Let us denote by  $\wedge\overrightarrow{E} \colon \X^i(V)\rightarrow \X^{i+1}(V) $ the map $\omega\mapsto\omega\wedge\overrightarrow{E}$.
The map $\wedge\overrightarrow{E}$ is a chain map from the restriction $\mathfrak i_W^*\X^\bullet$ to $W$ of the Koszul complex, shifted by one to itself. More precisely,   $\wedge\overrightarrow{E}$ in the diagram below is a chain map:
 $$
\xymatrix{ \cdots\ar[r]^{\iota_{\varphi}}&  \ar[r]^{\iota_{\varphi}}\mathcal \mathfrak i_W^*\X^{2}(V)\ar@<3pt>[d]^{\wedge\overrightarrow{E}}&  \ar[r]^{i_{\varphi}} \mathfrak i_W^*\X(V)\ar@<3pt>[d]^{\wedge\overrightarrow{E}}&\mathcal O_W \ar@<3pt>[d]^{\wedge\overrightarrow{E}}&\\ \cdots\ar[r]^{\iota_{\varphi}}& \ar[r]^{\iota_{\varphi}} \mathfrak i_W^* \mathcal \X^{3}(V)& \mathfrak i_W^*\X^{2}(V)  \ar[r]^{\iota_{\varphi}}&\X(W)  & } 
$$ 
\begin{lemma}\label{cone}
The chain map $\cdot\,\wedge\overrightarrow{E}$ is a quasi-isomorphism.
\end{lemma}
\begin{proof}
It is quite clear for $i\geq 1$ since the complex $(\E'[1],\iota_{\varphi})$ has no cohomology. Let us check bijectivity in the last column. Surjectivity follows from Proposition \ref{prop:vectorfieldsonW}. To  check injectivity consider a function $f\in\mathcal O$ such that \begin{align}\label{eq:inj} f\overrightarrow{E}=\iota_{\varphi}(\pi)+\varphi Q,\end{align}
for some bivector field $\pi\in\X^2(V)$ and vector field $Q\in\X(V)$. Upon taking $\iota_{\varphi}$ to both sides, Equation \eqref{eq:inj} implies, $$ f \lvert\varphi\rvert\varphi=\varphi \iota_{\varphi}(Q).$$ Hence, we obtain $f=\frac{1}{|\varphi|}\iota_{\varphi}(Q)$. 
\end{proof}

The lemma \ref{cone} implies that the mapping cone construction provides a free resolution  of the $\mathcal O_W$-module of vector fields of $\X(W)$ on $W$ namely: \begin{align}\label{resol-W}
\left( \E_{-i}=\mathfrak i_W^* \mathcal \X^{i-1}(V)\oplus\mathfrak i_W^* \mathcal \X^{i+1}(V),\,i\geq 1;D=\begin{pmatrix}
-\iota_{\varphi}&0\\-\wedge\overrightarrow{E}&\iota_{\varphi}
	\end{pmatrix},\pi\right).
	\end{align}
In degree $-1$ we read $\E_{-1}=\mathcal O_W \oplus \mathfrak i_W^* \mathcal \X^{2}(V)$ and $\pi$ is defined on the generators of $\E_{-1}$ as $\pi(1\oplus 0):=\overrightarrow{E}$ and $$\pi(0\oplus \partial_{x_i}\wedge\partial_{x_j} ):=\frac{\partial \varphi}{\partial x_i}\partial_{x_j} -\frac{\partial \varphi}{\partial x_j} \partial_{x_i} ,\;\text{for all},\;1\leq i<j\leq d.$$ Let us now describe some of the $k$-ary brackets:

\begin{proposition}
The Lie-Rinehart algebra $\mathfrak X (W) $ of vector fields on $W$ admits a universal Lie $\infty $-algebroid  whose
 \begin{enumerate}
   \item underlying  complex is \eqref{resol-W} (which is a free resolution of $\mathfrak{X}(W)$),
    \item  $1$-ary bracket is given by the resolution \eqref{resol-W}.
    \item  anchor map $\rho:=\pi$. 
      \item In the case where the generator $1\oplus 0$ appears together with a bivector field on $W$, one can define  the $2$-ary bracket on elements of degree $-1$ which makes the anchor map a morphism as follows,$$\{1\oplus 0,0\oplus\partial_{x_i}\wedge\partial_{x_j}\}_2:=(\lvert\varphi\rvert -2)0\oplus\partial_{x_i}\wedge\partial_{x_j}.$$ 
    \item The $k$-ary brackets can be chosen to be the same as in the structure given by Proposition \ref{prop-koszul} on $\mathfrak i_W^* \mathcal \X^{i+1}(V)$ for $i\geq 1$.
     \item the $3$-ary bracket can be chosen to be zero when evaluated at $1\oplus 0$ and with two other  elements of $\mathfrak i_W^* \mathcal \X^{2}(V)$.
\end{enumerate} 
\end{proposition}

\begin{corollary}
$\mathfrak{X}(W)$ does not come from a Lie algebroid of rank $1+ \frac{d(d+1)}{2}$.
\end{corollary}

\begin{remark}
In general, it is hard to compute the generators of $\mathfrak{X}(W)$. But, there is a case where we know all generators: it is when $W$ is a complete intersection with isolated singularity at zero such that $\mathcal{I}_W$ is generated by weight homogeneous polynomials $\varphi_1, \ldots,\varphi_r$. In that case, $\mathfrak{X}_W(\mathbb{C}^d)$ is generated by $\mathcal{I}_W\mathfrak{X}(\mathbb{C}^d)$, the Euler vector field, and the Hamiltonian vector fields (see \cite{HauserHerwig1993Aval,SiebertThomas1996Laod}).\\

\noindent
The computation of the Lie $\infty$-algebroid remains complicated.
\end{remark}

\vspace{2cm}

\begin{tcolorbox}[colback=gray!5!white,colframe=gray!80!black,title=Conclusion:]
As a particular case of Section \ref{Chap:main}, we notice that a Lie $\infty$-algebroid can be associated to any affine variety. Explicit computations are difficult. Some applications to Mohsen's
resolution of singularities will be given in Section \ref{blow-up-procedure}

\phantom{cc} There are still many open questions on the geometric meaning of this Lie $\infty$-algebroid structure. Although, we have not answered many questions , we state concepts, lemmas, and counter-examples that we hope to be able to use in the future.

\end{tcolorbox}


\chapter{Universal $Q$-manifolds of a singular foliation}\label{chap:oid1}

The aim of this chapter is to lay the ground for the subsequent results. More precisely, to explain how Theorem \ref{thm:existence} extends the results on singular foliations of  Sylvain Lavau's PhD \cite{LavauSylvain} followed by a referred version by C. Laurent-Gengoux, S. Lavau and T. Strobl in \cite{LLS}. The results of Section Chapter \ref{Chap:main} extend the latter for arbitrary Lie-Rinehart algebras and also to the infinite case, and it still holds even when we do not have a geometric resolution. We introduce the notion of longitudinal vector fields on a $NQ$-manifold and prove a new result on their cohomology.\\


This chapter is taken from the textbook \cite{LLL} in which I am co-author. We refer the reader to \cite{Voronov2,Poncin,LeanMadeleineJotz2020Maru,LLS} for more details on this topic.\\

Throughout of this chapter  $M$ is a smooth, real analytic or complex manifold  and $\mathbb{K}\in \{\mathbb{R},\mathbb{C}\}$. We denote by $\mathcal{O}$ the sheaf of functions on $M$. 


\section{$Q$-manifolds}\label{Q-manifold}
Let us first define $\mathbb N$-graded manifolds.
\subsection{Graded manifolds}
In words, as the name suggests graded manifolds are for graded vector spaces what manifolds are for vector spaces, in the sense that roughly speaking manifolds look locally like $\mathbb{R}^n$ and graded manifolds are locally like $\mathbb{R}^n\times V$ for some graded vector space $V$. In this chapter we introduce the notion of graded manifolds, their vector fields, their morphisms, etc.

\vspace{.5cm}

\begin{definition}
 A \emph{(positively) graded manifold} over  the \emph{base} manifold $M$ is a sheaf 
$$ \mathcal E \colon \mathcal U \mapsto \mathcal E({\mathcal U}) $$
of graded commutative algebras over $\mathbb K $
such that every $m\in M$ admits an open neighborhood $\mathcal U\subset M$ on which the sheaf structure takes the form
$$\E(\mathcal U)=\mathcal O_{\mathcal U }\otimes_\mathbb{K} S(E^*_{-\bullet})$$ for some graded vector space $E=\oplus_{i=1}^\infty E_{-i}$. Sections of the sheaf $\E$ are called \emph{functions on $E $}. It is convenient to denote a graded manifold as a pair $(M,\E)$ where $E$ is implicit. 
\end{definition}


\begin{remark}
A function $\xi\in\E_j$ is a formal sum \begin{equation}
    \xi=\sum_{i=0}^{+\infty}\xi^{(i)}
\end{equation}
with $\xi^{(i)}\in \E$ an element of polynomial-degree $i$ and degree $j$. For degree reasons, the sum must be finite.
\end{remark}

\subsubsection{Local coordinates of a graded manifold}

Recall that for $\mathcal{U}\subset M$ an open set, one has $({E^*_i})_\mathcal{U}\overset{\sim}{\longrightarrow}\mathcal{U}\times \mathbb{K}^{\mathrm{rk}(E^*_i)}$ for every $i\geq 1$. Hence, the \emph{graded coordinates} on the graded manifold $(M,\E)$ is the data made of:\\

\noindent
\textbf{In degree $0$}:  a system of coordinates $(x_1,\ldots,x_n)$ of $M$ on $\mathcal U$\\

\noindent
\textbf{In degree $i\geq 1$}: a local trivialization $(\xi_i^{1},\ldots, \xi_i^{\mathrm{rk}(E^*_i)})$ of $E^*_i$ on $\mathcal U$.\\

\noindent
That is, a system of graded coordinates of $(M, \E)$ on $\mathcal{U}$ is $$(x_1,\ldots,x_n,\xi_1^{1},\ldots, \xi_1^{\mathrm{rk}(E^*_1)},\ldots, \xi_i^{1},\ldots, \xi_i^{\mathrm{rk}(E^*_i)},\ldots\; ).$$
Elements of $\E(\mathcal U)$ are "polynomials" in $\{(\xi_i^{j})_{j=1,\ldots,\mathrm{rk}(E^*_{-i})},\;i\geq 1\}$ with coefficients in $\mathcal{O}(\mathcal{U})$.

\begin{example}
The sheaf of differential forms $(M, \E=\Omega(M))$ on a manifold $M$ is a graded manifold since for every point $m\in M$, it takes the form $\mathcal{O}_\mathcal{U}\otimes_\mathbb{K}\wedge^\bullet T_m^*M$ where $\mathcal U$ is an open neighborhood of $m$. 
Exterior forms can be seen as sections on the graded vector bundle $ E_{-1}=TM$.
\end{example}
\begin{example}
Let $k$ be a positive integer.
A finite dimensional vector space $E$ and its dual $E^*$ can be seen as graded vector bundles of respective degree $-k$ and $k$ over a point.
$E$ is a graded manifold  over $M=\{pt\}$, with functions isomorphic to $ \wedge E^*$ for $k$ odd and $S(E^*)$ for $k$ even.
\end{example}



\begin{definition}
A \emph{morphism of graded manifolds} between the two graded manifolds $(M,\mathcal \E)$ and $(M', \mathcal{E}')$ with respective base manifolds
$M$ and $M'$ is a pair made of a smooth or real analytic or holomorphic map $\phi\colon M\longrightarrow M'$ called the \emph{base
map} and a sheaf morphism over it, i.e. a family of graded algebra morphisms:
 $$ \mathcal E'(\mathcal U') \to \mathcal E(\phi^{-1}(\mathcal U')) ,$$
compatible with the restriction maps, such that\begin{equation}
    \Phi(f\alpha)=\phi^*(f )\Phi(\alpha).
\end{equation}
for all $f\in\mathcal O_{\mathcal U'}'$ and $\alpha\in\E'(\mathcal U')$.\\

A \emph{homotopy} between two morphisms of graded manifolds
$ \Phi, \Psi \colon  (M,\mathcal \E) \longrightarrow (N, \mathcal{E}')$ is a morphism
of graded manifold 
 $$ (M,\mathcal E) \times ([0,1], \Omega([0,1])) \longrightarrow (M',\mathcal E')$$
 whose restrictions to $t=0$ resp. $t=1$  coincide with $\Phi$ and $\Psi $ respectively.
\end{definition}

\subsubsection{Vector fields on graded manifolds}

Vector fields on manifolds are derivations of its algebra of functions. For a graded manifold, the analog of functions are  the sheaf of sections  $\Gamma(S^\bullet(\mathcal E^*))$. Since it is not commutative but graded commutative, one has to consider graded derivations. 
A \emph{graded derivation of degree $k$} of $\mathcal E $ is the data, for every $\mathcal U \subset M $ of a linear map
 $$ Q \colon \mathcal E_\bullet(\mathcal U) \longrightarrow   \mathcal E_{\bullet +k}(\mathcal U) ,$$
 compatible with all restriction maps, that increases the degree by $+k$ and satisfies: 
  $$  Q[FG] = Q[F] G +(-1)^{ki}FQ[G] $$
  for every $F \in \mathcal E_i(\mathcal U), G \in \mathcal E(\mathcal U)$.
   Since we think geometrically, we  say "\emph{vector fields of degree $k$}" instead of graded derivation.

\begin{definition}
Let $(M,\E)$ be a graded manifold. For $\mathcal U\subset M$ and $k\in\mathbb{Z}$ let $$\mathfrak{X}_k(E)(\mathcal U):=\text{Der}_k(\E(\mathcal U))$$be the $\E(\mathcal U)$-module of derivation of degree $k$ on $\E(\mathcal U)$. The correspondence $\mathcal U\longmapsto \mathfrak{X}_\bullet(E)(\mathcal U)$ is a sheaf of $\E$-module. Its sections are called \emph{vector fields on} $E$.
\end{definition}

Let us list some important facts on vector fields on $E$:
\begin{enumerate}
     \item the $\E$-module $\mathfrak X_\bullet(E):=\oplus_{k\in\mathbb{Z}}\mathfrak X_k(E)$ of vector fields on $E$ is naturally graded. The $\E$-module $\mathfrak X_\bullet(E)$ of vector fields on $E$ is a graded Lie subalgebra of the graded Lie algebra $\mathrm{Hom}_\mathbb{K}(\E,\E)$ whose graded Lie bracket is the graded commutator. Precisely, the graded Lie bracket \begin{equation}\label{graded:commutator}
         [P,Q]=P\circ Q-(-1)^{kl}Q\circ P\end{equation} of two vector fields $P,Q$ of degree $k,l$ respectively is a vector field of degree $k+l$. It is easily checked that the bracket \eqref{graded:commutator} fulfills
         
         \begin{enumerate}
             \item $[P,Q]=-(-1)^{jk}[Q,P]$,\quad(graded skew-symmetry)
             \item $(-1)^{jl}[P, [Q, R]] + (-1)^{jk}[Q, [R, P]]+(-1)^{kl}[R, [P, Q]]=0$,\quad(graded Jacobi identity)
             \item $[P, fQ]=P[f]Q+f[P,Q]$, \quad(Leibniz identity)
         \end{enumerate}
         for $f\in\mathcal{O}$ and $P,Q, R$ are vector fields  of degree $j,k$ and $l$ respectively.
    \item Their description in local coordinates: notice that any homogeneous element $e\in \Gamma(E_{-k})$ corresponds to a vertical\footnote{A vector field $P\in \mathfrak{X}(E)$ is said to be \emph{vertical} if it is linear with respect to functions on $M$, in other words if $P[f]=0$ for all $f\in \mathcal{O}$.}  vector field $\iota_e\in\mathfrak X_{-k}(E)$ (i.e. it is $\mathcal{O}$-linear) of degree $-k$ defined by contraction with $e$
\begin{equation}\label{eq:contraction}
    \iota_e (\xi):=\langle \xi, e\rangle, \quad \xi\in \Gamma(E^*)
\end{equation}
and we extend by $\mathcal{O}$-linear  derivation, where $\langle\cdot\,,\cdot\rangle$ is the dual pairing between $E^*$ and $E$. Notice that $\iota_e$ is by construction of polynomial-degree $-1$.

Let $(\mathcal U, x_1,\ldots,x_n)$ be a coordinate chart of $M$ and $(\xi_i^{j})_{j=1,\ldots, \mathrm{rk}(E^*_i)}$ with $i\geq 1$ be a homogeneous local trivialization of $E^*_{-i}$, it should be understood that $\xi_i^{j}$ is the $j$-th elements of the local frame in $\Gamma( E_{-i}^*)$. Let $(e_i^{j})_{j=1,\ldots, \mathrm{rk}(E_{-i})},i\geq 1$ be the dual basis of $(\xi_i^{j})_{j=1,\ldots, \mathrm{rk}(E^*_{-i})},i\geq 1$, then for every pair $i,j$, $\iota_{e_i^j}=\frac{\partial}{\partial \xi_i^j}$ is the partial derivative with respect to $\xi_i^j\in \Gamma( E^*_{-i})$. By choosing a $TM$-connection on $E$, it is easy to check that for any $k\in \mathbb Z$ the family $$\displaystyle{\left(\xi^{j_1}_{i_1}\odot\cdots\odot\xi^{j_l}_{i_l}\frac{\partial}{\partial x_j}\right)_{\tiny{\begin{array}{c}l\geq 0\\i_1\cdots i_l=k\\j_1,\ldots, j_l\\j=1,\ldots,n
\end{array}}}\cup\left(\xi^{j_1}_{i_1}\odot\cdots\odot\xi^{j_l}_{i_l}\frac{\partial}{\partial \xi_i^j}\right)_{\tiny{\begin{array}{c}l\geq 0\\i_1\cdots i_l-i=k\\j_1,\ldots, j_l\\i\geq 1,\,j=1,\ldots, \mathrm{rk}(E_{-i})
\end{array}}}}$$ form a basis for $\mathfrak X_k(E)(\mathcal{U})$ up to permutations of the $\xi^{j_1}_{i_1}\odot\cdots\odot\xi^{j_l}_{i_l}$'s. Here we adopt the convention $i_0=j_0=0$ and $\xi^0=1\in\Gamma(S^0(E^*))\simeq \mathcal O$. Whence, any  vector field $Q\in\mathfrak X_k(E)(U)$ admits coordinates decomposition as follows 
$$Q=\sum_{{\tiny{\begin{array}{c}l\geq 0\\i_1\cdots i_l=k\\ j_1,\ldots, j_l\\j=1,\ldots,n
\end{array}}}}\frac{1}{l!}\,{}^{j}Q^{j_1\cdots j_l}_{i_1\cdots i_l}\,\xi^{j_1}_{i_1}\odot\cdots\odot\xi^{j_l}_{i_l}\,\frac{\partial}{\partial x_j}+\sum_{{\tiny{\begin{array}{c}i\geq 1,\,l\geq 0\\i_1\cdots i_l-i=k\\j_1,\ldots, j_l\\j=1,\ldots,\mathrm{rk}(E_{-i})
\end{array}}}}\frac{1}{l!}\,{}^{ij}Q^{j_1\cdots j_l}_{i_1\cdots i_l}\, \xi^{j_1}_{i_1}\odot\cdots\odot\xi^{j_l}_{i_l}\frac{\partial}{\partial \xi_i^j}.$$ 
for some functions $Q^{j_1\cdots j_l}_{i_1\cdots i_l}\in \mathcal{O}$. These functions can be chosen in a unique manner to satisfy, e.g. ${}^{ij}Q^{j_{\sigma(1)}\cdots j_{\sigma(l)}}_{i_{\sigma(1)}\cdots i_{\sigma(l)}}=\epsilon(\sigma)Q^{j_1\cdots j_l}_{i_1\cdots i_l}$ for any permutation $\sigma$ of $\{1,\ldots,l\}$.\\

For example, if $Q$ is of degree $+1$, then it can be written in these notations as
$$Q=\sum_{_{\tiny{\begin{array}{c}1\leq u\leq\mathrm{rk}(E_{-1})\\j=1,\ldots,n
\end{array}}}}{}^{j}Q^{u}_{1}\,\xi^{u}_{1}\,\frac{\partial}{\partial x_j}+\sum_{_{\tiny{\begin{array}{c}i\geq 1,\,=l\geq 0\\i_1\cdots i_l-i=1\\j_1,\ldots, j_l\\j=1,\ldots,n
\end{array}}}}\frac{1}{l!}\,{}^{ij}Q^{j_1\cdots j_l}_{i_1\cdots i_l}\, \xi^{j_1}_{i_1}\odot\cdots\odot\xi^{j_l}_{i_l}\frac{\partial}{\partial \xi_i^j}.$$ 
\end{enumerate}

This following lemma says that vector fields of polynomial-degree $-1$ are all the types \eqref{eq:contraction}.
\begin{lemma}\label{lemma:contraction}
Let $(M, \E)$ be a graded  manifold. For $i\geq 1$, a vector field $P\in\mathfrak{X}_{-i}(E)$ of polynomial-degree $-1$ and of degree $-i$, is of the form $\iota_e$ for some section $e\in \Gamma(E_{-i})$.
\end{lemma}
\begin{proof}
Note that a vector field $P\in \mathfrak{X}_{-i}(E)$ of polynomial-degree $-1$ is vertical, i.e. $P(f)=0$, since functions of $M$ are of polynomial-degree zero. Hence, in local coordinates $(\mathcal U, x_1,\ldots,x_n)$  $M$ and $(\xi_i^{j})_{j=1,\ldots, \mathrm{rk}(E^*_{-i})}$ with  be a homogeneous local trivialization of $E^*_{-i}$ and $(e_i^{j})_{j=1,\ldots, \mathrm{rk}(E_{-i})}$ be the dual basis of $(\xi_i^{j})_{j=1,\ldots, \mathrm{rk}(E^*_{-i})}$: a polynomial-degree $-1$ vector field  $P\in\mathfrak X_{-i}(E)$ of degree $-i$ is forced to be of the form
\begin{equation}
P|_\mathcal U=\sum_{j=1}^{\mathrm{rk}(E_{-i})}f^j(x)\frac{\partial}{\partial \xi^j_i}
\end{equation}
with $f^j\in C^\infty(U)$. Now, choose a local section $e_U\in \Gamma_\mathcal{U}(E_{-i})$ of the form, \begin{equation}
    e_U:=\sum_{j\geq 1}f^j e_i^j.
\end{equation}
It is then clear that $\iota_{e_U}=P|_\mathcal U$. Since local sections of $E_{-i}$ given on different coordinates chart domains $\mathcal U_a$ and $\mathcal{U}_b$ coincide on $\mathcal{U}_a\cap \mathcal{U}_b$ and then lift to a global section of $E_{-i}$. Thus,  $\iota_{e}=P$ for some $e\in \Gamma (E_{-i})$.
\end{proof}
\subsection{$NQ$-manifolds}
\begin{definition}
A \emph{$dg$-manifold} or \emph{$NQ$-manifold} is a positively graded manifold $(M, \E)$ endowed with a
degree $+1$ homological vector field $Q$ on $E$, i.e.,   $Q\in\mathfrak{X}_1(E)$ is such that $Q^2 = 0$.\\

They shall be denoted as a triple $ (M,\mathcal E, Q)$.
\end{definition}

\begin{example}\label{ex-Lie-algebra}
Given a finite dimension Lie algebra $(\mathfrak g,[\cdot\,,\cdot\,])$ of dimension $d$. We assume that $\mathfrak{g}$ is concentrated in degree $-1$. It is clear that $(M=\{pt\},\E=\wedge^\bullet\mathfrak{g}^*)$ is a graded manifold over $M=\{pt\}$. This graded manifold carries a $dg$-manifold structure. Precisely, we define the corresponding homological vector field as follows: fix a basis $(e_i)_{i=1,\ldots,n}$ of $\mathfrak g$ and let these global coordinate functions $(\xi^i)_{i=1,\ldots,n}$ on $\mathfrak g$ be its dual. We have
$$[e_i,e_j]=\sum_{l=1}^n\lambda_{ij}^le_l$$
for some coefficients $\lambda_{ij}^l\in\mathbb{K}$. One can check that the degree $+1$ vector field  $$Q=\frac{1}{2}\sum_{i,j,l= 1}^n\lambda_{ij}^l\xi^i\wedge\xi^j\frac{\partial}{\partial\xi^l}$$
corresponds to the Chevalley-Eilenberg differential $(\dd^{\mathrm{CE}}, \wedge^\bullet\mathfrak{g}^*)$. Therefore, $Q^2=0$. Notice that $Q^2=0$ is equivalent to the Jacobi identity for $\lb$.
\end{example}

\begin{example}
Given a differential graded vector bundle $((E_{-i})_{i\geq 1},\dd)$
over $M$. There is a natural $dg$ manifold
given by its sheaf of sections $(M, \E = \Gamma(S(E^*))$. In particular, the deferential map  $\dd\colon E \longrightarrow E$ is dualized as a degree $+1$ map $S^1(E^*)\longrightarrow S^1(E^*)$ that we extend to a  $C^\infty(M)$-linear derivation on $\E$ squared to zero.
\end{example}

\begin{example}
Let $E=T[1]M$ be the shifted bundle of $M$. It induces a graded manifold structure $(M, \E=\Omega(M))$ over $M$. This graded manifold carries a $dg$-manifold structure $Q$ that corresponds to the de Rham differential on $\Omega(M)$. In terms of coordinates, the homological vector field $Q$ reads

$$\sum_{i=1}^n\dd x_i\frac{\partial}{\partial x_i}.$$
\end{example}
Let us introduce some vocabularies that will need to use.
\begin{definition}
Let $(M,\E', Q')$ and $(M,\E, Q)$ be two $NQ$-manifolds.
     \begin{enumerate}
         \item A linear map $\Phi\colon \E \longrightarrow \E'$ is said to be of \emph{polynomial-degree/degree} $j\in\mathbb{Z}$ provided that, for all function $\alpha\in \E$ of polynomial-degree/degree $i$, $\Phi(\alpha)$ is of polynomial-degree/degree $i+j$. Any map $\Phi\colon \E\longrightarrow \E'$ of degree $i$ decomposes w.r.t the polynomial-degree as follows:\begin{equation*}
        \Phi=\sum_{r\in\mathbb{Z}}\Phi^{(r)}
    \end{equation*}
with $\Phi^{(r)}\colon \E \longrightarrow \E'$ a map of polynomial-degree $r$. 
\end{enumerate}
\end{definition}

\begin{remark}
When  $\Phi\colon \E \longrightarrow \E'$ is a graded morphism of algebras, necessarily one has $\Phi^{(r)} = 0 $  for all $r<0$. Furthermore, for all $n, r \in\mathbb N$ and all $\xi_1 ,\ldots, \xi_k \in \Gamma (V)$ one has:
\begin{equation}
\Phi^{(r)}(\xi_1\odot\cdots\odot\xi_n)= \sum_{i_1+\cdots+i_n =r}\Phi^{(i_1 )} (\xi_1 )\odot \cdots \odot\Phi^{(i_n)} (\xi_n).\end{equation} Obviously, in this case $\Phi$ is determined uniquely by the image of $\Gamma(V)$.
\end{remark}
\begin{definition}{}{} \label{def:oid-morph}
Let $(M,\E,Q)$ and $(M,\E',Q')$ be two $NQ$-manifolds over $M$ with
sheaves of functions $\E$ and $\E'$ respectively. A \emph{morphism of $NQ$-manifold over $M$} from  $(M,\E',Q')$ to $(M,\E,Q)$ is a morphism of graded manifolds 
 $\Phi\colon \E \longrightarrow \E'$ (of degree $0$) over the identity map  which intertwines 
$Q$ and $Q'$, i.e.,
\begin{equation}\label{eq:Lie-infty-morphis2}
 \Phi\circ  Q=Q'\circ \Phi. 
\end{equation}
\end{definition}

\begin{remark} It is important to notice that
\begin{itemize}
    \item morphisms of $NQ$-manifolds over $M$ are by definition $\mathcal O$-linear, since they are defined over the identity map
    \item the component $\Phi^{(r)}$ of polynomial-degree $r\geq 0$ of any $\mathcal{O}$-linear map $\Phi\colon \E\longrightarrow \E'$ maps $\Gamma(E^*)$ to $\Gamma (S^{r+1}(E'^* ))$. By $\mathcal O$-linearity, it gives rise to a section  $\phi_r\in\Gamma (S^{r+1} (E'^*)\otimes E)$. Therefore, one has

       \begin{equation}
           \Phi^{(r)}(\xi) = \langle\phi_r,\xi\rangle
       \end{equation}
for all $\xi\in  \Gamma(E^*)$. It follows that $\Phi$ is entirely determined by the collection  $\left(\phi_r\in\Gamma (S^{k+1} (E'^*)\otimes E)\right)_{r\geq 0}$ when  $\Phi$ is an algebra
morphism or a $\Xi$-derivation for some map $\Xi\colon \E\longrightarrow\E'$. In such case, for $r\geq 0$, $\phi_r\in\Gamma (S^{r+1} (E'^*)\otimes E)$  is then called the \emph{$r$-th Taylor coefficient} of $\Phi$. We also call the $0$-th Taylor coefficient $\phi_0\colon E'\rightarrow E$ the \emph{linear part} of $\Phi$. The latter is a chain map
\end{itemize}


\begin{equation}
\xymatrix{
\cdots\ar[r]&E_{-3}'\ar[d]^{\phi_0}\ar[r]^{\dd'^{(3)}}&E_{-2}'\ar[d]^{\phi_0} \ar[r]^{\dd'^{(2)}}&E_{-1}'\ar[d]^{\phi_0}\ar[r]^{\rho'}&TM\ar[d]^{\mathrm{id}}\\
\cdots\ar[r]&E_{-3}\ar[r]^{\dd^{(3)}}&E_{-2} \ar[r]^{\dd^{(2)}}&E_{-1}\ar[r]^{\rho}&TM}
\end{equation}

\end{remark}

\subsection{$NQ$-manifolds - Lie $\infty$-algebroids}
We have seen in Section \ref{sec:co-version} that Lie $\infty$-algebroids  (possibly infinite dimension) over $\mathcal{O}$ are one-to-one with co-differentials of the graded symmetric algebra, which are compatible with the action of $\mathcal{O}$. This correspondence provides a simple  characterization of Lie $\infty$-algebroids over an arbitrary commutative unital algebra $\mathcal{O}$. In the finite dimensional case, we can work without co-differentials by using T. Voronov's higher derived brackets construction \cite{Voronov2}. Roughly speaking, it is shown in \cite{Voronov2} that (finite dimensional) Lie $\infty$-algebroids as in Remark \ref{rmk:NGLA} are the same as $NQ$-manifolds. However, it is important to note the latter correspondence fails in infinite dimension case, since the identification $\Gamma\left(S^\bullet (E^*)\right)\simeq\Gamma \left(S^\bullet (E)^*\right)$ fails.  We will not explain entirely the construction here, but we refer the reader to \cite{VoronovTheodore,Voronov,Voronov2} also to \cite{Poncin,LeanMadeleineJotz2020Maru} for more details on the construction.\\

The following statement is similar to our Proposition \ref{co-diff} the difference lies in the fact that ours remains valid in infinite dimension.
\begin{proposition}[T. Voronov]
\label{prop:dual}
Let  $(M,\E)$ be a graded manifold of finite dimension in each degree. There is a one-to-
one correspondence between:

\begin{enumerate}
    \item (finitely generated) negatively graded Lie $\infty$-algebroids $(E,(\ell_k)_{k\geq 1},\rho)$ over $M$,
   \item[]and
   
  \item  homological vector fields $Q\in \mathfrak{X}_{+1}(E)$.
\end{enumerate}
\end{proposition}
The relation stated in Theorem \ref{prop:dual} can be enlightened in terms of the $k$-ary brackets $(\ell_k)$ as follows:

\begin{align}
    \left[\cdots\left[\left[Q,\iota_{e_1}\right],\iota_{e_2}\right],\ldots,\iota_{e_k}\right]^{(-1)}&=\iota_{\ell_k(e_1,\ldots,e_k)},\qquad\;\;\,\text{for all $e_1,\ldots,e_k\in \Gamma(E)$}\\\nonumber&\\\rho(e)[f]&=[Q,\iota_e]^{(0)}(f),\,\qquad\text{for all $f\in\mathcal O,\, e\in \Gamma(E_{-1})$}
\end{align}

In particular\footnote{Our sign's convention are those of \cite{LLS}.},
\begin{enumerate}
\item for all $f\in \mathcal{O}$, $e \in\Gamma(E_{-1})$\begin{equation*}
    \langle Q^{(1)}[f], e\rangle = \rho(e)[f],
\end{equation*}
    \item for all $\alpha\in\Gamma(E^*)$ and $e\in\Gamma(E)$:
\begin{equation*}
\left\langle Q^{(0)}[\alpha], e\right\rangle = 
\left\langle\alpha, \ell_1(e)\right\rangle,\end{equation*}

\item for all homogeneous elements $e_1,e_2\in \Gamma(E)$ and $\alpha \in\Gamma(E^*)$
\begin{equation*}
 \left \langle Q^{(1)}[\alpha], e_1\odot e_2\right\rangle = \rho(e_1)[\langle\alpha,e_2\rangle]-\rho(e_2)[\langle \alpha, e_1\rangle]-
 \langle \alpha, \ell_2(e_1, e_2)\rangle,
\end{equation*}
with the understanding that the anchor $\rho$
vanishes on $E_{-i}$ when $i\geq 2$.
\item for every $k \geq 3$, the $k$-ary brackets $\ell_k \colon \Gamma( S_\mathbb{K}^k(E))\longrightarrow\Gamma(E)$ and the polynomial-degree $k-1$ component
$Q^{(k-1)}\colon\Gamma(E^*)~\rightarrow~\Gamma(S_\mathbb{K}^k( E^*))$ of $Q$ are dual to each other, i.e.
\begin{equation*}
    \left\langle Q^{(k-1)}[\alpha], e_1\odot\cdots\odot e_k\right\rangle=
    \left\langle\alpha, \ell_k(e_1, \ldots,e_k)\right\rangle,\qquad\text{for $\alpha\in \Gamma(E^*)$ and $e_1, \ldots,e_k\in \Gamma(E)$}.
\end{equation*}
\end{enumerate}

\begin{convention}
From now on, unless otherwise mentioned, we shall simply say "Lie $\infty$-algebroids over $M$" for "finitely generated $\infty$-algebroids over $M$". Also, Lie $\infty$-algebroids $(E,(\ell_k)_{k\geq 1},\rho)$ over $M$ shall be denoted as $(E,Q)$.
\end{convention}

\begin{remark}
Whether we see Lie $\infty$-algebroids as $NQ$-manifolds or as co-differentials, these two approches have in common to give the same notion of Lie $\infty$-morphism of $\infty$-algebroids. This can be seen directly by writing the conditions \eqref{def:morph} and \eqref{eq:Lie-infty-morphis2} in terms of the $k$-ary brackets.
\end{remark}

\section{Universal $Q$-manifolds}\label{sec:univ-q-manifold}
This section can be understood as a consequence of the main results of the first part of the thesis. We recover the result on universal $Q$-manifolds of a singular foliation \cite{LLS, LavauSylvain}, whose existence was proved under the condition that geometric resolutions exist. We recall that the existence of such a resolution is not guaranteed, unlike the case of resolution by free modules. We refer the reader to Appendix \ref{appendix:mod} and \ref{sec:geo-reso} for more details on resolutions of modules and geometric resolutions of singular foliations.

\begin{theorem}
Let $\mathfrak F$ be a singular foliation on a manifold $M$.
\begin{enumerate}
    \item Any resolution of $\mathfrak F $ by free $\mathcal{O}$-modules (which may not be a geometric resolution)
\begin{equation}
\cdots \stackrel{\dd} \longrightarrow P_{-3} \stackrel{\dd}{\longrightarrow} P_{-2} \stackrel{\dd}{\longrightarrow} P_{-1} \stackrel{\rho}{\longrightarrow} \mathfrak F\longrightarrow 0 \end{equation}
  carries a Lie $\infty $-algebroid structure over $\mathfrak F$ whose unary bracket is $\ell_1:=\dd $.
  \item In particular, when $\mathfrak F$ admits a geometric resolution $(E,\dd,\rho)$, there exists a Lie $\infty$-algebroid $(E,Q)$ over $\mathfrak F$ whose linear part is $(E,\dd,\rho)$.
\end{enumerate} 
\end{theorem}
\begin{proof}
Apply Theorem \ref{thm:existence} to  $\mathfrak F$ seen as a Lie-Rinehart algebra over $\mathcal{O}$.
\end{proof}

\begin{proposition}\label{th:universal2}
Let $\mathfrak F$ be a singular foliation over $M$. Given, 
\begin{enumerate}
    \item[a)] a Lie $ \infty$-algebroid $(M, \E',Q')$ that terminates in $\mathfrak F$, i.e, $\rho'(\Gamma(E_{-1}))\subseteq \mathfrak F$,
    \item[b)] a universal Lie $\infty $-algebroid $(M,\E,Q)$ of $\mathfrak F $,
\end{enumerate} 
then
\begin{enumerate} 
\item there exists a Lie $\infty $-morphism from $(M,\E',Q')$ to $(M,\E,Q)$.
\item and any two such morphisms  are homotopic.  
\end{enumerate}
\end{proposition}
\begin{proof}
Apply Theorem \ref{th:universal}.
\end{proof}

\begin{corollary}
Two universal Lie $\infty $-algebroid of a singular foliation are homotopy equivalent.

Moreover, the homotopy equivalence between them is unique up to homotopy.
\end{corollary}

\begin{proof}
Apply Corollary \ref{cor:unique}.
\end{proof}

\subsection{The complex defined by $\mathrm{ad}_{Q^{(0)}}$}\label{ad_qzero}
Let $\mathfrak F$ be a singular foliation on $M$ that admits a universal algebroid  $(E, Q)$. For a fixed $k\in\mathbb N_0\cup\{-1\}$  consider the bicomplex defined on  $\text{Hom}^\bullet_\mathcal O\left(\Gamma(E^*),\Gamma\left(S_\mathbb{K}^{k+1}(E^*)\right)\right)$ where the horizontal differential $\partial^h$ (resp. vertical differential $\partial^v$) are given by left composition (resp. right composition with $Q^{(0)}$), namely, for all $\Psi^{(k)}\in\text{Hom}^\bullet_\mathcal O\left(\Gamma(E^*),\Gamma\left(S_\mathbb{K}^{k+1}(E^*)\right)\right)$ one has, \begin{equation}
   \partial^h\left(\Psi^{(k)}\right):=\begin{cases}
   \Psi^{(k)}\circ Q^{(0)}\quad&\text{when $\Psi^{(k)}$ restricts to $\Gamma(E^*_{-i})$,\, with $i\neq 1$},\\\Psi^{(k)}\circ\rho^*\quad&\text{when $\Psi^{(k)}$ restricts to $\Gamma(E^*_{-1})$,}
   \end{cases}
\end{equation}and \begin{equation}
\partial^v\left(\Psi^{(k)}\right):= \begin{cases}\rho^*\circ\Psi^{(k)}\quad&\text{when $\Psi^{(k)}$ is of degree $i$ and restricts to $\Gamma(E^*_{-i-1})$},\\Q^{(0)}\circ\Psi^{(k)}\quad&\text{otherwise}.
 \end{cases}
\end{equation}With the understanding that $\Gamma\left(S_\mathbb{K}^{0}(E^*)\right)\simeq \mathcal{O}$. The total differential is given by the formula,$$
    \partial\left(\Psi^{(k)}\right)=\partial^h\left(\Psi^{(k)}\right)-(-1)^{\lvert \Psi^{(k)} \rvert}\partial^v\left(\Psi^{(k)}\right)\quad\text{for all}\quad\Psi^{(k)}\in\text{Hom}^\bullet_\mathcal O\left(\Gamma(E^*),\Gamma\left(S_\mathbb{K}^{k+1}(E^*)\right)\right).$$
The following diagram pictures the idea of the bicomplex. The total degree is given by the anti-diagonals lines.

\begin{equation}\label{recap2}
	\scalebox{0.7}{\hbox{$
 \begin{array}{ccccccccc}
		& & \vdots & & \vdots & & \vdots & & \\ 
		& & \uparrow & & \uparrow & & \uparrow & & \\  
		& \rightarrow & \text{Hom}_\mathcal O\left(\Gamma(E_{-2}^*),\Gamma\left(S_\mathbb{K}^{k+1}(E^*)\right)_{k+3}\right) & \overset{\partial^h}{\rightarrow} & \text{Hom}_\mathcal O\left(\Gamma(E_{-1}^*),\Gamma\left(S_\mathbb{K}^{k+1}(E^*)\right)_{k+3}\right) 
		& \overset{\cdot\,\circ\rho^*}{\rightarrow} & \text{Hom}_\mathcal O\left(\mathfrak F^*,\Gamma\left(S_\mathbb{K}^{k+1}(E^*)_{k+3}\right)\right) & \rightarrow & 0\\ 
		& & \partial^v\uparrow & & \partial^v\uparrow & & D^v\uparrow & & \\ 
		& \rightarrow & \text{Hom}_\mathcal O\left(\Gamma(E_{-2}^*),\Gamma\left(S_\mathbb{K}^{k+1}(E^*)\right)_{k+2}\right) & \overset{\partial^h}{\rightarrow} & \text{Hom}_\mathcal O\left(\Gamma(E_{-1}^*),\Gamma\left(S_\mathbb{K}^{k+1}(E^*)\right)_{k+2}\right) 
		& \overset{\cdot\circ\,\rho^*}{\rightarrow} & \text{Hom}_\mathcal O\left(\mathfrak F^*,\Gamma\left(S_\mathbb{K}^{k+1}(E^*)_{k+2}\right)\right) & \rightarrow& 0\\ 
		& & \partial^v\uparrow & & \partial^v\uparrow & & \partial^v\uparrow & &  \\ & \rightarrow &\text{Hom}_\mathcal O\left(\Gamma(E_{-2}^*),\Gamma\left(S_\mathbb{K}^{k+1}(E^*)\right)_{k+1}\right) & \overset{\partial^h}{\rightarrow} & \text{Hom}_\mathcal O\left(\Gamma(E_{-1}^*),\Gamma\left(S_\mathbb{K}^{k+1}(E^*)\right)_{k+1}\right) 
		& \overset{\cdot\,\circ\rho^*}{\rightarrow}& \text{Hom}_\mathcal O\left(\mathfrak{F}^*,\Gamma\left(S_\mathbb{K}^{k+1}(E^*)\right)_{k+1}\right) & \rightarrow & 0\\  
	& & \uparrow & & \uparrow & & \uparrow & & \\ 
	& & 0 & & 0 & & 0 & & \\   \end{array} 
 $}}
\end{equation}
	
There is no cohomology for the total complex which is governed by $\partial\equiv[\,\cdot\,,Q^{(0)}]$ since the lines are exact (see Proposition \ref{Acyclic-Assembly-Lemma}). When $k=-1$ the bicomplex \eqref{recap2} is just the line complex:
\begin{equation}\label{-1:case}
    \cdots \stackrel{\partial^h}{\longrightarrow}  \text{Hom}_\mathcal O\left( \Gamma(E^*_{-2}),\mathcal{O}\right) \stackrel{\partial^h}{\longrightarrow}\text{Hom}_\mathcal O\left( \Gamma(E^*_{-1}),\mathcal{O}\right)\stackrel{\partial^h}{\longrightarrow}\text{Hom}_\mathcal O\left(\mathfrak{F}^*,\mathcal{O}\right)\longrightarrow 0
\end{equation}which is also exact.

\subsection{A result on longitudinal vector fields and examples}
In this section, we study the cohomology of longitudinal vector fields, which will help in proving the main results stated in the beginning of Chapter \ref{sec:3}.\\

Let $\mathfrak{F}$ be a singular foliation over $M$.
\begin{definition}
Let $E$ be a splitted graded manifold over $M$ with sheaf of function $\E$. A vector field $L\in\mathfrak{X}(E)$ is said to be a \emph{longitudinal vector field for $\mathfrak{F}$} if there exists vector fields $X_1,\ldots,X_k\in \mathfrak{F}$ and functions $\Theta_1,\ldots,\Theta_k\in\E$ such that \begin{equation}
  L(f)=\sum_{i=1}^kX_i[f]\Theta_i,\qquad \forall f\in\mathcal{O}.
\end{equation}

\end{definition}
We also need to define the following
\begin{definition}
   A vector field $X\in\mathfrak{X}(M)$ is said to be an \emph{infinitesimal symmetry of $\mathfrak F$}, if $[X,\mathfrak F]\subset\mathfrak{F}$. The Lie algebra of infinitesimal symmetries of $\mathfrak{F}$ is denoted by, $\mathfrak{s}(\mathfrak F)$
\end{definition}
\begin{example}\label{longi:examples}Here are some examples.
\begin{enumerate}
\item Vertical vector fields on a graded manifold are longitudinal for any singular foliation.
    \item For any $Q$-manifold $(E,Q)$ over a manifold $M$. The homological vector field $Q\in\mathfrak{X}(E)$ is a longitudinal vector field for $\mathfrak{F}:=\rho(\Gamma(E_{-1}))$: in local coordinates $(\mathcal{U},x_1, \dots,x_n) $ on $M$ and a local trivialization $\xi^1,\xi^2,\ldots$ of graded sections in $\Gamma(E^*)$. The vector fields $Q $ are of the form:
    \begin{align}\label{eq:sym_long1}
    Q&=\displaystyle{\sum_j\sum_{k,\, |\xi^k|=1}Q^j_k(x) \xi^k\frac{\partial}{\partial x_j} + \sum_{j}\sum_{k,\iota_1,\ldots,\iota_k}\frac{1}{k!}Q^j_{\iota_1,\ldots,\iota_k}(x) \xi^1\odot\cdots\odot\xi^k\frac{\partial}{\partial \xi^j}}\\\nonumber&=\displaystyle{\sum_{k,\, |\xi^k|=1}\xi^k\left(\sum_jQ^j_k(x)\frac{\partial}{\partial x_j}\right) + \cdots}\end{align}
  Take $X_k:=\sum_jQ^j_k(x)\frac{\partial}{\partial x_j}\in \mathfrak{F}(\mathcal{U})$ for every $k$. One has for all $f\in C^\infty(\mathcal{U})$, $Q(f)=\sum_k\xi^kX_k[f]$.  
    \item For $(E, Q)$ a $Q$-manifold and $\mathfrak{F}:=\rho(\Gamma(E_{-1}))$ its basic singular foliation. For any extension of a symmetry $X\in\mathfrak{s}(\mathfrak{F})$ of $\mathfrak{F}$ to a degree zero vector field $\widehat{X}\in\mathfrak{X}(E)$, the degree $+1$ vector field $[Q,\widehat{X}]$ is longitudinal for $\mathfrak{F}$.
    \item[] Let us show this last point using local coordinates $(x_1, \dots,x_n) $ on $M$ and a local trivialization $\xi^1,\xi^2,\ldots$ of graded sections in $\Gamma(E^*)$. The vector fields $Q $ and $\widehat{X} $ take the form:
    \begin{equation}\label{eq:sym_long}
    \begin{array}{rcl}
    Q&=&\displaystyle{\sum_j\sum_{k,\, |\xi^k|=1}Q^j_k(x) \xi^k\frac{\partial}{\partial x_j} + \sum_{j}\sum_{k,\iota_1,\ldots,\iota_k}\frac{1}{k!}Q^j_{\iota_1,\ldots,\iota_k}(x) \xi^1\odot\cdots\odot\xi^k\frac{\partial}{\partial \xi^j}} \\
    \widehat{X}&= &\displaystyle{X+ \sum_{j}\sum_{k,\iota_1,\ldots,\iota_k}\frac{1}{k!}X^j_{\iota_1,\ldots,\iota_k}(x)\xi^1\odot\cdots\odot\xi^k\frac{\partial}{\partial \xi^j}} \end{array}
    \end{equation}
    where $\displaystyle{X=\sum_{i=1}^n X_i(x) \frac{\partial}{\partial x_i}}$. By using Equation \eqref{eq:sym_long} we note that all the terms of $[Q, \widehat{X}]$ are vertical except maybe for the ones where the vector field $X$ appears. For $k\geq 1$, the vector field   $[Q^j_{\iota_1,\ldots,\iota_k}\xi^1\odot\cdots\odot\xi^k\frac{\partial}{\partial \xi^j}, X]$ is vertical; and for every fix $k$, one has  \begin{align*}
        \left[\sum_{j=1}^n Q^j_k\xi^k\frac{\partial}{\partial x_j}, X\right]=\xi^k\left[\sum_{j=1}^n Q^j_k\frac{\partial}{\partial x_j}, X\right].
    \end{align*}
    
    Thus, $\displaystyle{\left[\sum_{j=1}^n Q^j_k\frac{\partial}{\partial x_j}, X\right]\in\mathfrak{F}}$, since $X$ is a symmetry for $\mathfrak{F}$ and $\displaystyle{\sum_{j=1}^n Q^j_k\frac{\partial}{\partial x_j}\in\mathfrak{F}}$. 
\end{enumerate}
\end{example}

\begin{remark}\label{longitudinal-stable}
Longitudinal vector fields are stable under the graded Lie bracket. We denote by $\mathfrak L(E)$ the graded Lie algebra of longitudinal vector fields for $\mathfrak{F}$.
\end{remark}
\begin{remark}
Let us study vector fields on $E$. 
\begin{enumerate}
\item Sections of $E$ are identified with derivations under the isomorphism mapping  $e\in\Gamma(E)\longmapsto \iota_e \in\mathfrak{X}(E)$. This allows us to identify a vertical vector field with (maybe infinite) sums of  tensor products of the form  $\Theta\otimes e$ with $\Theta\in \E, e\in \Gamma(E)$. 

\item Any connection on $\Gamma(E^*)$ 
induces a vector field of degree zero $\widetilde{\nabla}_X\in \mathfrak{X}(E)$ by setting for $f\in\mathcal{O}$, $\widetilde{\nabla}_X(f):= X[f]$. Once a connection is chosen, we have for all $k\in\mathbb{Z}$ \begin{equation}\label{eq:identification}
    \mathfrak{X}_k(E)\simeq\bigoplus_{j\geq 1}\E_{k+j}\otimes_\mathcal{O}\Gamma(E_{-j})\oplus\E_{k}\otimes_\mathcal{O}\mathfrak{X}(M)\simeq \oplus_{j\geq 1}\Gamma(S(E^*)_{k+j}\otimes E_{-j} )\oplus\Gamma(S(E^*)_k\otimes TM).
\end{equation}
Thus, one can realize a vector field $P\in\mathfrak{X}_k(E)$ as a sequence $P=(p_0,p_1,\ldots)$, where $p_0\in \Gamma(S(E^*)_k\otimes TM)$ and $p_i\in\Gamma(S(E^*)_{k+i}\otimes E_{-i})$ for  $i\geq 1$ are called \emph{components} of $P$. In the diagram \eqref{longitudinal:complex}, $P=(p_0,p_1,\ldots)$ is represented as an element of the anti-diagonal and $p_i$ is on column $i$. We say that $P$ is of \emph{depth} $n\in \mathbb{N}$ if $p_i=0$ for all $i< n $. In particular, vector fields of depth greater or equal to $1$ are vertical. Under the decomposition \eqref{eq:identification}, the differential map $\mathrm{ad}_Q$ takes the form\begin{equation}
    D=D^h+\sum_{s\geq 0}D^{v_s}
\end{equation}with $D^2=0$. Here $D^h=\mathrm{id}\otimes\dd\quad\text{or}\quad \mathrm{id}\otimes\rho$, and $D^{v_s}\colon \Gamma(S(E^*)_k\otimes E_{-i})\to \Gamma(S(E^*)_{k+s+1}\otimes ~E_{-i-s})$ for $i\geq 0,\,s\geq 0$, we shall denote $E_{0}:= TM$. We denote the latter complex by $(\mathfrak{X}, D)$. They can be represented as anti-diagonal lines in the following commutative diagram, whose lines are complexes of $\mathcal{O}$-modules

\begin{equation}\label{longitudinal:complex}
\xymatrix{  & \vdots &&\vdots && \vdots\\\cdots\ar[r] &\Gamma(S(E^*)_{k+2}\otimes E_{-2})\ar[u] \ar^{\mathrm{id}\otimes\dd}[rr]&&\Gamma(S(E^*)_{k+2}\otimes E_{-1})\ar[u] \ar^{\mathrm{id}\otimes\rho}[rr]&& \Gamma(S(E^*)_{k+2}\otimes TM)\ar[u]\\\cdots\ar[r] & \Gamma(S(E^*)_{k+1}\otimes E_{-2})\ar[luu]\ar[u]^\ddots_{Q\otimes\mathrm{id}\, +\, \cdots} \ar^{\mathrm{id}\otimes\dd}[rr]&&\Gamma(S(E^*)_{k+1}\otimes E_{-1})\ar[luu]\ar[u]^\ddots_{Q\otimes\mathrm{id}\,+\, \cdots} \ar^{\mathrm{id}\otimes\rho}[rr]&& \Gamma(S(E^*)_{k+1}\otimes TM)\ar[lluu]\ar[u]^\ddots_{Q\otimes\mathrm{id}\, +\, \cdots}\\\cdots \ar[r]& \Gamma(S(E^*)_k\otimes E_{-2})\ar[luu]\ar[u]^\ddots_{Q\otimes\mathrm{id}\, +\, \cdots}\ar^{\mathrm{id}\otimes\dd}[rr]&&\Gamma(S(E^*)_k\otimes E_{-1})\ar@/^/[llluuu]\ar[lluu]\ar[u]^\ddots_{Q\otimes\mathrm{id}\, +\, \cdots} \ar^{\mathrm{id}\otimes\rho}[rr]&& {\Gamma(S(E^*)_k\otimes TM)}\ar@/^/[llluuu]\ar[lluu]\ar[u]^\ddots_{Q\otimes\mathrm{id}\, +\, \cdots}\\ & \vdots\ar[u] &&\vdots\ar[u] && \vdots\ar[u]\\& \texttt{column $2$}&&\texttt{column $1$} && \texttt{column $0$}\\}
\end{equation}
\end{enumerate}
\end{remark}

Under this correspondence, we understand longitudinal vector fields as the following. 
\begin{lemma}
A graded vector field $P=(p_0,p_1,\ldots)\in\mathfrak{X}$ is longitudinal if $p_0\in\E\otimes_\mathcal{O}\mathfrak{F}$.
\end{lemma}


The following theorem is crucial for Chapter \ref{chap:symmetries}.
\begin{theorem}\label{thm:longitudinal}
Let $(E,Q)$ be a universal $Q$-manifold of $\mathfrak{F}$. 
\begin{enumerate}
    \item Longitudinal vector fields form an acyclic complex.
    
     More precisely, any longitudinal vector field on $E$ which is a $\mathrm{ad}_Q$-cocycle is the image through $\mathrm{ad}_Q$ of some vertical vector field on $E$.
     
     \item More generally, if a vector field on $E$ of depth $n$ is a $\mathrm{ad}_Q$-cocycle, then it is the image through $\mathrm{ad}_Q$ of some  vector field on $E$ of depth $n+1$. 
\end{enumerate}

\end{theorem}
\begin{proof}
Since $(E,Q)$ is an universal $Q$-manifold of $\mathfrak{F}$, lines in \eqref{longitudinal:complex} are exact. It is now a diagram chasing phenomena. Let $P=(p_0,p_1,\ldots,)\in\mathfrak{X}$ be a longitudinal element which is a $D$-cocycle. By longitudinality there exists an element  $b_1\in\Gamma(S(E^*)\otimes E_{-1})$ such that $(\mathrm{id}\otimes\rho)(b_1)=p_0$. Set $P_1=(0,b_1,0,\ldots)$, that is we extend $b_1$ by zero on $\Gamma(S(E^*)\otimes E_{\leq -2})$ and $\Gamma(S(E^*)\otimes TM)$. It is clear that $P-D(P_1)=(0,p'_1,p'_2,\ldots)$ is also a $D$-cocycle. In particular, we have $D^h(p'_1)=0$ by exactness there exists $b_2\in\Gamma(S(E^*)\otimes E_{-2})$ such that $D^h(b_2)=p'_1$. As before put $P_2=(0,0,b_2,0,\ldots)$. Similarly, $P-D(P_1)-D(P_2)=(0,0,p''_2,p''_3,\ldots)$ is a $D$-cocycle. By recursion, we end up to construct $P_1,P_2,\ldots$ that satisfy $P-D(P_1)-D(P_2)+\cdots =0$, that is, there exists an element $B=(0,b_1,b_2,\ldots)\in \mathfrak{X}$ such that $D(B)=P$. This proves item $1$.

To prove item $2$ it suffices to cross out in the diagram \eqref{longitudinal:complex} the columns number $0,\ldots,n-~1$, which does not break exactness. The proof now follows as for item $1$.
\end{proof}

In particular, we deduce from Theorem \ref{thm:longitudinal} the following exact subcomplex.
\begin{corollary}\label{cor:longitudinal}
Let $(E,Q)$ be a universal $Q$-manifold of $\mathfrak{F}$.
 The subcomplex $\mathfrak{V}_Q$ of $(\mathfrak{X}(E),\mathrm{ad}_Q)$ made of vertical vector fields $P\in\mathfrak{X}(E)$ that satisfy $P\circ Q(f)=~0$ for all $f\in\mathcal{O}$ is exact.
\end{corollary}

\begin{proof}
Let $P\in\mathfrak{X}(E)$ be a vertical vector field which is an $\mathrm{ad}_Q$-cocycle (note that we have automatically $P\circ Q(f)=0$ for all $f\in\mathcal{O}$). By Theorem \ref{thm:longitudinal} there exists a vertical vector field $\widetilde{P}\in\mathfrak{X}(E)$ such that $[Q,\widetilde{P}]=P$. Moreover, $\widetilde{P}\in \mathfrak{V}_Q$, since for all $f\in\mathcal{O}$, $$0=[Q,\widetilde{P}](f)=(-1)^{|\widetilde{P}|}\widetilde{P}\circ Q(f).$$
\end{proof}

The following remark will be used especially for the proof of Theorem \ref{proof:main} of Chapter \ref{chap:symmetries}.
\begin{remark}\label{rmk:arity}
For a cocycle  $P\in \mathfrak{V}_Q$ of degree $0$ one has $P^{(-1)}= 0$ (for degree reason). By Corollary \ref{cor:longitudinal}, $P$ is the image by $\mathrm{ad}_Q$ of an element $\widetilde P\in \mathfrak V_Q$ i.e. such that $[Q,\widetilde P]=P$. Also, one can choose ${\widetilde{P}}^{(-1)}=0$: we have  $$[Q, \widetilde P]^{(-1)}=[Q^{(0)}, {\widetilde{P}}^{(-1)}]=P^{(-1)}=0.$$ By exactness of $\mathrm{ad}_{Q^{(0)}}$ (see Section \ref{ad_qzero}), we have $P^{(-1)}=[Q^{(0)}, \vartheta]$ for some $\mathcal{O}$-linear map $\vartheta\in \mathrm{Hom}(\Gamma(E^*), \Gamma(S^0(E^*)))$ of degree $-2$. Put $\widebar{P}:= \widetilde P -[Q, \vartheta]$, where $\vartheta$ is extended to a derivation of polynomial-degree $-1$. Clearly, \begin{align*}
    [Q, \widebar{P}]=P\qquad\text{and}\qquad \widebar{P}^{(-1)}= \widetilde{P}^{(-1)}-[Q, \vartheta]^{(-1)}=\widetilde{P}^{(-1)}-[Q^{(0)}, \vartheta]=0.
\end{align*}
Therefore, $P=\mathrm{ad}_Q(\widebar P)$ with ${\widebar P}^{(-1)}=0$.
\end{remark}

Here is direct  consequence of Theorem \ref{thm:longitudinal} and Proposition \ref{prop:sub-complex-exact}.

\begin{corollary}
Let $\mathfrak{F}$ a singular foliation on $M$. Let $(E,Q)$ be a universal $Q$-manifold of $\mathfrak{F}$. Then,\begin{equation}
    H^\bullet\left(\mathfrak{X}(E),\mathrm{ad}_Q\right)\simeq H^\bullet\left(\frac{\mathfrak{X}(E)}{\mathfrak{L}(E)},\overline{\mathrm{ad}}_Q\right).
\end{equation}
where $\overline{\mathrm{ad}}_Q$ is induced by $\mathrm{ad}_Q$ on the quotient space.
\end{corollary}

Since all vertical vector fields are longitudinal, there is an isomorphism of $\mathcal{O}$-modules \begin{equation}
    \label{eq:sym_long2}
\frac{\mathfrak{X}(E)}{\mathfrak{L}(E)}\simeq \Gamma(S^\bullet(E^*))\otimes_\mathcal{O}\left(\frac{\mathfrak{X}(M)}{\mathfrak{F}}\right).
\end{equation}
Let us describe the differential map $\overline{\mathrm{ad}}_Q$ using this isomorphism: let  $(\mathcal{U}, x_1, \dots,x_n) $ be local coordinates on $M$ and take local trivialisations $\xi^{1},\xi^2,\ldots$ of graded sections in $\Gamma(E^*)$, also let $\xi_1,\xi_2,\ldots$ the corresponding dual sections in $\Gamma(E)$. Notice that locally, any  $\overline{P}\in\frac{\mathfrak{X}(E)}{\mathfrak{L}(E)}$ is represented by a vector field of the form $$\sum_{j,\,i_1,\ldots,i_k}f^j_{i_1,\ldots,i_k}(x) \xi^{i_1}\odot\cdots\odot\xi^{i_k}\otimes \frac{\partial}{\partial x_j},\quad f^j_{i_1,\ldots,i_k}\in C^\infty(\mathcal{U})$$
since the vector fields in $\tfrac{\partial}{\partial \theta^i}$ are vertical vector fields. That is to say, $\overline{P}$ can be represented by an element of the form $P=\Theta \otimes X$ with $\Theta\in \Gamma_\mathcal{U}(S^\bullet(E^*))$ and  $X\in \mathfrak{X}(\mathcal{U})$, using the decomposition \eqref{eq:sym_long2}. Since
\begin{align*}
    [Q,P]&=Q[\Theta]X +(-1)^{|\Theta|}\Theta\cdot[Q,X]\\&=Q[\Theta]X+(-1)^{|\Theta|}\Theta\odot\sum_{k,\, |\xi^k|=1}\xi^k\left[\sum_jQ^j_k(x)\frac{\partial}{\partial x_j},X\right] +\mathrm{vertical\; vector\; fields}\\&=Q[\Theta]X+(-1)^{|\Theta|}\sum_{k,\, |\xi^k|=1}\Theta\odot\xi^k\left[\rho(\xi_k),X\right]+\mathrm{vertical\; vector\; fields}
\end{align*}
we have \begin{equation}\label{eq:longitudinal-diff}
    \overline{\mathrm{ad}}_Q(\overline{P})=Q[\Theta]\otimes\overline{X} +(-1)^{|\Theta|}\sum_{k,\, |\xi^k|=1}\Theta\odot\xi^k\otimes\overline{\left[\rho(\xi_k),X\right]},
\end{equation} where $\overline{X}$ stands for the class of a vector field $X\in \mathfrak{X}(M)$ in $\frac{\mathfrak{X}(M)}{\mathfrak{F}}$. Therefore, locally $$\overline{\mathrm{ad}}_Q=\overline{Q\otimes \text{id}}+\sum_{k,\, |\xi^k|=1}\mathrm{id}\odot\xi^k\otimes\nabla^\mathrm{Bott}_{\xi_k}$$where $\nabla^\mathrm{Bott}\colon \Gamma(E_{-1})\times \frac{\mathfrak{X}(M)}{\mathfrak{F}} \longrightarrow \frac{\mathfrak{X}(M)}{\mathfrak{F}}$, is the Bott connection associated to $\mathfrak{F}$, i.e.
$\nabla^\mathrm{Bott}_e\overline{X}:=\overline{[\rho(e), X]}$.

\begin{remark}Assume that $M=\mathbb{R}^d$ and $\mathfrak{F}$ is a singular foliation that admits two leaves: $\{0\}$ and $\mathbb{R}^d\setminus {\{0\}}$. 
Then$$\frac{\mathfrak{X}(E)}{\mathfrak{L}(E)}\simeq \Gamma(S^\bullet(E^*))|_0\otimes_{\mathbb R}\left(\mathbb{R}^d\oplus\frac{\mathcal{I}_0\mathfrak{X}(M)}{\mathfrak{F}}\right).$$
The differential may be complicated to describe.
\end{remark}


\begin{example}
Let $\mathfrak F=\mathcal{I}_0\mathfrak X(\mathbb{R}^d)$ and $(E,Q)$ one of its universal Lie algebroid such that $E_{-1}=\simeq \mathfrak{gl}(\mathbb{R}^d)$ is the transformation Lie algebroid for the action of $\mathfrak{gl}(\mathbb{R}^d)$ on $\mathbb R^d $. For this singular foliation, $\frac{\mathfrak{X}(\mathbb{R}^d)}{\mathfrak{F}}\simeq \mathbb{R}^d.$ Thus, we obtain $$\frac{\mathfrak{X}(E)}{\mathfrak{L}(E)}\simeq \Gamma(S^\bullet(E^*))|_0\otimes_{C^\infty (\mathbb{R}^d)}\mathbb{R}^d.$$
and the Bott connection is simply the Chevalley-Eilenberg differential for the action of  $\mathfrak{gl}(\mathbb{R}^d)$ on $\mathbb R^d $.
In particular, it is easy to see that $$H^1(\mathfrak{X}(E))=H^1 \left( \frac{\mathfrak{X}(E)}{\mathfrak{L}(E)}\right) = H^1_{CE}(\mathfrak{gl}(\mathbb{R}^d), \mathbb R^d) =0 $$
In other words, the universal Lie $\infty$-algebroid is, in that case, rigid, i.e. any formal deformation $Q+\epsilon Q_1 + \epsilon^2 Q_2 + \cdots $ of its derivation $Q$ is formally trivial \cite{Lavau-L-G}.
\end{example}

Here is a second example where the universal Lie $\infty$-algebroid is rigid.  

\begin{example}
Let $\mathfrak{F}$ be the singular foliation given by the action of Lie algebra $\mathfrak{g}=\mathfrak{sl}(2,\mathbb{R})$ on $\mathbb{R}^2$ through the vector fields: $$\underline{e}=x\frac{\partial}{\partial y},\qquad\underline{f}=y\frac{\partial}{\partial x},\qquad \underline{h}=x\frac{\partial}{\partial x}-y\frac{\partial}{\partial y},$$
where $e,f,h$ are the standard basis elements of $\mathfrak{sl}(2,\mathbb{R})$. The singular foliation $\mathfrak{F}$ admits a universal Lie $\infty$-algebroid $(E,Q)$ built on a geometric resolution (\cite{LLS}, Example 3.31) of length $2$ $$0{\longrightarrow}E_{-2}\stackrel{\dd}{\longrightarrow}E_{-1}\stackrel{\rho}{\longrightarrow} TM,$$
where $E_{-1}\simeq \mathfrak{sl}(2,\mathbb{R})\times \mathbb{R}^2$ is a trivial vector bundle of rank 3, and $E_{-2}$ a  trivial vector bundle of rank 1. Here, the bracket between two constant sections of $E_{-1}$ is defined as being
their bracket in $\mathfrak{sl}(2,\mathbb{R})$ 
and all other brackets between generators is trivial. Also, $Q$ takes the form$$Q=xe^*\frac{\partial}{\partial y}+yf^*\frac{\partial}{\partial x}+h^*(x\frac{\partial}{\partial x}-y\frac{\partial}{\partial y})+e^*\odot f^*\frac{\partial}{\partial h^*}-2e^*\odot h^*\frac{\partial}{\partial e^*}+ 2f^*\odot h^*\frac{\partial}{\partial f^*}+(x^2 f^* + y^2 e^* -xy h^*)  \frac{\partial}{\partial \mu^*}
$$ where $\mu$ is the generator of $E_{-2}$, and $(e^*, f^*, h^*, \mu^*)$ is the dual basis of $((e, f, h, \mu))$.

\noindent
In this case, one has 
$$\frac{\mathfrak{X}(\mathbb{R}^2)}{\mathfrak{F}}\simeq \mathbb{R}^2\oplus \mathbb{R}$$ is generated by the classes of $$
\left\langle \frac{\partial}{\partial x}, \frac{\partial}{\partial y}, x\frac{\partial}{\partial x}+ y\frac{\partial}{\partial y}\right\rangle.
$$
Hence, as a graded vector space:
$$\frac{\mathfrak{X}(E)}{\mathfrak{L}(E)}\simeq \left(\wedge^\bullet\mathfrak{sl}_2^*\oplus \mathbb{R}[2]\right)\otimes_\mathcal{\mathbb{R}}(\mathbb{R}^2\oplus \mathbb{R}).$$
The differential induced by $Q $ is easily checked to be the Chevalley-Eilenberg differential for the natural $\mathfrak{sl}(2,\mathbb{R})$-action on these modules. In particular,  $$H^1(\mathfrak{X}(E))=H^1 \left( \frac{\mathfrak{X}(E)}{\mathfrak{L}(E)}\right) =
H^1\left(\mathfrak{sl}(2,\mathbb{R}), \frac{\mathfrak{X}(\mathbb{R}^2)}{\mathfrak{F}}\right)=0,$$ since  the Lie algebra $\mathfrak{sl}(2,\mathbb{R})$ is semisimple. This implies that the universal Lie $\infty$-algebroid $(E,Q)$ is rigid in the sense of \cite{Lavau-L-G}.
\end{example}

\vspace{2cm}
\begin{tcolorbox}[colback=gray!5!white,colframe=gray!80!black,title=Conclusion:]
We recalled the duality "Lie $\infty$-algebroid $\longleftrightarrow$ $Q$-manifolds of finite rank in all degree" so that the universal Lie $\infty$-algebroid becomes a universal $Q$-manifold. We recover \cite{LLS} the notion of universal $Q$-manifold $(E,Q)$ of a singular foliation. We prove that $\mathrm{ad}_Q$ has\footnote{Which squares to zero on vector fields on $E$.} no longitudinal cohomology.
\end{tcolorbox}

\chapter{Isotropy Lie algebras of a singular foliation}\label{chap:isotropy}
In this chapter, we look at an extension of the Androulidakis and Skandalis isotropy Lie algebra \cite{AZ} of a singular foliation  at a point.\\

Let us recall some definitions.

\begin{definition}
 Let $\mathfrak{F}$ be a singular foliation on a manifold $M$. Let $\mathcal I_m = \left\lbrace f \in
C^\infty (M )\mid  f (m) = 0\right\rbrace$.
 \begin{enumerate}
 \item The \emph{tangent space of $\mathfrak F$} at a point $m\in M$ is the vector space  $T_m\mathfrak{F}:=\{X|_m\mid X\in \mathfrak{F}\}$ whose elements are evaluations of the vector fields of $\mathfrak F$ at $m$. In other words, $T_m\mathfrak{F}$ is the image 
of the evaluation map $\mathrm{ev}_m\colon \mathfrak{F}\rightarrow T_mM,\; X\mapsto X|_m$. 
     \item
   The set $\left\{m\in M\,\middle|\, \frac{\mathfrak F}{\mathcal I_m\mathfrak F}\longrightarrow T_m\mathfrak{F} \;\text{is bijective}\right\}$ is open and dense in $M$ \cite{AndroulidakisIakovos}. It is the set of \emph{regular} points of $(M, \mathfrak F)$ and is denoted  by $M_{reg}$.
 \end{enumerate} 
 \end{definition}

For any point $m\in M$, the $\mathcal O$-module $\mathfrak F(m):= \left\lbrace X \in\mathfrak F \mid  X(m) =0\right\rbrace=\ker(\mathrm{ev}_m)$ is a Lie subalgebra of $\mathfrak{F}$. The evaluation map goes to quotient and induces the following exact sequence,
$$\xymatrix{0\ar[r]&\dfrac{\mathfrak{F}(m)}{\mathcal I_m\mathfrak F}\ar[r]&\dfrac{\mathfrak F}{\mathcal{I}_m\mathfrak{F}}\ar[r]&T_m\mathfrak{F}\ar[r]&0 .}$$
Since $\mathcal I_m\mathfrak F\subseteq \mathfrak F(m)$ is a Lie ideal,

\begin{definition}\label{def:isotropy}
     the quotient space $\mathfrak{g}_m = \dfrac{\mathfrak{F}(m)}{\mathcal I_m\mathfrak F}$ is a Lie algebra, which is called the \emph{isotropy Lie algebra of $\mathfrak{F}$ at $m$}. 
\end{definition}

\begin{remark}
The isotropy Lie algebras detect singular and regular points of $(M, \mathfrak F)$. It measures how far $\mathfrak F$ is of being regular in neighborhood of a point. This can be stated as follows (see Lemma 1.1 of \cite{AZ}). 
\end{remark}
\begin{proposition}\label{prop:isotropy-regualar}
For any point $m \in L$ of a leaf $L\subset M$, $\mathfrak g_m = \{0\}$ if and only if,  $L$ is a regular leaf. Also, the  $\mathfrak g_m$'s are all isomorphic as Lie algebras  while  $m\in L$.
\end{proposition}
For $\mathcal{A}$ a Lie-Rinehart algebra over $\mathcal{O}$, the same construction can be done for any maximal ideal $\mathfrak{m}$. Let $$\mathcal{A}|_\mathfrak{m}:=\{a\in \mathcal{A}\mid \rho(a)[\mathcal{O}]\subseteq \mathfrak{m}\}.$$The quotient $\frac{\mathcal{A}|_\mathfrak m}{\mathfrak m \mathcal{A}}$ is a Lie algebra over $\mathcal{O}/\mathfrak m$. For $\mathcal{O}=C^\infty(M)$ and $\mathfrak{m}=\mathcal{I}_m$ for a point $m\in M$, the following sequence
\begin{equation}
  \xymatrix{0\ar[r]&\frac{\mathcal{A}|_{\mathcal I_m}}{\mathcal I_m\mathcal{A}}\ar@{^{(}->}[r]&\frac{\mathcal{A}}{\mathcal I_m\mathcal{A}}\ar@{->>}[r]&T_m\rho(\mathcal{A})\ar[r]&0}  
\end{equation}
is exact, where $T_m\rho(\mathcal{A})\subset \mathrm{Hom}(\mathfrak{m}/\mathfrak{m}^2,\mathcal{O}/\mathfrak{m})$ is the image of $\mathcal{A}\rightarrow\mathrm{Hom}(\mathfrak{m}/\mathfrak{m}^2,\mathcal{O}/\mathfrak{m}),\, a\mapsto \overline{\rho(a)}$. Proposition \ref{prop:isotropy-regualar} is however subtle to extend to this setting.

\begin{proposition}
For a finitely generated Lie-Rinehart $\mathcal{A}$ over $C^\infty(M)$, there exists a Lie algebroid $(\mathcal{A},\rho_\mathcal{A}, \lb_A)$ such that $\mathcal{A}\simeq \Gamma(A)$ for some vector bundle $A\rightarrow M$ if and only if $\frac{\mathcal{A}}{\mathcal I_m\mathcal A}$, as a vector space, has constant  rank at all points. In that case $\frac{\mathcal{A}(m)}{\mathcal I_m\mathcal A}$ is $\ker(\rho_\mathcal{A}|_m)$ for all $m\in M$. In particular, it is the case in a neighborhood of every point where the dimension of $\frac{\mathcal{A}}{\mathcal I_m\mathcal{A}}$ is minimal.\\

\end{proposition}
\begin{proof}
This is a consequence of Nakayama Lemma \ref{Nakayama}. The germ $\mathcal{O}_m$ of functions on a neighborhood of $\mathcal{U}\subseteq M$ of $m$ is a local ring \cite{REQUEJO} with maximal ideal still denoted by $\mathcal{I}_m$. The tensor product $\mathcal{O}_m\otimes_\mathcal O \mathcal{A}$ is a module over $\mathcal{O}_m$. Let $r=\dim(\frac{\mathcal{A}}{\mathcal{I}_m\mathcal{A}})$. By Nakayama Lemma \ref{Nakayama} any basis $(\overline{e}_1,\ldots, \overline{e}_r)$ of $\frac{\mathcal{A}}{\mathcal{I}_m\mathcal{A}}$ lifts to minimal generators on $\mathcal{O}_m\otimes_\mathcal O \mathcal{A}$ in a neighborhood of $m$. Since $r=\dim(\frac{\mathcal{A}}{\mathcal {I}_{m'}\mathcal{A}})$ for all points $m'\in \mathcal{U}$, again Nakayama Lemma \ref{Nakayama} implies $({e}_1,\ldots, {e}_r)$ are minimal generators on a neighborhood of $m'$. Any representative $({e}_1,\ldots, {e}_r)$ of $(\overline{e}_1,\ldots, \overline{e}_r)$ are local sections of $\mathcal{A}$ in a neighborhood $\mathcal{U}$ of $m$ that span $\mathcal{A}$ on $\mathcal{U}$ so that $\mathcal{A}|_\mathcal{U}$ is a free module. This applies that $\mathcal{A}$ is a $C^\infty(\mathcal{U})$-module, which is projective. By Serre-Swan theorem, $\mathcal{A}=\Gamma(A)$ for some vector bundle $A$.
\end{proof}

\section{Specialization of a Lie $\infty$-algebroid at a point}\label{sec:evaluation-oid}

\begin{itemize}
    \item Let $(M,\E,Q)=\left(E_\bullet,\ell_\bullet, \rho\right)$ be a Lie $\infty$-algebroid  over a manifold $M$ with anchor $\rho$. For every point $m\in M$, 
the $k$-ary brackets restrict to the graded vector space $$ev(E,m):=\left(\bigoplus_{i\geq 2} {E_{-i}}_{|_m}\right)\oplus \ker (\rho_{m})$$
and equipped the latter with a Lie $\infty$-algebra structure that we denote by $Istropy_m(\E,Q)$. For every $k\geq 1$, the restriction goes as follows:
\begin{equation}
    \{x_1,\ldots,x_k\}_k:=\ell_k(s_1,\ldots,s_k)_{|_m}
\end{equation}
for all $x_1,\ldots,x_k\in ev(E,m) $ and $s_1,\ldots,s_k\in \Gamma(E)$ sections of $E$ such that ${s_i}(m)=x_i$ with $i=1,\ldots,k$. These brackets are well-defined. It is clear that for $k\neq 2$, since $\ell_k$ is linear over functions. But it is not immediate that the $2$-ary bracket is well-defined as well. Let us check that.

On one hand, the new brackets $\{\cdots\}_k$ have values in $ev(E,m)$ for degree reason, except may be for the $2$-ary bracket when applied to elements of degree $-1$ (i.e. elements of the kernel of $\rho_{m}$) but in that case it is in the kernel of $\rho_{m}$ since \begin{align*}
    \rho_{m}(\{x_1,x_2\}_2)&=    \rho_{m}(\ell_2(s_1,s_2)_{|_m})\\&=\rho(\ell_2(s_1,s_2))_{|_m}\\&=[\rho(s_1),\rho(s_2)]_{|_m}=0
\end{align*}
In the last line we have used the fact that the Lie bracket of two  vector fields that vanish at $m$ is a vector field that vanishes again at $m$.\\

On the other hand, the $2$-ary bracket $\{\cdot\,,\cdot\,\}_2$ is also well-defined when applied to elements of degree less or equal to $-2$, we need to verify when we take the bracket with at least an element of degree $-1$. Let  $(e^i_1,\ldots e^i_{\mathrm{rk}(E_{-i})})$ be a local trivialization of $E_{-i}$ on a neighborhood $\mathcal U$ of the point $m\in M$. For $x_1\in \ker(\rho_{m})$ and $x_2\in {E_{-i}}_{|_m}$ write $$\displaystyle{x_1=\sum_{k=1}^{{\mathrm{rk}(E_{-1})}}}\lambda_ke^{1}_k(m),\quad\displaystyle{x_2=\sum_{k=1}^{{\mathrm{rk}(E_{-i})}}}\mu_ke^{i}_k(m)$$ for some scalars $(\lambda_i)$ in $\mathbb{K}$. The scalars $(\lambda_k)$, $(\mu_k)$ extend to functions $(f_k)$,  $(g_k)$ on $\mathcal{U}$. Therefore, we have $$\{x_1,x_2\}_2=\ell_2(s_1,s_2)_{|_m}$$ with 
$$\displaystyle{s_1=\sum_{k=1}^{{\mathrm{rk}(E_{-1})}}}f_ke^{1}_k,\quad\displaystyle{s_2=\sum_{k=1}^{{\mathrm{rk}(E_{-i})}}}g_ke^{i}_k.$$ If $\widetilde{s}_2$ is another extension of $x_2$, then $(s_2-\widetilde{s}_2)(m)=0$ and this is equivalent to $(g_k-\widetilde{g}_k)(m)=0$ for  $k=1,\ldots, \mathrm{rk}(E_{-i})$. It follows that
\begin{align*}
    \ell_2(s_1,s_2-\widetilde{s}_2)_{|_m}&=\sum_{k=1}^{\mathrm{rk}(E_{-i})}\ell_2\left(s_1,(f_k-\widetilde{g}_k)e^{i}_k\right)_{|_m}\\&=\sum_{k=1}^{\mathrm{rk}(E_{-i})}\cancel{(f_k-\widetilde{g}_k)(m)}\ell_2\left(s_1,e^{i}_k\right)_{|_m}+\cancel{\rho(s_1)_{|_m}[f_k-\widetilde{g}_k]}e^{i}_k\\&=0.
\end{align*}

Hence, $$ev(E,m): \cdots \stackrel{\{\cdot\}_1=\ell_1|_m}{\longrightarrow} {E_{-3}}|_m\stackrel{\{\cdot\}_1=\ell_1|_m}{\longrightarrow}{E_{-2}}|_m\stackrel{\{\cdot\}_1=\ell_1|_m}{\longrightarrow}\ker(\rho_m)$$ comes equipped with a Lie $\infty$-algebra whose brackets are $\left(\{\cdots\}_k\right)_{k\geq 1}$.
\item Any Lie $\infty$-morphism of algebroids $\Phi\colon(M,\E',Q')\rightarrow (M,\E,Q)$ induces a  Lie $\infty$-algebra morphism $\Phi_{|_m}\colon S^\bullet(V'_{|m})\rightarrow S^\bullet(V_{|m})$ since it is $\mathcal{O}$-linear. 

\item 
We define the graded vector space \begin{equation}\label{eq:cohomology}
    \bigoplus_{i\geq 1}H^{-i}(E_{\bullet},m)
\end{equation}as the cohomology group of the complex $$\xymatrix{\cdots\ar[r]^{{\ell_1}_{|_m}}&{E_{-3}}_{|_m}\ar[r]^{{\ell_1}_{|_m}}&{E_{-2}}_{|_m}\ar[r]^{{\ell_1}_{|_m}}&\ker(\rho_{m})\ar[r]&0.}$$

One can check that when $(M,\E,Q)$ is universal for a singular foliation $\mathfrak F$, the graded space \eqref{eq:cohomology} does not depend on the underlying geometric resolution of $\mathfrak{F}$ and is denoted $H(\mathfrak{F},m)$.\\

This construction can be extended to any Lie $\infty$-algebroid over $\mathcal{O}$ for any maximal ideal $\mathcal{I}$. Let $(\E,\ell_\bullet,\rho)$ be a Lie $\infty$-alegrboid over $\mathcal{O}$. Define
\begin{equation}
     {\mathcal{E}_{-i}}|_\mathcal{I}= \left\{ \begin{array}{ll}\left\{e\in \E_{-1}\mid \rho(e)[\mathcal{O}]\subseteq \mathcal{I}\right\}& \text{for}\; i=1\\&\\\frac{\E_{-i}}{\mathcal{I}\E_{-i}}& \text{for}\; i\geq 1
      \end{array}\right.
\end{equation}
The construction is purely formal. Also, the cohomology of the complex
$$\xymatrix{\cdots\ar[r]^{{\ell_1|_\mathcal{I}}}&{\mathcal{E}_{-3}}|_\mathcal{I}\ar[r]^{\ell_1|_\mathcal{I}}&{\mathcal{E}_{-2}}|_\mathcal{I}\ar[r]^{\ell_1|_\mathcal{I}}&{\mathcal{E}_{-1}}|_\mathcal{I}}$$
denoted by $H^\bullet(\E,\mathcal{I})$ does not depend on the choice of a free resolution of the Lie-Rinehart algebra $\mathcal{A}$, we denote it by $H^\bullet(\E,\mathcal{I})$.
\end{itemize}

\subsubsection{The isotropy Lie $\infty$-algebra of a singular foliation}
We assume that $(M,\E,Q)$ is universal for $\mathfrak F$. Note that the Lie $\infty$-algberoid obtained by specialising at some point $m\in M$ does not induce directly a Lie $\infty$-algberoid on the graded space $H(\mathfrak F,m)$ but the $2$-ary bracket $\{\cdot\,,\cdot\,\}_2$ goes to quotient directly on elements of degree $-1$ i.e. to $H^{-1}(\mathfrak{F},m)$, because $$\{\dd^{(2)}_{m} (x_1),x_2\}_2=\dd^{(2)}_{m}(\{x_1,x_2\}_2)$$ for all $x_1\in {E_{-2}}_{|_m}$ and $x_2\in \ker(\rho_m)$. That endows $H^{-1}(\mathfrak{F},m)$ with Lie algebra structure. 
\begin{proposition}\label{prop:isotropy}
The Androulidakis and Skandalis isotropy Lie algebra $\mathfrak{g}_m=\frac{\mathfrak{F}(m)}{I_m\mathfrak{F}}$ of the singular foliation $\mathfrak{F}$ at a point $m\in M$, is isomorphic to  $H^{-1}(\mathfrak F,m)$ equipped with the induced Lie algebra structure.
\end{proposition}
\begin{proof}
For $m\in M$, we construct a Lie algebra  isomorphism $\zeta\colon\frac{\ker(\rho_m)}{\mathrm{im}(\dd^{(2)}_{m})}\rightarrow \mathfrak{g}_m$ as follows: For an element $u\in\ker(\rho_m)$, let $\widetilde{u}$ be an extension of $u$ to a local section on $E_{-1}$. By construction, one has $\rho(\widetilde{u})\in\mathfrak{F}(m)$. Let $\widetilde{\rho}_m$ be the surjective linear map defined by
\begin{align*}
    \widetilde{\rho}_m\colon\ker(\rho_m)\longrightarrow \mathfrak{g}_m,\, u\longmapsto [\rho(\widetilde{u})].
\end{align*}
 Since any other extension $\widetilde u$ for $u$ differs from the first one by a section in $\mathcal{I}_m\Gamma(E_{-1})$, the map $\widetilde{\rho}_m$ is well-defined. Surjectivity is due to the fact that every vector field of $\mathfrak{F}$ vanishing at $m\in M$ is of the form $\rho(e)$ with  $e$ a (local) section of $E_{-1}$ whose value at $m$ belongs to $\ker(\rho_m)$. In addition, it is not hard to see that $\widetilde{\rho}_m$ is a morphism of brackets.\\
 
 \noindent
 It remains to show that $\ker(\widetilde{\rho}_m)=\mathrm{im}(\dd^{(2)}_{m})$: let $u \in \ker(\widetilde{\rho}_m)\subset \ker({\rho}_m)$ and $\widetilde{u}$ be a local section of $E_{-1}$ that extends $u$. By definition of $u$, the class of  $\rho(\widetilde{u})$ is zero in $\mathfrak{g}_m$, therefore, there exists some functions $f_i\in \mathcal{I}_m$ and $X_i\in\mathfrak{F}, i=1,\ldots, k$,  local generators such that $\displaystyle{\rho(\widetilde{u})=\sum_{i=1}^kf_iX_i}$. This implies that,  $$\rho( \widetilde{u}-\sum_{i=1}^kf_ie_i)= 0.$$
where for $i = 1,\ldots,k$, $e_i$ is a (local) section of $E_{-1}$ whose image through $\rho$ is $X_i$. Since $(E_\bullet, \dd^{\bullet}, \rho)$
is a geometric resolution, there exists a (local) section  $q \in \Gamma(E_{-2})$ such that
\begin{equation}\label{eq:geo-iso}
    \widetilde{u}=\sum_{i=1}^kf_ie_i+ \dd^{(2)}q
\end{equation}
 By evaluating Equation \eqref{eq:geo-iso} at $m$, we find out that $u\in\mathrm{im}(\dd^{(2)}_{m})$. Conversely, for $v\in {E_{-2}}_{|_m}$, choose a (local) section $q$ of $E_{-2}$ through $v$. Therefore, $\dd^{(2)}q\in \ker\rho$, is a (local) extension of $\dd_m^{(2)}v\in \mathrm{im}(\dd^{(2)}_{m})$. The image of $\dd_m^{(2)}v$ through $\widetilde{\rho}_m$ is obviously zero. This proves that $\ker(\widetilde{\rho}_m)=\mathrm{im}(\dd^{(2)}_{m})$.
\end{proof}

However, if the underlying complex of $(M,\E,Q)$  is minimal at $m$ then, for  every $i\geq 2$, the
vector space $H^{-i}(\mathfrak{F}, m)$ is canonically isomorphic to ${E_{-i}}_{|_m}$. Also, $H^{-1}(\mathfrak{F}, m)$ is canonically isomorphic to $\ker (\rho_{m})$. 

\begin{definition}
Let $(M,\E,Q)$ be a universal Lie $\infty$-algebroid of a singular foliation $\mathfrak{F}$ whose underlying complex is minimal at $m$. Then, $H(\mathfrak{F},m)$ carries a Lie $\infty$-algebra structure given by $Istropy_m(\E,Q)$
called the \emph{isotropy Lie $\infty$-algebra of the singular foliation $\mathfrak F$ at $m$}.
\end{definition}

One can show that this definition is independent of any choices made in the construction.
\begin{remark}
By Proposition \ref{prop:isotropy}, the isotropy Lie algebra of the singular foliation $\mathfrak{F}$ at a point $m\in M$  in the sense of Androulidakis and Skandalis, is isomorphic to the degree minus one component $H^{-1}(\mathfrak F, m)\simeq\ker(\rho_{m})$ of the isotropy Lie $\infty$-algebra of $\mathfrak{F}$ at $m$.
\end{remark}
\begin{remark}
For an arbitrary Lie-Rinehart algebra $\mathcal A$ and a maximal ideal $\mathcal{I}$, it is still true that $(H^\bullet(\mathcal{A},\mathcal{I}), \dd=0)$ is quasi-isomorphic to $(\E|_\mathcal{I}, \ell_1|_\mathcal{I})$ and it is still true that the bracket of the universal Lie $\infty$-algrbroid over $\mathcal{O}$ constructed in Section \ref{sec:main} restricts to a $\infty$-algebra  structure on $(\E|_\mathcal{I},\ell_1|_\mathcal{I})$. By Homotopy transfer, this Lie $\infty$-algebra can be transferred to $(H^\bullet(\mathcal{A}, \mathcal{I}), \dd=0)$. It is not clear whether Proposition \ref{prop:isotropy} holds true or not.  
\end{remark}

\begin{lemma}
\label{lemma:M_reg}
Let $(M,\mathfrak{F})$ be a singular foliation. Let $(E,Q)$ be a universal  Lie $\infty$-algebroid of $\mathfrak{F}$ and let \begin{equation}
(E, \dd, \rho):\quad\cdots \stackrel{\dd^{(4)}} \longrightarrow E_{-3} \stackrel{\dd^{(3)}}{\longrightarrow} E_{-2} \stackrel{\dd^{(2)}}{\longrightarrow} E_{-1} \stackrel{\rho}{\longrightarrow} TM \end{equation}
be its linear part.
\begin{enumerate}
 \item For all $m\in M$, we have $\frac{\ker(\rho_m)}{\mathrm{im}(\dd^{(2)}_{m})}\simeq \mathfrak{g}_m$ as Lie algebras.
    \item The subset of regular points of $\mathfrak F$ in $M$ satisfies \begin{equation*}
    M_{\mathrm{reg}}=
    \{m\in M \mid \mathrm{rk}(\dd^{(2)}_m)=\dim(\ker\rho_m)\}.
    \end{equation*}
    
     It is open and
dense in $M$. 

\item The restriction of the foliation $\mathfrak{F}$ to $M_{\mathrm{reg}}$ is
the set of sections of a subbundle of TM, i.e., is a regular foliation.

\item If $(E,\dd,\rho)$ is of finite length, then the regular leaves have the same dimension.
\end{enumerate}

\end{lemma}

\begin{proof}
Item 1. is proved in Proposition \ref{prop:isotropy} (see also Proposition 4.14 \cite{LLS}). For items 2. and 3. we refer the reader to Proposition 1.5 of \cite{AndroulidakisIakovos}.  
\end{proof}

\section{A blow-up procedure for a singular foliation}\label{blow-up-procedure}

We present in this section an interpretation of a blow-up procedure invented by Mohsen in \cite{MohsenOmar}.\\


 \subsection{Grassmann bundle}\label{chart:rassmanian}
For $E$ a finite dimensional vector space, we denote by 
 $\mathrm{Gr}_{-\ell}(E)$ the set of all codimension $\ell\in \mathbb N$ vector subspaces in $E$. Let us recall a few facts on $\mathrm{Gr}_{-\ell}(E)$ see \cite{Smith-Karen-E,Harris-Joe,lakshmibai2015grassmannian}, 
 our main reference is \cite{Igor}:
 \begin{enumerate}
     \item It is a metric space for the distance $ d(V,V')=\lVert P_{V}-P_{V'}\rVert,$ where  $P_V$ stands for the orthogonal projection of $E$ onto $V\subset E$ with respect to an arbitrary metric on $E$.
     \item The topology does not depend on the metric and makes  $\mathrm{Gr}_{-\ell}(E)$ a manifold.
     \item It is moreover a compact manifold and an projective variety\footnote{For the notion of projective variety and notations, we refer the reader e.g. to the book \cite{Igor}.}.
 \end{enumerate}

 \noindent
 \textbf{Affine coordinates charts:}
 One of the ways of defining the standard affine coordinates on the Grassmanian $\mathrm{Gr}_{-\ell}(E)$ is as follows (Example 1.24 of \cite{Igor}): Fix a basis $e_1,\ldots,e_{d=\dim E}$ for $E$. Any element $V\in \mathrm{Gr}_{-\ell}(E)$, i.e. a vector subspace $V\subset E$ of codimension $\ell$, can be viewed as a $d\times (d-\ell)$ matrix $\mathfrak{M}_V=(C_1,\ldots, C_{d-\ell})$ whose columns are formed by linearly independent column vectors obtained by choosing a basis for $V$. The \emph{homogeneous coordinates} of $V$ in $\mathrm{Gr}_{-\ell}(E)$ are the components of the $n\times (d-\ell)$ matrix $\mathfrak{M}_V$. Any other choice of basis for $V$ gives another maximal rank  matrix $\mathfrak{M}_V'$ and an invertible $(d-\ell)\times(d-\ell)$-matrix $P\in GL(d-\ell,\mathbb{K})$ such that $\mathfrak{M}_V=\mathfrak{M}_V'\circ P$. In particular, since  $\mathfrak{M}_V$ has full  rank, there exists a family of integers ${\displaystyle 1\leq i_{1}<\cdots <i_{d-\ell}\leq d}$, such that $\mathfrak{M}_V$ is equivalent to a matrix $\mathfrak{M}'_V$ whose submatrix made of the rows $i_{1},\cdots ,i_{d-\ell}$ is the identity matrix.

 For example, if the first $d-\ell$ rows of $\mathfrak{M}_V$ are linearly independent, then the matrix is equivalent to the matrix $$\begin{pmatrix}
 I_{d-\ell}\\ A
 \end{pmatrix}$$ where $A=(a_{ij})$ is a $\ell\times (d-\ell)$-matrix. In that case, $V$  admits a basis of the form
\begin{equation}
v_j := e_j+\sum_{k=1}^\ell a_{kj}e_k,\quad j=1,\ldots,d-\ell.
\end{equation}$V$ is completely determined by $A$.\\


 One define an atlas on $Gr_{-l}(E)$ as follows: Consider the map
\begin{align*}
    \psi_{1,\ldots,d}\colon M_{l,d-l}(\mathbb{K})&\longrightarrow M_{d,d-l}(\mathbb{K})\\A\;&\longmapsto \begin{pmatrix}
 I_{d-\ell}\\ A
 \end{pmatrix}.
\end{align*}
For a permutation $\sigma\in \mathfrak{S}_d$, we define the map $\psi_{\sigma(1),\ldots,\sigma(d)}$ that associates  every $A\in M_{l,d-l}(\mathbb{K})$ the matrix in $M_{d,d-l}(\mathbb{K})$ that permutes the lines of $\begin{pmatrix}
 I_{d-\ell}\\ A
 \end{pmatrix}$ w.r.t $\sigma$,  that is to say $$\psi_{\sigma(1),\ldots,\sigma(d)}(A):=P(\sigma)\circ\psi_{1,\ldots,d}(A)$$where $P(\sigma)$ is the permutation matrix associated to $\sigma$.\\

\noindent
\underline{\textbf{Statement}}: For every ordered integers ${\displaystyle 1\leq i_{1}<\cdots <i_{d-\ell}\leq d}$, the coordinates chart on $\mathrm{Gr}_{-\ell}(E)$ is the open set  $\mathcal U_{i_{1},\cdots ,i_{d-\ell}}\subset \mathrm{Gr}_{-\ell}(E)$ consisting of all sub-vector space of $E$ such that for every basis the submatrix which is made  of the rows $i_{1},\cdots ,i_{d-\ell}$ is invertible, and the coordinate map is $\psi_{\sigma(1),\ldots,\sigma(d)}$ with $\sigma$ is a permutation sending $1,\ldots,d-\ell$ on $i_{1},\cdots ,i_{d-\ell}$.\\


 \noindent
 \textbf{Grassmann bundle:} For $ E_{-1} \to M$ a vector bundle of rank $d$, the disjoint union:
  $$ \mathrm{Gr}_{-\ell}({E_{-1}}):= \coprod_{m\in M} \mathrm{Gr}_{-\ell}({E_{-1}}|_{m})  $$
  comes equipped with a natural manifold structure and 
 $$\Pi\colon \mathrm{Gr}_{-\ell}({E_{-1}})\longrightarrow M$$
 is a fibration. It is called \emph{the Grassmann  $(d-\ell)$-plane bundle}. To fix some notations, the fiber at $m\in M$ is  $$\Pi^{-1} ( m )=\left\{(V,m)\,\middle|\, V\in \mathrm{Gr}_{-\ell} ( E_{-1}|_m )\right\}.$$
 For every open subset $\mathcal U\subset M$ on which $E_{-1}$ is trivial, $\Pi^{-1}(\mathcal U)\simeq \mathrm{Gr}_{-\ell}(\mathbb{R}^{d})\times \mathcal{U}.$
 
\subsection{A blow-up procedure}\label{sec:blow-up-procedure}
\noindent
\textbf{Settings:}  
 In what follows, we are given a foliated manifold $(M,\mathfrak F)$ with $M$ connected. 
 We assume that a geometric resolution of finite length exists.
 Under these assumptions,
 all the regular leaves have the same dimension. Let  $M_{reg}$ the set of  the regular points of $(M,\mathfrak F)$. Denote by $\ell$ the common dimension of the regular leaves. 
 
 \begin{remark}
 For most of the present discussion, the whole geometric resolution is not needed: is it sufficient to assume that there exists vector bundles $ E_{-1}, E_{-2}$ and that all regular leaves have the same dimension.
 \end{remark} 
 
 Let $(E_\bullet, \ell_\bullet, \rho)$ be a universal Lie $\infty$-algebroid of $\mathfrak F$. Consider the underlying geometric resolution 
\begin{equation}
(E, \dd, \rho):\quad\cdots \stackrel{\ell_1=\dd^{(4)}} \longrightarrow E_{-3} \stackrel{\ell_1=\dd^{(3)}}{\longrightarrow} E_{-2} \stackrel{\ell_1=\dd^{(2)}}{\longrightarrow} E_{-1} \stackrel{\rho}{\longrightarrow} TM.\end{equation}Notice that for every point $m\in M_{reg}$, $$\dim( \dd^{(2)}_m)=\dim \ker(\rho_m)=\mathrm{rk}(E_{-1})-\ell.$$ Therefore, there exists a natural section of $\Pi$ on $ M_{\mathrm{reg}}$ defined by: \begin{equation}
    \sigma\colon  M_{\mathrm{reg}}\longrightarrow \mathrm{Gr}_{-\ell}({E_{-1}}),\,  m\longmapsto \mathrm{im}(\dd^{(2)}_m)
\end{equation}
We consider $\widetilde{M}:= \overline{\sigma(M_{\mathrm{reg}})}$ the closure of the graph of $\sigma$ in ${\mathrm{Gr}_{-\ell}(E_{-1})}$, and denote by
 $\pi$ the restriction of $\Pi$ to  $\widetilde{M}$. \\

 Recall from Section \ref{sec:evaluation-oid} that $$ev(E,m): \cdots \stackrel{}{\longrightarrow} {E_{-3}}|_m\stackrel{\dd^{(3)}_m}{\longrightarrow}{E_{-2}}|_m\stackrel{\dd^{(2)}_m}{\longrightarrow}\ker(\rho_m)$$ comes equipped with a Lie $\infty$-algebra  $\left(\{\cdots\}_k\right)_{k\geq 1}$ whose unary bracket is ${\dd^{(\bullet)}_m}$.
\begin{theorem}\label{thm:grass}
The projection map  \begin{equation}
    \pi\colon \widetilde M\subset {\mathrm{Gr}_{-\ell}(E_{-1})}\longrightarrow M , \; (V, m)\longrightarrow m
\end{equation}
fulfills the following properties

\begin{enumerate}

    \item For each point $m\in M$, the set
    \begin{equation*}
        \pi^{-1}(m)=\left\{V\subset{E_{-1}}_{|_m}\;\middle|\; \exists\, (m_n)\in M_{\mathrm{reg}}^{\mathbb{N}},\, \text{such that},\,\; \mathrm{im}(\dd^{(2)}_{m_n})\underset{n \to +\infty}{\longrightarrow} V\; as \; m_n\underset{n \to +\infty}{\longrightarrow}m\right\}
    \end{equation*}
    is non-empty.
    
\item For all $m\in M$, and $V\in \pi^{-1}(m)$ one has, 

\begin{enumerate}
    \item $\mathrm{im}(\dd^{(2)}_{m})\subseteq V \subseteq \ker(\rho_m).$
    \item The $2$-ary bracket $\{\cdot\,,\cdot\,\}_2$ on $\ker \rho_m$ restricts to $V$ and the image of  $V$ in  $\frac{\ker(\rho_m)}{\mathrm{im}(\dd^{(2)}_{m})}\simeq \mathfrak{g}_m$, is a Lie subalgebra of codimension $\ell-\dim(L_m)$, where $L_m$ is the leaf through $m$.
\end{enumerate}
    \item For all $m\in M_{\mathrm{reg}}$,  $\pi^{-1}(m)=\ker(\rho_m)=\mathrm{im}(\dd^{(2)}_m)$ is reduced to a point in $\mathrm{Gr}_{-\ell}(\mathbb{R}^{\mathrm{rk}({E_{-1})}})$.
    \item $\widetilde M$ does not depend on the choice of a geometric resolution.
    \item The projection $\pi\colon \widetilde M\rightarrow M$ is proper and onto.
\end{enumerate}
\end{theorem}

\begin{proof}
Let us prove item 1 for $m\in M$. By compactness of the Grassmanian manifold $\mathrm{Gr}_{-\ell}({E_{-1}}_{|_m})\simeq \mathrm{Gr}_{-\ell}(\mathbb{R}^{\mathrm{rk}({E_{-1})}})$, out of any sequence in $M_{\mathrm{reg}}^\mathbb{N}$ that converges to $m$, we can extract a sequence $n\mapsto \kappa\left(\mathrm{im}(\dd^{(2)}_{m_n})\right)\in \mathrm{Gr}_{-\ell}({E_{-1}}_{|_m})$ that has a limit, where $\kappa$ is a local trivialization of the vector bundle $E_{-1}$ which identifies the fibers ${E_{-1}}_{|_m}$ and  ${E_{-1}}_{|_{m_n}}$. This proves item 1.\\

Let us show item $2.(a)$: Let $V\in \pi^{-1}(m)$ and $(m_n)\in M_{\mathrm{reg}}^{\mathbb{N}}$ such that $m_n\underset{n \to +\infty}{\longrightarrow}m$\;  and  \; $\mathrm{im}(\dd^{(2)}_{m_n})\underset{n \to +\infty}{\longrightarrow} V$. Let $v\in\mathrm{im}(\dd^{(2)}_{m})$. We have $v=\dd^{(2)}_{m}u$ for some $u\in {E_{-2}}_{|_m}$. Choose a (local) section $\widetilde{u}$ of $E_{-2}$ through $u$. It implies that $\dd^{(2)}_{m_n}\widetilde{u}(m_n)\underset{n \to +\infty}{\longrightarrow}\dd^{(2)}_{m}u$, hence $\dd^{(2)}_{m}u\in V$. Thus, $\mathrm{im}(\dd^{(2)}_{m})\subseteq V$. For any element $v\in V$, there exists a sequence $v_n\in \ker(\rho_{m_n})=\mathrm{im}(\dd^{(2)}_{m_n}),\, n\in \mathbb{N}$ that converges to $v$. In particular, $\rho_{m_n}(v_n)=0$ for all $n$. Hence, by continuity, one has $v\in \ker(\rho_{m})$.\\

\noindent
To prove item $(b)$, choose a local frame $e_1,\ldots,e_r,\ldots e_{\ell+r}$ of $E_{-1}$ such that $e_1(m),\ldots,e_r(m)$ is an orthogonal basis of $V$ for an arbitrary Hermitian structure on $E_{-1}$. For $i,j\in\{1,\ldots, l+r\}$, let $(c_{ij}^k)\in \mathcal{O}(\mathcal{U})$ be a family of functions over $\mathcal{U}$ such that $$\ell_2(e_i, e_j)=\sum_{k=1}^{l+r}c_{ij}^ke_k.$$ In particular, for every $\{i,j\}\subset\{1,\ldots, l+r\}$,\; $$\displaystyle{\left\{e_i(m), e_j(m)\right\}_2=\sum_{k=1}^{l+r}c_{ij}^k(m)e_k(m)}.$$ Let $v_1,v_2\in V$ with $\displaystyle{v_1=\sum_{i=1}^r\alpha^ie_i(m)}$ and  $\displaystyle{v_2=\sum_{i=j}^r\beta^ie_i(m)}$. There exists sequences $$v_1^n=\displaystyle{\sum_{i=1}^{r+l}\alpha_n^ie_i(m)}\underset{n \to +\infty}{\longrightarrow} v_1\quad  \text{and}\quad v_2^n=\displaystyle{\sum_{i=1}^{r+l}\beta_n^ie_i(m)}\underset{n \to +\infty}{\longrightarrow} v_2$$ with $v_1^n,v_2^n\in \ker\rho_{x_n}$, for all $n\in \mathbb{N}$. In particular, the sequences $(\alpha^i_n),(\beta^i_n)\in \mathbb{K}^\mathbb{N}$ with $i\in \{1,\ldots,r+l\}$ satisfy  $\alpha^i_n\underset{n \to +\infty}{\longrightarrow} \alpha^i,\quad \beta_n^i\underset{n \to +\infty}{\longrightarrow} \beta^i$ \quad($\alpha^i=\beta^i=0$ for $i\geq r+1$). For every $n\in\mathbb{N}$ we have,
\begin{equation}
    \sum_{i,j,k=1}^{r+l}\alpha^i_n\beta^j_nc_{ij}^k(m_n)e_k(m_n)=\{v_1^n,v_2^n\}_{2}\in \ker\rho_{x_n}=\mathrm{im}(\dd^{(2)}_{m_n}).
\end{equation}
Since $$\sum_{i,j,k=1}^{r+l}\alpha^i_n\beta^j_nc_{ij}^k(m_n)e_k(m_n)\underset{n \to+\infty}{\longrightarrow}\sum_{i,j,k=1}^{r+l}\alpha^i\beta^jc_{ij}^k(m)e_k(m)=\{v_1,v_2\}_{2},$$one has, $\{v_1,v_2\}_{2}\in V$.\\

Item 3. is a consequence of $2.(a)$, since $m\in M_{\mathrm{reg}}$ if and only if $\ker(\rho_m)=\mathrm{im}(\dd^{(2)}_m)$ (by  item 1. of Lemma \ref{lemma:M_reg}). Item 4. follows from the existence of homotopy equivalence between any two geometric resolutions. Item 5. follows from the fact that the projection $\Pi$ is with compact fibers. This concludes the proof.

\end{proof}

The corollary below is a direct consequence of item 2.(b) of Theorem \ref{thm:grass}.
\begin{corollary}
There is a natural inclusion
\begin{equation}
    \widetilde M\hookrightarrow \coprod_{m\in M} Gr_{\ell-\dim (L_m)}(\mathfrak{g}_{m}).  
\end{equation}
\end{corollary}

\begin{proof}
Let $m\in M$. 
Since elements  $V\in \pi^{-1}(m)$ satisfy $\mathrm{im}(\dd^{(2)}_{m})\subseteq V \subseteq \ker(\rho_m)$, they correspond injectively to a (unique) sub-vector space of codimension $\ell-\dim L_m$ in $\mathfrak g_m$. In particular, this implies $\pi^{-1}(m)\hookrightarrow Gr_{\ell-\dim (L_m)}(\mathfrak{g}_{m})$.

\end{proof}


\begin{lemma}\label{lemma:locally-affine}

For every $m\in M$, the set $K_m=\{V\in \mathrm{Gr}_{-\ell}({E_{-1}}_{|_m})\mid V\subseteq\ker\rho_m\}\subset \mathrm{Gr}_{-\ell}({E_{-1}}_{|_m})$ is locally an affine variety.
\end{lemma}
\begin{proof}
Let $e_1,\ldots,e_d$ be a local frame of $E_{-1}$ on an open subset $\mathcal{U}\subset M$. 
One has, \begin{equation}\label{eq:poly-generators}
    \displaystyle{\rho({e}_i)=\sum_{k=1}^{\dim M}f_i^k\frac{\partial}{\partial x_k}\in \mathfrak{F}},
\end{equation} for some local functions $f_i^k~\in~ C^\infty(\mathcal{U})$. Without any lost of generality, consider for example the standard coordinates chart $\mathcal U_{1,\ldots, d-\ell}$ for the grassmanian $\mathrm{Gr}_{-\ell}({E_{-1}}_{|_m})$. Let $V\in \mathcal U_{1,\ldots, d-\ell}$\; and let  $(a_{ij})$ be the homogeneous coordinates of $V$. Define the sections

\begin{equation}
\widetilde{v}_j:={e}_j+\sum_{k=1}^\ell a_{kj}{e}_k,\quad j=1,\ldots,d-\ell
\end{equation} By construction, the ${\widetilde{v}_j }$'s, evaluated at $m$, form a basis for $V$. We have, \begin{align*}
    \rho(\widetilde{v}_j)=\sum_{k=1}^{\dim M}\left(f_j^k+\sum_{s=1}^\ell a_{sj}f_s^k\right)\frac{\partial}{\partial x_k}.
\end{align*}$V\subseteq\ker\rho_m$ if and only if 
    
\begin{equation}\label{eq:affino1}
    \begin{array}{cc}
  \displaystyle{f_j^k(m)+\sum_{s=1}^\ell a_{sj}f_s^k(m)=0}, &j=1,\ldots, d-\ell\\
     &\qquad k=1,\ldots, \dim M. 
\end{array}
\end{equation}
Therefore, $K_m$ is defined by the polynomial Equations \eqref{eq:affino1}.    
\end{proof}
\begin{theorem}\label{thm:locally-affine}
If $\mathfrak F$ is a polynomial singular foliation on $M\in \left\{\mathbb{C}^N,\mathbb{R}^N\right\}$, then $\widetilde M$ is a locally affine variety.
\end{theorem}
\begin{proof}
We can choose a polynomial geometric resolution. Denote by $x=(x_1,\ldots, x_N)$ the local coordinates on $\mathcal{W}\subseteq M$. We are using notations of Lemma \ref{lemma:locally-affine}. In Equation \eqref{eq:poly-generators}, the functions $f^k_i$ are polynomial in $(x_1,\ldots, x_N)$. By item 1, Theorem \ref{thm:grass}, every element $V$ of  $\widetilde M=\bigcup_{x\in M}\pi^{-1}(x)$ is obtained as a limit $$ \mathrm{im}(\dd^{(2)}_{x_n})\underset{n \to +\infty}{\longrightarrow} V\in \pi^{-1}(x)$$   with $(x_n)\in M_{\mathrm{reg}}^{\mathbb{N}}$,\,\text{such that}, $x_n\underset{n \to +\infty}{\longrightarrow}x$. Fix $n\in \mathbb{N}_0$. In the notations of Lemma \ref{lemma:locally-affine}, take $V=\mathrm{im}(\dd^{(2)}_{x_n})$. Thus, in the coordinate chart $\mathcal{U}_{1,\ldots,d-\ell}$ (see Section \ref{chart:rassmanian}), the coordinates $(a_{ij}^n)$ of $\mathrm{im}(\dd^{(2)}_{x_n})$ satisfy the polynomial equations

\begin{equation}\label{eq:affino12}
    \begin{array}{cc}
  \displaystyle{f_j^k(x_n)+\sum_{s=1}^\ell a_{sj}^nf_s^k(x_n)=0}, &j=1,\ldots, d-\ell\\
     &\qquad k=1,\ldots, N. 
\end{array}
\end{equation}
One has, $$f_j^k(x_n)+\sum_{s=1}^\ell a_{sj}^nf_s^k(x_n)=0 \underset{n \to +\infty}{\longrightarrow} f_j^k(x)+\sum_{s=1}^\ell a_{sj}f_s^k(x)=0,$$

where for all $s,j$,\,$a_{sj}^n\underset{n \to +\infty}{\longrightarrow} a_{sj}$. Therefore, $\widetilde M$ is given in local coordinates $\mathcal{W}\times \mathcal{U}_{1,\ldots,d-\ell}$ by elements that satisfy \begin{enumerate}
    \item Equation \eqref{eq:affino1} 
    \item and that are limit of elements of \eqref{eq:affino1} in nearby regular points, elements which are unique. Hence, it is, on this affine variety, the irreducible components of \eqref{eq:affino1} that projects onto $M$.
    
\end{enumerate}  
\end{proof}
\hspace{1cm}

\noindent
\textbf{The pull-back of $\mathfrak F$ to $\widetilde{M}$}: Let $X\in \mathfrak{F}$. There exists a section $\upsilon$ of the vector bundle $p\colon E_{-1}\rightarrow M$ such that $\rho(\upsilon)=X$. Consider the linear vector field  $\widetilde X\in \mathfrak{X}(E_{-1})$ defined as follows
\begin{align*}
  \widetilde X[p^*f]&:=p^*(\rho(\upsilon)[f]),\;\forall\; f\in C^\infty(M),\\&\\ \langle\widetilde X[\alpha],e\rangle&:=\rho(\upsilon)[\langle \alpha ,e\rangle]-\langle \alpha, \ell_{2}(\upsilon,e)\rangle,\;\forall\; \alpha\in \Gamma(E^*),\;e\in\Gamma(E_{-1}).\end{align*}
Notice that $\widetilde X$ depends on the choice of the almost Lie algebroid bracket $\ell_2$. The following items hold (see \cite{LLL}, Lemma 2.1.19. p. 61).
 \begin{enumerate}
     \item The flow  $\phi^{\widetilde X}_t\colon E_{-1}\rightarrow {E_{-1}}$ of $\widetilde X$ when it is defined, is a vector bundle isomorphism over $\phi^X_t$.
     
     \item The diagram below commutes,
     \begin{equation}
         \xymatrix{E_{-1}\ar[r]^{\phi^{\widetilde X}}\ar[d]_\rho&E_{-1}\ar[d]^\rho\\TM\ar[r]_{d\phi^X_t}&TM}
     \end{equation}
     where $\phi^X_t$ is the flow of $X$.
 \end{enumerate}
In particular, $\phi^{\widetilde X}_t$ preserve the grassmanian $\mathrm{Gr}_{-\ell}(E_{-1})$.

\begin{proposition}\label{prop:lift-on-blowup}
For every $X\in \mathfrak F$, choose $\upsilon\in \Gamma(E_{-1})$ such that $\rho(\upsilon)=X$. The vector field  $\widetilde X$ induces a vector field $\widetilde{\widetilde X}$ on $\mathrm{Gr}_{-\ell}(E_{-1})$ such that
\begin{enumerate}
    \item $\widetilde{\widetilde X}$ is tangent to $\widetilde M$.
    \item $\widetilde{\widetilde X}$ projects onto $X$.
\end{enumerate}
\end{proposition}

\begin{proof}
Since  for  every $x\in M$, $\phi^{\widetilde X}_t|_x$ is an isomorphism of ${E_{-1}}|_x$, therefore $\phi^{\widetilde X}_t|_x$ preserves $\Pi^{-1}(x)=\mathrm{Gr}_{-\ell}({E_{-1}}|_x)$. We define $\widetilde{\widetilde X}$ for $(V,x)\in \mathrm{Gr}_{-\ell}(E_{-1})$ by \begin{equation}
    \label{eq:pull-back-gr}\widetilde{\widetilde X}(V):=\frac{d}{dt}|_{t=0}\left(\phi^{\widetilde X}_t|_x\right)|_V\in T_V\mathrm{Gr}_{-\ell}({E_{-1}}|_x).
\end{equation}This shows item 1.\\

Let us show item 1, $\phi^{\widetilde X}_t$ preserves $\widetilde M$: to see this take $(V,x)\in \widetilde M$,  let $x_n\underset{n \to +\infty}{\longrightarrow}x$ be such that $\mathrm{imd}^{(2)}_{x_n}\underset{n \to +\infty}{\longrightarrow} V$ with $(x_n)\subset{M_{reg}}$. For every $n\in\mathbb{N}_0$, one has  $$\phi_t^{\widetilde X}|_{x_n}\left(\mathrm{imd}^{(2)}_{x_n}\right)=\mathrm{imd}^{(2)}_{\phi^X_t(x_n)},$$ since $\rho\circ\phi_t^{\widetilde X}=d\phi_t^{X}\circ\rho$. Thus, 
\begin{align*}
    \phi_t^{\widetilde X}|_x(V)&=\lim_{n\rightarrow +\infty}\phi_t^{\widetilde X}|_{x_n}\left(\mathrm{imd}^{(2)}_{x_n}\right)\\&=\lim_{n\rightarrow +\infty}\left(\mathrm{imd}^{(2)}_{\phi^X_t(x_n)}\right)\in \pi^{-1}\left(\phi^X_t(x)\right).
\end{align*}
By consequence, for $V\in \widetilde M$, we have  $\widetilde{\widetilde X}(V)\in T_V\widetilde M$, by Equation \eqref{eq:pull-back-gr}.
\end{proof}

\begin{corollary}
$\widetilde{\widetilde X}$ does not depend on the choice of the almost bracket $\ell_2$.
\end{corollary}
\begin{proof}
This is an obvious consequence of item 3. in Proposition \ref{prop:lift-on-blowup}, since $\pi\colon \widetilde M\longrightarrow M$ is invertible on an open (dense) subset.
\end{proof}
Denote by $\widetilde{\mathfrak{F}}$ the module generated by the  $\widetilde{\widetilde X}$'s, with  $X\in \mathfrak{F}$. For $\mathfrak{F}$ polynomial, $\widetilde{\mathfrak{F}}$ is a singular foliation on the locally affine variety $\widetilde M$, because $\widetilde X$ is still polynomial.

\begin{corollary}
If $\mathfrak{F}$ is a polynomial singular foliation, then it is lifted to  a projective singular foliation  $\widetilde{\mathfrak{F}}$ on the locally affine variety $\widetilde M$.
\end{corollary}

\noindent
\textbf{Generalization.}
The same construction would apply if, instead of considering $\mathrm{im}(\dd^{(2)})\subseteq E_{-1}$, we consider $\mathrm{im}(\dd^{(i+1)})\subseteq E_{-i}$ or $\mathrm{im}(\rho)\subseteq TM$ for $i\geq 1$. Denote by $\widetilde M_i$ the transformation that we would obtain by such a construction. A lift $\widetilde{\mathfrak{F}}_i$ of $\mathfrak{F}$ can be defined exactly in the same manner by considering a lift $\widetilde X$ of $X\in \mathfrak{F}$ on $E_{-i}$ (or $TM$) associated to $\ell_2\colon \Gamma(E_{-1})\times \Gamma(E_{-i})\rightarrow \Gamma(E_{-i})$ or to $[X, \cdot\,]$.\\

\noindent
More precisely, consider for $i\geq 0$, the Grassmann bundle $Gr_{-\ell^i}(E_{-i})$, with $E_0:=TM$, where $\ell^i:=\mathrm{rk}(\dd^{i+1}_m)$ (with the understanding that $\dd^{(1)}=\rho$ for $i=0$) at regular points. These bundles are manifolds that project to $M$ through a proper map. Over $M_{reg}$, there exists a unique 
$V\in Gr_{-\ell^i}(E_{-i})$ such that $ \dd^{(i)}_m(V)=0$. Theses define maps
$$\sigma_i\colon M_{reg}\longhookrightarrow Gr_{-\ell^i}(E_{-i}).$$
We define $\widetilde{M}_i:=\overline{\sigma_i(M_{reg})}$. This is not a manifold in general. However, it is a locally affine variety if $\mathfrak{F}$ is a polynomial singular foliation on $\mathbb{R}^N$ or, $\mathbb{C}^N$ as in Theorem \ref{thm:locally-affine}. Moreover,

\begin{enumerate}
    \item independent of the choice of a geometric resolution, as in Theorem \ref{thm:grass},
    \item the projection $\widetilde M_i\rightarrow M$ is proper and onto, as in Theorem \ref{thm:grass},
    \item $\mathfrak{F}$ lifts to a singular foliation on $\widetilde M_i$, and it is one-to-one to $M_{reg}$, as in Proposition \ref{prop:lift-on-blowup}.
\end{enumerate}
For $i=0$ the same conclusion holds for \begin{align*}
    \sigma_0\colon M&\longrightarrow Gr_{\ell^0}(TM)\\&m\longmapsto T_m\mathfrak F.
\end{align*}

Let us give some examples.
\begin{example}Notice that $\mathrm{Gr}_{-\ell}(\mathbb{C}^\ell)=\{pt\}$, so that if $\dd^{(i)}$ is into on an open dense subset, the construction degenerates, in this case $\mathrm{Gr}_{-\ell}(\mathbb{C}^\ell)\times M\simeq M$
\begin{enumerate}
    \item If $\mathfrak{F}$ is a projective singular foliation, then $\widetilde M_1\simeq M$: because $\Gamma(E_{-1})\simeq \mathfrak{F}$ and $E_{-1}\stackrel{\rho}{\rightarrow}TM$ is injective on the open dense subset $M_{reg}$, i.e. $\ell^0=\mathrm{rk}(E_{-1})$. Hence $Gr_{-\ell^0}(E_{-1})\simeq M$.
    \item If $\mathfrak{F}$ admits open dense orbits, ${\widetilde M}_0\simeq M$, since $Gr_{-\dim M}(TM)\simeq M$.
    
    \item Let $W\subseteq M=\mathbb{C}^d$ be an affine variety generated by (independent functions)  $\varphi,\psi\in \mathbb{C}[x_1,\ldots,x_d]$, i.e. $\mathcal{I}_W=\langle\varphi,\psi\rangle$. Let $\mathfrak{F}_W$ be the singular foliation of (polynomial) vector fields vanishing on $W$. At regular points $x\in M_{reg}$,  $T_x\mathfrak{F}_W=d$, hence
    
    \begin{enumerate}
        \item $\widetilde M_0\simeq M$.
        
        \item 
        
        Also, \begin{itemize}
            \item $E_{-1}=(\mathbb{C}^d\oplus \mathbb{C}^d)\times M,\; \rho\colon e_i\oplus 0\mapsto \varphi\frac{\partial}{\partial x_i}$ and $0\oplus e_i\mapsto -\psi\frac{\partial}{\partial x_i}$, for all $i=1, \ldots, d$.
 \item $E_{-2}=\mathbb{C}^d\times M,$ and $\dd^{(2)}\colon e_i\mapsto \frac{\psi}{K}e_i\oplus \frac{\varphi}{K}e_i$ with, $K=gcd(\varphi, \psi)$$$ \mathrm{im}(\dd^{(2)}_x)=\ker\rho_x=\left\langle\frac{\psi(x)}{K(x)}u\oplus \frac{\varphi(x)}{K(x)}u, \;\text{with}\; u\in \mathbb{C}^d \right\rangle
 $$ For a  convergent sequence $(y_n)$ in $M\setminus W$, $ (\ker \rho_{y_n})$ converges if and only if $[\frac{\psi(y_n)}{K(y_n)}\colon \frac{\varphi(y_n)}{K(y_n)}]$ converges in $\mathbb{P}^1(\mathbb{C})$. In that case, $\widetilde M_1$ is the closure of the graph $\{(y, [\psi(y)\colon \varphi(y)]), y\in M\setminus W\}$, which is the blow up of $\mathbb{C}^d$ along $W$.
 \item $\widetilde M_2\simeq M$ since $\dd_x^{(2)}$ is injective at regular points, so that $Gr_{-d}(E_{-2})\simeq M$.
        \end{itemize}        
        
\end{enumerate}

    \item Let $W$ be the affine variety defined by $\varphi\in\mathbb{C}[x_1,\ldots,x_d]$. Consider the singular foliation $\mathfrak{F}_\varphi=\{X\in \mathfrak{X}(\mathbb{C}^d)\mid X[\varphi]=0\}$.  
    For every $y\in \mathbb C^d$, $(T_y\mathfrak{F}_\varphi)^\bot= \langle \nabla_y\varphi\rangle$. For convergent sequence $y_n\underset{n \to +\infty}{\longrightarrow}y\in W$. The sequence $\mathrm{im}(\rho_{y_n})$ converges if and only if $\nabla_{y_n}\varphi$ converges in $Gr_{-(d-1)}(\mathbb{C}^d)$, that is $\left[\frac{\partial \varphi}{\partial x_1}(y_n)\colon \cdots \colon\frac{\partial \varphi}{\partial x_d}(y_n)\right]$ converges in the projective space $\mathbb{P}^{d-1}(\mathbb{C})$. Therefore, $\widetilde M_0$ is the closure of the graph of the map, $y\mapsto (y, \left[\frac{\partial \varphi}{\partial x_1}(y)\colon \cdots \colon\frac{\partial \varphi}{\partial x_d}(y)\right])$ which is the  blow up of $\mathbb{C}^d$ along the singular locus of $W$.\end{enumerate}
\end{example}
\vspace{0.3cm}
Let us conclude this chapter by an open question. Can we desingularize a singular affine variety $W\subseteq \mathbb C^d$ by applying the constructions above to the singular foliation $\mathfrak{F}=\mathfrak X(W)$ of vector fields tangent to $W$? We should then understand $\widetilde W_0$ as the Nash modification of $W$ \cite{D.T}. The meaning of $\widetilde W_1$ is unclear. We would like to relate the $\widetilde W_i$'s and the $\widetilde {\mathfrak{F}}_i$ together and go further in the universal Lie $\infty$-algebroid, e.g. to understand the role of the $3$-ary bracket in this procedure.

\vspace{2cm}

\begin{tcolorbox}[colback=gray!5!white,colframe=gray!80!black,title=Conclusion:]
In this chapter, we recall the notion of isotropy Lie ($\infty$)-algebra of a singular foliation (and of a Lie-Rinehart algebra).

\phantom{cc}Then, we use the geometric resolution (i.e. the resolution on which the universal Lie $\infty$-algebroid is built) to recover several notions of resolution of singularities: one being due to Nash and a second one to Mohsen.
\end{tcolorbox}

\chapter{Symmetries of singular foliations through Lie $\infty$-algebroids}\label{chap:symmetries}
This chapter is one of the main application of results in Chapter \ref{main} and Section \ref{sec:univ-q-manifold}. These results are taken from my  article \cite{RL}.

We introduce the notion of weak symmetry actions of a lie algebra $\mathfrak g$ on a singular foliation $\mathfrak F$ and study the interaction of those on the universal Lie $\infty$-algebroids of $\mathfrak F$. Also, in the subsequent chapter, we apply these results to the problem of extending a strict Lie algebra action on a sub-affine variety on the ambient space. 

\begin{convention}
Throughout this chapter, $M$ stands for a smooth or complex manifold, or an affine variety over $\mathbb{C}$. We denote the sheaf of smooth or complex, or regular functions on $M$ by $\mathcal O$ and the sheaf of vector fields on $M$ by $\mathfrak{X}(M)$, and $X[f]$
stands for a vector field $X\in\mathfrak{X}(M)$ applied to $f\in \mathcal O$. Also, $\mathbb K$ stands for $\mathbb R$ or $\mathbb C$.
\end{convention}

\section{Definitions and examples}
\begin{definition}
Let $\mathfrak F\subset\mathfrak{X}(M)$ be a singular foliation over $M$. 
\begin{itemize}
    \item A diffeomorphism $\phi\colon M\longrightarrow M$ is said to be a \emph{symmetry} of $\mathfrak F$, if $\phi_*(\mathfrak F)=\mathfrak F$.
    
    \item A vector field $X\in\mathfrak{X}(M)$ is said to be an \emph{infinitesimal symmetry of $\mathfrak F$}, if $[X,\mathfrak F]\subset\mathfrak{F}$. The Lie algebra of infinitesimal symmetries of $\mathfrak{F}$ is denoted by $\mathfrak{s}(\mathfrak F)$.\\
\end{itemize}

    In particular, $\mathfrak F\subset\mathfrak{s}(\mathfrak F)$, since $[\mathfrak F,\mathfrak F]\subset\mathfrak{F}$.
\end{definition}

\begin{proposition}\cite{AndroulidakisIakovos,GarmendiaAlfonso}\label{prop:symm}
Let $M$ be a smooth or complex manifold. The flow of an infinitesimal symmetry of $\mathfrak F$, if it exists, is a symmetry of $\mathfrak{F}$.
\end{proposition}

As we will see later, one of the consequences of our future results is that any symmetry $X\in\mathfrak{s}(\mathfrak F)$ of a singular foliation $\mathfrak F$ admits a lift to a degree zero vector field on any universal $NQ$-manifold over $\mathfrak{F}$ that commutes with the  homological vector field $Q$. This allows us to have an alternative proof and interpretation of Proposition \ref{prop:symm} (see Section \ref{sec:3}).\\

Let $(\mathfrak g, \lb_{\mathfrak g})$ be a Lie algebra over $\mathbb{K}=\mathbb{R}$ or $\mathbb{C}$, depending on the context. From now on and in the sequel $\mathfrak g$ is concentrated in degree $-1$.

\begin{definition}\label{def:sym}
A \emph{weak symmetry action of the Lie algebra $\mathfrak{g}$} on  a singular foliation $\mathfrak F$ on $M$ is a $\mathbb K$-linear map $\varrho\colon\mathfrak{g}\longrightarrow \mathfrak X(M)$ that satisfies:\begin{itemize}
    \item $\forall\, x\in\mathfrak{g},\;[\varrho(x),\mathfrak F]\subseteq \mathfrak F$,
    \item  $\forall\, x,y\in\mathfrak{g},\;\varrho([x,y]_\mathfrak{g})-[\varrho(x),\varrho(y)]\in \mathfrak F$.
\end{itemize} 
When $x\longmapsto \varrho(x)$ is a Lie algebra morphism, we speak of \emph{strict symmetry action} of $\mathfrak{g}$ on $\mathfrak F$. There is an equivalence relation on the set of weak symmetry actions which is defined as follows: two weak symmetry actions, $\varrho, \varrho' \colon \mathfrak{g}\longrightarrow \mathfrak X(M)$ are said to be \emph{equivalent} if there exists a linear map $\varphi\colon\mathfrak{g}\longrightarrow \mathfrak{F}$ such that $\varrho-\varrho'=\varphi$.
\end{definition}
\begin{remark}It is important to notice that when $\mathfrak F$ is a regular foliation and $M/\mathfrak{F}$ is a manifold, any weak symmetry action of a Lie algebra $\mathfrak g$ on $\mathfrak{F}$ induces a strict action of $\mathfrak{g}$ over $M/\mathfrak{F}$. Definition \ref{def:sym} is a way of extending this idea to all singular foliations.
\end{remark}
Here is a list of some examples.

\begin{example}\label{ex:pull-back}
Let $\pi\colon M\longrightarrow N$ be a submersion. Since any vector field on $N$ comes from a $\pi$-projectable vector  field on $M$, therefore any Lie algebra morphism $\mathfrak{g}\longrightarrow \mathfrak{X}(N)$ can be lifted to a weak symmetry action $\mathfrak{g}\longrightarrow \mathfrak{X}(M)$  on the regular foliation $\Gamma(\ker\dd\pi)$, and any two such lifts are equivalent.

Furthermore, any weak action of a Lie algebra $\mathfrak{g}$ on a singular foliation $\mathfrak{F}$ on $N$ can be lifted to a class of weak symmetry actions on the pull-back foliation $\pi^{-1}(\mathfrak{F})$, (see Definition 1.9 in  \cite{AndroulidakisIakovos}). 
\end{example}
\begin{example}\label{ex:isotropy} Let $\mathfrak{F}$ be a singular foliation on $M$. For any point $m\in M$, consider  $\mathfrak{g}_m = \dfrac{\mathfrak{F}(m)}{\mathcal I_m\mathfrak F}$  the {isotropy Lie algebra of $\mathfrak{F}$ at $m$} (see Definition \ref{def:isotropy}). Let us denote its Lie bracket by  $\lb_{\mathfrak{g}_m}$.

\begin{enumerate}
    \item Consider $\varrho\colon \mathfrak{g}_m\to \mathfrak F(m)\subset\mathfrak{X}(M)$ a section of the projection map,
\begin{equation}\xymatrix{\mathcal I_m\mathfrak F\ar@{^{(}->}[r]&\mathfrak F(m)\ar@{->>}[r]&\mathfrak{g}_m\ar@/_/[l]_\varrho}\end{equation}Then, $[\varrho(x), \mathcal I_m\mathfrak F]\subset\mathcal{I}_m\mathfrak F$ and $\varrho([x,y]_{\mathfrak{g}_m})-[\varrho(x),\varrho(y)]\in\mathcal I_m\mathfrak F$. Hence, the map $\varrho\colon \mathfrak{g}_m\to \mathfrak{X}(M)$ is a weak symmetry action of the singular foliation $\mathcal I_m\mathfrak F$. A different section $\varrho'$ of  the projection map yields an equivalent weak symmetry action of $\mathfrak{g}_m$ on $\mathcal{I}_m\mathfrak{F}$. An obstruction class for having a strict symmetry action equivalent to $\varrho$ will be given later in Section \ref{sec:5}.

\item In particular, for $k\geq 1$, let us denote by $\mathfrak{g}_m^k$ the isotropy Lie algebra of the singular foliation $\mathcal I^k_m\mathfrak F$ at $m$. Any section $\varrho_k\colon\mathfrak{g}^k_m\longrightarrow \mathfrak{X}(M)$ of the projection map 

\begin{equation}\xymatrix{\mathcal I_m^{k+1}\mathfrak F\ar@{^{(}->}[r]&\mathcal{I}^k\mathfrak F\ar@{->>}[r]&\mathfrak{g}^k_m\ar@/_/[l]_{\varrho_k}}\end{equation}is a weak symmetry action of the Lie algebra $\mathfrak{g}^k_m$ on the singular foliation $\mathcal I_m^{k+1}\mathfrak F$. 
\end{enumerate}
\end{example}

\begin{example}
Let $(A, \lb_A, \rho)$ be an almost Lie algebroid on a smooth manifold $M$, and let $\mathfrak{F}=\rho(\Gamma(A))$. Assume there exists $p\in M$ such that $A|_{p}$ is a Lie algebra and $\rho|_p=0$. The question of finding a map $A|_{p}\longrightarrow \Gamma(A)$ such that the composition  \begin{equation}\label{eq:transformation-morphism}
    A|_{p}\longrightarrow \Gamma(A)\stackrel{\rho}{\longrightarrow}\mathfrak{X}(M) 
\end{equation}
be a Lie algebra morphism, i.e., to ask whether the singular foliation $\mathfrak{F}$ comes from the transformation Lie algebroid of $A|_p$, can be formulated as a weak symmetry action of $A|_{p}$ on the singular foliation $\mathcal I_p\mathfrak{F}$ as follows: 
Consider a linear map  \begin{align}\label{eq:transf-sym}
    A|_{p}&\longrightarrow \Gamma(A)\\&\nonumber a\longmapsto \widetilde{a}
\end{align}
such that for all $a\in A|_{p}$ and $\widetilde{a}$ a section of $A$ such that $\widetilde{a}|_p=a$. It is easily checked that the map in \eqref{eq:transf-sym} is not a Lie algebra morphism, but it satisfies$$\widetilde{[a,b]}_{A|_p}-\left[\widetilde a, \widetilde b\right]_A\in \mathcal{I}_p\Gamma(A)\qquad\text{and}\qquad\left[\widetilde{A|_p},\, \mathcal{I}_p\Gamma(A)\right]\subset{\mathcal{I}_p\Gamma(A)}.$$

\noindent
Therefore, the map $a\longmapsto \rho(\widetilde a)$ is a weak symmetry action of ${A|_p}$ on $\mathcal{I}_p\mathfrak{F}$. Notice that the isotropy Lie algebra $\mathfrak{g}^1_p=\frac{\mathcal{I}_p\mathfrak F}{\mathcal{I}^2_p\mathfrak{F}}$ of the singular foliation $\mathcal{I}_p\mathfrak{F}$, is Abelian. In Chapter \ref{chap:obstruction-theory}, we show that in this case,  the obstruction of having a Lie algebra morphism in \eqref{eq:transformation-morphism} is a cocycle of Chevalley-Eilenberg of $A|_p$ valued in $\mathfrak{g}^1_p$.
\end{example}

The next example comes from \cite{CLG-Ryv}, and follows the same patterns as in Examples \ref{ex:pull-back} and \ref{ex:isotropy}. It is based on the notion of Ehresmann connection. Let us first recall quickly this concept for the sake of completeness and clarity. There are several equivalent manners of viewing Ehresmann connections (see \cite{Ivan-Kolar} Section 9, Page 76), the most relevant approach in this context is the following: An \emph{Ehresmann connection} on a smooth fiber bundle\footnote{this generalizes connections to arbitrary
fiber bundle $\pi\colon E \rightarrow M$, here the total space $E$ is itself a smooth manifold, and has its
own tangent bundle $T E \rightarrow E$.} $\pi\colon E\rightarrow M$ is a vector subbundle $H$ of $TE$, such that $TE=H\oplus V$, where $V:=\{\xi\in TE\mid\pi_*(\xi)=0\}$ is called the "\emph{vertical bundle}" whose fiber at $e\in E$ is $V_e=T_e(E_{\pi(e)})$. The subbundle $H$ is called the "\emph{horizontal bundle}". Given a connection as defined above,  for every $e\in E$, the linear map $\pi_*\colon T_e E \rightarrow  T_{\pi(e)}M$ restricts to an isomorphism, $H_e \rightarrow T_{\pi(e)}M$, whose inverse $T_{\pi(e)}M \rightarrow H_e$ is called the \emph{horizontal lift}.

\begin{example}\label{ex:F-connection}
Let $(M,\mathfrak F)$ be a singular foliation on a smooth manifold $M$ and $L\subset M$ a leaf. Let $[L,M]$ be a neighborhood of $L$ in $M$ equipped with some projection\footnote{such projection comes with the Tubular neighborhood theorem \cite{Ana-Cannas}.} $\pi\colon[L,M]\rightarrow L$. 
 According to \cite{CLG-Ryv}, upon replacing  $[L,M]$ be a smaller neighborhood of $L$ if necessary,  there exists an Ehresmann connection (that is a vector sub-bundle $H \subset T[L,M]$ with $H\oplus\ker(\pi_*)=T[L,M]$)
which satisfies that $\Gamma(H) \subset \mathfrak F $. Such an Ehresmann connection is called an \emph{Ehresmann $\mathfrak F$-connection} and induces a $C^\infty(L)$-linear section $\varrho_H\colon\mathfrak X(L)\rightarrow \mathfrak{F}^{\mathrm{proj}}$ of the surjection $\mathfrak{F}^{ \mathrm{proj}} \rightarrow \mathfrak X(L)$, where $\mathfrak{F}^{ \mathrm{proj}}$ stands for  vector fields of $\mathfrak{F}$ $\pi$-projectable on elements of $\mathfrak{X}(L)$. The section $\varrho_H$ is a weak symmetry action of $\mathfrak{X}(L)$ on the \emph{transverse} foliation $\mathcal{T}:=\Gamma(\ker\pi_*)\cap~\mathfrak F$.
When the Ehresmann connection $H$ is flat,   $\varrho_H $ is bracket-preserving, 
and defines a strict symmetry of $\mathfrak{X}(L)$ on the transverse foliation $\mathcal T $.
\end{example}

\begin{example}
Consider, for a fixed $k$, the singular foliation $\mathfrak{F}_k:=\mathcal{I}_0^k\mathfrak{X}(\mathbb{R}^d)$ generated by all vector fields vanishing to order $k$ at the origin. The action of the Lie algebra $\mathfrak{gl}(\mathbb R^d)$ on $\mathbb{R}^d$ which is given by,  $$\mathfrak{gl}(\mathbb R)\longrightarrow \mathfrak{X}(\mathbb{R}^d),\,(a_{ij})_{1\leq i,j\leq d}\longmapsto\sum_{1\leq i, j\leq d}a_{ij}x_i\frac{\partial}{\partial x_j}$$ is a strict symmetry of $\mathfrak F_k$.
\end{example}
\begin{example}Let $\varphi := (\varphi_1,\ldots,\varphi_r )$ be a $r$-tuple of homogeneous polynomial functions in $d$ variables over $\mathbb K$. Consider the singular foliation $\mathfrak F_\varphi$ (see \cite{CLRL} Section 3.2.1) which is generated by all polynomial vector fields $X \in \mathfrak X(\mathbb K^d )$ that satisfy $X[\varphi_i ] = 0$ for all
$i \in\{1,\ldots, r\}$. The action $\mathbb K\rightarrow\mathfrak{X}(\mathbb{K}^d),\;\lambda\mapsto \lambda\overrightarrow{E}$, is a strict symmetry of  $\mathfrak{F}_\varphi$. Here $\overrightarrow{E}$ stands for the Euler vector field.
\end{example}
\begin{example}\label{ex:affine-action}
Let $W\subset \mathbb{C}^d$ be an affine variety and $\mathcal{I}_W\subset\mathbb{C}[x_1,\ldots,x_d]$ its corresponding ideal. Let us denote by $\mathfrak{X}(W):=\mathrm{Der}(\mathbb{C}[x_1,\ldots,x_d]/\mathcal{I}_W)$ the Lie algebra of vector field of $W$. Any vector field on $W$ can be extended (not unique) to a vector field on $\mathbb{C}^d$ (see Section \ref{three-main-constructions}). 

Let $\mathfrak{F}_W:=\mathcal I_W\mathfrak{X}(\mathbb{C}^d)$ the singular foliation made of vector fields vanishing on $W$. Since every vector field on $W$ can be extended to a vector field on $\mathbb{C}^d$ tangent to $W$. Any Lie algebra morphism $\varrho\colon\mathfrak{g}\longrightarrow \mathfrak{X}(W)$ extends to a linear map $\widetilde{\varrho}\colon\mathfrak{g}\longrightarrow \mathfrak{X}(\mathbb{C}^d)$ that makes this diagram commutes

\begin{equation*}
    \xymatrix{&\mathfrak{X}(\mathbb{C}^d)\ar@{->>}[d] \\\mathfrak{g}\ar[ru]^{\widetilde{\varrho}}\ar[r]_{\varrho}& \mathfrak{X}(W)}
\end{equation*}

This extension $\widetilde{\varrho}$ is a weak symmetry action of $\mathfrak{g}$ on $\mathfrak{F}_W$ over the ambient space $\mathbb{C}^d$. Two different extensions yield equivalent symmetry actions.
\end{example}
\section{A Lie $\infty$-morphism lifting a weak symmetry of a foliation}\label{chp:main2}\label{sec:3}
We recall that $\mathcal{O}$ is the sheaf of smooth/complex  functions on  a smooth/complex manifold $M$, or the algebra of regular functions on  an affine variety over  $\mathbb{C}$.
We refer the reader to the Chapters \ref{chap:oid1} and \ref{Chap:main} for the notion of (universal) Lie $\infty$-algebroid of a singular foliation. We denote them by $(E, Q)$ and their functions by $\E$.\\ 

For the sake of clarity, let put this chapter in context, and fix some notations.\\

Let $(\mathfrak{g},\lb_\mathfrak{g})$ a Lie algebra and $(E, Q)$ a Lie $\infty$-algebroid over $M$. In the sequel, the Lie algebra $\mathfrak{g}$ is concentrated in degree $-1$. The differential graded Lie algebra $(\mathfrak{X}(E),\lb, \mathrm{ad}_Q)$ of vector fields on $E$  is shifted by $1$, i.e. a derivation of degree $k$ in $\mathfrak{X}_k(E)$ is of degree $k-1$ as an element of the shifted space $\mathfrak{X}_k(E)[1]$. The graded symmetric Lie bracket on $\mathfrak{X}(E)[1]$ is of degree $+1$ and given on homogeneous elements $u,v\in\mathfrak{X}(E)[1]$ as  $$\{u,v\}:=(-1)^{\lvert v\rvert}[u,v].$$ 
In the sequel, we write $(\mathfrak{X}(E)[1],\lb,\mathrm{ad}_Q)$ instead of $(\mathfrak{X}(E)[1],\{\cdot\,,\cdot\},\mathrm{ad}_Q)$.\\

Let $\left(S_\mathbb{K}^\bullet\mathfrak{g}, Q_\mathfrak{g}\right)$ respectively $(S_\mathbb{K}^\bullet(\mathfrak{X}(E)[1]),\Bar{Q})$ be the corresponding formulations in terms of co-derivations of the differential graded Lie algebras $(\mathfrak{g},\lb_\mathfrak{g})$ and $(\mathfrak{X}(E)[1],\lb,\text{ad}_Q)$. Precisely, $Q_\mathfrak{g}$ is the co-derivation defined by putting for every homogeneous monomial $x_1\wedge\cdots\wedge x_k \in S^k_\mathbb{K} \mathfrak g$,
\begin{equation}
   Q_\mathfrak{g}(x_1\wedge\cdots\wedge x_k) :=\sum_{1\leq i<j\leq k}(-1)^{i+j-1}[x_i,x_j]_\mathfrak{g}\wedge x_1\wedge\cdots \widehat{x}_i\cdots \widehat{x}_j\cdots\wedge x_k,
\end{equation}
and  $\Bar{Q}=\Bar{Q}^{(0)}+\Bar{Q}^{(1)}$ is the co-derivation of degree $+1$ where the only non-zero Taylor coefficients are, $\Bar{Q}^{(0)}\colon S_\mathbb{K}^1(\mathfrak{X}(E)[1])\stackrel{\text{ad}_Q}{\longrightarrow}\mathfrak{X}(E)[1]$ and $\Bar{Q}^{(1)}\colon S_\mathbb{K}^2(\mathfrak{X}(E)[1])\stackrel{\{ \cdot\,,\cdot\}}{\longrightarrow}\mathfrak{X}(E)[1]$.\\

The following is a particular case of Example \ref{ex:dgla-morph}, see also \cite{LadaTom1994ShLa}.
\begin{definition}\label{def:morph2}
A \emph{Lie $\infty$-morphism}\footnote{Here, we use "$\rightsquigarrow$" to emphasize that $\Phi$ is not a DGLA morphism.} $\Phi\colon (\mathfrak{g},\lb_\mathfrak{g})\rightsquigarrow (\mathfrak X_\bullet(E)[1],\lb,\text{ad}_{Q})$ is a graded coalgebra morphism $\Bar{\Phi}\colon (S_\mathbb{K}^\bullet\mathfrak{g},Q_\mathfrak{g})\longrightarrow (S_\mathbb{K}^\bullet\left(\mathfrak X(E)[1]\right),\Bar{Q})$ of degree zero which satisfies,
\begin{equation}
    \Bar\Phi\circ Q_\mathfrak{g} =\Bar{Q}\circ\Bar\Phi.
\end{equation}In order words, it is the datum of degree zero linear maps $\left(\Bar{\Phi}_k\colon S_\mathbb{K}^{k+1}\mathfrak{g}\longrightarrow \mathfrak X_{-k}(E)[1]\right)_{k\geq 0}$ that satisfies \begin{align}\label{infty-morph-axiom}
    \sum_{1\leq i<j\leq n+2}(-1)^{i+j-1}\Bar{\Phi}_{n}([x_i,x_j]_\mathfrak{g},x_1,\ldots,\widehat{x}_{ij},\ldots,x_{n+2})&=[Q,\Bar{\Phi}_{n+1}(x_{1},\ldots,x_{n+2})] \,+&\nonumber\\\sum_{\tiny{\begin{array}{c}
        i+j= n \\i\leq j\\\sigma\in\mathfrak{S}_{i+1,j+1}\end{array}}}\epsilon(\sigma)[\Bar{\Phi}_i(x_{\sigma(1)}&,\ldots,x_{\sigma(i+1)}),\Bar{\Phi}_j(x_{\sigma(i+2)},\ldots,x_{\sigma(n+2)})]
\end{align}where $\widehat{x}_{ij}$ means that we take $x_i,x_j$ out of the list. When there is no risk of confusion, we write $\Phi$ for $\Bar{\Phi}$.
\end{definition}
\begin{remark}It is important to notice that:
\begin{enumerate}
    \item Definition \ref{def:morph2} and Definition \ref{def:oid-morph} are compatible when $M=\{\mathrm{pt}\}$. Therefore,  morphisms in both sense match.

\item In \cite{MEHTA2012576}, Definition \ref{def:morph2} corresponds to the definition of actions of a Lie $\infty$-algebras of finite dimension on Lie $\infty$-algebroids of finite rank. Here we only have a Lie algebra. In contrast to theirs, we do not assume that $\mathfrak{g}$ is finite dimensional.
\end{enumerate}
\end{remark}
\begin{remark}\label{Rmk:CE1}
It follows from the axioms \eqref{infty-morph-axiom} that for all $x,y\in \mathfrak{g}$, $[Q, \Phi_0(x)]=0$ and \begin{equation}\label{eq:second-condition}
    \Phi_0([x,y]_{\mathfrak{g}})-[\Phi_0(x),\Phi_0(y)]=[Q,\Phi_1(x,y)].
\end{equation} In particular, if the homological vector field $Q$  vanishes at some point $m\in M$, then the map $x\longmapsto (P\in\mathfrak{X}(E),\, P_{|_m}\mapsto[\Phi_0(x),\,P]_{|_m})$ endows the vector space $\mathfrak{X}(E)_{|_m}\simeq~(S(E^*)\otimes E)_{|_m}$ with a $\mathfrak{g}$-module structure. Moreover, the restriction of the map $\Phi_1\colon \wedge^2\mathfrak{g}\longrightarrow \mathfrak{X}_{-1}(E)_{|_m}$ at $m$ is a $2$-cocycle of Chevalley-Eilenberg.
\end{remark}
{\begin{remark}Let $(E,Q)$ be a Lie $\infty$-algebroid and $\mathfrak F$ its basic singular foliation. Any Lie $\infty$-morphism $\Phi\colon (\mathfrak{g},\lb_\mathfrak{g})\longrightarrow (\mathfrak X_\bullet(E)[1],\lb,\text{ad}_{Q})$ gives a weak symmetry action of $\mathfrak{g}$ on $\mathfrak F$. If $Q_{|_m}=0$ for some point $m\in M$, the $\mathfrak{g}$-action defined in Remark \ref{Rmk:CE1}, is independent of the equivalence class of the weak symmetry action.
\end{remark}}

The following lemma explains what the $0$-Taylor coefficient of a Lie $\infty$-morphism as in Definition \ref{def:morph2}  induces  on the linear part of $(E,Q)$. More details will be given in Proposition \ref{prop:induced-action} and Remark \ref{rmk:rk:low-dual-terms-Lie-infty}.
\begin{lemma}\label{lemma:basic-action}
The $0$-th Taylor coefficient $\Phi_0\colon \mathfrak{g}\longrightarrow \mathfrak X_0(E)$ induces 

\begin{enumerate}
    \item a linear map $\varrho\colon \mathfrak{g}\longrightarrow \mathfrak{X}(M),\, x\longmapsto \left(\varrho(x)[f]:=\Phi_0(x)[f],\;\;f\in \mathcal{O}\right)$ and
    \item a linear map
$x\in\mathfrak{g}\longmapsto \nabla_x\in\mathrm{Der}_0(E)$, i.e. for each $x\in \mathfrak{g}$, $\nabla_x\colon E\longrightarrow E$ is a degree zero map  that satisfies $$\nabla_x(fe)=~f\nabla_x(e) + \varrho(x)[f]e,\;\; \text{for}\;\; f\in \mathcal{O}, e\in \Gamma(E).$$
such that  \begin{equation}\label{eq:dual-action}
     \langle \Phi_0(x)^{(0)}(\alpha),e\rangle=\varrho(x)[\langle\alpha,e\rangle]-\langle\alpha,\nabla_x(e)\rangle,\;\;\text{for all $\alpha\in \Gamma(E^*),e\in \Gamma(E)$}.
 \end{equation}
 $\Phi_0(x)^{(0)}$ stands for the polynomial-degree zero of $\Phi_0(x)$.
\end{enumerate}

\end{lemma}
\begin{proof}
Using Lemma \ref{lemma:contraction}, we have for every $x\in\mathfrak g$, and $e\in \Gamma(E)$,\,  $$[\Phi_0(x), \iota_e]^{(-1)}=\iota_{\nabla_xe},$$ for some $\mathbb{K}$-bilinear map $\nabla_x\colon \Gamma(E_{-\bullet})\longrightarrow \Gamma(E_{-\bullet})$ that depends linearly on $x\in\mathfrak g$ and that satisfies \begin{equation}
    \label{eq:linear-part-der}
    \nabla_x(fe)=~f\nabla_x(e) + \varrho(x)[f]e,\; \text{for $f\in \mathcal{O}, e\in \Gamma(E)$}.
\end{equation}
To see \eqref{eq:linear-part-der}, compute $[\Phi_0(x), \iota_{fe}]^{(-1)}$: \begin{align*}
    \iota_{\nabla_x(fe)}&=[\Phi_0(x), \iota_{(fe)}]^{(-1)}\\&=\Phi_0(x)[f]\iota_e + f[\Phi_0(x), \iota_e]^{(-1)}\\&=\iota_{\varrho(x)[f]e + \nabla_xe}.
\end{align*}

In particular, one has for all $\alpha\in \Gamma(E^*),e\in \Gamma(E)$, \begin{align*}
    \langle \Phi_0(x)^{(0)}(\alpha),e\rangle&=\Phi_0(x)^{(0)}[\langle \alpha,e\rangle]-[\Phi_0(x)^{(0)}, \iota_e]^{(-1)}(\alpha)\\&=\varrho(x)[\langle\alpha,e\rangle]-\langle\alpha,\nabla_x(e)\rangle.
\end{align*}
\end{proof}

\subsection{Homotopies}

The following is a particular case of Definition \ref{def:homotopy} of Section \ref{sec:Homtopies1}. We rewrite it in this special context for the sake of clarity.
\begin{definition}\label{homp:def}
Let $\Bar{\Phi},\Bar{\Psi}\colon (S_\mathbb{K}^\bullet\mathfrak{g},Q_\mathfrak{g})\rightsquigarrow \left(S_\mathbb{K}^\bullet(\mathfrak X(E)[1]),\Bar{Q}\right)$ be Lie $\infty$-morphisms. We say $\Bar{\Phi},\Bar{\Psi}$ are \emph{homotopic over the identity of $M$} if the following conditions hold:\begin{enumerate}
    \item there a piecewise rational continuous path $t\in[a,b]\mapsto\Xi_t\colon(S_\mathbb{K}^\bullet\mathfrak{g},Q_\mathfrak{g})\rightsquigarrow \left(S_\mathbb{K}^\bullet(\mathfrak X(E)[1]),\Bar{Q}\right)$ made of Lie $\infty$-morphisms  that coincide with $\Bar{\Phi}$ and $\Bar{\Psi}$ at $t=a$ and $b$, respectively,
     
 \item and a piecewise rational path $t\in[a,b]\mapsto H_t$ of $\Xi_t$-co-derivations of degree $-1$ such that \begin{equation}
        \frac{\dd\Xi_t }{\dd t}=\Bar{Q}\circ H_t+H_t\circ Q_{\mathfrak{g}}.
    \end{equation}
\end{enumerate}
\end{definition}
\begin{remark}
Homotopy equivalence in the sense of the Definition \ref{homp:def} is an equivalence relation, and it is compatible with composition of Lie $\infty$-morphisms. Also, we "glue" infinitely many equivalences, as in Lemma \ref{gluing-lemma}.
\end{remark}
\begin{convention}
In the sequel, $Q_\mathfrak{g}$ and $\Bar Q$ will be in implicit.
\end{convention}
\begin{remark}
Definition \ref{homp:def} is slightly  more general than the equivalence relation \cite{MEHTA2012576}. In \cite{MEHTA2012576}, it is explained that Lie $\infty$-oid morphisms are Maurer-Cartan elements in some Lie $\infty$-algebroid $\mathfrak{g}\oplus E$ of certain form, and they define equivalence as gauge-equivalence of the Maurer-Cartan elements. This gauge equivalence corresponds to homotopies as above for which all functions are smooth. Also, we do not require nilpotence unlike in Definition 5.1 of  \cite{MEHTA2012576}. Last, we do not assume $\mathfrak{g}$ to be of finite dimension.
\end{remark}

\begin{proposition}\label{prop:induced-action}
Let $\mathfrak{g}$ be a Lie algebra and $(E,Q)$ a Lie $\infty$-algebroid over $M$.  

\begin{enumerate}
    \item Any {Lie $\infty$-morphism} $\Phi\colon (\mathfrak{g},\lb_\mathfrak{g})\rightsquigarrow (\mathfrak X_\bullet(E)[1],\lb,\mathrm{ad}_{Q})$ induces a weak symmetry action of $\mathfrak{g}$ on  the basic singular foliation $\mathfrak F=\rho(\Gamma(E_{-1}))$ of $(E,Q)$.
    \item Homotopic Lie $\infty$-morphisms ${\Phi},{\Psi}\colon (\mathfrak{g},\lb_\mathfrak{g})\rightsquigarrow (\mathfrak X_\bullet(E)[1],\lb,\mathrm{ad}_{Q})$ induce equivalent weak symmetry actions $\varrho_a, \varrho_b$ of $\mathfrak{g}$ on  the basic singular foliation $\mathfrak F$.
\end{enumerate}
\end{proposition}

\begin{proof}
Item \emph{1.} is a consequence of Remark \ref{Rmk:CE1}. Indeed,  take $\varrho\colon \mathfrak{g}  \longrightarrow \mathfrak{X}(M)$ as in Lemma \ref{lemma:basic-action} (\emph{1.}) We claim that $\varrho$ is a weak symmetry action of $\mathfrak{g}$ on  $\mathfrak{F}$: Let $x,y\in \mathfrak g$, and $e\in \Gamma(E_{-1})$ and $f\in \mathcal{O}$.
\begin{itemize}
    \item  $[\Phi_0(x), Q]=0$ entails, \begin{align*}
        \left\langle\Phi_0(x)^{(0)}\left[Q^{(1)}(f)\right],e\right\rangle&=\left\langle Q^{(1)}\left(\Phi_0(x)^{(0)}[f]\right),e\right\rangle\\&\\\varrho(x)[\langle Q[f],e\rangle]-\left\langle Q[f],\nabla_x(e)\right\rangle&=\rho(e)[\varrho(x)],\qquad\text{(by Lemma \ref{lemma:basic-action} (\emph{2.}))}\\&\\\varrho(x)[\rho(e)][f]-\rho(\nabla_x(e))[f]&=\rho(e)[\varrho(x)]
    \end{align*}
By consequence, $[\varrho(x), \rho(e)]=\rho(\nabla_x(e))\in \mathfrak{F}$. Therefore, $[\varrho(x), \mathfrak{F}]\subseteq \mathfrak{F}$.

    \item By Lemma \ref{lemma:contraction}, there exists a skew-symmetric linear map $\eta\colon \wedge^2 \mathfrak{g}\longrightarrow \Gamma(E_{-1})$ such that $\Phi_1(x,y)^{(-1)}=\iota_{\eta(x,y)}$. Therefore, the polynomial-degree zero of Equation \eqref{eq:second-condition} evaluated at an arbitrary function $f\in \mathcal{O}$ yields:
\begin{align*}
     &\Phi_0([x,y]_{\mathfrak{g}})^{(0)}(f)-[\Phi_0(x),\Phi_0(y)]^{(0)}(f)=[Q,\Phi_1(x,y)]^{(0)}(f)\\&\\\Longrightarrow \;&\Phi_0([x,y]_{\mathfrak{g}})(f)-\left[\Phi_0(x)^{(0)},\Phi_0(y)^{(0)}\right](f)=\left[Q^{(1)},\Phi_1(x,y)^{(-1)}\right](f)\\&\\\Longrightarrow &\; \varrho([x,y]_\mathfrak{g})[f]-[\varrho(x), \varrho(y)][f]=\left[Q^{(1)},\iota_{\eta(x,y)}\right](f)\\&\\\phantom{\Longrightarrow} &\; \phantom{\varrho([x,y]_\mathfrak{g})[f]-[\varrho(x), \varrho(y)][f]}=\rho(\eta(x,y))[f].
\end{align*}Since $f$ is arbitrary, this proves item \emph{1.} Using Proposition \ref{prop:homotpy-usual}, $\Phi\sim\Psi$ implies for $x\in \mathfrak{g}$ that
\begin{align}
   \nonumber \Psi(x)-\Phi(x)&=\Bar{Q}\circ H(x)+\cancel{H\circ Q_{\mathfrak{g}}(x)}\\\label{eq:sym-equi}&=[Q, H(x)]
\end{align}
with $H\colon \mathfrak{g}\longrightarrow \mathfrak{X}_{-1}(E)$ a linear map. Let $\beta\colon \mathfrak{g}\longrightarrow \Gamma(E_{-1})$ be a linear map such that $H(x)^{(-1)}=\iota_{\beta(x)}$. Taking the polynomial-degree zero of both sides in Equation \eqref{eq:sym-equi} and evaluating at  $f\in \mathcal{O}$ we obtain that \begin{align*}
        \left(\varrho_a(x)-\varrho_b(x)\right)[f]= \left[Q^{(1)}, H(x)^{(-1)}\right]=\left[Q^{(1)}, \iota_{\beta(x)}\right][f]=\rho(\beta(x))[f].
\end{align*}
Since $f$ is arbitrary, this proves item $\emph{2}$.
\end{itemize}
\end{proof}

Proposition \ref{prop:induced-action} tells us that {Lie $\infty$-morphism} $\Phi\colon (\mathfrak{g},\lb_\mathfrak{g})\rightsquigarrow (\mathfrak X_\bullet(E)[1],\lb,\text{ad}_{Q})$ induces weak symmetry action on the base manifold $M$. The aim of the next section is to look at the opposite direction. It responds to the following question: Do any weak symmetry action of a Lie algebra on a singular foliation comes from a Lie $\infty$-morphism? If so, can we extend uniquely?\\

Now we define what we call "lift" of a weak symmetry action of a Lie algebra $\mathfrak{g}$ on a singular foliation $\mathfrak{F}$ to a Lie $\infty$-algebroid $(E,Q)$ over $\mathfrak{F}$. 

\begin{definition}\label{def:lift}
Let $\mathfrak{F}$ be a singular foliation over $M$ and $(E, Q)$ a Lie $\infty$-algebroid over $\mathfrak F$. Consider a weak symmetry action $\varrho\colon\mathfrak{g}\longrightarrow\mathfrak{X}(M)$ of $\mathfrak g$ on  $\mathfrak F$.
\begin{itemize}
    \item We say that a Lie $\infty$-morphism of differential graded Lie algebras $$\Phi\colon (\mathfrak{g},\lb_\mathfrak{g})\rightsquigarrow (\mathfrak X_\bullet(E)[1],\lb,\text{ad}_{Q})$$ \emph{lifts the weak symmetry action $\varrho$ to $(E, Q)$} if for all $x\in\mathfrak{g}, f\in\mathcal{O}$,\, $\Phi_0(x)(f)=\varrho(x)[f]$.
    \item When $\Phi$ exists, we say then $\Phi$ is a \emph{lift} of $\varrho$ on $(E,Q)$.
\end{itemize}
\end{definition}

\section{Main statements}
 We now state the main theorem of this chapter. Proposition \ref{prop:induced-action} showed that a Lie $\infty$-morphism between a Lie algebra $\mathfrak{g}$ and a Lie $\infty$-algebroid $(E,Q)$ induces a weak action of $\mathfrak{g}$ on the basic foliation. In this section, we show  that any weak symmetry action of a Lie algebra $\mathfrak{g}$ on a singular foliation $\mathfrak{F}$ arises this way.\\

 More precisely,
\begin{theorem}\label{main}Let $\mathfrak F$ a be a singular foliation over a smooth manifold (or an affine variety) $M$ and $\mathfrak{g}$ a Lie algebra. Let $\varrho\colon\mathfrak{g}\longrightarrow\mathfrak{X}(M)$ be a weak symmetry action  of $\mathfrak g$ on $\mathfrak F$. The following assertions hold:
\begin{enumerate}
    \item for any universal Lie $\infty$-algebroid $(E,Q)$ of the singular foliation $\mathfrak F$, there exists a Lie $\infty$-morphism $\Phi\colon(\mathfrak{g},\lb_\mathfrak g)\rightsquigarrow\left(\mathfrak X_\bullet(E)[1],\lb, \emph{ad}_Q \right)$ that lifts $\varrho$ to $(E,Q)$,
    \item any two such Lie $\infty$-morphisms are homotopy equivalent over the identity of $M$,
    \item any two such lifts of any two  equivalent weak symmetry actions of $\mathfrak{g}$ on $\mathfrak{F}$ are homotopy equivalent.
\end{enumerate}

\begin{remark}
Again, Lie $\infty$-morphisms in item $\emph{1}$ of Theorem \ref{main} are $\mathfrak{g}$-actions on $(E,Q)$ in \cite{MEHTA2012576}.
\end{remark}

\begin{remark}\label{rk:low-terms-Lie-infty}
The item $\emph{1}$ in Theorem \ref{main} means that there exists a linear map $\Phi_0\colon\mathfrak{g}\longrightarrow\mathfrak{X}_0(E)$ such that \begin{equation}\label{eq:compatibility-with-Q}
    \Phi_0(x)[f]=\varrho(x)[f], \;\text{and}\,\; [Q,\Phi_0(x)]=0,\quad \forall x\in \mathfrak{g}, f\in \mathcal{O}.
\end{equation}
This morphism is not a graded Lie algebra morphism, but there exist a linear map $\Phi_1\colon\wedge^2\mathfrak{g}\longrightarrow \mathfrak{X}_{-1}(E)$ such that for all $x,y, z\in \mathfrak{g}$,\begin{equation}\label{eq:term2}
    \Phi_0([x,y]_{\mathfrak{g}})-[\Phi_0(x),\Phi_0(y)]=[Q,\Phi_1(x,y)].
\end{equation}
Also, \begin{equation}\label{eq:low-term2}
    \Phi_1\left([x,y]_\mathfrak{g},z\right)-[\Phi_0(x),\Phi_1(y,z)]+\circlearrowleft(x,y,z)=[Q,\Phi_2(x,y,z)]
\end{equation}for some linear map $\wedge^3\mathfrak{g}\longrightarrow \mathfrak{X}_{-2}(E)$. These sets of compatibility conditions continue to higher multilinear maps.  
\end{remark}

\end{theorem}
\begin{corollary}\label{1symmetry}
Any symmetry $X\in\mathfrak{X}(M)$ of the singular foliation $\mathfrak F$ can be lifted to a degree zero vector field $Z\in\mathfrak X_0(E)$ that commutes with $Q$, i.e. such that $[Z,Q]=0$. 
\end{corollary}
\begin{proof}
To construct $Z$, it suffices to apply Theorem \ref{main} for $\mathfrak{g}=\mathbb{R}$ and take $Z$ to be the image of $1$ through $\Phi_0\colon \mathbb R\longrightarrow \mathfrak{X}_0(E)$. 
\end{proof}

\begin{remark}
 In particular, Corollary \ref{1symmetry} has the following consequences: \begin{enumerate}
    \item for any admissible $t$, the flow $\Phi^Z_t\colon \E \longrightarrow \E$ of $Z$ induces an isomorphism of vector bundles $E_{-1}\longrightarrow E_{-1}$. Since $[Q,Z]=0$, the following diagram commutes,
    
    $$\xymatrix{\Gamma(E_{-1})\ar[d]^{\rho}\ar[r]^{(\Phi^Z_t)^{(0)}}&\Gamma(E_{-1})\ar[d]^\rho\\\mathfrak{X}(M)\ar[r]_{(\varphi^X_t)_*}&\mathfrak{X}(M)}$$

    where  $\phi^X_t$  is the flow of $X$ at $t$.
    \item Consequently, for any open set $U\subset M$ which is stable under $\varphi^X_t$, there exists an invertible matrix $\mathfrak{M}^t_X$ with coefficients in $\mathcal O(U)$ that satisfies $$\left(\phi^X_t\right)_*\begin{pmatrix}X_1\\\vdots\\X_n\end{pmatrix}=\mathfrak{M}^t_X\begin{pmatrix}X_1\\\vdots\\X_n\end{pmatrix},$$
   for some generators $X_1,\ldots,X_n$ of $\mathfrak{F}$ over $U$. As announced earlier, we recover Proposition \ref{prop:symm}, that is, $\left(\phi^X_t\right)_*(\mathfrak{F})=\mathfrak{F}$.
   
\end{enumerate} 
\end{remark}

 Let $(E,Q)$ and $(E',Q')$ be two universal Lie $\infty$-algebroids of $\mathfrak{F}$. A direct consequence of Ricardo Campos's Theorem 4.1 in \cite{Campos} is that the differential graded Lie algebras $\left( \mathfrak X_\bullet(E)[1],\lb, \mathrm{ad}_Q \right)$ and $\left( \mathfrak X_\bullet(E)[1],\lb, \mathrm{ad}_{Q'}\right)$ are homotopy equivalent over the identity of $M$. This leads to the following statement.
\begin{corollary}
 Let $\varrho\colon\mathfrak{g}\longrightarrow\mathfrak{X}(M)$ be a weak symmetry action of a Lie algebra $\mathfrak{g}$ on $\mathfrak F$. Then,   there exist Lie $\infty$-morphisms, $\Phi\colon\mathfrak{g}\rightsquigarrow \left( \mathfrak X_\bullet(E)[1],\lb, \emph{ad}_Q \right)$ and $\Psi\colon\mathfrak{g}\rightsquigarrow \left(\mathfrak X_\bullet(E')[1],\lb, \emph{ad}_{Q'} \right)$ that lift $\varrho$, and $\Phi,\Psi$ make the following diagram commute up to homotopy

\begin{equation}\label{diagram:Campos}\xymatrix{ &\mathfrak{g}\ar[dl]_\Phi\ar[dr]^\Psi& \\\left( \mathfrak X_\bullet(E)[1],\lb, \emph{ad}_Q \right)\ar@{<->}[rr]^\sim & &\left( \mathfrak X_\bullet(E')[1],\lb, \emph{ad}_{Q'} \right).}\end{equation}

\end{corollary}

\begin{proof}
The composition of $\Phi$ with the horizontal map in the diagram \eqref{diagram:Campos} is a lift of the action $\varrho$. It is necessarily homotopy equivalent to $\Psi$ by item $(2)$ in Theorem \ref{main}.
\end{proof}

\subsection{Proof of \ref{main}}\label{results:proofs}
This section is devoted to the proof of the main results, i.e. Theorem \ref{main}.\\

Let $\mathfrak F$ be a singular foliation, and $(E,Q)$ a universal Lie $\infty$-algebroid of $\mathfrak F$. We start with the following lemma.
\begin{lemma}\label{initia2:lemm}
  For every weak symmetry Lie algebra action of $\mathfrak g$ on $\mathfrak{F}$ there exists a linear map, $\Phi_0\colon \mathfrak g\rightarrow \mathfrak X_0(E)$, such that $[Q,\Phi_0(x)]=0$ and $\Phi_0(x)[f]=\varrho(x)[f]$ for all $x\in\mathfrak g$, $f\in\mathcal{O}$. 
\end{lemma}

\begin{proof}
For $x\in\mathfrak{g}$, let $\widehat{\varrho(x)}\in \mathfrak X_0(E)$ be any arbitrary extension of $\varrho(x)\in\mathfrak{s}(\mathfrak{F})$ to a degree zero vector field on $E$. Since $\varrho(x)$ is a symmetry of $\mathfrak{F}$, the degree $+1$ vector field $[\widehat{\varrho(x)},Q]$ is also a longitudinal vector field on $E$, see Example \ref{longi:examples} item 3. In addition, $[\widehat{\varrho(x)},Q]$ is a $\mathrm{ad}_Q$-cocycle. By item $1$ of Theorem \ref{thm:longitudinal}, there exists a vertical vector field $Y(x)\in\mathfrak a(E)$ of degree zero such that \begin{equation}
    [Q,Y(x)+ \widehat{\varrho(x)})]=0.
\end{equation}

Let us set for $x\in\mathfrak{g}$,\;$\Phi_0(x):=Y(x)+\widehat{\varrho(x)}$. By construction, we have, $[Q,\Phi_0(x)]=0$ and $\Phi_0(x)[f]=\varrho(x)[f]$ for all $x\in\mathfrak g,\,f\in\mathcal{O}$. 
\end{proof}

We will need the following lemma.
\begin{lemma}\label{lem:length}Assume $(E,Q)$ is a universal Lie $\infty$-algebroid over $M$. Let $\Bar{\Phi}\colon (S_\mathbb{K}^\bullet\mathfrak{g},Q_\mathfrak{g})\longrightarrow (S_\mathbb{K}^\bullet\mathfrak X(E)[1],\Bar{Q})$ be a coalgebra morphism which is a Lie $\infty$-morphism up to polynomial-degree $n\geq 0$, i.e. $$\left(\Bar\Phi\circ Q_\mathfrak{g} -\Bar{Q}\circ\Bar\Phi\right)^{(i)}=0\;\text{ for all integer} \;\;i\in\{0,\ldots,n\}.$$ Then, $\Bar{\Phi}$ can be lengthened to a $\infty$-morphism up to polynomial-degree $n+1$. 
\end{lemma}
\begin{proof} For convenience, we omit the variables $x\in S_\mathbb{K}^\bullet\mathfrak{g}$. The identity, $$\Bar{Q}\circ\left(\Bar\Phi\circ Q_\mathfrak{g} -\Bar{Q}\circ\Bar\Phi\right)+\left(\Bar\Phi\circ Q_\mathfrak{g} -\Bar{Q}\circ\Bar\Phi\right)\circ Q_\mathfrak{g}=0$$ taken in polynomial-degree $n+1$ yields, \begin{align*}
    0=\left(\Bar{Q}\circ(\Bar\Phi\circ Q_\mathfrak{g} -\Bar{Q}\circ\Bar\Phi)\right)^{(n+1)}&=[Q,(\Bar\Phi\circ Q_\mathfrak{g} -\Bar{Q}\circ\Bar\Phi)^{(n+1)}],
\end{align*}
\text{since $Q_\mathfrak{g}^{(0)}=0$ and $\left(\Bar\Phi\circ Q_\mathfrak{g} -\Bar{Q}\circ\Bar\Phi\right)^{(i)}=0$ for $i\in\{0,\ldots,n\}$}. It is clear that for all $n\geq 0$ the map $\left(\Bar\Phi\circ Q_\mathfrak{g} -\Bar{Q}\circ\Bar\Phi\right)^{(n+1)}\colon S^{n+2}_\mathbb{K}\mathfrak{g}\longrightarrow\mathfrak{X}_{-n}(E)[1]$ take value in vertical vector fields on $E$. By virtue of Lemma \ref{cor:longitudinal} there exists a map  $\zeta\colon S^{n+2}_\mathbb{K}\mathfrak{g} \longrightarrow\mathfrak{X}_{-n-1}(E)[1]$  such that  \begin{equation}[Q,\Bar{\Phi}^{(n+1)}+\zeta]=\Bar\Phi^{(n)}\circ Q_\mathfrak{g}^{(1)} -\Bar{Q}^{(1)}\circ\Bar\Phi^{(n)}.
\end{equation}
By redefining the polynomial-degree $n+1$ of $\Bar{\Phi}$ as $\Bar{\Phi}^{(n+1)}:=\Bar{\Phi}^{(n+1)}+\zeta$. One obtains a Lie $\infty$-morphism up to polynomial-degree $n+1$. The proof continues by recursion.
\end{proof}

\begin{proof}[Proof of Theorem \ref{main}]\label{proof:main} Let us show Item 1. Note that
Lemma \ref{initia2:lemm}  gives the existence of a linear map $\Phi_0\colon\mathfrak{g}\longrightarrow \mathfrak{X}_0(E)$ such that, $[Q,\Phi_0(x)]=0$ for all $x\in \mathfrak{g}$. For $x,y\in \mathfrak{g}$, consider \begin{equation}
    \Lambda(x,y)=\Phi_0([x,y]_\mathfrak{g})-[\Phi_0(x),\Phi_0(y)].\end{equation}Since $\varrho([x,y]_\mathfrak{g})-[\varrho(x),\varrho(y)]\in \mathfrak F$ for all $x,y\in\mathfrak g$, and since $\rho\colon\Gamma(E_{-1})\longrightarrow\mathfrak{F}$ surjective, we have $\varrho([x,y]_\mathfrak{g})-[\varrho(x),\varrho(y)]=\rho\left(\eta(x,y)\right)$ for some element $\eta(x,y)\in\Gamma(E_{-1})$ depending linearly on $x$ and $y$. Now we consider the vertical vector field of degree $-1$, $\iota_{\eta(x,y)}\in\mathfrak{X}_{-1}(\mathbb{U}^\mathfrak F)$ which is defined on $\Gamma(E^*)$ as:
$$\iota_{\eta(x,y)} (\alpha):= \langle\alpha, {\eta(x,y)}\rangle\;\; \text{for all}\;\; \alpha\in\Gamma(E^*),$$ and extended it by derivation on the whole space. For every $f\in\mathcal O$,\begin{align*}
\left(\Lambda(x,y)-[Q,\iota_{\eta(x,y)}]\right)(f)&=\left(\varrho([x,y]_\mathfrak{g})-[\varrho(x),\varrho(y)]-\rho(\eta(x,y) \right)[f]\hspace{1cm}\text{(by definition of $\Phi_0$)}\\&=0\hspace{7.47cm}\text{(by definition of $\eta$)}
\end{align*}
It is clear that $\Lambda(x,y)+[Q,\iota_{\eta(x,y)}]$ is a $\text{ad}_Q$-cocycle. Also, $\left(\Lambda(x,y)+[Q,\iota_{\eta(x,y)}]\right)^{(-1)}$: for every $\alpha\in \Gamma(E^*)$, \begin{align*}
    [Q,\iota_{\eta(x,y)}]^{(-1)}(\alpha)=[Q^{(0)}, \iota_{\eta(x,y)}](\alpha)=\cancel{Q^{(0)}[\langle\alpha,\eta(x,y)\rangle]}+\cancel{\langle Q^{(0)}[\alpha],\eta(x,y)\rangle}=0,
\end{align*}
where the first term (resp. the second term) is cancelled by $\mathcal{O}$-linearity of $Q^{(0)}$ (resp. for degree reason). Hence, by Corollary \ref{cor:longitudinal} and Remark \ref{rmk:arity}, the degree zero vector field $\Lambda(x,y)+[Q,\iota_{\eta(x,y)}]$ is of the form $[Q, \Upsilon(x,y)]$ for some vertical vector field $\Upsilon(x,y)\in\mathfrak{X}_{-1}(E)$ of degree $-1$ with $\Upsilon (x,y)^{(-1)}=0$. For all $x,y\in\mathfrak{g}$, we define the Taylor coefficient $\Phi_1\colon\wedge^2\mathfrak{g}\longrightarrow \mathfrak{X}(E)$  as $\Phi_1(x,y):=\Upsilon(x,y)+\iota_{\eta(x,y)}$. By construction,  we have the following relation,\\
\begin{equation}
    \Phi_0([x,y]_\mathfrak{g})-[\Phi_0(x),\Phi_0(y)]=[Q, \Phi_1(x,y)],\;\forall x,y\in\mathfrak{g}
\end{equation}

\noindent
Consider for $x,y,z\in\mathfrak{g}$, \begin{equation}
    \vartheta(x,y,z)=\Phi_1\left([x,y]_\mathfrak{g},z\right)-[\Phi_0(x),\Phi_1(y,z)]+\circlearrowleft(x,y,z).
\end{equation}Here, $\circlearrowleft(x,y,z)$ stands for circular permutation of $x,y$ and $z$ with Koszul sign. For degree reason $\vartheta(x,y,z)$ is $\mathcal O$-linear. Moreover, $ \vartheta(x,y,z)$ is a $\text{ad}_Q$-cocycle: 
\begin{align*}
    \left[Q, \Phi_1([[x,y]_\mathfrak{g},z]_\mathfrak{g})\right]+\circlearrowleft(x,y,z)&=-\left[\Phi_0\left([x,y]_\mathfrak{g}\right),\Phi_0(z)\right]+\circlearrowleft(x,y,z)\\&=[[\Phi_0(z),Q],\Phi_1(x,y)]-[[\Phi_1(x,y),\Phi_0(z)],Q]+\circlearrowleft(x,y,z)\\&=[Q,[\Phi_0(x),\Phi_1(y,z)]]+\circlearrowleft(x,y,z).
    \end{align*}
   
\noindent
Here, we have used the fact that $[Q,\Phi_0(x)]=0$ for all $x\in \mathfrak{g}$, and the Jacobi identity for the Lie brackets $\lb_\mathfrak{g}$ and $\lb$. By Corollary \ref{cor:longitudinal}, there exists a derivation of degree $-2$ denoted by $\Phi_2(x,y,z)\in\mathfrak{X}_{-2}(E)[1]$ that satisfies, \begin{equation}
    \vartheta(x,y,z)=[Q,\Phi_2(x,y,z)].
\end{equation}

So far, in the construction of the Lie $\infty$-morphism, we have shown the existence of a Lie $\infty$-morphism $\Bar{\Phi}\colon S_\mathbb{K}^\bullet\mathfrak{g}\longrightarrow S_\mathbb{K}^\bullet\left(\mathfrak{X}(E)[1] \right)$ up to polynomial-degree $2$ that is $(\Bar{\Phi}\circ Q_\mathfrak{g})^{(i)}=(\Bar{Q}\circ \Bar{\Phi})^{(i)}$ with $i=0,1,2$. The proof continues by recursion or by applying directly Lemma \ref{lem:length}. This proves the part $1.$ of the theorem.
\end{proof}
Before proving item 3 of Theorem \ref{main} we will need the following lemma. For convenience,  we sometimes omit the variables in $\mathfrak{g}$.
\begin{lemma}\label{lemma:homotopy}
For any two Lie $\infty$-morphisms $\Gamma,\Omega\colon (S_\mathbb{K}^\bullet\mathfrak{g},Q_\mathfrak{g})\rightsquigarrow (S_\mathbb{K}^\bullet(\mathfrak{X}(E)[1]),\Bar{Q})$ which coincide up to polynomial-degree $n\geq 1$, i.e. $\Gamma^{(i)}=\Omega^{(i)}$, for $0\leq i\leq n$, their difference in polynomial-degree $n+1$, namely, $$\Gamma^{(n+1)}-\Omega^{(n+1)}\colon S_\mathbb{K}^{n+2}\mathfrak{g}\longrightarrow\mathfrak{X}_{-n-1}(E)[1]$$ is valued  in $\mathrm{ad}_Q$-coboundary. 
\end{lemma}

\begin{proof}
Indeed, a direct computation yields \begin{align*}
    \Bar{Q}\circ(\Gamma-\Omega)=(\Gamma-\Omega)\circ Q_\mathfrak{g}&\Longrightarrow \Bar{Q}^{(0)}\circ(\Gamma-\Omega)^{(n+1)}-\underbrace{\left((\Gamma-\Omega)\circ Q_\mathfrak{g}\right)^{(n+1)}}_{=0}=0\\&\Longrightarrow[Q,\Gamma^{(n+1)}-\Omega^{(n+1)}]=0\\&\Longrightarrow \Gamma^{(n+1)}-\Omega^{(n+1)}=[Q,
    H^{(n+1)}]\hspace{1cm}\text{(by item 1 of Theorem \ref{thm:longitudinal})}
\end{align*}for some linear map $H^{(n+1)}\colon S_\mathbb{K}^{n+2}\mathfrak{g}\longrightarrow\mathfrak{X}_{-n-2}(E)[1]$.
\end{proof}

Let us show item 2 of Theorem \ref{main}. Let $\Phi,\Psi\colon \mathfrak g\longrightarrow\mathfrak X({E})[1]$ be two different lifts of the action $\mathfrak{g}\longrightarrow\mathfrak{X}(M)$. We denote by $\Bar{\Phi},\Bar{\Psi}\colon S_\mathbb{K}^\bullet\mathfrak{g}\longrightarrow S_\mathbb{K}^\bullet(\mathfrak{X}(E)[1])$ the unique comorphisms given respectively by the Taylor coefficients  \begin{equation}
    \begin{cases}
    &\Bar{\Phi}^{(r)}\colon S_\mathbb{K}^{r+1}\mathfrak{g}\xrightarrow{\Phi_{r}}\mathfrak{X}_{-r}(E)[1]\\&\Bar{\Psi}^{(r)}\colon S_\mathbb{K}^{r+1}\mathfrak{g}\xrightarrow{\Psi_{r}}\mathfrak{X}_{-r}(E)[1]
    \end{cases},\; \text{for $r\geq 0$.}
\end{equation}

For any $x\in\mathfrak{g}$, the degree zero vector field\, $\Phi_{0}(x)-\Psi_{0}(x)\in\mathfrak{X}_0(E)$ is vertical. Moreover, we have, $[Q,\Phi_{0}(x)-\Psi_{0}(x)]=0$. By Corollary \ref{cor:longitudinal} there exists a vector field $H_{0}\in\mathfrak{X}_{-1}(E)$ of degree $-1$, such that $\Psi_{0}(x)-\Phi_{0}(x)=[Q,H_{0}(x)]$

\begin{equation}
   \xymatrix{&\mathfrak{g}\ar[d]^{\Psi_{0}-\Phi_{0}}  \ar@{-->}[ld]_{H_{0}}\\\mathfrak{X}_{-1}(E)[1]\ar[r]^{\text{ad}_Q} & \mathfrak{X}_0(E)[1]}  
\end{equation}


Consider the following differential equation \begin{equation}\label{diff-eq1}
    \begin{cases}
    \frac{\dd\Xi_t}{\dd t}&=\Bar{Q}\circ H_t+H_t\circ Q_\mathfrak g,\hspace{1cm}t\in[0,1]\\\Xi_0&=\Bar \Phi
    \end{cases}
\end{equation}where $(\Xi_t)_{t\in[0,1]}$ is as in Definition \ref{homp:def}, and for $t\in[0,1]$, $H_t$ is the unique $\Xi_t$-co-derivation where the only non-zero polynomial-degree is $H^{(0)}=H_{0}$. Equation \eqref{diff-eq1} gives a homotopy between $\Bar{\Phi}$ and $\Xi_1$. When we consider the polynomial-degree zero component in Equation \eqref{diff-eq1}, one obtains\begin{align*}
    \frac{\dd\Xi_t^{(0)}}{\dd t}&=\Bar{Q}^{(0)}\circ H_t^{(0)}+H_t^{(0)}\circ Q_\mathfrak g^{(0)}\\&=[Q,H_{0}]\\&=\Psi_{0}-\Phi_{0}=\Bar{\Psi}^{(0)}-\Bar{\Phi}^{(0)}.
\end{align*}Therefore, $\Xi_t^{(0)}=\Bar{\Phi}^{(0)}+t(\Bar{\Psi}^{(0)}-\Bar{\Phi}^{(0)})$, and $\Bar{\Phi}\sim \Xi_1$ with $\Bar{\Psi}^{(0)}=\Xi_1^{(0)}$. Using Lemma \ref{lemma:homotopy}, the image of  any element through the map  $\Bar{\Psi}^{(1)}-\Xi_1^{(1)}\colon S_\mathbb{K}^{2}\mathfrak{g}\longrightarrow\mathfrak{X}_{-1}(E)[1]$ is a $\text{ad}_Q$-coboundary. Thus, $\Bar{\Psi}^{(1)}-\Xi_1^{(1)}$ can be written as \begin{equation}
    \Bar{\Psi}^{(1)}-\Xi_1^{(1)}=[Q, H^{(1)}],\quad \text{with}\, \;H^{(1)}\colon S_\mathbb{K}^{2}\mathfrak{g}\longrightarrow\mathfrak{X}_{-2}(E)[1].
\end{equation}Let us go one step further by considering the differential equation on $[0,1]$ given by \begin{equation}\label{equa:diff2}
    \begin{cases}
    \frac{\dd\Theta_t}{\dd t}&=\Bar{Q}\circ H_t+H_t\circ Q_\mathfrak g\\\Xi_0&=\Bar \Xi_1
    \end{cases}
\end{equation}Here $H_t$ is the extension of $H^{(1)}$ as the unique $\Theta_t$-co-derivation where all its arities vanish  except the polynomial-degree 1 which is given by $H^{(1)}$. In polynomial-degree zero, $(\Theta^{(0)}_t)_{t\in[0,1]}$ is constant and has value $\Theta_1^{(0)}=\Bar{\Psi}^{(0)}$. In polynomial-degree one, we have, \begin{align*}
    \frac{\dd\Theta_t^{(1)}}{\dd t}&=\Bar{Q}^{(0)}\circ H_t^{(1)}\\&=[Q,H^{(1)}]=\Bar{\Psi}^{(1)}-\Xi_1^{(1)}.
\end{align*}Hence, $\Theta_t^{(1)}=\Bar{\Phi}^{(1)}+t(\Bar{\Psi}^{(1)}-\Xi_1^{(1)})$ with $\Bar{\Psi}^{(i)}=\Theta_1^{(i)}$ for $i=0,1$. We then continue this
procedure by gluing all these homotopies as in the proof of item \emph{2}  of Theorem \ref{th:universal}. We obtain at last a Lie $\infty$-morphism $\Omega$ such that  $\Bar{\Phi}\sim\Omega$ and $\Omega^{(i)}=\Bar{\Psi}^{(i)}$ for $i\geq 0$. That means $\Omega=\Bar{\Psi}$, therefore $\Bar{\Phi}\sim\Psi$. This proves item 2. of Theorem \ref{main}.\\

Let us prove item 3 of Theorem \ref{main}. Given two equivalent weak symmetry actions $\varrho,\varrho'$ of $\mathfrak{g}$ on a singular foliation $\mathfrak{F}$, i.e. $\varrho,\varrho'$ differ by a linear map $\mathfrak g\longrightarrow \mathfrak{X}(M)$ of the form $x\mapsto \rho(\beta(x))$ for some linear map $\beta\colon\mathfrak g\longrightarrow \Gamma(E_{-1})$. Let $\Phi,\Phi'\colon\mathfrak{g}\rightsquigarrow \left( \mathfrak X_\bullet(E)[1],\lb, \text{ad}_Q \right)$ be a lift into a Lie $\infty$-morphism of the action $\varrho$ and $\varrho'$ respectively. One has for all $x\in\mathfrak{g}$ and $f\in\mathcal{O}$, \begin{align*}
        \left(\Phi_0(x)-\Psi_0(x)-[Q,\iota_{\varphi(x)}]\right)(f)&=\rho(\varphi(x))[f]-\langle Q(f),\varphi(x)\rangle\\&=0.
    \end{align*}
 Since $[Q, \Phi_0(x)-\Psi_0(x)-[Q,\iota_{\varphi(x)}]]=0$, by Corollary \ref{cor:longitudinal} there exists a vertical derivation $\widehat{H}(x)\in \mathfrak{X}_{-1}(E)$ of degree $-1$ depending linearly on $x\in\mathfrak{g}$ such that $$\Phi_0(x)-\Psi_0(x)=[Q,\widehat{H}(x)+\iota_{\varphi(x)}].$$ Let $H(x):=\widehat{H}(x)+\iota_{\varphi(x)}$, for $x\in\mathfrak{g}$. The proof continues the same as for item 2 of Theorem \ref{main}

\subsection{Particular examples}
We recall that for a regular foliation $\mathfrak{F}$ on a manifold $M$, the Lie algebroid $T F \subset T M$, whose sections form $\mathfrak F$, is a universal Lie $\infty$-algebroid of $\mathfrak F$. Its corresponding $Q$-manifold is given by the leafwise De Rham differential on $\Gamma(\wedge^\bullet T^*F)$. 

\begin{example}
Let $\mathfrak{F}$ be a regular foliation on a manifold $M$. Any weak symmetry action $\mathfrak{g}\longrightarrow \mathfrak X(M),\,x\longmapsto \varrho(x)$,  of $\mathfrak F$, can be lifted to Lie $\infty$-morphism $\Phi\colon\mathfrak{g}\rightsquigarrow \left( \mathfrak X_\bullet(E)[1],\lb, \text{ad}_Q \right)$ 
given explicitly as follows:
\begin{align}
   x\in\, &\mathfrak{g}\longmapsto \Phi_0(x)={\mathcal{L}}_{\varrho(x)}\in\mathfrak{X}_0(\wedge^\bullet T^* F)\\x\wedge y\in\wedge^2&\mathfrak g\longmapsto \Phi_1(x,y)=\iota_{\chi(x,y)}\in\mathfrak{X}_{-1}(\wedge^\bullet T^* F)\end{align}
   and $\left(\Phi_i\colon \wedge^{i+1}\mathfrak{g}\longrightarrow \mathfrak{X}_{-i}(\wedge^\bullet T^* F)\right)\equiv 0$, for all $i\geq 2$, where $\chi(x,y):=\varrho([x,y]_\mathfrak{g})-[\varrho(x),\varrho(y)]$ for $x,y\in\mathfrak{g}$. Also, ${\mathcal{L}}_{X}$ stands for the Lie derivative on multi-forms w.r.t $X\in \mathfrak{X}(M)$, and $\iota_X$ is the internal product.
\end{example}
\begin{example}
Let $\mathfrak F$ be a singular foliation on a manifold $M$ together with a strict symmetry action $\varrho\colon\mathfrak{g}\longrightarrow\mathfrak{X}(M)$ such that $\underline{\mathfrak{g}}\subset\mathfrak{F}$. Hence, $C^\infty(M)\underline{\mathfrak{g}}$ is a singular foliation which is the image of the transformation Lie algebroid $\mathfrak g\times M$. The  universality theorem (see \cite{LLS,CLRL}) provides the existence of a Lie $\infty$-morphism $\nu\colon\mathfrak{g}\longrightarrow\mathbb{U}^\mathfrak F$. Let us call its Taylor coefficients $\nu_n\colon\wedge^{n+1}\mathfrak{g}\longrightarrow E_{-n-1},\,n\geq 0$. We may take for example the $0$-th and $1$-th Taylor coefficients of a Lie $\infty$-morphism that lifts $\varrho$ as: \begin{align*}\label{natural}\Phi_0(x)&:=[Q,\iota_{\nu_0(x)}]\in\mathfrak{X}_{0}(\mathbb{U}^\mathfrak F),\; \text{for}\; x\in \mathfrak g.\\\Phi_1(x,y)&:=[Q,\iota_{\nu_1(x,y)}]^{(-1)}-\sum_{k\geq 0}[[Q,\iota_{\nu_0(x)}],\iota_{\nu_0(y)}]^{(k)}\in\mathfrak{X}_{-1}(\mathbb{U}^\mathfrak F),\; \text{for}\; x,y\in \mathfrak g.\end{align*}
Note that in this case the action $\varrho$ is equivalent to zero, therefore by item 3 of Theorem \ref{main} the Lie $\infty$-morphism $\Phi$ is homotopic to zero. 
\end{example}

\section{Lifts of weak symmetry actions and Lie $\infty$-algebroids}\label{sec:4}

In this section, $\mathfrak{g}$ is a finite dimensional Lie algebra that we see as the trivial vector bundle over $M$  with fiber $\mathfrak{g}$. 

The following theorem says that any lift of strict symmetry action of $\mathfrak{g}$ on a singular foliation $\mathfrak{F}$ induces a Lie $\infty$-algebroids with some special properties and vice versa. See \cite{MEHTA2012576}, Prop. 3.3, for a proof of the following statement.
\begin{proposition}\label{alt:thm-res}Let $(E,Q)$ be a Lie $\infty$-algebroid over a singular foliation $\mathfrak F$. 
Any Lie $\infty$-morphism 
  $\Phi\colon(\mathfrak{g},\lb_\mathfrak{g})\rightsquigarrow (\mathfrak X_\bullet(E)[1],\lb,\mathrm{ad}_{Q})$ with $\mathfrak{g}$ of finite dimension
  induces a Lie $\infty$-algebroid  $(E\oplus\mathfrak{g},Q')$ with \begin{equation}\label{def:Q}
  Q':=\dd^{\text{CE}}+ Q + \sum_{
k\geq 1,i_1,\ldots,i_k=1,\ldots,\mathrm{dim}(\mathfrak{g})}\frac{1}{k!}\xi^{i_1}\odot\cdots\odot\xi^{i_k}\Phi_{k-1}(\xi_{i_1},\ldots,\xi_{i_{k}}),
\end{equation}
where $\dd^{\text{CE}}$ is the Chevalley-Eilenberg complex of $\mathfrak{g}$, and   $\xi^{1},\ldots,\xi^{\mathrm{dim}(\mathfrak{g})}\in\mathfrak{g}^
*$ is the dual basis of some basis $\xi_{1},\ldots,\xi_{\mathrm{dim}(\mathfrak{g})}\in\mathfrak{g}$ and for all $k\geq 0$, $\Phi_k\colon S^{k+1}\mathfrak{g}\longrightarrow \mathfrak{X}_{-k}(E)[1]$ is the $k$-th Taylor coefficients of $\Phi$.\\

In the dual point of view, \eqref{def:Q} corresponds to a Lie $\infty$-algebroid over the complex
\begin{equation}\label{semi-resol1}
 \begin{array}{c} \cdots\stackrel{\ell_1}{\longrightarrow}  E_{-3}\stackrel{\ell_1}{\longrightarrow}  E_{-2} \stackrel{\ell_1}{\longrightarrow}\mathfrak g\oplus E_{-1} \stackrel{\rho'}{\longrightarrow}TM\end{array}
\end{equation}
whose brackets satisfy
\begin{enumerate}\item\label{item_1} the anchor map $\rho'$ sends an element $x\oplus e\in\mathfrak{g}\oplus E_{-1}$ to $\varrho(x)+\rho(e)\in\varrho(\mathfrak g)+T\mathfrak F$,\item \label{item_2} the binary bracket satisfies $$\ell_2\left(\Gamma(E_{-1}),\Gamma(E_{-1})\right)\subset \Gamma(E_{-1})\quad \text{and}\quad\ell_2(\Gamma(E_{-1}),x)\subset \Gamma(E_{-1}),\; \forall\, x\in\mathfrak{g}$$\item \label{item_3}the $\mathfrak g$-component of the binary bracket on constant sections of $\mathfrak{g}\times M$ is the Lie bracket of $\mathfrak{g}$.\end{enumerate}Conversely, if there exists a Lie $\infty$-algebroid $(E',Q')$ whose underlying complex of vector bundles is of the form $\eqref{semi-resol1}$ and that satisfies item \ref{item_1}, \ref{item_2} and \ref{item_3}, then there is a Lie $\infty$-morphism $$\Phi\colon(\mathfrak{g},\lb_\mathfrak{g})\rightsquigarrow \left( \mathfrak X_\bullet(E)[1],\lb, \emph{ad}_{Q} \right)$$  which is defined on a given basis $\xi_1,\ldots,\xi_d$ of $\mathfrak{g}$ by:
\begin{equation}\label{expl:lift}
\Phi_{k-1}(\xi_{i_1},\ldots,\xi_{i_k})=\emph{pr}\circ[\cdots[[Q',\iota_{\xi_{i_1}}],\iota_{\xi_{i_2}}],\ldots,\iota_{\xi_{i_k}}]\subset\mathfrak{X}(E)[1],
\;k\in\mathbb{N},\end{equation}
where $\emph{pr}$ stands for the projection map $\mathfrak{X}{(E')}[1]\longrightarrow \mathfrak{X}{(E)[1]}$. 
\end{proposition}


\begin{proof}We explain the idea of the proof. A direct computation gives the first implication. Conversely, let us denote by $Q'$ the homological vector fields of Lie $\infty$-algebroid whose underlying complex of vector bundles is of the form $\eqref{semi-resol1}$. The map defined in Equation \eqref{expl:lift} is indeed a lift into a Lie $\infty$-morphism of the weak symmetry action $\varrho$:

\begin{itemize}
    \item It is not difficult to check that, for any $\xi\in\mathfrak{g}$, one has $[Q,\Phi_0(\xi)]=0$.
    \item The fact that $\Phi$ defines a Lie $\infty$-morphism can be found using Voronov trick \cite{VoronovTheodore, Voronov}, i.e, doing Jacobi's identity inside the null derivation\begin{equation}\label{trick}
        0=\text{pr}\circ[\cdots[[[Q',Q'],\iota_{\xi_{i_1}}],\iota_{\xi_{i_2}}],\ldots,\iota_{\xi_{i_k}}].
    \end{equation}
\end{itemize}
A direct computation of Equation \eqref{trick} falls exactly on the requirements of Definition \ref{def:morph2}.

\end{proof}
\begin{remark}
Let us compute Equation \eqref{trick} for a small number of generators (e.g $k=2,3$) in order to show how it works: from the identity $$\left[\left[\left[Q',Q'\right],\iota_{\xi_{i_1}}\right],\iota_{\xi_{i_2}}\right]=0,$$ one obtains by using twice the Jacobi identity the following relation,
\begin{equation}\label{trick:rel}
    \left[Q',\left[\left[Q',\xi_{i_1}\right],\xi_{i_2}\right] \right]-\left[\left[Q',\xi_{i_1}\right],\left[Q',\xi_{i_2}\right]\right]=0.
\end{equation}
\noindent
One should notice that $\left[\left[Q',\xi_{i_1}\right],\xi_{i_2}\right]$ splits into two parts. One part  where the Chevalley-Eilenberg acts to give $\left[\left[\dd^{\text{CE}},\xi_{i_1}\right],\xi_{i_2}\right]=\iota_{[\xi_{i_1},\xi_{i_2}]_\mathfrak{g}}$, while the other part is $\left[\left[Q'-\dd^{\text{CE}},\xi_{i_1}\right],\xi_{i_2}\right]$. Hence, by putting them in Equation \eqref{trick:rel}, afterwards projecting on $\mathfrak{X}\left(S^\bullet (E^*)\right)$, we get\begin{align*}
    &\text{pr}\circ\left[Q',\iota_{[\xi_{i_1},\xi_{i_2}]_\mathfrak{g}}\right]+\text{pr}\circ[Q',\left[\left[Q'-\dd^{\text{CE}},\xi_{i_1}\right],\xi_{i_2}\right]]-\text{pr}\circ\left[\left[Q',\xi_{i_1}\right],\left[Q',\xi_{i_2}\right]\right]=0.
\end{align*}
  From here, we deduce that \begin{equation*}
      \Phi_0([\xi_{i_1},\xi_{i_2}]_\mathfrak{g})=[Q,\Phi_1(\xi_{i_1},\xi_{i_2})]+[\Phi_0(\xi_{i_1}),\Phi_0(\xi_{i_2})].
  \end{equation*}  
\end{remark}
Here is an application of Theorem \ref{alt:thm-res}.
\begin{corollary}\phantom{}
\begin{enumerate}
    \item Let $\mathfrak{g}$ be a Lie algebra and $G$ its Lie group.
    \item Let $(M, \mathfrak{F})$ a singular foliation together with a weak symmetry action $\varrho\colon\mathfrak{g}\longrightarrow \mathfrak{X}(M)$. 
\end{enumerate}The following assertions hold:
\begin{enumerate}
    \item On $G\times M$ the $C^\infty(G\times M)$-module generated by 
   \begin{equation}
        \begin{cases}
               (\overleftarrow{u}, \varrho(u))  & u\in \mathfrak{g}\\
            (0, X) & X\in\mathfrak{F} 
    \end{cases}
   \end{equation}
is a singular foliation that we denote by $(\mathfrak{X}(G)\times_\varrho \mathfrak{F})$. 
\item If $(E, \dd, \varrho)$ is a geometric resolution of $(M, \mathfrak{F})$, then \begin{equation}\label{eq:MxG}
    \cdots\stackrel{}{\longrightarrow} p^*E_{-3}\stackrel{\dd}{\longrightarrow} p^*E_{-2}\stackrel{\dd}{\longrightarrow} \mathfrak{g}\oplus p^*E_{-1}\stackrel{\rho'}{\longrightarrow} T(G\times M)
\end{equation}
    with $\rho'(e, u)=\left(\overleftarrow{u}, \varrho(u)+ \rho(e)\right)$, is a geometric resolution of $(\mathfrak{X}(G)\times_\varrho \mathfrak{F})$. Here, $p\colon G\times M\longrightarrow M$ is the projection on $M$.
\end{enumerate}
Let $(E,Q)$ be a universal Lie $\infty$-algebroid structure of $(M, \mathfrak{F})$. Let $\Phi_k\colon \wedge^{k+1}\mathfrak{g}\longrightarrow \mathfrak{X}_{-k}(E)$ be the Taylor coefficients of a Lie $\infty$-morphism $\mathfrak{g}\rightsquigarrow \mathfrak{X}(E)$ that lifts $\varrho$. Then 

\begin{equation}\label{def:Q2}
  Q':=\dd^{\text{G}}_{dR}+ Q + \sum_{
k\geq 1,i_1,\ldots,i_k=1,\ldots,\mathrm{dim}(\mathfrak{g})}\frac{1}{k!}\xi^{i_1}\odot\cdots\odot\xi^{i_k}\Phi_{k-1}(\xi_{i_1},\ldots,\xi_{i_{k}}),
\end{equation}
with  $(\xi_{1},\ldots,\xi_{\mathrm{dim}(\mathfrak{g})})$, $(\xi^{1},\ldots,\xi^{\mathrm{dim}(\mathfrak{g})})$ be dual basis of $\mathfrak{g}$ and $\mathfrak{g}^*$
\begin{enumerate}
    \item is a universal Lie $\infty$-algebroid of  $(\mathfrak{X}(G)\times_\varrho \mathfrak{F})$
    \item whose coefficients are left invariant for the action of $G$ on $G\times M$ given by $g\cdot(h,m):=(gh,m)$.
\end{enumerate}
Conversely, a left invariant Lie $\infty$-algebroid on \eqref{eq:MxG} can be interpreted as Taylor coefficients of a lift of $\varrho$.
\end{corollary}

\subsection{A more general statement of Proposition \ref{alt:thm-res}}
We end the chapter with a generalization of Proposition \ref{alt:thm-res}. \\

In the previous section, Proposition \ref{alt:thm-res} is stated in the finite dimensional context, i.e. it needs $\mathfrak g$ to be finite dimensional and the existence of a geometric resolution for the singular foliation $\mathfrak{F}$. In this section we prove that given a weak symmetry action of a Lie algebra $\mathfrak{g}$ (may be of infinite dimensional)  on a Lie-Rinehart algebra $\mathfrak{F}\subset \mathfrak{X}(M)$ (we do not require $\mathfrak{F}$ being locally finitely generated), such Lie $\infty$-algebroid described at the sections level of the complex \eqref{semi-resol1} as Proposition \ref{alt:thm-res} exists.\\ 

We state the following theorem in the context of singular foliations, but the same statement and the same proof are valid word-by-word by replacing $\mathfrak{F}$ by a Lie-Rinehart algebra.
\begin{theorem}\label{alt-thm-res}Let $\mathfrak{g}$ be a (possibly infinite dimensional) Lie $\mathbb{K}$-algebra  and let  $\varrho\colon\mathfrak{g}\longrightarrow \mathfrak{X}(M)$ be a weak symmetry action of $\mathfrak{g}$ on  a singular foliation $\mathfrak F$. Let $\left((\mathcal{K}_{-i})_{i\geq 1},\dd, \rho\right)$ be a free resolution of the singular foliation $\mathfrak F$ over $\mathcal{O}$. 
The complex of trivial vector bundles over $M$
\begin{equation}\label{semi-resol}
  \begin{array}{c} \cdots\stackrel{\dd}{\longrightarrow}  E_{-3}\stackrel{\dd}{\longrightarrow}  E_{-2} \stackrel{\dd}{\longrightarrow}\mathfrak g\oplus E_{-1} \stackrel{\rho'}{\longrightarrow}TM
    \end{array}
\end{equation}
where $\Gamma(E_{-1})=\mathcal{K}_{-i}$, comes equipped with a Lie $\infty$-algebroid structure
\begin{enumerate}
\item whose unary bracket is $\dd$ and whose anchor map $\rho'$, sends an element $x\oplus e\in\mathfrak{g}\oplus E_{-1}$ to $\varrho(x)+\rho(e)\in\varrho(\mathfrak{g})+T\mathfrak F$,
\item the binary bracket satisfies $$\ell_2\left(\Gamma(E_{-1}),\Gamma(E_{-1})\right)\subset \Gamma(E_{-1})\quad \text{and}\quad\ell_2(\Gamma(E_{-1}),\Gamma(\mathfrak{g}))\subset \Gamma(E_{-1}),$$
\item the $\mathfrak g$-component of the binary bracket on constant sections of $\mathfrak{g}\times M$ is the Lie bracket of $\mathfrak{g}$.
\end{enumerate}
\end{theorem}

\begin{remark}
 When we have $\varrho(\mathfrak{g})\cap T_m\mathfrak F=\{0\}$ for all $m$ in $M$, Equation \eqref{semi-resol} is a free  resolution of the singular foliation $C^{\infty}(M)\varrho(\mathfrak{g})+\mathfrak F$ and we can apply directly the Theorem \ref{thm:existence}. Otherwise, we need to show there is no obstruction in degree $-1$ while doing the construction of the brackets if the result still needs to hold. 
\end{remark}

\begin{proof}{(of Theorem \ref{alt-thm-res})} The complex of Equation \eqref{semi-resol} being exact everywhere except in degree $-1$ we cannot apply directly Theorem 2.1 in \cite{CLRL} but we can mimic the proof given for Theorem 2.1 in \cite{CLRL} to construct the higher brackets when there is no obstruction in degree $-1$.
\noindent
 For convenience, let us denote $\mathcal R_{-1}:=\Gamma({\mathfrak{g}})\oplus\Gamma(E_{-1})$ and $\mathcal R_{-i}:=\Gamma(E_{-i})$ for $i\geq 2$. Given a natural number $k\geq 0$, we consider the total complex $\left(\widehat{\mathfrak{Page}}^{(k)}_\bullet(\mathcal R),D=[\dd,\cdot]_{\text{RN}}\right)$ of the following bicomplex

\begin{equation}\label{eqbicomplex}
\scalebox{0.8}{\hbox{$ 
	\begin{array}{ccccccccccc}
		& & \vdots & & \vdots & & \vdots & & 
		\\ 
		& & \uparrow & & \uparrow & & \uparrow & & 
		\\ 
		\cdots& \rightarrow & \text{Hom}_\mathcal O\left(\bigodot^{k+1} \mathcal R\,_{|_{-k-3}},\mathcal R_{-3}\right) & \overset{\dd}{\rightarrow} & \text{Hom}_\mathcal O\left(\bigodot^{k+1} \mathcal R\,_{|_{-k-3}},\mathcal R_{-2}\right) 
		& \overset{\dd}{\rightarrow} & \text{Hom}_\mathcal O\left(\bigodot^{k+1} \mathcal R\,_{|_{-k-3}},\dd\mathcal R_{-2}\right) & \rightarrow & 0
		\\ 
		& & \delta\uparrow & & \delta\uparrow & & \delta\uparrow & & 
		\\ 
		\cdots& \rightarrow & \text{Hom}_\mathcal O\left(\bigodot^{k+1} \mathcal R\,_{|_{-k-2}},\mathcal R_{-3}\right) & \overset{\dd}{\rightarrow} & \text{Hom}_\mathcal O\left(\bigodot^{k+1} \mathcal R\,_{|_{-k-2}},\mathcal R_{-2}\right) 
		& \overset{\dd}{\rightarrow} & \text{Hom}_\mathcal O\left(\bigodot^{k+1}\mathcal R\,_{|_{-k-2}},\dd\mathcal R_{-2}\right) & \rightarrow& 0
		\\ 
		& & \delta\uparrow & & \delta\uparrow & & \delta\uparrow & &  
		\\ 
		\cdots& \rightarrow & \text{Hom}_\mathcal O\left(\bigodot^{k+1} \mathcal R\,_{|_{-k-1}},\mathcal R_{-3}\right) & \overset{\dd}{\rightarrow} & \text{Hom}_\mathcal O\left(\bigodot^{k+1}\mathcal R\,_{|_{-k-1}},\mathcal R_{-2}\right) 
		& \overset{\dd}{\rightarrow}& \text{Hom}_\mathcal O\left(\bigodot^{k+1}\mathcal R\,_{|_{-k-1}},\dd\mathcal R_{-2}\right) & \rightarrow & 0
	\\
	& & \uparrow & & \uparrow & & \uparrow & & 
	\\ 
	& & 0 & & 0 & & 0 & &  \\
	& & \hbox{\small{\texttt{"-$3$ column"}}} & & \hbox{\small{\texttt{"-$2$ column"}}} & & \hbox{\small{\texttt{"-$1$ column"}}} & &  	
\end{array}$}}
\end{equation}
The map $\delta$ stands for the vertical differential which is defined for all $\Phi\in\text{Hom}_\mathcal O\left(\bigodot^{k+1} \mathcal R,\mathcal R\right)$  by
$$\delta(\Phi) \, (r_1,\ldots,r_{k+1}):=\Phi\circ\dd \, (r_1\odot\ldots\odot r_{k+1}),\hspace{1cm} \forall \,r_1,\dots, r_{k+1}\in \mathcal R,$$
where here $\dd$ acts as an $\mathcal O$-derivation on $r_1\odot\ldots\odot r_{k+1}\in\bigodot^k\mathcal R$ and the horizontal differential given by 
 $$\Phi \mapsto \dd \circ \Phi .$$

Since the line the bicomplex is exact, the total complex $\left(\widehat{\mathfrak{Page}}^{(k)}_\bullet(\mathcal R),D=[\dd,\cdot]_{\text{RN}}\right)$ is also exact.

\noindent 
\textbf{Construction of the $2$-ary bracket}: its construction is almost the same as  in \cite{CLRL} we adapt what has been done to our case. We first construct a $2$-ary bracket on $\mathcal R_{-1}$ to extend on every degree. For all $ k \geq 1$, let us denote by 
$(e_i^{(-k)})_{i \in I_k}$
a basis of $\Gamma(E_{-k})$. The set $\{X_i=\rho(e_i^{(-1)})\in \mathfrak F\mid i\in I_{1}\}$ is a set of generators of $\mathfrak F$. In particular, there exists elements $c^k_{ij}\in\mathcal{O}$
and satisfying the skew-symmetry condition $c^k_{ij}=-c^k_{ji}$ together with
\begin{equation} 
\left[ X_i,X_j\right]=\sum_{k\in I}c^k_{ij} X_k \hspace{.5cm} \forall i,j \in I_{1}.
\end{equation}
By definition of weak symmetry one has \begin{equation}
  [\varrho(\xi_i),\rho(e_j^{(-1)})]\in \mathfrak F\quad\text{and}\quad \varrho([\xi_i,\xi_j])_\mathfrak{g}-[\varrho(\xi_i),\varrho(\xi_j)]\in\mathfrak F\;\;\text{for all}\; (i,j)\in I_{\mathfrak{g}}\times I_{-1}.  
\end{equation}
Here, $(\xi_i)_{i\in I_\mathfrak{g}}$ is a basis for $\mathfrak{g}$. Since $\left((\mathcal{K}_{-i})_{i\geq 1},\dd, \rho\right)$ is a free resolution of $\mathfrak F$, there exists two $\mathcal O$-bilinear maps $\chi\colon\Gamma({\mathfrak{g}})\times \Gamma (E_{-1})\to \Gamma(E_{-1})$,  $\eta\colon\Gamma({\mathfrak{g}})\times\Gamma({\mathfrak{g}})\to \Gamma(E_{-1})$ defined on generators $\xi_i, e_j^{(-1)}$ by the relations\begin{equation*}
    [\varrho(\xi_i),\rho(e_j^{(-1)})]=\rho(\chi(\xi_i,e_j^{(-1)}))\qquad\text{and}\qquad \varrho([\xi_i,\xi_j]_\mathfrak{g})-[\varrho(\xi_i),\varrho(\xi_j)]=\rho(\eta(\xi_i,\xi_j)).
\end{equation*}
We first define a naive $2$-ary bracket on $\Gamma(E_{-1})$ as follows: 
\begin{enumerate}
\item an anchor map by $\rho'(e_i^{(-1)})=X_i$, and $\rho'(\xi_i)=\varrho(\xi_i)$, for all $i\in I,I_{\mathfrak{g}}$,
\item a degree $ +1$ graded symmetric operation $\tilde{\ell}_2 $ on $ \mathcal R_{\bullet}$ as follows:
\begin{enumerate}
    \item $\tilde{\ell}_2\left(e_i^{(-1)},e_j^{(-1)} \right) =\sum_{k\in I}c^k_{ij}e_k^{(-1)}$ for all $i,j\in I_{-1}$,
    \item $\tilde{\ell}_2\left(\xi_i,e_j^{(-1)} \right)=\chi\left(\xi_i,e_j^{(-1)}\right)$,
    \item $\tilde{\ell}_2\left(\xi_i,\xi_j\right)=[\xi_i,\xi_j]_{\mathfrak{g}}+\eta(\xi_i,\xi_j)$,
    \item $\widetilde{\ell}_2$ is zero on the other generators,
\item we extend $\tilde{\ell}_2$ to $\mathcal R $ using $\mathcal O$-bilinearity and Leibniz identity with respect to the anchor $\rho'$.
\end{enumerate} 
\end{enumerate} 

By $(a),(b),(c),(d),(e)$, $\tilde{\ell}_2$ satisfies the Leibniz identity with respect to the anchor $\widetilde{\rho}$ and $(a),(b),(c)$  makes the latter a bracket morphism. The map defined for all homogeneous $r_1,r_2\in \mathcal{R_\bullet}$ by\begin{equation}
  [\dd, \tilde{\ell}_2]_{\hbox{\tiny{RN}}}(r_1,r_2) = \dd \circ \tilde \ell_2 \left( r_1,r_2\right)+ \tilde \ell_2 \left(\dd r_1,r_2 \right) +(-1)^{\lvert r_1\rvert} \tilde \ell_2 \left(r_1,\dd r_2 \right),   
\end{equation}
is a graded symmetric degree $ +2$ operation $ (\mathcal R \otimes \mathcal R)_\bullet \longrightarrow \mathcal R_{\bullet +2} $, and $[\dd, \tilde{\ell}_2]_{{\hbox{\tiny{RN}}}_{|_{\mathcal R_{-1}}}}=0$. It is $\mathcal O$-bilinear, i.e. for all $f \in \mathcal O, r_1,r_2 \in \mathcal R$
 $$[\dd, \tilde{\ell}_2]_{\hbox{\tiny{RN}}}(r_1,f r_2 )- f[\dd, \tilde{\ell}_2](r_1,r_2 )=0.$$ We also have that $\rho([\dd, \tilde{\ell}_2]_{\hbox{\tiny{RN}}}(r_1,f r_2 ))=\rho(\tilde{\ell}_2(\dd r_1,r_2))=0$, for all $r_1\in\mathcal R_{-2}, r_2\in\mathcal R_{-1}$, since $\rho\circ\dd=0$. Thus, ${[\dd, \tilde{\ell}_2]_{\hbox{\tiny{RN}}}}_{|\mathcal R_{-2}\times\mathcal R_{-1}}\in \dd\mathcal R_{-2}$, because $\left((E_{-i})_{i\geq 1},\dd, \rho\right)$ is a geometric resolution.\\
 
 \noindent
 Therefore, $ [\dd, \tilde{\ell}_2]_{\hbox{\tiny{RN}}}$ is a degree $+2$ element in the total complex $\widehat{\mathfrak{Page}}^{(1)}(\mathcal R)$. The $\mathcal O$-bilinear operator $[\dd, \tilde{\ell}_2]_{\hbox{\tiny{RN}}}$ is $D$-closed in  $\widehat{\mathfrak{Page}}^{(1)}(\mathcal R)$, since $[ \dd, [\dd, \tilde{\ell}_2]_{\hbox{\tiny{RN}}}]_{{\hbox{\tiny{RN}}}_{|_{\mathcal R_{\leq -2}}}}=0$. So there exists  $\tau_2 \in \oplus_{j\geq 2} \text{Hom}_\mathcal O\left(\bigodot^2\mathcal R_{-j-1},\mathcal R_{-j}\right)$ such as $D(\tau_2)= -[\dd, \tilde{\ell}_2]_{\hbox{\tiny{RN}}}.$  By replacing  $ \tilde{\ell}_2$  by $\tilde{\ell}_2+\tau_2$ we get a $2$-ary bracket $\ell_2$ of degree $+1$ which is compatible with the differential map $\dd$ and the anchor map $\widetilde{\rho}$. \\
 
 \noindent
\textbf{Construction of higher brackets}: notice that by construction of the $2$-ary bracket $\ell_2$ one has, $\text{Jac}(r_1,r_2,r_3)\in\dd\mathcal R_{-2}$ for all $r_1,r_2,r_3\in\mathcal R_{-1}$. In other words, $\text{Jac}\in\text{Hom}_\mathcal O(\bigodot^3\mathcal{R}_{-1},\dd\mathcal R_{-2})$. A direct computation shows\begin{equation*}
\dd \text{Jac}(r_1,r_2,r_3)=\text{Jac}(\dd r_1,r_2,r_3)+(-1)^{\lvert r_1\rvert}\text{Jac}(r_1,\dd r_2,r_3)+(-1)^{\lvert r_1\rvert+\lvert r_2\rvert}\text{Jac}(r_1,r_2,\dd r_3)
\end{equation*}\text{for all}\;$r_1,r_2,r_3 \in \mathcal R$. Which means, $[\text{Jac},\dd]_{\hbox{\tiny{RN}}}(r_1,r_2,r_3)=0$ \text{for all}\;$r_1,r_2,r_3 \in \mathcal R$.

\noindent
Thus, $D(\text{Jac})=0$. It follows that,  $\text{Jac}$ is a $D$-coboundary, there exists an element $\ell_3=\sum_{j\geqslant 2} \ell_3^{j}\in\widehat{\mathfrak{Page}}_1^{(2)}(\mathcal R)$ with $\ell_3^{j}\in\text{Hom}(\bigodot^3\mathcal R\, _{|_{-j-1}},\mathcal R_{-j})$ such that \begin{equation}
D(\ell_3) =-\text{Jac}.
\end{equation}
We choose the $3$-ary bracket to be $\ell_3$. For degree reason, the remaining terms of the $k$-ary brackets for $k\geq 3$ have trivial components on the column $-1$ of the bicomplex \eqref{eqbicomplex}. From this point, the proof continues exactly as in Section \ref{thm:existence-proof}.
\end{proof}

\begin{example}\label{ex:F-connection-lift}
We return to Example \ref{ex:F-connection}. In that case, the Lie algebra $\mathfrak{g}=\mathfrak{X}(L)$ that acts on the singular foliation $\mathcal{T}$. In that case, we have $\varrho(\mathfrak{g})\cap T_m\mathcal T=\{0\}$ for all $m$ in $[L,M]$. We can apply directly Theorem \ref{thm:existence}, to obtain a Lie $\infty$-algebroid structure on the complex 
\begin{equation}
  \begin{array}{c} \cdots\stackrel{\dd}{\longrightarrow}  \Gamma(E_{-3})\stackrel{\dd}{\longrightarrow}  \Gamma(E_{-2}) \stackrel{\dd}{\longrightarrow}\mathfrak g\oplus \Gamma(E_{-1}) \stackrel{\rho'}{\longrightarrow}\mathfrak{X}([L,M]).
    \end{array}
\end{equation}
Notice that here the Lie algebra $\mathfrak g$ is infinite dimensional, therefore we are not allowed to use the duality between Lie $\infty$-algebroid and $Q$-manifold. Therefore, we cannot use the explicit Formula \eqref{expl:lift} to define at lift $\Phi$. However, we have to rely on the existence theorem \ref{main} to assure the existence of a lift $\Phi$.
\end{example}

\vspace{2cm}

\begin{tcolorbox}[colback=gray!5!white,colframe=gray!80!black,title=Conclusion:]
We show that actions of a Lie algebra $\mathfrak g$ on the leaf space $M/\mathfrak{F}$ i.e. weak symmetry actions, lift to Lie $\infty$-algebra morphisms $\mathfrak{g}\rightsquigarrow\mathfrak{X}(E)$ on the DGLA of  vector fields on an universal Lie $\infty$-algebroid $(E,Q)$, provided it exists, in a unique up to homotopy manner.\\

We explain how to  use chapter \ref{Chap:main} to get rid of all finite ranks/dimensions assumptions, using the universal  Lie $\infty$-algebroids of Lie-Rinehart algebras.
\end{tcolorbox}

\chapter{On weak and strict symmetries: an obstruction theory}\label{sec:5}\label{chap:obstruction-theory}
In this chapter, we apply the main theorems of \ref{sec:3} of Chapter \ref{chap:symmetries}
to define a class obstructing the existence of strict symmetry action equivalent to a given weak symmetry action.\\

\section{Introduction}
Recall that Theorem \ref{main} assures that any weak symmetry action $\varrho\colon\mathfrak g\rightarrow \mathfrak{X}(M)$ of a Lie algebra $\mathfrak g$ on a singular foliation $\mathfrak F$ admits a lift to a Lie $\infty$-morphism \begin{equation}
    \label{eq:Lie-morp-obstruction-theory}\Phi\colon (\mathfrak{g},\lb_\mathfrak{g})\rightsquigarrow (\mathfrak X_\bullet(E)[1],\lb,\text{ad}_{Q}).
\end{equation} To understand the compatibility conditions between  the low terms of $\Phi$ see Remark \ref{rk:low-terms-Lie-infty}. By playing with the definition of $\Phi$ we make the following observation. 
\begin{remark}\label{rmk:rk:low-dual-terms-Lie-infty} What does $\Phi$ induces on the linear part of $(E,Q)$?  We have seen in Lemma \ref{lemma:basic-action} and Proposition \ref{prop:induced-action} that the $0$-th Taylor coefficient of the lift $\Phi$ induces a  linear map
$x\in\mathfrak{g}\longmapsto \left(\nabla_x\colon E_{-i}\longrightarrow E_{-i}\right)$ for every $i\geq 1$   that satisfies $$\nabla_x(fe)=~f\nabla_x(e) + \varrho(x)[f]e,\;\; \text{for}\;\; f\in \mathcal{O}, e\in \Gamma(E).$$
\begin{enumerate}

\item 

Also, for every $x\in \mathfrak{g}$ and $e\in \Gamma(E_{-1})$, $$\rho(\nabla_x(e))=[\varrho(x), \rho(e)].$$ 
\item A direct computation gives
\begin{align*}
   0=\left[\left[Q,\Phi_0(x)\right], \iota_e\right]^{(-1)}&=\left[Q,\left[\Phi_0(x), \iota_e\right]\right]^{(-1)}-\left[\Phi_0(x), \left[Q,\iota_e\right]\right]^{(-1)}\\&\\&=\left[Q^{(0)},\left[\Phi_0(x), \iota_e\right]^{(-1)}\right]-\left[\Phi_0(x)^{(0)}, \left[Q,\iota_e\right]^{(-1)}\right]\\&\\&=\left[Q^{(0)},\iota_{\nabla_x(e)}\right]-\left[\Phi_0(x)^{(0)}, \iota_{\ell_1(e)}\right]\\&=\iota_{\ell_1\circ\nabla_x(e)}-\iota_{\nabla_x\circ\ell_1(e)}
\end{align*}
We have used graded Jacobi identity and the dual correspondence between Lie $\infty$-algebroids and $NQ$-manifolds (see Proposition \ref{prop:dual}). We recapitulate in the following commutative diagram:

\begin{equation}\label{diagram:1}
    \xymatrix{ \cdots\ar[r]^{\dd}&  \ar[r]^{\dd}\Gamma(E_{-2})\ar@<3pt>[d]^{\nabla_x}&  \ar[r]^{\rho} \Gamma(E_{-1})\ar@<3pt>[d]^{\nabla_x} &  \mathfrak F \ar@<3pt>[d]^{\text{ad}_{\varrho(x)}} &  \\\cdots\ar[r]^{\dd}& \ar[r]^{\dd}  \Gamma(E_{-2})&  \Gamma(E_{-1})\ar[r]^{\rho}  &\mathfrak F  } 
\end{equation}which means,
 $$\ell_1\circ\nabla_x=   \nabla_x\circ \ell_1\quad\text{and}\quad  \rho\circ\nabla_x=\text{ad}_{\varrho(x)}\circ \rho.$$
Here, $\ell_1$ stands for the corresponding unary bracket of $(E, Q)$. Also, for $X\in \mathfrak{X}(M)$, $\text{ad}_{X}:=[X,\,\cdot\,]$.\\

\item
For $x,y\in\mathfrak{g}$, and $e\in\Gamma(E)$, the relation \eqref{eq:term2} yields
\begin{align*}
\left[\Phi_0([x,y]_{\mathfrak{g}})-\left[\Phi_0(x),\Phi_0(y)\right],\iota_e\right]^{(-1)}&=\left[\left[Q,\Phi_1(x,y)\right], \iota_e\right]^{(-1)}\\&\\\iota_{\nabla_{[x,y]_\mathfrak{g}}(e)}-[\Phi_0(x),[\Phi_0(y),\iota_e]]^{(-1)}+[\Phi_0(y),[\Phi_0(x),\iota_e]]^{(-1)}&=\qquad {}^{"}\\&\\\iota_{\nabla_{[x,y]_\mathfrak{g}}(e)}-\left[\Phi_0(x)^{(0)}, \iota_{\nabla_y(e)} \right]+\left[\Phi_0(y)^{(0)}, \iota_{\nabla_x(e)} \right]&=\left[\left[Q,\Phi_1(x,y)\right], \iota_e\right]^{(-1)}.\end{align*}
Thus \begin{equation}\label{eq:1}
    \left[\left[Q,\Phi_1(x,y)\right], \iota_e\right]^{(-1)}=\iota_{\nabla_{[x,y]_\mathfrak{g}}(e)}- \iota_{\nabla_x\circ \nabla_y(e)}+\iota_{\nabla_y\circ \nabla_x(e)}.
\end{equation}

On the other hand, $\Phi_1(x,y)$ admits a polynomial decomposition$$\Phi_1(x,y)^{(-1)}+\sum_{i\geq 0}\Phi_1(x,y)^{(i)}= \iota_{\eta(x,y)}+\sum_{i\geq 0}\Phi_1(x,y)^{(i)}$$ and that $$\displaystyle{\left[\sum_{i\geq 0}\Phi_1(x,y)^{(i)}, \iota_e\right]^{(-1)}=\left[\Phi_1(x,y)^{(0)}, \iota_e\right]}=\iota_{\gamma(x,y)(e)},$$ for some linear map $\gamma(x,y)\colon \Gamma(E_{-\bullet})\longrightarrow \Gamma(E_{|e|-1})$ depending linearly on $x, y$. Therefore, 
\begin{align*}
    \left[\left[Q,\Phi_1(x,y)\right], \iota_e\right]^{(-1)}&=\left[\left[Q,\iota_{\eta(x,y)}\right], \iota_e\right]^{(-1)}+\left[\left[Q,\sum_{i\geq 0}\Phi_1(x,y)^{(i)}\right], \iota_e\right]^{(-1)}\\&=\iota_{\ell_2(\eta(x,y),e)}+\left[Q,\left[\sum_{i\geq 0}\Phi_1(x,y)^{(i)}, \iota_e\right]\right]^{(-1)}-\left[\sum_{i\geq 0}\Phi_1(x,y)^{(i)},\left[Q, \iota_e\right]\right]^{(-1)}\\&=\iota_{\ell_2(\eta(x,y),e)}+\left[Q^{(0)},\left[\sum_{i\geq 0}\Phi_1(x,y)^{(i)}, \iota_e\right]^{(-1)}\right]-\left[\Phi_1(x,y)^{(0)},\left[Q, \iota_e\right]^{(-1)}\right]\\&=\iota_{\ell_2(\eta(x,y),e)}+\left[Q^{(0)}, \iota_{\gamma(
    x,y)(e)}\right]-\left[\Phi_1(x,y)^{(0)}, \iota_{\ell_1(e)}\right]\\&=\iota_{\ell_2(\eta(x,y),e)}+\iota_{\ell_1(\gamma(
    x,y)(e))}-\iota_{\gamma(x,y)(\ell_1(e))}
\end{align*}
By equating the latter with \eqref{eq:1} we can recapitulate as follows:\\

\noindent
\textbf{Conclusion}: In general, the map\;$\mathfrak g\longrightarrow \mathrm{Der}(E),\,x\mapsto \nabla_x$ is not a Lie algebra morphism  even when the action $\varrho$ is strict. In fact, there exists a bilinear map $\gamma\colon \wedge^2\mathfrak g\longrightarrow \mathrm{End}(E)[1]$ of degree ~$0$ that satisfies

$$\nabla_{[x,y]_\mathfrak g}-[\nabla_x, \nabla_y]=\gamma(x,y)\circ\ell_1- \ell_1\circ \gamma(x,y)+ \ell_2(\eta(x,y),\cdot\,),$$here $\ell_2$ is the corresponding $2$-ary bracket of $(E,Q)$, and  $\eta\colon\wedge^2\mathfrak g\longrightarrow \Gamma(E_{-1})$ is such that $\varrho([x,y]_\mathfrak{g})-[\varrho(x),\varrho(y)]=\rho(\eta(x,y))$.\\

\item Also, Equation \eqref{eq:low-term2} taken in polynomial-degree $-1$ implies, that, for all $\alpha\in \Gamma(E_{-1}^*)$
\begin{align*}
&\Phi_1([x,y]_\mathfrak g,z)^{(-1)} -[\Phi_0(x)^{(0)}, \Phi_1(y,z)^{(-1)}]+\circlearrowleft(x,y,z)=[Q^{(0)}, \Phi_2(x,y,z)^{(-1)}].\\&\\\Longrightarrow&\;\iota_{\eta([x,y]_\mathfrak g,z)}-[\Phi_0(x)^{(0)}, \iota_{\eta(y,z)}]+\circlearrowleft(x,y,z)=[Q^{(0)},\iota_{\zeta(x,y,z)}],\quad \text{with}\quad\zeta\colon \wedge^3\mathfrak{g}\rightarrow \Gamma(E_{-2})\\&\\\Longrightarrow&\; \langle \alpha, \eta([x,y]_\mathfrak g,z)\rangle-\left(\Phi_0(x)^{(0)}[\langle \alpha, \eta(y,z)\rangle]-\langle \Phi_0(x)^{(0)}(\alpha), \eta(y,z)\rangle\right)+\circlearrowleft=\langle Q^{(0)}[\alpha],\zeta(x,y,z)\rangle,\\&\\\Longrightarrow& 
\end{align*}
\begin{align*}
     \langle \alpha, \eta([x,y]_\mathfrak g,z)\rangle-\left(\cancel{\varrho(x)[\langle \alpha, \eta(y,z)\rangle]}-\cancel{\varrho(x)[\langle\alpha, \eta(y,z)\rangle]}+\langle\alpha, \nabla_x\eta(y,z)\rangle\right)&+\circlearrowleft(x,y,z)=\\&\langle \alpha,\, \ell_1(\zeta(x,y,z))\rangle.
\end{align*}
We have used Equations \eqref{eq:compatibility-with-Q} and  \eqref{eq:dual-action} in the last line. Hence, 

\begin{equation*}
    \langle \alpha, \eta([x,y]_\mathfrak g,z)- \nabla_x\eta(y,z)\rangle+\circlearrowleft(x,y,z)=\langle \alpha,\, \ell_1(\zeta(x,y,z))\rangle.
\end{equation*}Since $\alpha$ is arbitrary, one obtains \begin{equation}
    \nabla_x\eta(y,z)-\eta([x,y]_\mathfrak g,z)+\circlearrowleft(x,y,z)= \ell_1(\zeta(x,y,z)).
\end{equation}

In particular, if $m\in M$ is such that $\ell_1|_m=0$ and $\ell_2|_m=0$, then the map $x\mapsto \nabla_x$ defines an action on the isotropy Lie algebra $\mathfrak g_m$ of $\mathfrak F$, since $\nabla_x$ preserves the kernel of $\rho$. If in addition $\eta(x,y)\in \ker \rho_m$, then $\eta|_m\colon\wedge^2\mathfrak g\longrightarrow \mathfrak g_m$ is a cocycle of Chevalley-Eilenberg.\\
\end{enumerate}
\end{remark}
\begin{remark}
Notice that by using the duality which is given in Proposition \ref{alt:thm-res}, the linear map $x\mapsto \nabla_x$ is simply $$x\mapsto \ell_2'(x,\,\cdot),$$ 
where $\ell_2'$ is the $2$-ary bracket between sections of $\mathfrak g$ and $E$  of the Lie $\infty$-algebroid $(Q', E\oplus \mathfrak g)$ over the complex
\begin{equation}\label{semi-resol2}
 \begin{array}{c} \cdots\stackrel{\ell_1}{\longrightarrow}  E_{-3}\stackrel{\ell_1}{\longrightarrow}  E_{-2} \stackrel{\ell_1}{\longrightarrow}\mathfrak g\oplus E_{-1} \stackrel{\rho'}{\longrightarrow}TM\end{array}
\end{equation}like in Proposition \ref{alt:thm-res}. Indeed, for $\alpha\in \Gamma(E_{-1}^*)$ and $e\in \Gamma(E_{-1})$, \begin{align*}
\Phi_0(x)^{(0)}&= \text{pr}\circ[Q', \iota_x]^{(0)}\\&=\text{pr}\circ [Q'^{(1)},\iota_x].
\end{align*}
This implies that:
\begin{align*}
   \langle \Phi_0(x)^{(0)}\alpha, e \rangle&= \langle  Q'^{(1)}\alpha, x\odot e \rangle\\&=\rho'(x)[\langle \alpha, e \rangle]-\cancel{\rho'(e)[\langle \alpha,x\rangle]}-\langle \alpha, \ell_2'(x,e) \rangle\\&=\varrho(x)[\langle \alpha, e \rangle]-\langle \alpha, \ell_2'(x,e) \rangle.
\end{align*}
\end{remark}

\section{An obstruction theory}
Let us start with some generalities. Assume we are given\begin{itemize}
\item a Lie algebra $\mathfrak g$,
\item a weak symmetry action $\varrho\colon \mathfrak g\longrightarrow \mathfrak X(M)$ of $\mathfrak g$ on a singular foliation $\mathfrak{F}$, together  with $\eta\colon\wedge^2\mathfrak{g}\longrightarrow \Gamma(E_{-1})$   such that $x,y\in\mathfrak{g}$\begin{equation}\label{eq:def:eta}
    \varrho([x,y]_\mathfrak{g})-[\varrho(x),\varrho(y)]=\rho(\eta(x,y)).
\end{equation}
    \item an universal Lie $\infty$-algebroid $(E,Q_E)$  of $\mathfrak{F}$, 
 \end{itemize}  
Theorem \ref{main} assures  $\varrho\colon\mathfrak g\rightarrow \mathfrak{X}(M)$ admits a lift to a Lie $\infty$-morphism \begin{equation}
    \label{eq:Lie-morp-obstruction-theory2}\Phi\colon (\mathfrak{g},\lb_\mathfrak{g})\rightsquigarrow (\mathfrak X_\bullet(E)[1],\lb,\text{ad}_{Q}).
\end{equation}Equivalently, if $\mathfrak{g}$ is of finite dimension, \eqref{eq:Lie-morp-obstruction-theory2} corresponds (by Proposition \ref{alt:thm-res}) to a Lie $\infty$-algebroid $(E',Q')$ over $M$ such that
    
 \begin{itemize}
     \item $(E, Q_E)$ is included as a sub-Lie $\infty$-algebroid in a Lie algebroid $(E', Q)$ over $M$,\item its underlying complex is, $E'_{-1}:=~\mathfrak{g}\oplus E_{-1}$, and for any $i\geq 2$, $E_{-i}'=E_{-i}$, 
    namely \begin{equation}\label{semi-res}
  \begin{array}{c} \cdots\stackrel{\dd}{\longrightarrow}  E_{-3}\stackrel{\dd}{\longrightarrow}  E_{-2} \stackrel{\dd}{\longrightarrow}\mathfrak g\oplus E_{-1} \stackrel{{\rho'}}{\longrightarrow}TM,
    \end{array}
\end{equation}\item we have,  $$\ell'_2(x\oplus 0,y\oplus 0)=[x,y]_{\mathfrak{g}}\oplus \eta(x,y)$$ 
and $$\ell_2'(x,\Gamma(E_{-1}))\subset\Gamma(E_{-1})$$ for all $x\in\mathfrak{g}$.
\end{itemize}   

\begin{remark}
It is important to note that the Lie $\infty$-algebroid $(E',Q')$ can be constructed directly out of the weak symmetry action $\varrho\colon \mathfrak{g}\longrightarrow\mathfrak{X}(M)$, even if $\mathfrak{g}$ is of infinite dimension (see Theorem \ref{alt-thm-res}).
\end{remark}

\begin{remark}
In Equation \eqref{semi-res}, the complex $(E,\ell_1)$ can be chosen to be  minimal at a point $m\in M$, i.e., ${\ell_1}_{|_m}=0$, provided that a geometric resolution of $\mathcal{F}$ exists. By Proposition 4.14 in \cite{LLS} the isotropy Lie algebra $\mathfrak{g}_m=\frac{\mathcal{F}(m)}{\mathcal{I}_m\mathcal{F}}$ of the singular foliation $\mathcal{F}$ at the point $m\in M$ is isomorphic to $\ker(\rho_m)$. This allows to denote the latter space also by $\mathfrak{g}_m$.
\end{remark}

In what follows, we will use the $\infty$-algebroid which is described on the complex \eqref{semi-res}.

\begin{lemma}\label{class}
Let $m\in M$ be a fixed point of the $\mathfrak g$-action $\varrho$. 
Assume that the underlying complex $(E,\ell_1)$ is minimal at a point $m$, i.e. ${\ell_1}_{|_m}=0$. The map
\[{\nu}\colon\mathfrak{g}\longrightarrow \emph{End}\left(\mathfrak{g}_m\right),\;x\longmapsto\ell_2'(x\,,\cdot)_{|_m}\]satisfies \begin{itemize}
    \item[(a)] $\nu([x,y]_\mathfrak{g})-[\nu(x),\nu(y)]+\ell_2(\,\cdot,\eta(x,y))_{|_m}=0$,
    \item[(b)] $\nu(z)\left(\eta(x,y)_{|_m}\right)-\eta([x,y]_\mathfrak{g},z)_{|_m}+\circlearrowleft(x,y,z)=0$.
\end{itemize}

\end{lemma}
\begin{proof}
The map $\nu$ is well-defined. The Jacobi identity on elements $x,y\in \mathfrak{g},\,e\in \Gamma(E_{-1})$, evaluated at the point $m$, implies that $$\nu({[x,y]_\mathfrak{g}})(e_{|_m})-[\nu(x),\nu(y)](e_{|_m})+\ell_2(\eta(x,y),e)_{|_m}=0.$$ This proves item (a). Likewise,  Jacobi identity on  elements $x,y,z\in\mathfrak{g}$ and since  ${\ell_1}_{|_m}=0$ give:\begin{align*}
    \ell'_2(\ell_2'(x,y),z)_{|_m} +\circlearrowleft(x,y,z)=0\,&\Longrightarrow\,\ell'_2([x,y]_\mathfrak{g},z)_{|_m}+\ell'_2(\eta(x,y),z)_{|_m} +\circlearrowleft(x,y,z)=0,\\&\Longrightarrow \,\nu(z)\left(\eta(x,y)_{|_m}\right)-\eta([x,y]_\mathfrak{g},z)_{|_m}+\circlearrowleft(x,y,z)=0.
\end{align*}
Here we have used the definition of $\ell'_2$ on degree $-1$ elements and Jacobi identity for the bracket $[\cdot\,,\cdot]_\mathfrak{g}$. This proves item (b).
\end{proof}

By Lemma \ref{class}, $\mathfrak g_m$ is equipped with a $\mathfrak{g}$-module structure when $\eta(x,y)_{|_m}$ is  for all $x,y\in \mathfrak{g}$ valued in the center the Lie algebra $\mathfrak g_m$. Recall that $\eta$ is defined by Equation \eqref{eq:def:eta}. Notice that if $m\in M$ is a fixed point of the $\mathfrak g$-action $\varrho$, then  this implies in particular that $\eta(x,y)|_m\in \ker \rho_m=\mathfrak g_m$.

\begin{proposition}\label{Prop:class}
Let $m\in M$ be a fixed point of the $\mathfrak g$-action $\varrho$. Assume that 
\begin{itemize}
\item the underlying complex $(E,\ell_1)$  of $(E,Q)$ is minimal at $m$,
      \item for all $x,y\in \mathfrak{g}$, $\eta(x,y)_{|_m}$\,is valued in the center\footnote{In particular, when the $2$-ary bracket $\ell_2$ is zero at $m$, on elements of degree $-1$ we have, $Z(\mathfrak{g}_m)=\mathfrak{g}_m$.} $Z(\mathfrak{g}_m)$ of  $\mathfrak{g}_m$.
\end{itemize}
 Then,
\begin{enumerate}
    \item the restriction of the $2$-ary bracket 
\[\ell_2'\colon\mathfrak{g}\otimes Z(\mathfrak{g}_m)\longrightarrow Z(\mathfrak{g}_m)\]endows $Z(\mathfrak{g}_m)$ with a $\mathfrak{g}$-module structure which does not
depend neither on the choices of weak symmetry action  $\varrho$, a universal Lie $\infty$-algebroid of $\mathfrak{F}$, nor of the Lie $\infty$-morphism $\Phi\colon \mathfrak{g}\longrightarrow \mathfrak{X}(E)$.
\item the restriction of the map $\eta\colon\wedge^2\mathfrak{g}\longrightarrow \Gamma({E_{-1}})$ at $m$ \[\eta_{|_m}\colon\wedge^2\mathfrak{g}\longrightarrow Z(\mathfrak{g}_m)\]is a $2$-cocycle for the Chevalley-Eilenberg complex of $\mathfrak{g}$ valued in $Z(\mathfrak{g}_m)$,
\item the cohomology class of this cocycle does not depend on the representatives of the equivalence class of $\varrho$,
\item  if $\varrho$ is equivalent to a strict symmetry action, then $\eta_{|_m}$ is exact.
\end{enumerate}
\end{proposition}
\begin{proof}
$(E, \ell_1)$ being minimal at $m$,  $\ell_2'|_m$ satisfies the Jacobi identity. In particular,  for every $x\in \mathfrak g[1]$, $\ell_2'(x,\cdot\,)|_m$ preserves $Z(\mathfrak{g}_m)$. By item $(a)$ of Lemma \ref{class}, the restriction of the $2$-ary bracket 
\[\ell_2'\colon\mathfrak{g}\otimes Z(\mathfrak{g}_m)\longrightarrow Z(\mathfrak{g}_m)\] endows $Z(\mathfrak{g}_m)$ with a $\mathfrak{g}$-module, since $\ell_2(\cdot,\,\eta(x,y))|_m=0$ by assumption. It is easy to see that if we change the action $\varrho$ to $\varrho+\rho\circ\beta$  for some vector bundle morphism $\beta\colon\mathfrak{g}\longrightarrow E_{-1}$ such that $\beta|_m\colon\mathfrak{g}\longrightarrow Z(\mathfrak{g}_m)$, the new $2$-ary bracket between sections of $\mathfrak{g}[1]$ and $E_{-1}$ constructed as in the  proof of Theorem \ref{alt-thm-res} is modified by $(x,e)\mapsto\ell_2'(x,e)+\ell_2(\beta(x),e)$. The second term of the latter vanishes at $m$, by definition of $\beta|_m$. As a result, the action of $\ell_2'$ on $Z(\mathfrak{g}_m)$ does not depend on the choices made in the construction. This proves  item ${1}$.

{Item ${2}$ follows from item (b) of Lemma \ref{class} that tells that $\eta_{|_m}\colon\wedge^2\mathfrak{g}\longrightarrow Z(\mathfrak{g}_m)$ is a $2$-cocycle for the Chevalley-Eilenberg complex of $\mathfrak{g}$ valued in $Z(\mathfrak{g}_m)$.}

Let $\varrho'$ be a weak symmetry action of $\mathfrak{g}$ on $\mathcal{F}$ which is equivalent to $\varrho$, i.e., there exists a vector bundle morphism $\beta\colon\mathfrak{g}\longrightarrow E_{-1}$ (with $\beta|_m\colon\mathfrak{g}\longrightarrow Z(\mathfrak{g}_m)$) such that $\varrho'(x)=\varrho(x)+\rho(\beta(x))$ for all $x\in\mathfrak{g}$. Let $\eta'\colon\wedge^2\mathfrak{g}\longrightarrow E_{-1}$ be such that $\varrho'([x,y]_\mathfrak{g})-[\varrho'(x),\varrho'(y)]=\rho(\eta'(x,y))$ for all $x,y\in\mathfrak{g}$. Following the constructions in the proof of Theorem \ref{alt-thm-res}, this implies that\begin{equation}\label{eq:eta-exa}
    \eta'(x,y)=\eta(x,y)+\beta([x,y]_\mathfrak{g})- \ell_2'(x,\beta(y))+\ell'_2(y,\beta(x))-\ell_2(\beta(x),\beta(y))),\quad \text{for all $x,y\in\mathfrak{g}$}.
\end{equation}
Equation \eqref{eq:eta-exa} implies that $\eta'(x,y)|_m-\eta(x,y)|_m=\dd^{CE}(\beta|_m)(x,y)$, where $\dd^{CE}$ stands for the Chevalley-Eilenberg differential. As a consequence, $\eta'|_m$ and $\eta|_m$ define the same class in the Chevalley-Eilenberg complex of $\mathfrak{g}$ valued in $Z(\mathfrak{g}_m)$. This proves item $3$ and $4$.
\end{proof}

\begin{remark}
When ${\ell_2}_{|_m}\neq 0$. The weak symmetry action $\varrho$ is equivalent to strict one if  the Maurer-Cartan-like equation \eqref{eq:eta-exa} has no solution with ${\eta'}_{|_m}=0$.
\end{remark}

Let $\mathfrak{F}$ be a singular foliation. Let us choose a universal Lie $\infty$-algebroid $(E,Q)$ such that $(E,\ell_1)$ is minimal at a point $m\in M$. Such a structure always exists (see Proposition \ref{prop:minimal-resol}). By Proposition 4.14 in \cite{LLS} the isotropy Lie algebra $\mathfrak{g}_m$ of the singular foliation $\mathfrak{F}$ at the point $m\in M$ is isomorphic to $\ker(\rho_m)$. The following is a direct consequence of Proposition \ref{Prop:class}.
\begin{corollary}\label{prop:class}
Let $m\in M$ be a fixed point of the $\mathfrak g$-action $\varrho$. Assume that the isotropy Lie algebra $\mathfrak{g}_m$ of $\mathfrak F$ at $m$ is Abelian. Then, for any weak symmetry action  $\varrho$ of a Lie algebra action $\mathfrak{g}$ on $\mathfrak{F}$ such that $\varrho([x,y]_\mathfrak{g})-[\varrho(x),\varrho(y)]\in\mathfrak{F}(m)$ for all $x,y\in\mathfrak{g}$

 \begin{enumerate}
    \item $\mathfrak{g}_m$ is a $\mathfrak{g}$-module.
    \item The bilinear map, $\eta_{|_m}\colon\wedge^2\mathfrak g\to \mathfrak{g}_m$, is a Chevalley-Eilenberg $2$-cocycle  of $\mathfrak g$ valued in $\mathfrak{g}_m$. \item Its  class $\emph{cl}(\eta)\in H^2(\mathfrak g,\mathfrak{g}_m)$ does not depend on the choices made in the construction.
    \item Furthermore, $\emph{cl}(\eta)$ is an obstruction of having a strict symmetry action equivalent to $\varrho$.
\end{enumerate}
\end{corollary}

\begin{example}
We return to Example \ref{ex:isotropy} with $m\in M$ a leaf of $\mathfrak{F}$. Since the isotropy Lie algebra $\mathfrak{g}^k_m$ is Abelian for every $k
\geq 2$ the following assertions hold by Corollary \ref{prop:class}:

\begin{enumerate}
\item For each $k\geq 1$, the vector space $\mathfrak{g}_m^{k+1}$ is a $\mathfrak{g}_m^{k}$-module.
\item The obstruction of having a strict symmetry action equivalent to $\varrho_k$ is a Chevalley-Eilenberg cocycle valued in $\mathfrak{g}_m^{k+1}$.
\end{enumerate}
\end{example}
Here is a particular case of this example.
\begin{example}
Let $\mathfrak{F}:=\mathcal{I}_0^3\mathfrak{X}(\mathbb{R}^n)$ be the singular foliation generated by vector fields vanishing to order $3$ at the origin. The quotient $\mathfrak{g}:=\frac{\mathcal{I}_0^2\mathfrak{X}(\mathbb{R}^n)}{\mathcal{I}_0^3\mathfrak{X}(\mathbb{R}^n)}$ is a trivial Lie algebra. There is a weak symmetry action of $\mathfrak{g}$ on $\mathfrak{F}$ which assigns to an element in $\mathfrak{g}$ a representative in $\mathcal{I}^2_0\mathfrak{X}(\mathbb{R}^n)$. In this case, the isotropy Lie algebra of $\mathfrak{F}$ at zero is Abelian  and $\ell'_2(\mathfrak{g},\mathfrak{g}_0)_{|_0}=0$. Thus, the  action  of $\mathfrak{g}$ on $\mathfrak{g}_0$ is trivial. One can choose $\eta\colon\wedge^2\mathfrak{g}\longrightarrow \mathfrak{g}_0$ such that $\eta\left(\overline{x_i^2\frac{\partial}{\partial x_i}},\overline{x_i^2\frac{\partial}{\partial x_j}}\right)=2e_{ij}$, with $e_{ij}$ a constant section in a set of generators of degree $-1$ whose image by the anchor is $x_i^3\frac{\partial}{\partial x_j}$. Therefore, $\eta_{|_0}\left(\overline{x_i^2\frac{\partial}{\partial x_i}},\overline{x_i^2\frac{\partial}{\partial x_j}}\right)\neq 0$. This implies that the class of $\eta$ is not zero at the origin. Therefore, by item 2 of Corollary \ref{prop:class} the weak symmetry action of $\mathfrak{g}$ on $\mathfrak{F}$ is not equivalent to a strict one.
\end{example}

\begin{example}
Consider again the Example \ref{ex:F-connection-lift} and let $(E, \ell_\bullet, \rho)$ be a universal Lie $\infty$-algebroid of $\mathcal{T}$. Recall that a \emph{flat Ehresmann} connection is a horizontal distribution whose sections are closed under the Lie bracket of vector fields. Let $m\in M$. Assume in Example \ref{ex:F-connection} that the  section $\varrho_H\colon\mathfrak X(L)\rightarrow \mathfrak{F}^{\mathrm{proj}}$ satisfies $$\varrho_H([a,b])-[\varrho_H(a),\varrho_H(b)]=\rho(\eta(a,b))\in \mathcal{T}(m)$$
for some bilinear map $\eta\colon \wedge^2\mathfrak X(L)\rightarrow E_{-1}$. By Proposition \ref{prop:class}, the isotropy Lie algebra $\mathfrak g_m^\mathcal{T}$ of $\mathcal T$  at $m$ is a $\mathfrak X(L)$-module,  and $\eta$ is a cocycle of Chevalley-Eilenberg. This provides an  obstruction  for the $\mathcal{ F}$-connection to be flat.
\end{example}

Also, we have the following consequence of Corollary \ref{prop:class} for Lie algebra actions on affine varieties, as in Example \ref{ex:affine-action}. Before going to Corollary \ref{cor:affine} let us write definitions and some facts.\\

\noindent
\textbf{Settings:} Let $W$ be an affine variety realized as a subvariety of $\mathbb{C}^d$, and defined by some ideal $\mathcal I_W\subset\mathbb{C}[x_1,\ldots,x_d]$. We denote by $\mathfrak{X}(W):=\mathrm{Der}(\mathcal O_W)$ the Lie algebra of vector fields on $W$, where $\mathcal O_W$  is  coordinates ring of $W$. 
\begin{definition}
A point $p\in W$ is said to be $\emph{strongly singular}$ if for all $f\in \mathcal I_W$, $\dd_pf\equiv0$ or equivalently if for all $f\in \mathcal I_W$ and $X\in \mathfrak{X}(\mathbb{C}^d)$, one has $X[f](p)\in \mathcal I_p$.
\end{definition}

\begin{example}
Any singular point of a hypersurface $W$ defined by a polynomial $\varphi\in\mathbb{C}[x_1,\ldots,x_d]$ is strongly singular.
\end{example}
The lemma below is immediate.
\begin{lemma}
In a strongly singular point, the isotropy Lie algebra of the singular foliation $\mathfrak{F}=\mathcal I_W\mathfrak{X}(\mathbb{C}^d)$ is Abelian.
\end{lemma}
\begin{corollary}\label{cor:affine}
Let $\varrho\colon\mathfrak{g}\longrightarrow \mathfrak{X}(W) $ be a Lie algebra morphism.
\begin{enumerate}
    \item Any extension $\widetilde{\varrho}$ as in Example \ref{ex:affine-action} is a weak symmetry action for the singular foliation $\mathfrak{F}=\mathcal I_W\mathfrak{X}(\mathbb{C}^d)$.
    
    \item 
    For any strongly singular point $p$ in $W$ if the class $cl(\eta)$ does not vanish the strict action $\varrho\colon\mathfrak{g}\longrightarrow \mathrm{Der}(\mathcal O_W)$ can not be extended to the ambient space.
\end{enumerate}
\end{corollary}

Let us give an example of a Lie algebra action on an affine variety that do not extend to the ambient space.

We hope to construct an example as follows

\begin{example}
Let $W\subset\mathbb{C}^d$ be an affine variety generated by a regular homogeneous polynomial $\varphi\in\mathcal{O}=\mathbb{C}[x_1,\ldots,x_d]$ of degree $\geq 2$. Assume there exists two 
vector fields 
$X,Y\in \mathfrak{X}(\mathbb{C}^d)$ that satisfy $X[\varphi]=f\varphi,\,Y[\varphi]=g\varphi$, with $f,g\in\mathcal{I}_0$, and such that $[X,Y]=\varphi Z$, for some $Z\in\mathfrak{X}(\mathbb{C}^d)$.  Consider the action of the trivial $\mathfrak{g}=\mathbb{R}^2$ on $W$ that sends its canonical basis to $X$, and $ Y$ respectively. It is  a weak symmetry action on the singular foliation $\mathfrak{F}^\varphi:=\langle\varphi\rangle\mathfrak{X}(\mathbb{C}^d)$ and induces a Lie algebra map, $\mathfrak g\longrightarrow \mathfrak{X}(W)$.\\

Notice that the universal Lie algebroid of $\mathfrak{F}^\varphi$ is a Lie algebroid (see Example of \cite{CLRL}) because,
$$ \xymatrix{0\ar[r] & \mathcal{O}\mu \otimes_\mathcal O \X(\mathbb C^d) \ar^{\varphi\frac{\partial}{\partial\mu}\otimes_\mathcal O \text{id}}[rr] & & \mathfrak{F}^\varphi.}$$
is a $\mathcal{O}$-module isomorphism. Here $\mu$ is a degree $ -1$ variable, so that $\mu^2=0$. The universal algebroid structure over that resolution is given on the set of generators by: $$\ell_2\left(\mu\otimes_\mathcal O \dfrac{\partial}{\partial x_a},\mu\otimes_\mathcal O \dfrac{\partial}{\partial x_b}\right):=\frac{\partial\varphi}{\partial x_a}\,\mu\otimes_\mathcal O \partial_{x_b}-\frac{\partial\varphi}{\partial x_b}\,\mu\otimes_\mathcal O \partial_{x_a}$$ and $\ell_k:=0$ for every $k\geq 3$. Write $Z=\displaystyle{\sum_{i=1}^d}f_i\dfrac{\partial}{\partial x_i}$, with $(f_i)_{i=1,\ldots,d}\subset \mathcal{O}$. We have, $$\eta(e_1,e_2):=\sum_{i=1}^df_i\mu\otimes_\mathcal{O}\dfrac{\partial}{\partial x_i}.$$ where $e_1,e_2$ is the canonical basis of $\mathbb{R}^2$.\\
This Lie $\infty$-algebroid structure satisfies all the assuptions of Proposition \ref{prop:class}. Assume that the vector field $Z$ does not vanish at zero. Since the action is trivial at zero and $\eta_{|_0}\neq 0$, therefore its class is non-zero. By consequence, such action cannot be extended to ambient space.
\end{example}

Let us make it explicit
\begin{example}

Let $W\subset\mathbb{C}^2$ be the affine variety generated by the polynomial $\varphi=FG$ with $F,G \in\mathbb{C}[x,y]=:\mathcal{O}$. We consider the vector fields $U=F\mathcal{X}_G,\, V=G\mathcal{X}_F\in\mathfrak X(\mathbb{C}^2)$, where $\mathcal{X}_F$ and $\mathcal{X}_G$ are Hamiltonian vector fields w.r.t the Poisson structure $\{x,y\}:=1$. Note that $U, V$  are tangent to $W$, i.e. $U[\varphi], V[\varphi]\in \langle \varphi \rangle$. It is easily checked that $[U,V]=\varphi \mathcal{X}_{\{F,G\}}$.\\

The action of the trivial Lie algebra  $\mathfrak{g}=\mathbb{R}^2$ on $W$ that sends its canonical basis $(e_1, e_2)$ to $U$, and $V$ respectively, is a weak symmetry action on the singular foliation $\mathfrak{F}^\varphi:=\langle\varphi\rangle\mathfrak{X}(\mathbb{C}^2)$, and induces a Lie algebra map, \begin{equation}\label{eq:LA-Mor-AV}
   \varrho\colon\mathfrak g\longrightarrow \mathfrak{X}(W)
.\end{equation}

A universal Lie algebroid of $\mathfrak{F}^\varphi$ is a Lie algebroid (see Example 3.19 of \cite{CLRL}) because,
$$ \xymatrix{0\ar[r] & \mathcal{O}\mu \otimes_\mathcal O \X(\mathbb C^2) \ar^{\varphi\frac{\partial}{\partial\mu}\otimes_\mathcal O \text{id}}[rr] & & \mathfrak{F}^\varphi}$$
is a $\mathcal{O}$-module isomorphism. Here $\mu$ is a degree $ -1$ variable, so that $\mu^2=0$. The universal algebroid structure over that resolution is given on the set of generators by: \begin{equation}
    \ell_2\left(\mu\otimes_\mathcal O \dfrac{\partial}{\partial x},\mu\otimes_\mathcal O \dfrac{\partial}{\partial y}\right):=\frac{\partial\varphi}{\partial x}\,\mu\otimes_\mathcal O \dfrac{\partial}{\partial y}-\frac{\partial\varphi}{\partial y}\,\mu\otimes_\mathcal O \dfrac{\partial}{\partial x}\end{equation} and $\ell_k:=0$ for every $k\geq 3$. Write $\mathcal{X}_{\{F,G\}}=\dfrac{\partial \{F,G\}}{\partial y}\dfrac{\partial }{\partial x}-\dfrac{\partial \{F,G\}}{\partial x}\dfrac{\partial }{\partial y} $. Therefore,  we can put \begin{equation}\label{eta:Poisson}
         \eta(e_1,e_2):=\dfrac{\partial \{F,G\}}{\partial y}\,\mu\otimes_\mathcal{O}\dfrac{\partial}{\partial x}-\dfrac{\partial \{F,G\}}{\partial x}\,\mu\otimes_\mathcal{O}\dfrac{\partial}{\partial y}.\end{equation}

Take for example,  $F(x,y)= y-x^2$ and $G(x,y)=y+x^2$. The isotropy Lie algebra $\mathfrak{g}_{(0,0)}$ of $\mathfrak{F}^\varphi$ is Abelian, since zero is a strong singular point of $W$. By Corollary \ref{prop:class} (1), $\mathfrak{g}_{(0,0)}$ is a $\mathbb{R}^2$-module. A direct computation shows that the action on $\mathfrak{g}_{(0,0)}$ is not trivial, but takes value in $\mathcal{O}\, \mu\otimes_\mathcal O \dfrac{\partial}{\partial x}$. Besides, Equation \eqref{eta:Poisson} applied to $\{F,G\}=4x$ gives \begin{equation}\label{eta:Poisson2}
            \eta(e_1,e_2)=-4\,\mu\otimes_\mathcal O \dfrac{\partial}{\partial y}.
\end{equation}

If $\eta_{|_{(0,0)}}$ were a coboundary of Chevalley Eilenberg, we would have (in the notations of Proposition \ref{Prop:class}) that \begin{equation}\label{eq:nonzero-class}
    \eta(x,y)_{|_{(0,0)}}=\beta([x,y]_\mathfrak{\mathbb{R}^2})- \ell_2'(x,\beta(y))+\ell'_2(y,\beta(x))\in \mathcal{O}\, \mu\otimes_\mathcal O \dfrac{\partial}{\partial x},\quad \text{for all $x,y\in \mathfrak{g}$}
\end{equation} for some linear map $\beta\colon  \mathfrak{g}\longrightarrow \mathfrak{g}_{(0,0)}$. Therefore, Equation \eqref{eq:nonzero-class} is impossible by Equation \eqref{eta:Poisson2} and since $\eta_{|_{(0,0)}}\neq 0$. In orther words, its class $\mathrm{cl}(\eta)$ does not vanish. By Corollary \ref{cor:affine} (2), the action $\varrho$ given in Equation \eqref{eq:LA-Mor-AV} cannot be extended to ambient space.
\end{example}

\vspace{2cm}

\begin{tcolorbox}[colback=gray!5!white,colframe=gray!80!black,title=Conclusion:]
We applied the existence of a Lie $\infty$-morphism to the question of lifting to $M$ an $\mathfrak{g}$-action on $M/\mathfrak{F}$, i.e. of making strict a weak $\mathfrak{g}$-action.

\phantom{cc}We apply this question to several geometric issues, e.g. neighbourhood of leaves, Lie algebra actions on an affine variety.
\end{tcolorbox}

\chapter{Bi-submersion towers}\label{chap:tower}
\section{Symmetries of bi-submersions}\label{sec:6}
In this chapter, we introduce the notion \textquotedblleft bi-submersion towers\textquotedblright. The work contained in this chapter is entirely original, except for the notion below that arose in a discussion between C. Laurent-Gengoux, L.~ Ryvkin, and I, and will be the object of a separate study.\\

\noindent
Let us firstly recall the  definition of bi-submersion. The concept of bi-submersion over singular foliations has been introduced in \cite{AndroulidakisIakovos}.

\begin{definition}
Let $M$ be a manifold endowed with a singular foliation $\mathfrak F $.
A \emph{bi-submersion} $\xymatrix{B\ar@/^/[r]^{s}\ar@/_/[r]_{t}&M}$ over $\mathfrak F$ is a triple $(B,s,t) $ where:
\begin{itemize}
    \item $B$ is a manifold,
    \item $s,t\colon B \to M$ are 
    submersions, respectively called \emph{source} and \emph{target},
\end{itemize}
such that the pull-back singular foliations $ s^{-1} \mathfrak F$ and $t^{-1} \mathfrak F $ are both equal to the space of vector fields of the form $\xi+\zeta $ with $\xi\in \Gamma(\ker (\dd s)) $ and $\zeta \in \Gamma(\ker (\dd t))$. Namely,

\begin{equation}
\label{eq:defbisub}
  s^{-1} \mathfrak F=t^{-1} \mathfrak F = \Gamma(\ker (\dd s))+\Gamma(\ker (\dd t)).
\end{equation}
In that case, we also say that $(B,s,t)$ is a bi-submersion over $(M, \mathfrak{F})$.
\end{definition}

\begin{example}\label{ex:holonomy-biss}
Let $\mathfrak{F}$ be a singular foliation over a manifold $M$. For $x\in M$ and $X_1,\ldots, X_n\in \mathfrak{F}$ inducing  generators for $\mathfrak{F}_x:=\mathfrak{F}/\mathcal{I}_x\mathfrak{F}$. We know from  \cite{AndroulidakisIakovos} that there is an open neighborhood $\mathcal{W}$ of $(x, 0)\in M \times \mathbb R^n$ such that $(\mathcal{W}, t, s)$ is a bi-submersion over $\mathfrak{F}$, where\begin{equation}
    s(x,y)=x\quad  \text{and}\quad \displaystyle{t (x, y)=\exp_x\left(\sum_{i=1}^ny_iX_i\right)}=\varphi_1^{\sum_{i=1}^ny_iX_i}(x),
\end{equation}
where for $X\in\mathfrak X(M)$,\;$\varphi^X_1$ denotes the time-1 flow of $X$. 
 Such bi-submersions are called \emph{path holonomy bi-submersions} \cite{AZ}.
 \end{example}

Now we can introduce the following definition.
\begin{definition}\label{def:tower-bi}
A \emph{bi-submersion tower over a singular foliation $\mathfrak{F}$ on $M$}, is a (finite or infinite) sequence of manifolds and maps as follows
\begin{equation}\label{eq:tower}\mathcal{T}_B:\xymatrix{\cdots\phantom{X}\ar@/^/[r]^{s_{i+1}}\ar@/_/[r]_{t_{i+1}}&{B_{i+1}}\ar@/^/[r]^{s_{i}}\ar@/_/[r]_{t_{i}}&{B_{i}}\ar@/^/[r]^{s_{i-1}}\ar@/_/[r]_{t_{i-1}}&{\cdots}\ar@/^/[r]^{s_1}\ar@/_/[r]_{t_1}&B_1\ar@/^/[r]^{s_0}\ar@/_/[r]_{t_0}&B_0,}
\end{equation}
together with a sequence $\mathfrak F_i $ of singular foliations on $B_i$, with the convention that $B_0=M$ and $\mathfrak F_0 = \mathfrak F $, such that
\begin{itemize}
\item for all $i \geq 1$, $\mathfrak F_i\subset\Gamma(\ker ds_{i-1})\cap\Gamma(\ker dt_{i-1})$, 
\item for each $i\geq 1$, $\xymatrix{{B_{i+1}}\ar@/^/[r]^{s_{i}}\ar@/_/[r]_{t_{i}}&{B_{i}}}$ is a bi-submersion over $\mathfrak F_i$.
    
\end{itemize}
A bi-submersion tower over $(M,\mathfrak F) $ shall be denoted as $(B_{i+1},s_i,t_i,\mathfrak F_i)_{i \geq 0}$.
The bi-submersion tower over $\mathfrak{F}$ in  \eqref{eq:tower} is said to be of \emph{length $n\in\mathbb{N}$} if $B_j=B_n, s_j=t_j=\mathrm{id}\;$ and $\mathfrak{F}_j=\{0\}$ for all $j\geq n$.
\end{definition}

\begin{remark}\label{Rmk:Geom_consequences}
Let us spell out some consequences of the axioms. For $i\geq 1$,
two points $b,b'\in B_i$ of the same leaf of $\mathfrak F_i$ satisfy $s_{i-1}(b)=s_{i-1}(b')$ and $t_{i-1}(b)=t_{i-1}(b')$. Also, for all $b\in B_i,\; T_b\mathfrak{F}_i\subset (\ker ds_{i-1})_{|_b}\cap(\ker dt_{i-1})_{|_b}$.
\end{remark}

Let us explain how such towers can be constructed out of a singular foliation. Let $\mathfrak{F}$ be a singular foliation on $M$. Then, 
\begin{enumerate}
\item By Proposition 2.10 in \cite{AndroulidakisIakovos}, there always exists a bi-submersion $\xymatrix{B_1\ar@/^/[r]^{s_0}\ar@/_/[r]_{t_0}&M}$ over $\mathfrak F$.
\item The $C^\infty(B_1)$-module $\Gamma(\ker ds_{0})\cap\Gamma(\ker dt_{0}) $ is closed under Lie bracket. When it is locally finitely generated, it is a singular foliation on $B_1$.  
Then, it admits a bi-submersion $\xymatrix{B_2\ar@/^/[r]^{s_1}\ar@/_/[r]_{t_1}&B_1} $.
Therefore, we have obtained the two first terms of a bi-submersion tower.
\item We can then continue this construction provided that $\Gamma(\ker ds_{1})\cap\Gamma(\ker dt_{1}) $ is locally finitely generated as a $C^\infty(B_2)$-module, and that it is so at each step\footnote{In real analytic case, the module $\Gamma(\ker ds_1)\cap\Gamma(\ker dt_1)$ is locally finitely generated because of the noetherianity of the ring of germs of real analytic functions \cite{FrischJacques,Yum-TongSiu}.}. 
\end{enumerate}

\begin{definition}
A bi-submersion tower  $ \mathcal{T}_B=(B_{i+1},s_i,t_i,\mathfrak F_i)$  over $ (M,\mathfrak F)$ is called \emph{exact bi-submersion tower over $ (M,\mathfrak F)$} when $ \mathfrak F_{i+1}= \Gamma(\ker ( ds_i))\cap  \Gamma(\ker ( dt_i))$ for all $i \geq 0$.
It is called a \emph{path holonomy bi-submersion tower} (resp. \emph{path holonomy atlas bi-submersion tower}) if  $\xymatrix{B_{i+1}\ar@/^/[r]^{s_i}\ar@/_/[r]_{t_i}&B_i}$ is a path holonomy bi-submersion (resp. a path holonomy atlas) for $\mathfrak F_i $ for each $i \geq 0$. When a path holonomy bi-submersion tower is exact, we speak of \emph{exact path holonomy bi-submersion tower}.
\end{definition}

The following theorem gives a condition which is equivalent to the existence of a bi-submersion tower over a singular foliation. The proof uses Lemma \ref{lemma:invariant-vector-fields} which is stated in the next section. 
\begin{theorem}\label{thm:equivalence}
Let $\mathfrak F$ be a singular foliation on $M$. The following items are equivalent:
\begin{enumerate}
    \item $\mathfrak F$ admits a geometric resolution.
    \item There exists an exact path holonomy bi-submersion tower over $(M, \mathfrak F)$.
\end{enumerate}
\end{theorem}

\begin{convention}\label{conven:proj}For a submersion $\phi\colon M\longrightarrow N$ and a smooth map $ \psi\colon M\longrightarrow N$, we denote by ${}^\psi\Gamma(\ker d\phi)$ the space of $\psi$-projectable vector fields in $\Gamma(\ker d\phi)\subset \mathfrak{X}(M)$.\end{convention}
\begin{proof}
\noindent
$(1)\Rightarrow (2)$ : Assume that $\mathcal F$ admits a geometric resolution $(E, \dd, \rho)$. 
In particular, $(E_{-1}, \rho)$ is an anchored bundle over $\mathcal{F}$. {We need to show by recursion on $i\geq 0$ that $\Gamma(\ker ds_i
    )\cap\Gamma(\ker dt_i
    )$ is locally finitely generated because $\ker\dd^{(i+1)}$ or $\ker\rho$ is locally  finitely generated. We actually repeat at each step  $i\geq 0$, the general fact  that the pull-back complex of vector bundle by the submersion $\varphi=t_0\circ t_1\circ \cdots \circ t_i\colon B_{i+1}\longrightarrow M$ 
\begin{equation}\label{eq:pull-back-geo-resol}
         \xymatrix{\cdots\ar[r]& \varphi ^*E_{-i-3}\ar[r]^{\tiny{\varphi ^*\dd^{(i+3)}}}&\varphi ^*E_{-i-2}\ar[r]^{\tiny{\varphi ^*\dd^{(i+2)}}}&\varphi ^*E_{-i-1}\ar[r]& TB_{i+1} }
     \end{equation} remains exact at the sections level (at degree\footnote{We shall  understand that the degree of elements of $\varphi ^*E_{-i-j}$ is $-j$.} $\leq -1$), since $C^\infty(B_{i+1})$ is a flat $C^\infty(M)$-module. In addition, for $i\geq 1,$ the complex \eqref{eq:pull-back-geo-resol} defines a geometric resolution of $\mathcal{F}_{i+1}$.}
    
Let $(B_1, s_0, t_0)$  be a path holonomy bi-submersion over $(M,\mathcal{F})$. Consider the map 
\begin{align}\label{eq:induced-map-tower}
    R\colon\Gamma(t^*_0E_{-1})&\longrightarrow {}^{t_0}\Gamma(\ker ds_0
    )\subset\mathfrak{X}(B_1)\\\nonumber t^*_0e\;&\longmapsto \overrightarrow{e}
\end{align}
defined as in Proposition \ref{prop:left-arrows}. {By Lemma \ref{lemma:invariant-vector-fields} (1), the map $R$ in \eqref{eq:induced-map-tower}  comes from a vector bundle morphism $t_0^* E_{-1} \longrightarrow \ker ds_0$ 
 and is surjective on an open subset $V_1\subset B_1$, by item $2$ of Lemma \ref{lemma:invariant-vector-fields}}. 
 In particular the map $R$ in \eqref{eq:induced-map-tower} restricts to a surjective map
\begin{align}\label{eq:restriction-map-tower}
    \ker(t_0^*\rho)&\longrightarrow \ker\left( dt_0|_{\Gamma(\ker ds_0)}\right)=\Gamma(\ker ds_0
    )\cap\Gamma(\ker dt_0
    )\subset\mathfrak{X}(B_1)
\end{align}
By exactness in degree $-1$, $\ker\rho=\dd^{(2)}(\Gamma(E_{-2}))$. Therefore,  $\ker (t^*_0\rho)$ is locally finitely generated. By surjectivity of the map \eqref{eq:restriction-map-tower}, $\Gamma(\ker ds_0
    )\cap\Gamma(\ker dt_0)=:\mathcal{F}_1$ is also locally finitely generated on $V_1$, in particular $\mathcal{F}_1$ is a singular foliation on $V_1\subset B_1$ (we may assume that $B_1=V_1$.). Thus, one can take a path holonomy bi-submersion $(B_2,s_1,t_1)$ over $(B_1,\mathcal{F}_1)$. The proof continues the same as the previous step.

    Let us  make a step further for clarity. The composition \begin{equation*}
        \xymatrix{\Gamma(E_{-2})\ar@{.>>}@/_1pc/[rr]\ar@{->}[r]^<<<<{\dd^{(2)}}&\mathrm{im}(\dd^{(2)})=\ker\rho\ar@{->>}[r]&\mathcal{F}_1}
    \end{equation*} together with $E_{-2}$ is an anchored bundle over $\mathcal{F}_1$.  Just like in the first step,  define the surjective $C^\infty(B_2)$-linear map,
    \begin{align}\label{eq:induced-map-tower2}
    \Gamma(t_1^*E_{-2})&\longrightarrow {}^{t_1}\Gamma(\ker ds_1
    )\subset\mathfrak{X}(B_2)\\\nonumber t_1^*e\;&\longmapsto \overleftarrow{e}.
\end{align}
By Lemma \ref{lemma:invariant-vector-fields} (2), the map \ref{eq:induced-map-tower2} restricts (upon taking $B_2$ smaller) to a surjective map
\begin{align}\label{eq:restriction-map-tower2}
    \ker\left(\Gamma(t_1^*E_{-2})\stackrel{t_1^*\dd^{(2)}}{\longrightarrow} \Gamma(t_1^*E_{-1})\right)&\longrightarrow \ker\left( dt_1|_{\Gamma(\ker ds_1)}\right)=\Gamma(\ker ds_1
    )\cap\Gamma(\ker dt_1
    )\subset\mathfrak{X}(B_2)
\end{align}
By exactness in degree $-2$, the $C^{\infty}(M)$-module $$\ker\left(\Gamma(E_{-2})\stackrel{\dd^{(2)}}{\longrightarrow}\Gamma(E_{-1})\right)=\dd^{(3)}(\Gamma(E_{-3}))$$ is (locally) finitely generated, hence $\mathcal{F}_2:=\Gamma(\ker ds_1
    )\cap\Gamma(\ker dt_1
    )$ is a singular foliation on $B_2$. Thus, one can take a path holonomy bi-submersion $(B_3,s_2,t_2)$ over $(B_2,\mathcal{F}_2)$. By recursion on $i\geq 1$, we use a path holonomy bi-submersion $(B_{i+1},s_i,t_i)$ over $(B_i,\mathcal{F}_i)$ and  construct an anchor bundle over $\mathcal{F}_i$  by the composition \begin{equation*}
        \xymatrix{\Gamma(\varphi^* E_{-i-1})\ar@{.>>}@/_1pc/[rr]\ar@{->}[r]^<<<<{\varphi^*\dd^{({i+1})}}&\mathrm{im}(\varphi^*\dd^{(i+1)})=\ker(\varphi^*\dd^{(i)})\ar@{->>}[r]&\mathcal{F}_i}
    \end{equation*} with $\varphi=t_0\circ t_1\circ \cdots \circ t_{i-1}\colon B_{i}\longrightarrow M$ and show as for $i=0, 1$ that $\mathcal{F}_{i+1}:=\Gamma(\ker ds_i
    )\cap\Gamma(\ker dt_i
    )$ is a singular foliation on $B_{i+1}$. The proof follows.\\

\noindent
$(2)\Rightarrow (1)$ is proven by Lemma \ref{prop:sequence} and Remark \ref{rmk:pull-back-on-M} below.
\end{proof}

\begin{lemma}\label{prop:sequence}
Let $\mathfrak{F}$ be a singular foliation on $M$. Assume that there exists a bi-submersion tower $\mathcal{T}_B=(B_i,t_i,s_i, \mathfrak{F}_i)_{i\geq 0}$ over $\mathfrak F$. Then,
\begin{equation}\label{eq:exa-bisub}
        \xymatrix{
\cdots\ar[r]&\ker ds_2 \ar[d] \ar[r]^{ dt_2}&\ker ds_1\ar[d] \ar[r]^{ dt_1}&\ker ds_0 \ar[d] \ar[r]^{ dt_0} &TM\ar[d]\\
\cdots\ar[r]&B_3 \ar[r]_{t_2}&B_2\ar[r]_{t_1}&B_1 \ar[r]_{t_0} &M.}
    \end{equation}
 is a complex of vector bundles, which is exact on the sections level\footnote{Let us explain the notion of exactness at the level of sections when the base manifolds are not the same: what we mean is that for all $n\geq 0$, $\Gamma(\ker dt_n)\cap \Gamma(\ker ds_{n})=(t_{n+1})_*(\Gamma(\ker ds_{n+1}))$.
 
 Equivalently, it means that the pull-back of the vector bundles in \eqref{bi-thm:eq} to any one of the manifold $B_m $ with $m \geq n $ is exact at the level of sections, i.e\begin{equation}\label{complex-projectable}\xymatrix{ \Gamma(t_{n+1,m}^* \ker ds_{n+1})  \ar[r]^{ dt_{n+1}}&\Gamma(t_{n,m}^* \ker ds_n) \ar[r]^{ dt_n}&\Gamma(t_{n-1,m}^* \ker ds_{n-1}})\end{equation}is a short exact sequence of $C^\infty(B_m) $-modules, with $ t_{n,m}= t_{n} \circ \dots \circ t_{m}$ for all $m \geq n$.
 }
 if $\mathcal{T}_B$ is an exact bi-submersion tower, i.e. if $\mathfrak F_i=\Gamma(\ker ds_{i-1})\cap\Gamma(\ker dt_{i-1})$ for all $i\geq 1$.
\end{lemma}
\begin{proof}
 For any element $b \in B_{i+1}$ and any vector $v\in \ker ds_i \subset T_{b} B_{i+1}$ one has\begin{align*}
    & dt_i(v)\in T_{t_i(b)}\mathfrak{F}_{i},\quad\text{(since $\Gamma(\ker ds_i)\subset t^{-1}_i(\mathfrak{F}_{i})$)}.\\\Longrightarrow\quad&  dt_i(v)\in\left(\ker ds_{i-1}\cap\ker dt_{i-1}\right)|_{t_i(b)}\quad \text{by Definition  \ref{def:tower-bi}}.\\\Longrightarrow \quad& dt_i(v)\in \ker ds_{i-1}\quad \text{and}\quad dt_{i-1}\circ dt_i (v)=0,\; \text{for all $i\geq 1$}.
\end{align*}
This shows the sequence \eqref{eq:exa-bisub} is a well-defined complex of vector bundles.

Let us prove that it is exact when  $\mathfrak F_i=\Gamma(\ker ds_{i-1})\cap\Gamma(\ker dt_{i-1})$ for all $i\geq 1$. Let $\xi\in\Gamma\left(\ker ds_{i-1}\right)$ be a $t_{i-1}$-projectable vector field that projects to zero, i.e.  $ dt_{i-1}(\xi)=0$. This implies that   $\xi\in\Gamma(\ker ds_{i-1})\cap\Gamma(\ker dt_{i-1})=\mathfrak F_i$. Since $t_{i}$ is a submersion, there exists a $t_i$-projectable vector field $\zeta\in t_{i}^{-1}(\mathfrak{F}_i)$ that satisfies $ dt_{i}(\zeta)=\xi$. The vector field $\zeta$ can be written as $\zeta=\zeta_1+\zeta_2$ with $\zeta_1\in\Gamma\left(\ker dt_{i}\right)$ and $\zeta_2\in\Gamma\left(\ker ds_{i}\right)$, because $t_i^{-1}(\mathfrak{F}_i)=\Gamma(\ker ds_{i})+\Gamma(\ker dt_{i})$. One has, $ dt_i(\zeta_2)=\xi$. A similar argument shows that the map, $\Gamma(\ker ds_0)\overset{ dt_0}{\longrightarrow}t^*_0\mathfrak{F}$, is surjective. This proves exactness in all degree.
\end{proof}

\begin{remark}\label{rmk:pull-back-on-M}One of the consequence of Lemma \ref{prop:sequence} is that: 
\begin{enumerate}
    \item If there exists a sequence of maps \begin{equation}
    \xymatrix{M\ar@{>->}[r]^{\varepsilon_0}&B_1\ar@{>->}[r]^{\varepsilon_1}&B_2\,\ar@{>->}[r]^{\varepsilon_2}&\cdots}
\end{equation}\label{sequence-sections}where for all $i\geq 0$, $\varepsilon_i$ is a section for both $s_i$ and $t_i$ then the pull-back of \eqref{eq:exa-bisub} on $M$ through the sections $(\varepsilon_i)_{i\geq 0}$ i.e. \begin{equation}\label{eq:comp-vect-bund}
    \xymatrix{\cdots\ar[r]^{dt_3}&\varepsilon_{2,0}^*\ker ds_2\ar[r]^{dt_2}&\varepsilon_{1,0}^*\ker ds_1\ar[r]^{dt_1}&\varepsilon_0^*\ker ds_0\ar[r]^{dt_0}&TM}
\end{equation}
is a complex of vector bundles, with the convention $\varepsilon_{n,0}= \varepsilon_{n}\circ\cdots \circ\varepsilon_{0}$. If $\mathcal{T}_B$ is an exact bi-submersion tower then, \eqref{eq:comp-vect-bund} is a geometric resolution of $\mathfrak{F}$.
\item In case that $\mathcal{T}_B$ is an exact path holonomy bi-submersion tower, such a sequence \eqref{sequence-sections} always exists, since the bi-submersions $(B_{i+1}, s_i, t_i)$ are as in Example \ref{ex:holonomy-biss}. For such bi-sumersions, the zero section $x\mapsto ( x, 0)$ is a section for both $s_i$ and $t_i$.
\end{enumerate}
\end{remark}

\begin{corollary}\label{cor:exa}
Under the assumptions of Lemma \ref{prop:sequence}, assume the  tower  of  bi-submersion $\mathcal{T}_B$ is of length $n+1$. 
Then, the pull-back of the sequence of vector bundles
  \begin{equation}\label{bi-thm:eq}
      \xymatrix{\ker ds_{n}\ar[rrrrd] \ar[r]^-{ dt_{n}}&t^*_{n}\ker ds_{n-1}\ar[r]\ar[rrrd]&{\;\cdots}\ar[rrd] \ar[r]^-{ dt_2}&t_{2,n}^*\ker ds_1\ar[rd] \ar[r]^-{ dt_1}&TB_{n+1}\times_{TM}\ker ds_0 \ar[d] \ar[r]^-{\mathrm{pr}_1} &TB_{n+1}\ar[ld]\\& & &&B_{n+1}
  \footnote{ $TB_{n+1}\times_{TM}\ker ds_0=\{(u,v)\in TB_{n+1}\times\ker ds_0\mid  dt_{0,n}(u)= dt_0(v)\}$.}
  }
 \end{equation}
    is a geometric resolution of the pull-back foliation $t_{0,n}^{-1}(\mathfrak{F})\subset \mathfrak{X}(B_{n+1})$, where  $\mathrm{pr}_1$ is the projection on $TB_{n+1}$ and for $i\geq 1$, $t_{i,j}$ is the composition $t_i\circ\cdots\circ t_j\colon B_{j+1}\to B_{i}$.
\end{corollary}
\begin{proof}
By Lemma \ref{prop:sequence}, the complex in Equation \eqref{bi-thm:eq} is exact. By construction, the projection of the fiber product $TB_{n+1}\times_{TM}\ker ds_0$ to $TB_{n+1}$ induces the singular foliation $t_{0,n}^{-1}(\mathfrak{F})$. 
\end{proof}

\subsection{Lift of a symmetry to the bi-submersion tower}
Let us investigate what an action   $\varrho\colon\mathfrak g\rightarrow \mathfrak{X}(M)$ of a Lie algebra $\mathfrak{g}$ on $(M,\mathfrak F)$ would induce on a bi-submersion tower $\mathcal{T}_B$ over $\mathfrak{F}$.\\

We start with some vocabulary and preliminary results.

\begin{definition}Let $(B,s,t)$ be a bi-submersion of a singular foliation $\mathfrak F$ on a manifold $M$. We call \emph{lift} of a vector field $X\in\mathfrak X(M)$ to the bi-submersion $(B,s,t)$ a vector field $\widetilde{X}\in \mathfrak{X}(B)$ which is both $s$-projectable on $X$ and $t$-projectable on $X$. 
\end{definition}

The coming proposition means that the notion of lift to a bi-submersion only makes sense for symmetries of the singular foliation.

\begin{prop}\label{lift-symmetry-bi-susbmersion} If a vector field on $M$ admits a lift to a bi-submersion $(B,s,t)$, then it is a symmetry of $\mathfrak F$.
\end{prop}
\begin{proof}
Let $\widetilde{X}\in\mathfrak{X}(B)$ be a lift of $X\in\mathfrak{X}(M)$. Since $\widetilde{X}$  is $s$-projectable, $[\widetilde{X},\Gamma(\ker \dd s)]\subset\Gamma(\ker \dd s)$. Since $\widetilde{X}$  is $t$-projectable, $[\widetilde{X},\Gamma(\ker \dd t)]\subset\Gamma(\ker \dd t)$. Hence:
\begin{align*} [\widetilde{X},s^{-1}(\mathfrak{F})] &= [\widetilde{X},\Gamma(\ker \dd s)+\Gamma(\ker \dd t)] \\ &= [\widetilde{X},\Gamma(\ker (\dd s)]+[\widetilde{X},\Gamma(\ker \dd t)]  \\ &\subset  \Gamma(\ker \dd s)+\Gamma(\ker \dd t) =s^{-1}(\mathfrak{F}) .\end{align*}
In words, $\widetilde{X}$ is a symmetry of $s^{-1} \mathfrak F$.
Since $\widetilde{X}$ projects through $s$ to $X$, $X$ is a symmetry of ~$\mathfrak F $.
\end{proof}

We investigate the existence of lifts of symmetries of $ \mathfrak F$ to bi-submersions over $\mathfrak F $.

\begin{remark}\label{rmk:uniqueness} For a given $X\in\mathfrak{s}(\mathfrak F)$, 
\begin{enumerate}
    \item the lift $\widetilde{X}$ to a given bi-submersion is not unique, even when it exists. However, two different lifts of a $X\in\mathfrak s(\mathfrak{F})$ to a bi-submersion $(B,s,t)$ differ by an element of the intersection $\Gamma(\ker (\dd s))\cap\Gamma(\ker (\dd t))$. 
    \item $\widetilde X$ is a symmetry of $\Gamma(\ker (\dd s))\cap\Gamma(\ker (\dd t))$: $[\widetilde{X}, \Gamma(\ker (\dd s))\cap\Gamma(\ker (\dd t))]\subset \Gamma(\ker (\dd s))\cap\Gamma(\ker (\dd t))$,  since $\widetilde{X}$ is $s$-projectable and $t$-projectable.
\end{enumerate}
\end{remark}

As the following example shows, the lift of a symmetry  to a bi-submersion may not exist.

\begin{example}
Consider the trivial foliation $\mathfrak{F}:=\{0\}$ on $M$. For any diffeomorphism $\phi\colon M\longrightarrow M$, $(M,\text{id}, \phi)$ is a bi-submersion over $\mathfrak{F}$. Every vector field $X\in\mathfrak{X}(M)$ is a symmetry of $\mathfrak{F}$. If it exists, its lift has to be given by, $\widetilde{X}=X $ since the source map is the identity. But $\widetilde{X}=X $ is $t$-projectable if and only if
$X$ is $\phi$-invariant. A non $\phi$-invariant vector field $X$ therefore admits no lift to $(M,\text{id}, \phi)$.
\end{example}

However, internal symmetries, i.e. elements in $\mathfrak F$ admit lifts to any bi-submersion.

\begin{prop}\label{prop:lift-internal}
Let $(B,s,t)$ be any bi-submersion of a singular foliation $\mathfrak F$ on a manifold $M$. 
Every internal symmetry, i.e. every vector field in $\mathfrak{F}$, admits a lift to $(B,s,t)$.
\end{prop}
\begin{proof}
Let $X\in\mathfrak{F}$. Since $s\colon B\longrightarrow M $ is a submersion, there exists $X^s\in\mathfrak{X}(B)$ $s$-projectable on $X$. 
Since $t $ is a submersion, there exists $X^t\in\mathfrak{X}(B)$ $t$-projectable on $X$.
By construction $X^s \in s^{-1}(\mathfrak F) $
 and  $X^t \in t^{-1}(\mathfrak F) $.
 Using the property \eqref{eq:defbisub} of the bi-submersion $(B,s,t)$, the vector fields  $X^s$ and $X^t$ decompose as $$\begin{cases}
X^s=X^s_s+X^s_t\quad \text{with\; $X^s_s\in \Gamma(\ker (\dd s))$,\;$X^s_t\in\Gamma(\ker (\dd t))$},\\X^t=X^t_s+X^t_t\quad \text{with\; $X^t_s\in \Gamma(\ker (\dd s))$,\;$X^t_t\in\Gamma(\ker (\dd t))$}.
\end{cases}$$
By construction, $ X^s_t  $ is $s$-projectable to $X$ and $t$-projectable to $0$ while  $ X^t_s  $ is $s$-projectable to $0$ and $t$-projectable to $X$.
It follows that, $\widetilde{X}:=X_s^t+X_t^s$, is a lift of $X$ to the bi-submersion $(B,s,t)$. \\
\end{proof}

Let us make the Proposition \ref{prop:lift-internal}  more precise.
{The proof we give for Theorem \ref{thm:equivalence} uses the notion of left-invariant, resp. right-invariant, vector fields on a bi-submersion over a singular foliation. We define the latter in the next proposition. It uses the notion of anchored bundle over a singular foliation and almost Lie algebroid, see \cite{LLS, LLL} for more details. } 

\begin{proposition}\label{prop:left-arrows}
Let $(B,s,t)$ be a bi-submersion over a singular foliation $\mathcal F$ on a manifold $M$. Let $(A, \rho)$ be an anchored bundle over $\mathcal{F}$, i.e., $A\longrightarrow M$ is a vector bundle and $\rho\colon A\longrightarrow TM$ is a vector bundle morphism such that $\rho(\Gamma(A))=\mathcal F$. There exists two maps
\begin{equation}
\label{def:gauchedroite}
     \begin{array}{rcl} \Gamma(A) & \longrightarrow & \mathfrak X(B) \\ a & \longmapsto & \overleftarrow{a}\\  a & \longmapsto & \overrightarrow{a}\end{array}
\end{equation}
fulfilling the following conditions:
\begin{enumerate}
\item the vector field $\overrightarrow{a}\in \mathfrak{X}(B)$
(resp. $\overleftarrow{a}\in \mathfrak{X}(B)$) is $t$-related (resp. $s$-related) with $\rho(a) \in~\mathcal{F}
$,
    \item the vector field $\overrightarrow{a}$
(resp. $\overleftarrow {a}$) is tangent to the fibers of $s$ (resp. $t$),
\item $\overrightarrow{fa}=t^*(f)\overrightarrow{a}$ and $\overleftarrow{fa}=s^*(f)\overleftarrow{a}$ for all $a\in \Gamma(A),\,f\in C^\infty(M)$.
\end{enumerate}

{The vector fields $\overleftarrow{a}$ (resp. $\overrightarrow{a}$) for $a\in\Gamma(A)$ are called \emph{left-invariant} (resp. \emph{right-invariant}) vector fields of $(B,s,t)$.}
\end{proposition}
\begin{proof}
By Proposition \ref{prop:lift-internal}, given a section $a\in\Gamma(A)$ the vector field $\rho(a)\in\mathcal{F}$ admits a lift $\widetilde {\rho(a)}\in \mathfrak{X}(B)$ on $(B,s,t)$ of the form $$\widetilde{\rho(a)}:=\rho(a)_s^t+\rho(a)_t^s$$ with $\rho(a)^t_s\in \Gamma(\ker (\dd s))$ and $\rho(a)^s_t\in\Gamma(\ker (\dd t))$. Also, $dt(\rho(a)_s^t)=\rho(a)$ and  $ds(\rho(a)_t^s)=\rho(a)$. Let $b\in B$ and $\mathcal{U}_b$ an open neighborhood of $b$. Let $(a_1,\ldots,a_r)$ be a local trivialization of $A$ on the open subset $\mathcal U=s(\mathcal{U}_b)\subset M$.  We define a map $R_U$ on local generators by 
\begin{align}
    R_U\colon\Gamma_{\mathcal{U}_b}(t^*A)&\longrightarrow \Gamma_{\mathcal{U}_b}(\ker (\dd s))\\\nonumber t^*a_i\;&\longmapsto \rho(a_i)^t_s
\end{align}
and extend by $C^\infty(\mathcal U_b)$-linearity. 
These maps can be glued using partitions of unity. More precisely, let $(\chi_\lambda)_{\lambda\in \Lambda}$ be a partition of unity subordinate
to an open cover $(\mathcal{U}_\lambda)_{\lambda\in \Lambda}$ by open sets that trivialize the vector bundle $A$. We define a map $R$ on $\Gamma(t^*A)$ as $$\displaystyle{\sum_{\lambda\in \Lambda}\chi_\lambda R_{\mathcal{U}_\lambda}} .$$ 
Now for  $a\in \Gamma(a)$ we define $\overrightarrow{a}:= R(s^*a)$. The map $\overleftarrow{\bullet}$ is defined similarly. Item $1$, $2$ and $3$ hold by construction.

Assume that  $(A,\rho)$ is equipped with an almost Lie algebroid bracket $\lb_A$. For all $a,b\in \Gamma(A)$ one has
\begin{align*}
    ds\left(\overleftarrow{[a,b]_{A}}- [\overleftarrow{a},\overleftarrow{b}]\right)&=\rho([a, b]_{A})-[\rho(a),\rho(b)]\\&=0,\end{align*}because $\overleftarrow{a}$ is $s$-projectable to $\rho(a)$ and $\rho$ is a morphism of brackets. Since left-invariant vector fields are tangent to the fibers of $t$, one has $dt\left(\overleftarrow{[a,b]_{A}}- [\overleftarrow{a},\overleftarrow{b}]\right)=~0$. The proof is similar for $\overrightarrow{[a,b]_{A}}- [\overrightarrow{a},\overrightarrow{b}]$. This ends the proof. 
\end{proof}
The following lemma is important in the proof of Theorem \ref{thm:equivalence}.
\begin{lemma}\label{lemma:invariant-vector-fields}
Let $(B,s,t)$ be any bi-submersion over a singular foliation $\mathcal F$ on a manifold $M$, and $(A,\rho)$ an anchored bundle over $\mathcal{F}$. 
\begin{enumerate}
   \item  {There exists vector bundle morphisms $R\colon t^*A \longrightarrow \ker ds$ 
 and $L\colon s^*A \longrightarrow \ker dt$ inducing \eqref{def:gauchedroite}}.
 
\item  {Let $x\in M$.  If $(B,s,t)$ is a path holonomy bi-submersion over $\mathcal{F}$ near $(x, 0)$  then, every $b\in B$ such that $t(b)=x$  admits a neighborhood $V$ such that every $t$-projectable vector field of $\Gamma_V(\ker ds)$ 
is of the form $R(\xi)$ 
for some $\xi\in \Gamma_{V}(t^*A)$.}
\end{enumerate}
\end{lemma}
\begin{remark}
In item 1 of Lemma \ref{lemma:invariant-vector-fields}, by \textquotedblleft\,inducing \eqref{def:gauchedroite}\textquotedblright  we mean that  for every $a\in \Gamma(A)$, $\overrightarrow{a}:=R(s^*a)$ and $\overleftarrow{a}:=L(t^*a)$.
\end{remark}
\begin{proof}
{Item $1$ is an immediate consequence of Proposition \ref{prop:left-arrows} and is in fact used in its proof. Let us prove item ${2}$. By assumption, $B$ is a neighborhood of $(0,x)$ in $M\times \mathbb{R}^n$ with $n=\mathrm{rk}_x(\mathcal{F})=\dim (\mathcal{F}_x:=\mathcal{F}/\mathcal{I}_x\mathcal{F})$  near $x\in M$ (see Example \ref{ex:holonomy-biss}). Let $b\in B$ and $\mathcal{U}_b$ an open neighborhood of $b$. Let $(a_1,\ldots,a_r)$ be a local trivialization of $A$ on the open subset $\mathcal U=t(\mathcal{U}_b)\subset M$.  One has by definition of right-invariant vector fields of $(B,s, t)$ that $dt(\overrightarrow{a_i})=t^*\rho(a_i)$ for $i=1,\ldots, \mathrm{rk}(A)$. The vector fields $\rho(a_i)$ are generators of $\mathcal{F}$ on $\mathcal{U}$. We necessarily have $n\leq\mathrm{rk}(A)$. Since the $\rho(a_i)(x)'s$ are generators of $\mathcal{F}_x$, without loss of generality we can assume that $\rho(a_1)(x), \ldots, \rho(a_n)(x)$ is  a basis of $\mathcal{F}_x$. Since $\mathrm{rk}(\ker ds)=n$,  $\left(\overrightarrow{a_i}(b)\right)_{i=1,\ldots, n}$ form a basis of $\ker ds|_x$. Therefore, the $\overrightarrow{a_i}$’s are independent at every point of some neighborhood $V\subset\mathcal{U}_b$ of $b$ i.e., they form a local trivialization of the  vector bundle $\ker ds\longrightarrow B$. As a result, vector fields of $\Gamma_V(\ker ds)$ are of the form $\sum _i^nf_i\overrightarrow{a_i}$ with $f_i\in C^\infty(V)$ for $i=1, \ldots,n$. This ends the proof.}
\end{proof}

We can now state one of the important results of this section. It uses several concepts introduced in \cite{AndroulidakisIakovos}, which are recalled in the proof. 

\begin{prop}\label{prop:lift}
Let $\mathfrak{F}$ be a singular foliation on a manifold $M$. Any symmetry $X\in \mathfrak{s}(\mathfrak F)$ admits a lift
 \begin{enumerate}
    \item to any path holonomy bi-submersion  $(B,s,t)$,
    \item to Androulidakis-Skandalis' path holonomy atlas,
    \item to a neighborhood of any point in a bi-submersion through which there exists a local bisection that induces the identity.
\end{enumerate}
\end{prop}
\begin{remark}
In cases (1) or (2) in Proposition \ref{prop:lift}, a linear lift $$ X \to \widetilde{X}$$ can be defined on the whole space $\mathfrak{s}(\mathfrak F)$ of symmetries of $\mathfrak F$. 
As an immediate consequence of Remark \ref{rmk:uniqueness}, 
we obtain that for all $X,Y\in\mathfrak{s}(\mathfrak F)$,
\begin{equation}
   \widetilde{[X,Y]}-[\widetilde{X},\widetilde{Y}]\in \Gamma(\ker\dd s)\cap\Gamma(\ker\dd t).
\end{equation}
\end{remark}
\begin{proof}[Proof of (Proposition \ref{prop:lift})]
Let $X\in\mathfrak{s}(\mathfrak{F})$. Assume that $(B,s,t)=(\mathcal{W}, s_0,t_0)$ is a path holonomy bi-submersion associated to some generators $X_1,\ldots,X_n\in\mathfrak{F}$ as in Example \ref{ex:holonomy-biss}. Fix $(u,y=(y_1,\ldots,y_n))\in \mathcal{W}\subset M\times\mathbb{R}^n$, set $Y:=\sum_{i=1}^dy_iX_i$. Since $d\varphi_1^Y(X)=(\varphi^Y_1)_*(X)\in X+\mathfrak{F}$, there exists $Z_y\in\mathfrak{F}$ depending in smoothly on $y$ such that $dt_0(X,0)=X+Z_y$. Take $\widetilde{Z}_y\in t_0^{-1}(\mathfrak{F})$ such that $dt_0(\widetilde{Z}_y)=Z_y$. One has, 
\begin{equation*}
    dt_0\left((X,0)-\widetilde{Z}_y\right)=X= ds_0(X,0).
\end{equation*}
We can write $\widetilde{Z}_y=\widetilde{Z}_y^1+\widetilde{Z}_y^2$, with $\widetilde{Z}_y^1\in\Gamma(\ker  ds_0)$, $\widetilde{Z}_y^2\in\Gamma(\ker dt_0)$. By construction, $\widetilde{X}:=(X,-\widetilde{Z}_y^1)$ is a lift of $X$ to the bi-submersion $(\mathcal{W},s_0,t_0)$. This proves item 1.\\

If $X_B\in\mathfrak{X}(B)$ and  $X_{B'}\in\mathfrak{X}(B')$ are two lifts of the symmetry $X$ on the path holonomy bi-submersions  $(B, s,t )$ and $(B',s',t')$ respectively, then $(X_B,X_{B'})$ is a lift of $X$ on the composition bi-submersion $B\circ B'$. This proves item 2, since the path holonomy atlas is made of fibered products and inverse of holonomy path holonomy bi-submersions \cite{AndroulidakisIakovos}. Also $X_a$ is a symmetry of $(B,t,s)$.\\

Item $2$ in Proposition 2.10 of \cite{AndroulidakisIakovos} states that if the identity of $M$ is carried by $(B,s,t)$ at some point $v\in B$, then there exists an open neighborhood $V\subset B$ of $v$ that satisfies $s_{|_V}=s_0\circ g$ and $t_{|_V}=t_0\circ g$, for some submersion $g\colon V\longrightarrow \mathcal{W}$, for $\mathcal{W}$ of the form as in item 1. Thus, for all $X\in\mathfrak{s}(\mathfrak{F})$ there exists a vector field $\widetilde{X}\in\mathfrak{X}(V)$ fulfilling $ ds_{|_V}(\widetilde{X})=dt_{|_V}(\widetilde{X})=X$. This proves item 3.\\
\end{proof}

\begin{definition}
A \emph{symmetry of the tower of bi-submersion} $\mathcal{T}_B=(B_{i+1},s_i,t_i,\mathcal F_i)_{i \geq 0} $ is a family $X=(X_i)_{i \geq 0} $, with  the \emph{$i$-th component} $X_i $ in $\mathfrak{s}(\mathcal{F}_i)$,  such that $ ds_{i-1}(X_i)= dt_{i-1}(X_i)=X_{i-1}$ for all $i\geq 1$.
We denote by $\mathfrak{s}(\mathcal{T}_B)$ the Lie algebra of symmetries of $\mathcal{T}_B$.
\end{definition}

The next theorem gives a class of bi-submersion tower to which any symmetry of the base singular foliation $\mathfrak F $ lifts.

\begin{theorem}\label{thm:symm-tower-path} Let $\mathfrak F $  be a  foliation. Let $\mathcal{T}_B$ be a 
path holonomy bi-submersion tower (or an exact path holonomy atlas bi-submersion tower). 
A vector field $X\in\mathfrak{X}(M)$ is a symmetry of $\mathfrak{F}$, i.e. $[X,\mathfrak{F}]\subset\mathfrak{F}$, if and only if it is the component on $M$ of a symmetry of $\mathcal{T}_B$.
\end{theorem}
\begin{proof}
It is a direct consequence of Proposition \ref{lift-symmetry-bi-susbmersion} and of item 1. resp. item 2. in Proposition \ref{prop:lift}. It is due to the fact that the tower $\mathcal{T}_B$ is generated by path  holonomy bi-submersions, and then we can lift  symmetries at every stage of the tower $\mathcal{T}_B$. 
\end{proof}

\begin{remark}\label{rmk:sym-chain-map}
Let $(X^i)_{i\geq 0}$ be a lift of $X_0:=X\in\mathfrak{s}(\mathfrak{F})$. For $i\geq 1$, $\nabla^i_X:=\mathrm{ad}_{X_i}$ preserves $\Gamma(\ker  ds_{i-1})$, since $X_i$ is $s_{i-1}$-projectable. Altogether, they define a chain map $(\nabla_X^i)_{i\geq 0}$ at the section level of the complex \eqref{eq:exa-bisub}, on projectable vector fields in  \eqref{complex-projectable}, since for every $i\geq 0$ and any $t_{i}$-projectable vector field $\xi\in \ker  ds_{i}$, \begin{align*}
     dt_{i}([X_{i+1},\xi])&=[ dt_{i}(X_{i+1}), dt_{i}(\xi)]\\&=[X_{i},  dt_{i}(\xi)], 
\end{align*}that is $ dt_{i}\circ \nabla^{i+1}_X=\nabla^i_X\circ  dt_{i}$. 
\end{remark}

\begin{remark}
In \cite{GarmendiaAlfonso2}, under some assumptions, it is shown that if a Lie group  $G$ acts on a foliated manifold $(M,\mathfrak{F})$, then it acts on its holonomy groupoid. It is likely that this result follows from Theorem \ref{thm:symm-tower-path}, this will be addressed in another study.
\end{remark}
\subsection{Lifts of actions of a Lie algebra on a bi-submersion tower}
We end the section with the following constructions and some natural questions.\\

Let $\mathcal{T}_B=(B_{i+1},s_i,t_i,\mathfrak F_i)_{i \geq 0} $ be an exact path holonomy  bi-submersion tower over a singular foliation $(M,\mathfrak F)$. 

By Theorem \ref{thm:symm-tower-path}, any vector field $X\in \mathfrak{s}(\mathfrak{F})$ lifts to a symmetry $(X^i)_{i\geq 0}$ of $\mathcal{T}_B$. Once a lift is chosen, we can define a linear map, 
$$X\in \mathfrak{s}(\mathfrak{F})\mapsto (X^i)_{i\geq 1}\in \mathfrak{s}(\mathcal{T}_B).$$
Let $\varrho\colon\mathfrak g\rightarrow \mathfrak{X}(M)$ be a strict symmetry action of a Lie algebra $\mathfrak{g}$ on $(M,\mathfrak F)$.  For $x\in\mathfrak g$, there exists   $(\varrho(x)^i)_{i \geq 0}$, with  $\varrho(x)^i\in\mathfrak{s}(\mathfrak{F}_i)\subset \mathfrak{X}(B_i)$ a symmetry of $\mathcal{T}_B$ such that $\varrho(x)^0=\varrho(x)\in \mathfrak s(\mathfrak F)$, by Theorem \ref{thm:symm-tower-path}. Consider the composition, \begin{equation}
x\in\mathfrak g\longmapsto \varrho(x)\in \mathfrak{s}(\mathfrak F)\longmapsto (\varrho(x)^i)_{i \geq 0}\in \mathfrak{s}(\mathcal{T}_B)\mapsto \varrho(x)^1\in \mathfrak{X}(B_1).
\end{equation}

\begin{lemma}\label{lemma:sym-action-tower}
For all $x,y\in \mathfrak{g}$, $$[\varrho(x),\varrho(y)]^1-[\varrho(x)^1, \varrho(y)^1]= dt_1(C_1(x,y))$$
with $C_1(x,y)\in \Gamma(\ker  ds_1\rightarrow B_2)$ a $t_1$-projectable vector field, for some bilinear map  $$C_1\colon \wedge^2\mathfrak g \longrightarrow \Gamma(\ker  ds_1\rightarrow B_2).$$
\end{lemma}
\begin{proof}
This follows from Lemma \ref{prop:sequence}, because $[\varrho(x),\varrho(y)]^1-[\varrho(x)^1, \varrho(y)^1]\in \Gamma(\ker ds_{0})\cap\Gamma(\ker dt_{0})$. 
\end{proof}

\begin{theorem}\label{thm:final-sym}
The map $C_1\colon \wedge^2\mathfrak g \longrightarrow \Gamma(\ker  ds_1\rightarrow B_2)$ of Lemma \ref{lemma:sym-action-tower} satisfies for all $x,y, z\in \mathfrak{g}$,
\begin{equation}
    C_1([x,y]_\mathfrak{g},z) + \nabla_{\varrho(x)}^2(C_1(y,z)) + \circlearrowleft(x,y,z)= dt_2(C_2(x,y,z)) 
\end{equation}
for some tri-linear map $C_2\colon\wedge^3\mathfrak g \longrightarrow \Gamma(\ker  ds_2\rightarrow B_3)$. Here, $\nabla^2$ is as in Remark \ref{rmk:sym-chain-map}.
\end{theorem}
\begin{proof}
For $x,y, z\in \mathfrak{g}$, 
\begin{align*}
    dt_1\left(C_1([x,y]_\mathfrak g,z)\right) +\circlearrowleft(x,y,z)&=[\varrho([x,y]_\mathfrak g),\varrho(z)]^1-[\varrho([x,y]_\mathfrak g)^1,\varrho(z)^1] +\circlearrowleft(x,y,z)\\&=\cancel{[[\varrho(x),\varrho(y)],\varrho(z)]^1}-[[\varrho(x),\varrho(y)]^1,\varrho(z)^1]+\circlearrowleft(x,y,z)\\&=-\cancel{[[\varrho(x)^1,\varrho(y)^1],\varrho(z)^1]}+[dt_1(C_1(x,y)),\varrho(z)^1]+\circlearrowleft(x,y,z)\\&=dt_1([C_1(x,y),\varrho(z)^2]) +\circlearrowleft(x,y,z).
\end{align*}
We have used Jacobi identity and $dt_1(\varrho(z)^2)=\varrho(z)^1$. This implies that 
\begin{equation}
    dt_1\left(C_1([x,y]_\mathfrak g,z)-[C_1(x,y),\varrho(z)^2] +\circlearrowleft(x,y,z)\right)=0.
\end{equation}Again Lemma \ref{prop:sequence} implies the result.\end{proof}

Here is a natural question:\\

\noindent
\textbf{Question}:  Can we continue the construction in Theorem \ref{thm:final-sym} to a Lie $\infty$-morphism?\\

\noindent
There is another natural type of questions:\\

\noindent
\textbf{Question}: Can a strict symmetry action $\varrho\colon\mathfrak g\rightarrow \mathfrak{X}(M)$ of a Lie algebra $\mathfrak{g}$ on a singular foliation $(M,\mathfrak F)$ lift to a given geometric resolution $(E_\bullet, \dd, \rho)$ of $\mathfrak{F}$?\\

\noindent
\textbf{Discussion}

    \begin{itemize}
        \item Can we answer this question with what we have? 
        \item We know by Theorem \ref{main} that $\varrho$ lifts on any universal Lie $\infty$-algebroid $(E,Q)$ of $\mathfrak{F}$ to a Lie $\infty$-morphism $\Phi\colon\mathfrak g\longrightarrow \mathfrak{X}_\bullet(E)$. If we can choose at least the polynomial-degree zero of the Taylor coefficient $\Phi_1\colon \wedge^2\mathfrak g\rightarrow \mathfrak{X}_{-1}(E)$ to be zero, then the answer is yes.
        \item If yes, can we assume that the previous action to preserve the $dg$-almost Lie algebroid? Again, if yes, then the polynomial-degree $+1$ can be chosen to be zero.
        \end{itemize}
        It seems that being able to suppress the higher Taylor coefficients has a strong geometric meaning. We intend to study that problem in a subsequent paper.\\

\noindent
\textbf{Question}: Can a bi-submersion tower be equipped with a (local) Lie $\infty$-groupoid structure?

\appendix
\begin{appendices}

\chapter{Tensor algebra}\label{appendix:tensor}
We have used almost everywhere in the thesis, modules or vector spaces that arise as the quotient of the tensor algebra. It is worth it to dedicate a section to recall the construction and some basic facts on tensor algebras. In this chapter we assume that the reader is familiar with the notion of $\mathcal{O}$-modules, graded $\mathcal{O}$-modules, and vector spaces.\\

 It is known that many algebras such as the exterior algebra, symmetric algebra, Clifford algebras \cite{Todorov-Ivan}, universal enveloping algebras \cite{Xavier-Bekaert} and many other algebras are the quotient of tensor algebras. These make the tensor algebra a fundamental and a very useful notion. In this thesis we have dealt a lot with $\mathcal{O}$-multilinear maps, it simply means that these maps are $\mathcal{O}$-linear w.r.t each argument while we fix the other arguments. The tensor algebra is  used to characterize multilinear relations between algebraic objects related to  modules or vector spaces. There are several ways to construct the tensor algebra, we refer the reader  e.g. to chapter 16 of \cite{lang2005algebra} or \cite{A.Gathmann,KEITH-CONRAD} for more details on the matter.\\

When $\mathcal{O}=\mathbb{K}$ the reader may replace "$\mathcal{O}$-module" by "$\mathbb{K}$-vector space" and the construction of the tensor algebra that we give below works the same.
 \section{Tensor product}
 Let $\mathcal{V}$ and $\mathcal{W}$ and $\mathcal{Z}$ be $\mathcal{O}$-modules. The \emph{tensor product $\mathcal{V}\otimes_\mathcal{O} \mathcal{W}$ of $\mathcal{V}$ and $\mathcal{W}$ over $\mathcal{O}$} is the $\mathcal{O}$-module which is defined by the following universal property 
 \begin{theorem}\label{th:tensors-univ}
 There exists a $\mathcal{O}$-bilinear map, $p\colon\mathcal{V}\times \mathcal W\rightarrow \mathcal{V}\otimes_\mathcal{O} \mathcal{W}$, that satisfies the following  properties: 
     \begin{enumerate}
     \item any  $\mathcal{O}$-bilinear map $B\colon \mathcal{V}\times \mathcal{W}\rightarrow \mathcal{Z}$ admits a unique $\mathcal{O}$-linear map $\gamma\colon\mathcal{V}\otimes_\mathcal{O} \mathcal{W}\rightarrow \mathcal{Z}$, such that $\gamma\circ p=B$. In other words, such that the following diagram commutes:
\begin{equation}\label{univ:tensor}
    \xymatrix{\mathcal{V}\times \mathcal{W}\ar[rr]^p\ar[dr]_B &&\mathcal{V}\otimes_\mathcal{O}\mathcal{W}\ar@{.>}[ld]_\gamma&\\&\mathcal{Z}&}
\end{equation}

\item Furthermore, if there is another $\mathcal{O}$-module $\mathcal{Q}$ and a $\mathcal{O}$-bilinear map, $p'\colon\mathcal{V}\times \mathcal W\rightarrow \mathcal Q$ with the property \eqref{univ:tensor}, then there exists a unique isomorphism $i\colon \mathcal{V}\otimes_\mathcal{O} \mathcal{W}\rightarrow \mathcal{Q}$ such that $p'=i\circ p$. 
 \end{enumerate}
 \end{theorem}
 \begin{proof}We first show "uniqueness" when such $\mathcal{O}$-module exits, then show existence.
 
 \noindent
 \textbf{Uniqueness}: Suppose that there exist two pairs $(\mathcal{V}{\otimes}\mathcal{W}, \,{p}\colon\mathcal{V}\times \mathcal W\rightarrow \mathcal{V}{\otimes}_\mathcal{O} \mathcal{W})$ and $(\mathcal{V}\tilde{\otimes}\mathcal{W}, \,\tilde{p}\colon\mathcal{V}\times \mathcal W\rightarrow \mathcal{V}\tilde{\otimes}_\mathcal{O} \mathcal{W})$ satisfying the property \ref{th:tensors-univ}. By using two times \ref{th:tensors-univ}, we deduce the existence of a unique $\mathcal{O}$-linear map $i\colon \mathcal{V}\otimes_\mathcal{O} \mathcal{W}\rightarrow \mathcal{V}\tilde{\otimes}\mathcal{W}$, and another one  $j\colon \mathcal{V}\tilde{\otimes}_\mathcal{O} \mathcal{W}\rightarrow \mathcal{V}{\otimes}\mathcal{W}$ such that $i\circ p=\tilde{p}$ and $j\circ \tilde{p}=p$. This implies that $j\circ i\circ p=p$ and $i\circ j \circ \tilde{p}=\tilde{p}$. By unicity of the factorization in \ref{th:tensors-univ}, we must have $i\circ j=\mathrm{id}$ and $j\circ i=\mathrm{id}$. This proves item 2.\\
 
 \noindent
 \textbf{Existence}: Let us consider the free $\mathcal{O}$-module $\mathcal{P}$ generated by elements of the Cartesian product $\mathcal{V}\times \mathcal{W}$, i.e. elements of $\mathcal{P}$ are formal finite sums of elements of the form $f(v,w)$, with $f\in\mathcal{O}$ and $v\in\mathcal{V}, w\in\mathcal{W}$. Next, we divide by the submodule $\mathcal{R}\subset \mathcal{P}$ of $\mathcal P$ generated by  
 elements of the following types:
 
 \begin{align*}
    & (v_{1}+v_{2},w)-(v_{1},w)-(v_{2},w),\\&(v,w_{1}+w_{2})-(v,w_{1})-(v,w_{2}),\\&(fv,w)-f(v,w),\,\;\text{and}\,\;(v,fw)-f(v,w)
 \end{align*}
which are the relations that elements of the tensor product must satisfy with $v,v_1,v_2\in\mathcal V$ and $w,w_1, w_2\in\mathcal W$ and $f\in\mathcal{O}$. In those notations, the tensor product of $\mathcal{V}$ and $\mathcal{W}$ over $\mathcal{O}$ is defined as the quotient $$\mathcal{V}\otimes_\mathcal{O} \mathcal{W}:=\mathcal{P}/\mathcal{R}.$$

For $v\in\mathcal{V}, w\in\mathcal{W}$, we denote the class of $(v,w)\in \mathcal{V}\times \mathcal W$ by $v\otimes w\in \mathcal V\otimes_\mathcal{O}\mathcal W$. This quotient comes with the natural map, $p\colon \mathcal{V}\times \mathcal{W}\rightarrow \mathcal{V}\otimes_\mathcal{O} \mathcal{W},\; (v,w)\mapsto v\otimes w$. For any $\mathcal{O}$-linear map $B\colon \mathcal{V}\times \mathcal{W}\rightarrow \mathcal{Z}$, the $\mathcal{O}$-linear map $q\colon \mathcal P\rightarrow \mathcal{Z}$ which is given on the basis of $\mathcal{P}$ by $(x,y)\mapsto q((x,y)):= B(x,y)$ clearly goes to quotient to define a $\mathcal{O}$-linear map $\widebar{q}\colon \mathcal{V}\otimes_\mathcal{O} \mathcal{W}\rightarrow \mathcal Z$, since e.g. 
\begin{align*}
    q\left((v_{1}+v_{2},w)-(v_{1},w)-(v_{2},w)\right)&=q((v_{1}+v_{2},w))-q((v_{1},w))-q((v_{2},w))\\&=B(v_{1}+v_{2},w)-B(v_{1},w)-B(v_{2},w)\\&=0,\qquad\text{by bilinearity of $B$}
\end{align*}
also,

\begin{align*}
    q((fv,w)-f(v,w))&=q((fv,w))-fq((v,w)),\qquad\text{by $\mathcal{O}$-lineraity of $q$}\\&=B(fv,w)-fB(v,w)=0,\,\quad\text{by $\mathcal{O}$-bilineraity of $B$.}
\end{align*}We did everything to get,  $\widebar{q}\circ p=B$. Therefore, we can take $\gamma=\widebar{q}$ to satisfy \eqref{univ:tensor}.\\

\end{proof}
 \begin{proposition}
 Let $\mathcal{V},\mathcal{W},\mathcal{Z}$ be $\mathcal{O}$-modules. We have the following isomorphisms

 \begin{enumerate}
     \item  $\mathcal V\otimes_\mathcal{O}\mathcal W\simeq  \mathcal W\otimes_\mathcal{O}\mathcal V$.
    \item $\mathcal{O}\otimes_\mathcal{O}\mathcal V\simeq \mathcal{V}$, and for any ideal $\mathcal{I}\subset \mathcal{O}$, Also, $\frac{\mathcal{O}}{\mathcal{I}}\otimes_\mathcal{O}\mathcal V\simeq \frac{\mathcal{V}}{\mathcal{I}\mathcal{V}}$
    \item $(\mathcal V\otimes_\mathcal{O}\mathcal W) \otimes_\mathcal{O}\mathcal Z\simeq \mathcal V\otimes_\mathcal{O}(\mathcal W \otimes_\mathcal{O}\mathcal Z)$.
\end{enumerate}
\end{proposition}

 \begin{proof}
 For item 1, the map $(v,w)\in\mathcal{V}\times\mathcal{W}\mapsto w\otimes v \in\mathcal{V}\otimes_\mathcal{O}\mathcal{W}$ goes to quotient to \emph{the twist map} $v\otimes w\mapsto w\otimes v$  and give the isomorphism  whose inverse is $w\otimes v\in\mathcal{W}\otimes_\mathcal{O}\mathcal{V}\mapsto v\otimes w\in \mathcal{V}\otimes_\mathcal{O}\mathcal{W}$.
 
 For item 2, the map $(f,v)\in\mathcal{O}\times\mathcal{V}\rightarrow fv$ also induces an isomorphism whose inverse is $v\in\mathcal{V}\mapsto 1\otimes v\in\mathcal{O}\otimes_\mathcal{O}\mathcal{V}$. A similar map gives the second clause.

For item 3, the isomorphism is, obviously,  $(v_1\otimes v_2)\otimes v_3\mapsto v_1\otimes (v_2\otimes v_3)$.  This  is trivially extended to more $\mathcal{O}$-modules. 
 \end{proof}
  \begin{remark}
A similar  construction as in Theorem \ref{th:tensors-univ} can be done for a finite family of $\mathcal{O}$-modules $\mathcal V_1,\ldots, \mathcal V_r$, as in the case of bilinear maps. Then $\mathcal V_1\otimes_\mathcal{O}\cdots\otimes_\mathcal{O}\mathcal V_r$ is a universal object that factorizes $r$-multilinear maps, defined on $\mathcal V_1\times\cdots\times\mathcal V_r$. We naturally have, $$(\mathcal V_1\otimes_\mathcal{O}\mathcal{V}_2)\otimes_\mathcal{O}\mathcal V_3\simeq\mathcal V_1\otimes_\mathcal{O}(\mathcal{V}_2\otimes_\mathcal{O}\mathcal V_3)\simeq \mathcal V_1\otimes_\mathcal{O}\mathcal{V}_2\otimes_\mathcal{O}\mathcal V_3.$$So, in this thesis, we make no difference in how we denote the elements $(v_1\otimes v_2)\otimes v_3$, $v_1 \otimes(v_2\otimes v_3)$, or $v_1\otimes v_2 \otimes v_3$.
\end{remark}

 \begin{remark}\label{rmk:product-on-tensor} It is important to notice that if $\mathcal{C}$ and $\mathcal{C}'$ are graded $\mathcal O$-algebras,  then $\mathcal{C}\otimes_\mathcal{O}\mathcal{C}'$ is a graded $\mathcal O$-algebra with product \begin{equation}\label{produt-overtensors}
  (c_1\otimes c_1')(c_2\otimes c_2'):=(-1)^{|c_1'||c_2|}c_1c_2\otimes c_1'c_2'   
 \end{equation}
for homogeneous elements $c_1, c_2\in \mathcal{C}$ and $c_1', c_2'\in \mathcal{C}'$. Moreover, if  $\mathcal{C}$ and $\mathcal{C'}$ are unitary with units $1_C$ and $1_{C'}$, respectively, then the algebra $\mathcal{C}\otimes_\mathcal{O}\mathcal{C}$ is also unitary with unit $1_\mathcal{C}\otimes 1_\mathcal{C'}$. Here, to avoid explosion of notations, we denote the product of two elements $a, b$ by $ab$. Also, we will not make notational distinctions between the unit elements.
\end{remark}
 
\begin{convention}\label{conv:tensor-of-maps}
For $\Phi \colon \mathcal C \rightarrow \mathcal C'$ and
$\Psi \colon \mathcal C'' \rightarrow \mathcal C'''$  two homogeneous morphisms of $\mathbb{Z}$-graded $\mathcal{O}$-modules, then $\Phi\otimes\Psi : \mathcal C\otimes_\mathcal{O} \mathcal C''\rightarrow \mathcal C' \otimes_\mathcal{O}\mathcal C'''$ stands for the following morphism:
\begin{equation*}
    (\Phi\otimes\Psi)(x\otimes y) = (-1)^{\lvert \Psi\rvert\lvert x\rvert}\Phi(x)\otimes\Psi(y),\; \text{for all homogeneous}\; x \in \mathcal C , y \in \mathcal C''.\end{equation*}
 \end{convention}
 
 \begin{remark}
 In virtue of Remark \ref{rmk:product-on-tensor} and Convention \ref{conv:tensor-of-maps}, if  $\Phi \colon \mathcal C \rightarrow \mathcal C'$ and
$\Psi \colon \mathcal C'' \rightarrow \mathcal C'''$ are graded algebra morphisms, then $\Phi\otimes\Psi$ is also a graded algebra morphism w.r.t to the product defined in \eqref{produt-overtensors}: 
\begin{align*}
    \Phi\otimes\Psi((c_1\otimes c_1')(c_2\otimes c_2'))&=(-1)^{|c_1'||c_2|}    \Phi\otimes\Psi (c_1c_2\otimes c_1'c_2' )\\&=(-1)^{|c_1'||c_2|}\Phi(c_1c_2)\otimes\Psi(c_1'c_2')\\&=(-1)^{|c_1'||c_2|}\Phi(c_1)\Phi(c_2)\otimes\Psi(c_1')\Psi(c_2')\\&=\left(\Phi(c_1)\otimes\Phi(c_1')\right)\left(\Psi(c_2)\otimes\Psi(c_2')\right)\\&= \left(\Phi\otimes\Psi(c_1\otimes c_1')\right)\left( \Phi\otimes\Psi(c_2\otimes c_2')\right),\quad\text{since $\Phi, \Psi$ are of degree $0$.}
\end{align*}

\end{remark}
 
\section{The tensor algebra of a linear space} 
 Let $\mathcal{V}$ be an $\mathcal{O}$-module. For $k\in \mathbb{N}_0$, the \emph{$k$-th tensor power $T^k_\mathcal O \mathcal{V}$ over $\mathcal O$ of $\mathcal V$} (\emph{elements of polynomial-degree $k$}) is the tensor product of $\mathcal{V}$ with itself $k$ times, namely
$$ T^k_\mathcal O \mathcal{V}:=
\underbrace{ \mathcal V \otimes_\mathcal O \cdots \otimes_\mathcal O \mathcal V}_{\hbox{\small{$k$ times}}}
$$
we also adopt the convention $T^0_\mathcal O \mathcal{V}\simeq \mathcal{O}$, and $\mathcal{V}\simeq T^1_\mathcal{O}\mathcal{V}$. This leads to consider the $\mathcal{O}$-module constructed as the direct sum of the tensor powers, i.e. $$ T^\bullet_\mathcal O \mathcal V:= \bigoplus_{k =1}^\infty T^k_\mathcal O \mathcal{V}=\mathcal{O}\oplus \mathcal{V}\oplus (\mathcal{V}\otimes_\mathcal{O} \mathcal{V})\oplus(\mathcal{V}\otimes_\mathcal{O} \mathcal{V}\otimes_\mathcal{O} \mathcal{V})\oplus\cdots$$

\begin{proposition}
The $\mathcal{O}$-module $T^\bullet_\mathcal O \mathcal V$  comes equipped with a graded unital $\mathcal{O}$-algebra structure, which is induced by the canonical map $$T^k_\mathcal O \mathcal V\times T^\ell_\mathcal O \mathcal V \longrightarrow T^k_\mathcal O \mathcal V\otimes_\mathcal{O} T^\ell_\mathcal O \mathcal V\simeq T^{k+\ell}_\mathcal O \mathcal V,$$
that is extended by bilinearity to all $T_\mathcal{O}^\bullet\mathcal{V}$. We denote this product by $\otimes$. The unit element is $1\in\mathcal{O}\simeq T_\mathcal{O}^0\mathcal{V}$, in particular, we have $1\otimes v=1\cdot v=v$  for every $v\in \mathcal{V}$. 
\end{proposition}

 \subsection{$T_\mathcal{O}^\bullet\mathcal{V}$ as a co-algebra}
 In this section, we consider an  $\mathcal{O}$-module $\mathcal{V}$ which can be possibly graded. However, the construction is independent whether the module is graded or not.\\

 The notion of co-algebra structure is important in the context of the thesis, since it allows dealing with infinite dimension objects. It appears all along the thesis. For this concept, see Definition \ref{def:coproduct}. A natural way to construct a co-product (see e.g \cite{Loday-Vallette,Kassel}) structure \footnote{here ${\otimes}$ is an "outer" tensor product, it should not be confused with the internal tensor product in $T_\mathcal{O}^\bullet\mathcal{V}$ that denotes also its graded algebra structure.} on $T_\mathcal{O}^\bullet\mathcal{V}$ $$\Delta\colon T_\mathcal{O}^\bullet\mathcal{V}\longrightarrow T_\mathcal{O}^\bullet\mathcal{V}\bigotimes \,T_\mathcal{O}^\bullet\mathcal{V}$$  is to define it on elements $v\in\mathcal{V}\simeq T_\mathcal{O}^1\mathcal V$ of polynomial-degree $1$, and on the unit element $1\in \mathcal{O}\simeq T_\mathcal{O}^0\mathcal V$ and extend it to a (degree $0$) $\mathcal{O}$-algebra morphism to the whole $T_\mathcal{O}^\bullet \mathcal V$, namely for $v_1\otimes\cdots\otimes v_k\in T_\mathcal{O}^k\mathcal V$, $$\Delta(v_1\otimes\cdots\otimes v_k)=\Delta(v_1)\cdots \Delta(v_k).$$ The one which is defined by \begin{align*}
     v&\mapsto v\otimes 1 + 1 \otimes v\\1&\mapsto 1 \otimes1
 \end{align*}
endows $T_\mathcal{O}^\bullet\mathcal{V}$ with a co-associative (co-commutative) co-algebra structure, that is, it satisfies the axioms of Definition \ref{def:coproduct}. Indeed, the maps $$\Delta\otimes\mathrm{id}, \mathrm{id} \otimes \Delta\colon T_\mathcal{O}^\bullet\mathcal{V}\otimes T_\mathcal{O}^\bullet\mathcal{V} \rightarrow T_\mathcal{O}^\bullet\mathcal{V} \otimes T_\mathcal{O}^\bullet\mathcal{V}\otimes T_\mathcal{O}^\bullet\mathcal{V}$$ are algebra morphisms, so are $(\Delta \otimes \mathrm{id} )\circ\Delta$ and $(\mathrm{id} \otimes \Delta)\circ\Delta$, therefore it suffices to check that they coincide on $\mathcal{V}$. Let us check that. For $v\in \mathcal{V}$,

\begin{align*}
    (\Delta \otimes \mathrm{id} )\circ\Delta(v)&= \Delta \otimes \mathrm{id} (v\otimes 1 + 1 \otimes v)\\&=\Delta(v){\otimes} 1 + \Delta(1){\otimes} v\\&= (v\otimes 1 + 1 \otimes v)\otimes 1+ (1 \otimes 1)\otimes v\\&=(v\otimes 1)\otimes 1 + ( 1 \otimes v)\otimes 1+ (1 \otimes 1)\otimes v\\&=(\mathrm{id} \otimes \Delta)\circ\Delta(v), \quad\text{by associativity of $\otimes$.}
\end{align*}
By extending $\Delta$ to an algebra morphism, one gets explicit expressions as follows. For ${\displaystyle v\otimes w\in T^{2}_\mathcal{O}\mathcal{V}}$, one has 

\begin{align*}
    \Delta(v\otimes w)&=\Delta (v)\Delta (w)\\&=(v\otimes 1+1\otimes v) (w\otimes 1+1\otimes w)\\&=(v\otimes w)\otimes 1+v\otimes w+(-1)^{|v||w|}w\otimes v+1\otimes (v\otimes w).
\end{align*}We have used  Formula \eqref{produt-overtensors} and $ 1\otimes v=1\cdot v=v$, each time $\otimes$ is the tensor symbol in $T^\bullet_\mathcal{O}\mathcal{V}$. If  $\mathcal{V}$ is not graded, i.e., concentrated in degree zero, there is no sign $(-1)^{|v||w|}$. More generally, for every $v_1,\ldots, v_n\in \mathcal{V}$,

\begin{equation}
\Delta(v_1\otimes\cdots\otimes v_n)=\sum_{i=1}^{n-1}\epsilon(\sigma)\sum_{\sigma\in\mathfrak{S}({i,n-i})}v_{\sigma(1)}\otimes\cdots \otimes v_{\sigma(i)}\bigotimes v_{\sigma(i+1)}\otimes\cdots\otimes v_{\sigma(n)},\end{equation}
 where $\sigma\in\mathfrak{S}({i,n-i})$ is a $(i,n-i)$-shuffle and  $\epsilon(\sigma)$ is the Koszul sign associated to the $n$-uplet $v_1,\ldots, v_n$. See Section \ref{conventions} for more details.
\chapter{Homological algebra}\label{appendix:mod}

The goal of this chapter is to introduce some important results on homological algebra, which are used in this thesis. Although, these are classical notions in commutative algebra, I think it is important to make a brush-up on them for the readability of the thesis. Most of the notions of this chapter can be found in \cite{Weibel,zbMATH00704831,Matsumura,Hazewinkel-Michiel}, I also have learned a lot from the Lecture notes \cite{Caroline-Lassueur}.

\section{Complexes of modules}\label{app:complexes-modules}

We recall that a module over $\mathcal{O}$ is like a vector space in the sense that all the axioms still hold, except that the underlying field is replaced by $\mathcal{O}$. In this section, when $\mathcal{O}=\mathbb{K}$, the reader may replace "module" by "vector space".\\

For us, a \emph{$\mathbb{Z}$-graded module over $\mathcal{O}$} is a module $\mathcal{V}$ endowed with a direct sum decomposition $\mathcal{V}=\oplus_{i\in \mathbb Z} \mathcal{V}_i$ of $\mathcal{O}$-modules. We simply say "$\mathcal{O}$-modules" for graded modules which are concentrated in degree zero. For every $i\in \mathbb{Z}$, elements of $\mathcal{V}_i$ are said to be of \emph{degree} $i$. Let $\mathcal{V}$ and $\mathcal{W}$ be graded $\mathbb Z$-modules over $\mathcal{O}$, a $\mathcal{O}$-linear map $L\colon \mathcal{V}\rightarrow \mathcal{W}$ is said to be \emph{homogeneous \emph{or a} morphism of $\mathbb{Z}$-graded $\mathcal{O}$-modules of degree $|L|:=\ell\in \mathbb{Z}$}, if $L(\mathcal{V}_k)\subseteq \mathcal{V}_{k+\ell}$ for all $k\in \mathbb{Z}$. The set of  all $\mathcal{O}$-linear maps of degree $\ell$ from $\mathcal{V}$ to $\mathcal{W}$ form an $\mathcal{O}$-module that  we denote by $\mathrm{Hom}_\mathcal{O}^\ell(\mathcal{V},\mathcal{W})$. This implies that $\mathrm{Hom}_\mathcal{O}(\mathcal{V},\mathcal{W}):=\bigoplus_{\ell\in\mathbb{Z}}\mathrm{Hom}_\mathcal{O}^\ell(\mathcal{V},\mathcal{W})$ is a graded module. Also,  this graded module comes equipped with natural graded Lie bracket given by the graded commutator, namely \begin{equation}
    [F,G]:=F\circ G -(-1)^{|F||G|}G\circ F
\end{equation} 
for any homogeneous elements $F,G\in\mathrm{Hom}_\mathcal{O}(\mathcal{V},\mathcal{W})$. It is easily checked that the bracket satisfies

\begin{enumerate}
             \item $[F,G]=-(-1)^{|F||G|}[G,F]$\quad (\emph{graded skew-symmetry})
             \item $(-1)^{|F||H|}[F, [G, H]] + (-1)^{|H||G|}[H, [F, G]]+(-1)^{|G||F|}[G, [H, F]]=0,$\quad (\emph{graded Jacobi identity})
\end{enumerate}
         for homogeneous $\mathcal{O}$-linear maps $F, G, H\in\mathrm{Hom}_\mathcal{O}(\mathcal{V},\mathcal{W})$.


\begin{definition}
A \emph{complex of $\mathcal{O}$-modules} $(\mathcal{V}_\bullet,\dd)$ is a graded module $\mathcal{V}=\bigoplus_{i\in \mathbb{Z}} \mathcal{V}_i$ together with a  squared to zero $\mathcal{O}$-linear map $\dd\colon \mathcal{V}\rightarrow \mathcal{V}$ of degree $+1$ called the \emph{differential map}. In other words, it is a sequence 

\begin{equation}\label{eq:complex}
     \quad  \cdots {\longrightarrow}\mathcal{V}_{i-1} \stackrel{\dd}{\longrightarrow} \mathcal{V}_i \stackrel{\dd}{\longrightarrow}\mathcal{V}_{i+1}{\longrightarrow}\cdots 
\end{equation}

$\mathcal{O}$-linear maps such that  $\dd^2=\dd\circ\dd=0$. 
\begin{enumerate}
\item For every $i\in\mathbb Z$, elements of $\mathcal{V}_i$ are called \emph{cochains of degree $i$}. 
    \item We say  that \eqref{eq:complex} is  \emph{bounded below/above} if $\mathcal{V}_i=0$ for $i\leq n/i\geq n$, for some $n\in \mathbb{Z}$.
    
    \item A \emph{subcomplex} of a complex $(\mathcal{V}, \dd)$ a collection  of $\mathcal{O}$-modules $\left(\mathcal{V}'_i\subseteq \mathcal{V}_i\right)_{i\in \mathbb{Z}}$ such that $\dd(\mathcal{V}'_i)\subset \mathcal{V}_{i+1}$ for each $i\in\mathbb Z$. In particular $(\mathcal{V}', \dd'=\dd|_{\mathcal{V}'})$ is a complex of $\mathcal{O}$-modules.
    
    In particular, it induces a complex $\left(\mathcal{V}/\mathcal{V}',\, \overline{\dd}\right)$ called the \emph{quotient complex} where for $i\in \mathbb{Z}$,\, $\left(\mathcal{V}/\mathcal{V}'\right)_i:=\mathcal{V}_i/\mathcal{V}_i'$ and $$\dd\colon \mathcal{V}_i/\mathcal{V}_i'\longrightarrow \mathcal{V}_{i+1}/\mathcal{V}_{i+1}'$$ is determined uniquely by the universal property of the quotient.
\end{enumerate}
\end{definition}
\begin{remark}
Let $(\mathcal{V}, \dd)$  be a complex of $\mathcal{O}$-modules. Denote by $\dd_i\colon \mathcal{V}_i\rightarrow \mathcal{V}_{i+1}$ the restriction of $\dd$ to $\mathcal{V}_i$. Then $\dd^2=0$ means that $\dd_i\circ \dd_{i-1}$ for each $i\in \mathbb{Z}$. In particular, $\mathrm{Im}\dd_{i-1}\subseteq\ker \dd_{i}$ for every $i\in\mathbb Z$. This leads us to the next definition.
\end{remark}

\begin{definition}
   Let $(\mathcal{V}, \dd)$ be a complex of $\mathcal{O}$-modules. For each $i\in \mathbb{Z}$,
   
   \begin{enumerate}
       \item a \emph{$i$-cocycle of $(\mathcal{V},\dd)$} is an  element of $\ker(\mathcal{V}_i \stackrel{\dd}{\longrightarrow}\mathcal{V}_{i+1})$;
       \item a \emph{$i$-coboundary of $(\mathcal{V},\dd)$} is an element of $\mathrm{Im} (\mathcal{V}_{i-1} \stackrel{\dd}{\longrightarrow}\mathcal{V}_{i})$;
       \item the \emph{$i$-th cohomology group of $(\mathcal{V},\dd)$} is the quotient $H^i(\mathcal{V}):=\dfrac{\ker(\mathcal{V}_i \stackrel{\dd}{\longrightarrow}\mathcal{V}_{i+1})}{\mathrm{Im} (\mathcal{V}_{i-1} \stackrel{\dd}{\longrightarrow}\mathcal{V}_{i})}$.
   \end{enumerate}
   The complex $(\mathcal{V}, \dd)$ is \emph{exact at $i$} if $H^i(\mathcal{V})=\{0\}$. It is said to be \emph{exact} or \emph{acyclic} if it is exact at every degree $i\in\mathbb{Z}$.
\end{definition}

\begin{definition}
Let $(\mathcal{V},\dd^\mathcal{V})$ and $(\mathcal{W},\dd^\mathcal{W})$ be complexes of $\mathcal{O}$-modules.  

\begin{enumerate}
    \item 
A \emph{chain map} or \emph{complex of $\mathcal{O}$-modules morphism} between the complexes $(\mathcal{V},\dd)$ and $(\mathcal{W},\dd)$  is a $\mathcal{O}$-linear map $L\colon \mathcal{V}\rightarrow \mathcal{W}$ of degree $0$, which commutes with the differentials, that is a  collection of $\mathcal{O}$-linear map $L_\bullet\colon \mathcal{V}_{\bullet}\longrightarrow \mathcal{W}_{\bullet}$, such that the following diagram commutes

\begin{equation}\label{}
        \xymatrix{
\cdots\ar[r]&\mathcal{V}_{i}\ar[d]_{L_{i}} \ar[r]^{\dd^\mathcal{V} }&\mathcal{V}_{i+1}\ar[d]^{L_{i+1}} \ar[r]&\cdots \\
\cdots\ar[r]&\mathcal{W}_{i} \ar[r]^{\dd^\mathcal{W}}&\mathcal{W}_{i+1}\ar[r]&\cdots}\end{equation}

i.e. $\dd^\mathcal{W}\circ L_{i}= L_{i+1}\circ \dd^{\mathcal{V}}$ for every $i\in \mathbb{Z}$.

\item A \emph{homotopy} between two chain maps $K_\bullet,L_\bullet\colon \mathcal{V}_{\bullet}\longrightarrow \mathcal{W}_{\bullet}$ is the datum $\{h_i\colon \mathcal{V}_{i}\longrightarrow \mathcal{W}_{i-1}\}_{i\geq 1}$ of $\mathcal{O}$-linear maps, that satisfies 
for each $i\in \mathbb{Z}$, $K_i- L_i= \dd^{\mathcal{W}}\circ h_i +h_{i-1} \circ \dd^{\mathcal{V}}$. These maps are displayed in the following diagram as

\begin{equation}\label{}
        \xymatrix{
\cdots\ar[r]&\mathcal{V}_{i-1}\ar[d]_{K_{i-1}- L_{i-1}}\ar[r]^{\dd^{\mathcal{V}}}&\mathcal{V}_{i}\ar@{.>}[ld]_{h_i} \ar[d]<5pt>_{K_i- L_i} \ar[r]^{\dd^{\mathcal{V}} }&\mathcal{V}_{i+1}\ar[d]^{K_{i+1}- L_{i+1}}\ar@{.>}[ld]_{h_{i+1}} \ar[r]&\cdots \\
\cdots\ar[r]&\mathcal{W}_{i-1}\ar[r]^{\dd^{\mathcal{W}}}&\mathcal{W}_{i} \ar[r]^{\dd^{\mathcal{W}}}&\mathcal{W}_{i+1}\ar[r]&\cdots}
\end{equation}

\begin{enumerate}
\item When there is a homotopy between two chain maps, $L_\bullet,K_\bullet\colon \mathcal{V}_{\bullet}\longrightarrow \mathcal{W}_{\bullet}$, we often  write $L\sim K$. One can check that $\sim$ is indeed an equivalence relation.
    \item Two complexes of $\mathcal{O}$-modules $(\mathcal{V}, \dd^\mathcal{V})$ and $(\mathcal{W}, \dd^\mathcal{W})$  are said to be \emph{homotopy equivalent}, if there exist chain maps $L_\bullet\colon \mathcal{V}_{\bullet}\longrightarrow \mathcal{W}_{\bullet}$ and $K_\bullet\colon \mathcal{W}_{\bullet}\longrightarrow \mathcal{V}_{\bullet}$ such that $L\circ K \sim \mathrm{id}_{\mathcal{W}_\bullet}$ and $K\circ L \sim \mathrm{id}_{\mathcal{V}_\bullet}$. Likewise, one can check that homotopy equivalence between complexes of $\mathcal O$-modules is an equivalence relation.

\end{enumerate}
\end{enumerate}
\end{definition}

\begin{remark}
Note that in particular,  if $L_\bullet\colon \mathcal{V}_{\bullet}\longrightarrow \mathcal{W}_{\bullet}$ is a chain map that is an $\mathcal{O}$-linear isomorphism, then its inverse is also a chain map. Thus, they define a homotopy equivalence between $(\mathcal{V}, \dd^\mathcal{V})$ and $(\mathcal{W}, \dd^\mathcal{W})$.
\end{remark}

\begin{remark}
Let $L_\bullet\colon \mathcal{V}_{\bullet}\longrightarrow \mathcal{W}_{\bullet}$ be a chain map. For every $i\in\mathbb{Z}$, we have $$L\left(\ker(\mathcal{V}_i \stackrel{\dd^\mathcal{V}}{\longrightarrow}\mathcal{V}_{i+1})\right)\subseteq \ker(\mathcal{W}_i \stackrel{\dd^\mathcal{W}}{\longrightarrow}\mathcal{W}_{i+1}),\,\; \text{and}\,\; L\left(\mathrm{Im}(\mathcal{V}_{i-1} \stackrel{\dd^\mathcal{V}}{\longrightarrow}\mathcal{V}_{i})\right)\subseteq \mathrm{Im}(\mathcal{W}_{i-1} \stackrel{\dd^\mathcal{W}}{\longrightarrow}\mathcal{W}_{i}).$$ $L$ induces naturally a well-defined $\mathcal{O}$-linear map $H(L)\colon H(\mathcal{V})\rightarrow H(\mathcal{W}), [v]\mapsto [L(v)]$: for every $v\in \ker(\mathcal{V}_i \stackrel{\dd^\mathcal{V}}{\longrightarrow}\mathcal{V}_{i+1})$ and $v_0\in \mathcal{V}_{i-1}$
\begin{align*}
    H(L)([v+\dd^\mathcal{V}(v_0)])&=[L(v+\dd^\mathcal{V}(v_0))]=[L(v)+\dd^\mathcal{W}\circ L(v_0))]=[L(v)]=H(L)([v]).
\end{align*}
Notice that
\begin{enumerate}
    \item homotopic chain maps induce the same map on cohomology groups.
    
    \item if the complexes of $\mathcal{O}$-modules $(\mathcal{V}, \dd^\mathcal{V})$ and $(\mathcal{W}, \dd^\mathcal{W})$ are homotopy equivalent through $L$, then $H(L)\colon H(\mathcal{V})\rightarrow H(\mathcal{W})$ is an isomorphism.
\end{enumerate}
\end{remark}
\begin{proposition}\label{prop:sub-complex-exact}
If $(\mathcal{V}',\dd')$ is an acyclic subcomplex of a complex $(\mathcal{V},\dd)$, then $$H^\bullet(\mathcal{V}/\mathcal{V}')\simeq{H^\bullet(\mathcal{V})}.$$
\end{proposition}

\begin{proof}
The projection  $p_\bullet\colon \mathcal{V}\longrightarrow \mathcal{V}_\bullet/\mathcal{V}'_\bullet$ is a chain map from $(\mathcal{V},\dd)$ and $(\mathcal{V}/\mathcal{V}', \overline{\dd})$. We claim  $p_\bullet$ induces an isomorphism on the cohomology groups. To show injectivity of  $H(p)\colon H^\bullet(\mathcal{V})\longrightarrow H^\bullet(\mathcal{V}/\mathcal{V}')$, let $e\in \ker \dd$ such that $p(e)=\overline{e}\in \mathrm{Im}\,\overline{\dd}$, i.e. there is $u\in\mathcal{V}_{|e|-1}$ such that $\overline{e}=\overline{\dd}(\overline{u})$. It follows that $ e-\dd u\in \mathcal{V}'_{|e|}$. This implies,
\begin{equation}
    \dd'(e-\dd u)=\dd(e-\dd u)=0.
\end{equation}
By Exactness of $(\mathcal{V}',\dd')$,
\begin{align*}
&\Longrightarrow e-\dd u= \dd'(v),\hspace{1cm}\text{(for some $v\in \mathcal{V}'_{|e|-1}\subset  \mathcal{V}_{|e|-1}$)}\\&\Longrightarrow\; e=\dd(u+v)\in \mathrm{Im}\,\dd.
\end{align*}
This proves injectivity. Surjectivity goes as follows: let $e\in \mathcal{V}_{i}$ such that $\overline{\dd}(\overline{e})=0$. \begin{align*}
    \overline{\dd}(\overline{e})=0&\Longrightarrow\dd(e)\in \mathcal{V}_{i+1}'
\end{align*}We have,  $\dd'(\dd(e))=\dd\circ\dd(e)=0$. By exactness of $(\mathcal{V}',\dd')$, we can write $\dd(e)=\dd'(v)\; \text{for some}\; v\in \mathcal{V}'_{i}$. This implies that, $e-v=u\in \ker\dd$. Therefore, $$H(p)([u])= [p(u)]=[p(e)-p(v))]=[p(e)]=[\overline{e}].$$ This completes the proof.


\end{proof}

\subsubsection{The Chevalley-Eilenberg complex}
The following example of complex is important. We have used it several times in this thesis. Especially in Chapter \ref{chap:obstruction-theory} to define obstruction classes. Let us recall the definition. We refer the reader e.g. to \cite{Wagemann} for more details.\\

Let $(\mathfrak g, \lb_\mathfrak{g})$ be a Lie algebra and $V$ a $\mathbb{K}$-vector space. 

\begin{definition}
 A \emph{representation or action} of  $\mathfrak{g}$ on $V$ is a Lie algebra morphism \begin{equation}\label{eq:represetation}
     \nu\colon \mathfrak (\mathfrak g, \lb_\mathfrak{g})\longrightarrow (\mathrm{End}(V), \lb)
 \end{equation}
 where $(\mathrm{End}(V), \lb)$ denotes the vector space $\mathrm{End}(V)$ of endomorphisms of $V$ together with the Lie bracket $\lb$  which is the commutator: $[\alpha, \beta]=\alpha\circ \beta -\beta\circ \alpha, \; \forall\,\alpha, \beta\in \mathrm{End}(V)$. In this case, $V$ is then called a \emph{$\mathfrak g$-module} (w.r.t to $\nu$). In the literature, the action $\nu$ is often denoted by $\cdot$.\\

\noindent
Equation \eqref{eq:represetation} means that for all $x,y\in \mathfrak{g}$, $$\nu([x,y]_\mathfrak{g})=[\nu(x),\nu(y)]=\nu(x)\circ\nu(y)-\nu(y)\circ\nu(x).$$
\end{definition}
\begin{example}Here are two important examples of $\mathfrak{g}$-modules that we often use.
\begin{itemize}
\item \emph{The adjoint action}: $\mathfrak g$ acts on itself by the Lie bracket, i.e.  $\nu(x) ( y) := [x, y]_\mathfrak{g}$ for all $x, y \in \mathfrak g$. Indeed, $\nu\colon \mathfrak g \longrightarrow \mathrm{End}(\mathfrak g),\, x\longmapsto [x,\cdot\,]_\mathfrak g=:\mathrm{ad}_x$ is a Lie algebra morphism:
\begin{align*}
    \mathrm{ad}_{[x,y]_\mathfrak g}&=[[x,y]_\mathfrak g,\cdot\,]_\mathfrak g\\&=-[[y,\cdot\,]_\mathfrak g,x]_\mathfrak g-[[\cdot\,,x]_\mathfrak g,y]_\mathfrak g, \qquad(\text{by identity of Jacobi})\\&=[x,[y,\cdot\,]_\mathfrak g]_\mathfrak g-[y,[x,\cdot\,]_\mathfrak g]_\mathfrak g\\&=\mathrm{ad}_x\circ\mathrm{ad}_y-\mathrm{ad}_y\circ\mathrm{ad}_x=[\mathrm{ad}_x,\mathrm{ad}_y].
\end{align*}
\item \emph{The trivial action}: i.e. $\mathbb{K}$ is a $\mathfrak{g}$-module through the action $\nu(x)(\lambda) := 0$ for all $x\in \mathfrak{g}$ and all $\lambda \in  \mathbb{K}$.
\end{itemize}`
\end{example}
Now let us recall the definition of the Chevalley-Eilenberg complex. 
\begin{definition}Let $\nu\colon\mathfrak g \longrightarrow \mathrm{End}(V)$ be an action of $\mathfrak g$ on $V$. The \emph{Chevalley-Eilenberg complex of $\mathfrak g$ valued in $V$} is the  complex 

\begin{equation}\label{eq:CE-complex}
     \quad  \cdots {\longrightarrow}\mathrm{Hom}_\mathbb K(\wedge^{i-1}\mathfrak g,V) \stackrel{\dd^{\mathrm{CE}}}{\longrightarrow} \mathrm{Hom}_\mathbb K(\wedge^{i}\mathfrak g,V) \stackrel{\dd^{\mathrm{CE}}}{\longrightarrow}\mathrm{Hom}_\mathbb K(\wedge^{i+1}\mathfrak g,V){\longrightarrow}\cdots 
\end{equation}
whose $i$-th cochains  space is defined  to be
$\mathrm{Hom}_\mathbb K(\wedge^i\mathfrak g,V)$, the vector space of $i$-linear skew-symmetric linear maps from $\underbrace{\mathfrak g\times \cdots \times\mathfrak{g}}_{\text{$i$-times}}$ to $V$, under the convention $\mathrm{Hom}_\mathbb K(\wedge^0\mathfrak g,V)\simeq V$. The differential map is defined for $\mu\in\mathrm{Hom}_\mathbb K(\wedge^{i}\mathfrak g,V) $ by

\begin{align*}
    \left(\dd^{\mathrm{CE}}\mu\right)(x_1 ,\ldots,x_{i+1})&=\sum_{k=1}^{i+1}(-1)^{k-1}\nu(x_k)(\mu(x_1 ,\ldots,\widehat{x}_k, \ldots,x_{i+1}))\\&+\sum_{1\leq k<l\leq i+1}(-1)^{k+l}\mu([x_k,x_l]_\mathfrak{g},x_1,\ldots,\widehat{x}_{kl},\ldots, x_{i+1})
\end{align*}
where $\widehat{x}_k$ means ${x}_k$ is missing in the $k$-th place also $\widehat{x}_{kl}$  means that $x_k, x_l$ are missing in the $k$-th and $l$-th place respectively.
\end{definition}

\begin{remark}\phantom{}
\begin{enumerate} 
\item It follows from the identity of Jacobi that $\dd^{CE}\circ \dd^{CE}=0$.
\item For all  $x,y, z\in \mathfrak{g}$ we have, 
\begin{itemize}
    \item $\dd^{\mathrm CE}(x)=\nu(x)$.
    \item For $\mu\in\mathrm{Hom}_\mathbb K(\mathfrak g,V)$, $$\left(\dd^{CE}\mu\right)(x,y)=\nu(x)(\mu(y))-\nu(y)(\mu(x))-\mu([x,y]_\mathfrak{g}).$$
    \item For $\eta\in\mathrm{Hom}_\mathbb K(\wedge^2\mathfrak g,V)$, 
    \begin{align*}
        \left(\dd^{\mathrm CE}\eta\right)(x,y,z)&=\nu(x)(\eta(y,z))-\nu(y)(\eta(x,z))+ \nu(z)(\eta(x,y))\\&\quad-\eta([x,y]_\mathfrak{g},y)+\eta([x,z]_\mathfrak{g},y)-\eta([y,z]_\mathfrak{g},x).
    \end{align*}
\end{itemize}
\end{enumerate}
\end{remark}
\begin{remark}
 When $\mathfrak{g}$ is of finite dimension, the Chevalley-Eilenberg complex \eqref{eq:CE-complex} is canonically isomorphic to the complex $(\wedge^ \bullet\mathfrak{g}^*\otimes V,\; \dd)$ and for $\xi=\xi_1\wedge\cdots\wedge \xi_k\in \wedge^k\mathfrak{g}^*$,
  \begin{equation}\label{eq:CE-new-formula}
     \dd(\xi\otimes v)=\sum_{i=1}^{n}(\xi\wedge \xi_i)\otimes (\xi_i\cdot v) - \dd_\mathfrak{g}(\xi)\otimes v
 \end{equation}
 where $\xi_1,\ldots, \xi_n$ is a basis of $\mathfrak g$. In the formula \eqref{eq:CE-new-formula}, $\dd_\mathfrak{g}$ is the Chevalley-Eileberg differential of $\mathfrak{g}$ w.r.t the trivial action on $\mathbb{K}$.
\end{remark}

\subsection{Operations on complexes}
We have  used and adapted the following lemma many times in this thesis, e.g. Section \ref{interpretaion-bicomplex}.
\begin{lemma}[]\label{lemma:hom-complex}
Let $(\mathcal{V}_\bullet, \dd^\mathcal{V})$ and $(\mathcal{W}_\bullet, \dd^\mathcal{W})$ be complexes of $\mathcal{O}$-modules. Then, the pair $\left(\mathrm{Hom}^\bullet_\mathcal{O}(\mathcal{V},\mathcal{W}\right), \partial)$ is a complex of $\mathcal{O}$-modules, where the differential map is given by 

\begin{equation}
    \partial(F):=\dd^\mathcal{W}\circ F-(-1)^{|F|}F\circ \dd^\mathcal{V},
\end{equation}
for all homogeneous element $F\in \mathrm{Hom}_\mathcal{O}(\mathcal{V},\mathcal{W})$.
\end{lemma}

\begin{proof}
The $\mathcal{O}$-linear map, $\partial\colon\mathrm{Hom}_\mathcal{O}(\mathcal{V},\mathcal{W})\rightarrow\mathrm{Hom}_\mathcal{O}(\mathcal{V},\mathcal{W})$, is clearly of degree $+1$, since the maps $\dd^\mathcal{V}$ and $\dd^\mathcal{W}$ are. Moreover,  for every $F\in\mathrm{Hom}_\mathcal{O}(\mathcal{V},\mathcal{W})$

\begin{align*}
    \partial^2(F)&=\partial(\dd^\mathcal{W}\circ F-(-1)^{|F|}F\circ \dd^\mathcal{V})\\&=\underbrace{\dd^\mathcal{W}\circ\dd^\mathcal{W}\circ F}_{=0} -(-1)^{|F|+1}\cancel{\dd^\mathcal{W}\circ F\circ\dd^\mathcal{V}} -(-1)^{|F|}\left(\cancel{\dd^\mathcal{W}\circ F\circ\dd^\mathcal{V}}-(-1)^{|F|+1}\underbrace{F\circ\dd^\mathcal{V}\circ \dd^\mathcal{V}}_{=0}\right)\\&=0.
\end{align*}
\end{proof}

\begin{remark}It's worth it to notice that
\begin{enumerate}
    \item the cocycles $F$ of  $\left(\mathrm{Hom}^\bullet_\mathcal{O}(\mathcal{V},\mathcal{W}), \partial\right)$ are those that satisfy $\dd^\mathcal{V}\circ F=(-1)^{|F|}F\circ \dd^\mathcal{W}$. In particular, the chain maps between $(\mathcal{V},\dd^\mathcal V)$ and $(\mathcal{W},\dd^\mathcal{W})$ are the $0$-cocyles of $\left(\mathrm{Hom}^\bullet_\mathcal{O}(\mathcal{V},\mathcal{W}\right), \partial)$.
    
    \item two chain maps $F_\bullet,G_\bullet\colon \mathcal{V}_{\bullet}\longrightarrow \mathcal{W}_{\bullet}$ are homotopic if and only if $G-F$ is $0$-coboundary of $\left(\mathrm{Hom}^\bullet_\mathcal{O}(\mathcal{V},\mathcal{W}\right), \partial)$, that is, there exists a $\mathcal{O}$-linear map $H\in \mathrm{Hom}^{-1}_\mathcal{O}(\mathcal{V},\mathcal{W})$ of degree $-1$ such that $$F- G= \dd^{\mathcal{W}}\circ H +H \circ \dd^{\mathcal{V}}.$$
    
    \item when $(\mathcal{V}_\bullet, \dd^\mathcal{V})=(\mathcal{W}_\bullet, \dd^\mathcal{W})$, we have  $\left(\mathrm{Hom}^\bullet_\mathcal{O}(\mathcal{V},\mathcal{V}), \partial= [\dd^\mathcal{V},\cdot\,]\right)$.
\end{enumerate} 
\end{remark}

\subsubsection{Direct sum and tensor product of complexes}
Let $(\mathcal V_\bullet,\dd^{\mathcal V})$ and $(\mathcal W_\bullet,\dd^{\mathcal W})$ be complexes of $\mathcal O$-modules. Then, the tensor product $\mathcal V_\bullet\otimes_\mathcal O\mathcal W_\bullet$ together with the grading  $$(\mathcal V\otimes_\mathcal O\mathcal W)_k=\bigoplus_{{red}{i+j=k}} \mathcal V_i\otimes_\mathcal O\mathcal W_j$$ for $k\in\mathbb{Z}$, comes equipped with a differential map classically defined by $$\partial=\dd^{\mathcal V}\otimes \mathrm{id} + \mathrm{id}\otimes \dd^{\mathcal W},$$ namely,
 $$ \partial(v\otimes w)=\dd^{\mathcal V}(v)\otimes w + (-1)^{|v|}v\otimes \dd^{\mathcal W}(w),$$ for all homogeneous elements $v,w\in\mathcal V\otimes_\mathcal O\mathcal W$, is complex of $\mathcal O$-modules. The operator $\partial$ is indeed of degree $+1$. It is easily checked that $\partial^2=0$.

\subsubsection{Bi-complex}
\begin{definition}
 A  \emph{bi-complex} or \emph{double complex} is a collection of $\mathcal{O}$-modules $\mathcal V=(\mathcal V_{i,j})_{i,j\in \mathbb{Z}}$ together with two families of $\mathcal O$-linear maps
 \begin{enumerate}
     \item  $d_{i,j}^{h}\colon \mathcal V_{i,,j}\rightarrow  \mathcal V_{i+1,j}$,\; such that\; $d_{i,j}^{h}\circ d_{i-1,j}^{h}=0,\quad \text{for}\;i,j \in \mathbb Z$\;\;\text{called \emph{horizontal differential map}}
     \item $d_{i,j}^{v}\colon \mathcal V_{i,j}\rightarrow  \mathcal V_{i,j+1}$,\; such that\; $d_{i,j}^{v}\circ d_{i,j-1}^{v}=0,\quad\text{for all}\;i,j \in \mathbb Z$\;\;\text{called  \emph{vertical differential map}}
 \end{enumerate}
 that obey  for all $i,j \in \mathbb Z$ the identity
 $$d_{i,j+1}^{h}\circ d_{i,j}^{v}=d_{i,j}^{v}\circ d_{i+1,j}^{h}.$$
In this thesis we only consider first quadrant bi-complexes that  is, $\mathcal{V}_{i,j}=0$ for all $i\in \mathbb{Z}_{\leq 0}$ and $j\in \mathbb{ N}_0$. These can be represented as the commutative diagram

 \begin{equation}\label{app:bi-cpmplex}
     { \hbox{$
	\begin{array}{ccccccccccc}
		& & \vdots & & \vdots & & \vdots & & 
		\\ 
		& & \uparrow & & \uparrow & & \uparrow & & 
		\\ 
		\cdots& \rightarrow & \mathcal{V}_{i,j} & \overset{d^h}{\rightarrow} & \mathcal{V}_{i,j}
		& \overset{d^h}{\rightarrow} & \mathcal{V}_{i,j} & \rightarrow & 0
			\\ 
		& & d^v\uparrow & & d^v\uparrow & & d^v\uparrow & & 
		\\ 
		\cdots& \rightarrow & \mathcal{V}_{i,j} & \overset{d^h}{\rightarrow} & \mathcal{V}_{i,j}
		& \overset{d^h}{\rightarrow} & \mathcal{V}_{i,j} & \rightarrow& 0
		\\ 
		& & d^v\uparrow & & d^v\uparrow & & d^v\uparrow & &  
		\\ 
		\cdots& \rightarrow & \mathcal{V}_{i,j} & \overset{d^h}{\rightarrow} & \mathcal{V}_{i,j} 
		& \overset{d^h}{\rightarrow}& \mathcal{V}_{i,j} & \rightarrow & 0
	\\
	& & \uparrow & & \uparrow & & \uparrow & & 
	\\ 
	& & 0 & & 0 & & 0 & &  \\
	& & \hbox{\small{\texttt{"-$2$ column"}}} & & \hbox{\small{\texttt{"-$1$ column"}}} & & \hbox{\small{\texttt{"last column"}}} & &  	
\end{array}
$}}
 \end{equation}
\end{definition}

\noindent
One associate to the bi-complex \eqref{app:bi-cpmplex} the so-called \emph{total complex} which is defined by the anti-diagonals of  \eqref{app:bi-cpmplex}, namely $\left(T_r:=\bigoplus _{i+j=r}\mathcal{V}_{i,j}\right)_{r}$ with \emph{total differential} $D\colon T_r\rightarrow T_{r+1}$ defined by $$D(\tau_{ij}):=d^h(\tau_{ij}) - (-1)^r d^v(\tau_{ij}),\;\;\text{for}\;\;\tau_{ij}\in \mathcal{V}_{i,j}.$$
Indeed, \begin{align*}
    D^2&=(d^h - (-1)^{r+1} d^v)\circ (d^h - (-1)^r d^v)\\&=\underbrace{(d^h)^2}_{=0}-\cancel{(-1)^rd^h\circ d^v}- \cancel{(-1)^{r+1} d^v\circ d^h} -\underbrace{(d^v)^2}_{=0}\\&=0.
\end{align*}

\begin{proposition}[Acyclic Assembly Lemma \cite{Weibel}]\label{Acyclic-Assembly-Lemma}
Let $(\mathcal V_{i,j})_{i,j\in \mathbb{Z}}$ be a first quadrant bi-complex like in the notation above such that the rows are exact, then the total complex $(T_\bullet, D)$ is exact.
\end{proposition}


\noindent
\textbf{Mapping cone}. Let $(\mathcal{V}, \dd^\mathcal{V})$ and $(\mathcal{W}, \dd^\mathcal{W})$ be two complexes and $L\colon \mathcal{V}_\bullet\rightarrow \mathcal{W}_\bullet$ a chain map. The \emph{mapping cone} is the complex $(\mathcal{C}, \partial)$ whose degree $i$ is given by $\mathcal{V}_{i-1}\oplus \mathcal{W}_{i}$ and  whose differential is defined as 

$$\partial=\begin{pmatrix}
-\dd^\mathcal{V}&0\\-L& \dd^\mathcal{W}
\end{pmatrix}
$$
That is, the differential is given on  elements $(v,w)\in \mathcal{V}\oplus \mathcal{W}$ by $$\partial(v,w)=(-\dd^\mathcal{V}(v), \dd^\mathcal{W}(w)-L(v)).$$ The mapping cone $(\mathcal{C}, \partial)$ is exact if and only if $L\colon \mathcal{V}_\bullet\rightarrow \mathcal{W}_\bullet$ is  a quasi-isomorphism, \cite{Weibel}, Cor. 1.5.4 p.19.


\section{Resolutions of a module}\label{app:proj-res}
\begin{definition}
   An $\mathcal{O}$-module $\mathcal{V}$ is said to be \emph{projective} if it fulfills the following: given a $\mathcal{O}$-linear map $L\colon\mathcal{V}\colon\rightarrow \mathcal{Z}$, every surjective $\mathcal{O}$-linear map   $J\colon\mathcal{W}\rightarrow \mathcal{Z}$  admits a  $\mathcal{O}$-linear map $\widetilde{L}\colon\mathcal{V}\rightarrow \mathcal{W}$ such that the following diagram commutes
   \begin{equation}
       \xymatrix{&\mathcal{V}\ar@{.>}[ld]\ar[d]&\\\mathcal{W}\ar@{->}[r]&\mathcal{Z}\ar[r]&0}
   \end{equation}
\end{definition}
\begin{example}
Free modules are projective modules (see e.g. Proposition 5.1.2 of \cite{Hazewinkel-Michiel}).
\end{example}

\begin{definition}
Let $\mathcal A$ be a $\mathcal O$-module. A \emph{free/projective resolution} (or \emph{resolution by free/projective modules}) of $\mathcal{A}$ is an exact complex  $(\mathcal{V}_\bullet, \dd)$ 
\begin{equation}\label{eq:complex2}
     \quad  \cdots {\longrightarrow}\mathcal{V}_{-i-1} \stackrel{\dd}{\longrightarrow} \mathcal{V}_{-i} \stackrel{\dd}{\longrightarrow}\mathcal{V}_{-i+1}\stackrel{\dd}{\longrightarrow}\cdots \stackrel{\dd}{\longrightarrow} \mathcal{V}_{-2} \stackrel{\dd}{\longrightarrow} \mathcal{V}_{-1}\stackrel{\pi}{\longrightarrow}\mathcal{A}\longrightarrow 0
\end{equation}
such that the $\mathcal{V}_{-i}$'s are free/projective modules.
\end{definition}

\begin{remark}
The complex \eqref{eq:complex2} may be of infinite length. When the length is finite, we say that we have a \emph{finite} free/projective resolution. In that case,  the sequence 
\begin{equation}
     \quad  0 {\longrightarrow}\mathcal{V}_{-n} \stackrel{\dd}{\longrightarrow} \mathcal{V}_{-n+1} \stackrel{\dd}{\longrightarrow}\mathcal{V}_{-n+2}\stackrel{\dd}{\longrightarrow}\cdots \stackrel{\dd}{\longrightarrow} \mathcal{V}_{-2} \stackrel{\dd}{\longrightarrow} \mathcal{V}_{-1}\stackrel{\pi}{\longrightarrow}\mathcal{A}\longrightarrow 0
\end{equation}

is exact in every degree. In particular, the  sequence $0 \rightarrow \mathcal{V}_{-n}\stackrel{\dd}{\rightarrow} \mathcal{V}_{n-1}$ is exact at $-n$. For instance, the map $\mathcal{V}_{-n}\stackrel{\dd}{\rightarrow}\mathcal{V}_{n-1}$ injective. Thus, $\mathcal{V}_{-n}\simeq \mathrm{Im}(\mathcal{V}_{-n}\stackrel{\dd}{\rightarrow}\mathcal{V}_{n-1})$. By exactness, we have that $\mathcal{V}_{-n}\simeq \mathrm{Im}(\mathcal{V}_{-n}\stackrel{\dd}{\rightarrow}\mathcal{V}_{n-1})=\ker(\mathcal{V}_{-n+1}\stackrel{\dd}{\rightarrow}\mathcal{V}_{-n})$.
\end{remark}

\begin{proposition}\label{prop:free-resol}
Every $\mathcal{O}$-module $\mathcal{A}$ admits a free/projective resolution.
\end{proposition}

\begin{proof}
Firstly, notice that every module is isomorphic to a quotient of a free module: to see this, choose a set of generators $\{v_i\in \mathcal{A} \mid  i\in I\}$ 
of the module $\mathcal{A}$ so that $\displaystyle{\mathcal{V}=\sum_{i\in I}\mathcal{O}v_i}$. The $\mathcal{O}$-linear map $$\pi\colon \bigoplus_{i\in I}\mathcal{O}\rightarrow \mathcal{A},\; (f_j)_{j\in{J}\subseteq I}\mapsto \sum_{j\in J}f_jv_j, \; \text{$J$ is finite}$$ is surjective. By the first isomorphism theorem, it follows that $\mathcal{A}\simeq (\bigoplus_{i\in I}\mathcal{O})/\ker\pi$. Now put $\bigoplus_{i\in I}\mathcal{O}=:\mathcal{V}_{-1}$. This yields an exact sequence
$$0\longrightarrow \ker\pi\longhookrightarrow\mathcal{V}_{-1}\stackrel{\pi}{\longrightarrow} \mathcal{A}{\longrightarrow} 0.$$ But $\ker\pi$ does not  need to be a free $\mathcal{O}$-module, but there is a free $\mathcal{O}$-module $\mathcal{V}_{-2}$ together with a surjective $\mathcal{O}$-linear map $\xymatrix{\mathcal{V}_{-2}\ar@{->>}[r]^{\pi_1}&\ker \pi}$  such that $\mathcal{V}_{-2}/\ker \pi_1\simeq \ker\pi$. This is added to the previous sequence as follows

$$\xymatrix{0\ar[r]&\ker {\dd_1}\; \ar@{^{(}->}[r]&\mathcal{V}_{-2}\ar@{->>}[r]^{\pi_1}\ar@{.>>}@/_2pc/[rr]_{\dd_1}&\ker \pi\;\ar@{^{(}->}[r]& \mathcal{V}_{-1}\ar@{->>}[r]^\pi&\mathcal{A}}.$$ Once again, $\ker {\dd_1}$ does not  need to be a free $\mathcal{O}$-module, therefore we can continue the procedure. Inductively, assume that we have constructed a $\mathcal{O}$-linear map, $\dd_n \colon \mathcal{V}_{-n-1}\rightarrow \mathcal{V}_{-n}$, for $n\geq 2$. Thus, there exists a free $\mathcal{O}$-module $\mathcal{V}_{-n-2}$ together with a surjective $\mathcal{O}$-linear map $\pi_{n+1}\colon \mathcal{V}_{-n-2}\rightarrow \mathcal{V}_{-n-1}$ such that
$\mathcal{V}_{-n-2}/\ker \pi_{n+1}\simeq \ker \dd_n$. Joining this to the previous sequence, one get

\begin{equation}\label{eq:free-resol}
    \xymatrix{0\ar[r]&\ker {\dd_{n+1}}\; \ar@{^{(}->}[r]&\mathcal{V}_{-n-2}\ar@{->>}[r]^{\pi_{n+1}}\ar@{.>>}@/_2pc/[rr]_{\dd_{n+1}}&\ker \dd_n\;\ar@{^{(}->}[r]& \mathcal{V}_{-n-1}\ar@{->>}[r]^{\dd_n}&\mathcal{V}_{-n}\ar[r]&\cdots \ar[r]& \mathcal{V}_{-2}\ar[r]^{\dd_1}&\mathcal{V}_{-1}\ar@{->>}[r]^\pi&\mathcal{A}}
\end{equation}
By construction, we have $\ker (\dd_n)=\mathrm{Im}(\dd_{n+1})$. Therefore, we have built an exact sequence up to length $n+1$. This completes the proof.

One shall notice that the process \eqref{eq:free-resol} may be continued forever without reaching a free kernel.
\end{proof}

Here is an important operation on complexes. The following Proposition states the localization  in the sense of item \ref{def:localisation} of Section \ref{LR-construction} preserves exactness. More precisely,
\begin{proposition}[\cite{stacks-project}, Section 10.9  or \cite{A.Gathmann}]\label{prob:exact-localization}
Let $\mathcal{A}$ be an $\mathcal{O}$-module and $S\subset{\mathcal{O}}$ a multiplicative subset. For any free resolution of $\mathcal{A}$ $$\xymatrix{\cdots\ar[r]^\dd&\mathcal{V}_{-3}\ar[r]^\dd&\mathcal{V}_{-2}\ar[r]^\dd&\mathcal{V}_{-1}\ar[r]^\pi& \mathcal{A}},$$ the complex

$$\xymatrix{\cdots\ar[r]^{S^{-1}\dd}&S^{-1}\mathcal{V}_{-3}\ar[r]^{S^{-1}\dd}&S^{-1}\mathcal{V}_{-2}\ar[r]^{S^{-1}\dd}&S^{-1}\mathcal{V}_{-1}\ar[r]^{S^{-1}\pi}& S^{-1}\mathcal{A}}$$is a free resolution of $S^{-1}\mathcal{A}$ by $S^{-1}\mathcal{O}$-modules.
\end{proposition}

It is well-known that a submodule of a finitely generated module is not finitely generated in general. The following assertion guarantees this for finitely generated modules over Noetherian rings (see Proposition 1.4 of \cite{zbMATH00704831}, Page 28).

\begin{proposition}\label{prop:noetherian}
If $\mathcal{O}$ is Noetherian as a ring, then all submodules of
an $\mathcal{O}$-module $\mathcal{V}$ are finitely generated if and only if $\mathcal{V}$ is finitely generated.
\end{proposition}

The following theorem assures existence of free resolution of finite length in a special case.
\begin{theorem}[Hilbert Syzygy Theorem]\label{app:Syzygies} \cite{PeevaIrena, EisenbudSyzygies} Assume $\mathcal{O}=\mathbb{C}[x_1,\ldots, x_d]$. Any finitely generated (graded) $\mathcal{O}$-module $\mathcal{A}$ admits a finite graded free resolution by finitely generated $\mathcal{O}$-modules
$$\xymatrix{\cdots\ar[r]^\dd&\mathcal{V}_{-3}\ar[r]^\dd&\mathcal{V}_{-2}\ar[r]^\dd&\mathcal{V}_{-1}\ar[r]^\pi& \mathcal{A}}$$
of length  $N \leq  d+1$.
\end{theorem}
\subsubsection{Tor complex} \label{app:tor}
Let $(\mathcal V_\bullet,\dd^{\mathcal V}, \pi)$ be a projective resolution of an $\mathcal{O}$-module $\mathcal{A}$. For any $\mathcal{O}$-module $\mathcal{B}$, we consider the complex 

$$\xymatrix{\cdots \ar[r]&\mathcal{V}_{-3}\otimes_\mathcal{O}\mathcal{B}\ar[r]^{\dd^{\mathcal V}\otimes \mathrm{id}}&\mathcal{V}_{-2}\otimes_\mathcal{O}\mathcal{B}\ar[r]^{\dd^{\mathcal V}\otimes \mathrm{id}}&\mathcal{V}_{-1}\otimes_\mathcal{O}\mathcal{B}\ar[r]^{\dd^{\mathcal V}\otimes \mathrm{id}}&\mathcal{A}\otimes_\mathcal{O}\mathcal{B}}.$$

We define for $i\geq 0$
$$\mathrm{Tor}^{-i} ( \mathcal{A} ,\mathcal{B}):= H^{-i}( \mathcal{V}_\bullet\otimes_\mathcal{O} \mathcal{B}).$$
Here are some properties of $\mathrm{Tor}$ \cite{Caroline-Lassueur}.
\begin{proposition} The functor $\mathrm{Tor}$ satisfies
\begin{enumerate}
    \item $\mathrm{Tor}^{0}(\mathcal{A},\mathcal{B})\simeq \mathcal{A}\otimes_\mathcal{O}\mathcal{B}$.
    \item If $\mathcal A$ is projective, then $\mathrm{Tor}^{-i}(\mathcal{A},\mathcal{B})=0$ for every $i\geq 1$.
    
    \item If $\mathcal{B}$ is flat, then $\mathrm{Tor}^{-i}(\mathcal{A},\mathcal{B})=0$ for every $i\geq 1$.
    
    \item For all $i$, $\mathrm{Tor}^{-i}(\mathcal{A},\mathcal{B})=\mathrm{Tor}^{-i}(\mathcal{B},\mathcal{A})$.
    \item The construction of $\mathrm{Tor}(\mathcal A,\mathcal B)$ is independent of the choice of the resolution, i.e.  any other projective resolution of $\mathcal A$ or $\mathcal{B}$ yields the same $\mathrm{Tor}$ groups.

\end{enumerate}
\end{proposition}
\subsubsection{Minimal resolutions}
Detailed definitions on local rings can be found in \cite{Nagata-Masayoshi}.

\begin{definition}
   \label{def:regular-local-ring}
   \begin{enumerate}
       \item A ring $R$ is
said to be a  \emph{local ring} if it has a "unique maximal ideal", i.e., a proper ideal $\mathfrak m\subset R$  such that $\mathfrak m$ contains every other ideal of $R$.
\item A local ring $R$ with maximal ideal $\mathfrak m$ is called  \emph{regular} if $\mathfrak{m}$ can be generated by $n$ elements, where $n=\dim R$ is the krull dimension\footnote{the Krull dimension of $R$ is by definition  the supremum of lengths of all chains of prime ideals in $R$ \cite{Matsumura}, p. 30} of $R$.
   \end{enumerate}
\end{definition}

Assume now that $R$ is a local ring with maximal ideal $\mathfrak{m}$, and let $\mathbb{K}:=R/\mathfrak m$. Geometrically, local rings correspond to germs of functions on a manifolds or affine variety at a point.\\

We assume that $(R, \mathfrak m)$ is a local Noetherian commutative ring. Here is an important lemma that uses definition of local rings.
\begin{lemma}[Nakayama]\label{Nakayama} Let $\mathcal{V}$ be a finitely generated $R$-module such that $r=\dim (\mathcal{V}/\mathfrak m \mathcal{V})< \infty$. Then, any  basis of the vector space $\mathcal{V}/\mathfrak m \mathcal{V}$ lifts to a (minimal)
generating set for $\mathcal{V}$ as a $R$-module. In particular, $\mathcal{V}$ can be generated by $r$ elements.
\end{lemma}

\begin{proof}
\cite{EisenbudSyzygies} Lemma 1.4,  p. 6.
\end{proof}

\begin{remark}
By Nakayama Lemma, a local ring $R$ is regular if  the dimension of the $R/\mathfrak m$-vector space $\mathfrak m/\mathfrak m^2$ is $\dim R$. 
\end{remark}

One can construct free resolutions of finitely generated modules over $R$ like in the proof of Proposition \ref{prop:free-resol} by taking a minimal set of homogeneous generators at each step. This can be formalized as follows,
\begin{definition}
   A projective resolution
   \begin{equation}\label{eq:minimal-resol1}
       \xymatrix{\cdots\ar[r]^\dd&\mathcal{V}_{-3}\ar[r]^\dd&\mathcal{V}_{-2}\ar[r]^\dd&\mathcal{V}_{-1}\ar[r]\ar[r]^\pi&\mathcal{A}\ar[r]&0}
   \end{equation}
   of a finitely generated $R$-module $\mathcal A$ is said to be \emph{minimal} if the differential map satisfies $\dd(\mathcal{V}_{-i})\subseteq\mathfrak m\mathcal V_{-i+1}$ for all $i\geq 2$.
\end{definition}
\begin{proposition}\cite{EisenbudSyzygies}\label{prop:minimal-resol}
Every finitely  generated $R$-module $\mathcal{A}$ admits a minimal free resolution.
\end{proposition}
\begin{proof}
The proof goes just like in Proposition \ref{prop:free-resol}, one just need to take minimal set of generators in the construction. By doing so, it suffices to check that $\dd(\mathcal{V}_{-2})\subset \mathcal{V}_{-1}$:  Let $r=\dim (\mathcal{A}/\mathfrak m \mathcal{A})$. By Nakayama Lemma, we can take $\mathcal{V}_{-1}=R^r$. One has a short exact sequence,$$0\longrightarrow \dd(\mathcal{V}_{-2})=\ker \pi\longhookrightarrow \mathcal{V}_{-1}\longrightarrow \mathcal{A}\longrightarrow 0.$$ It induces an exact sequence $$\dd(\mathcal{V}_{-2})/\mathfrak m\, \dd(\mathcal{V}_{-2}) \longhookrightarrow \mathcal{V}_{-1}/\mathfrak m\mathcal{V}_{-1}\longrightarrow \mathcal{A}/\mathfrak{m}\mathcal{A}\longrightarrow 0,$$ by tensorizing with $\mathbb{K}=R/\mathfrak{m}$. Since by construction $\mathcal{V}_{-1}/\mathfrak m\mathcal{V}_{-1}\simeq \mathcal{A}/\mathfrak{m}\mathcal{A}\simeq \mathbb{K}^r$, the image of the map $\dd(\mathcal{V}_{-2})/\mathfrak m\, \dd(\mathcal{V}_{-2}) \longhookrightarrow \mathcal{V}_{-1}/\mathfrak m\mathcal{V}_{-1}$ is zero, i.e., $\dd(\mathcal{V}_{-2})\subseteq \mathfrak m\mathcal{V}_{-1}$. Likewise,  by Noetheriality of $R$, $\ker \pi\subset \mathcal{V}_{-1}$ is finitely generated. One can repeat the procedure by starting with a minimal set of generators of $\ker \pi$ and so on. The proof continues by recursion.
\end{proof}
\begin{remark}
Minimal resolutions and the Tor complex: given a minimal projective resolution as in \eqref{eq:minimal-resol1}, its quotient by the maximal ideal $\mathfrak{m}$ corresponds to the complex
$$ \xymatrix{\cdots\ar[r]&\mathcal{V}_{-3}\otimes_\mathcal{O}\mathbb{K}\ar[r]^{\dd\otimes \mathrm{id}}&\mathcal{V}_{-2}\otimes_\mathcal{O}\mathbb{K}\ar[r]^{\dd\otimes \mathrm{id}}&\mathcal{V}_{-1}\otimes_\mathcal{O}\mathbb{K}\ar[r]^{\pi\otimes \mathrm{id}}&\mathcal{A}\otimes_\mathcal{O}\mathbb{K}\ar[r]&0}$$
whose cohomology compute $\mathrm{Tor}^\bullet(\mathcal{A}, \mathbb{K})$. By minimality of $(\mathcal{V}_\bullet,\dd, \pi)$ one has  $\dd\otimes \mathrm{id}\equiv 0$, therefore $\mathrm{Tor}^{-i}(\mathcal{A}, \mathbb{K})\simeq  \mathcal{V}_{-i}\otimes_\mathcal{O}\mathbb{K}$ for every $i\geq 2$. In particular, the ranks of the $\mathcal{V}_{-i}$'s  for $i\geq 2$ are independent of the choices made in the construction of the minimal projective resolution.
\end{remark}

\subsubsection{Koszul complex }

Assume that  $\mathcal{V}$ is a free $\mathcal{O}$-module of finite rank $n$, which is concentrated in degree $-1$. Here, we denote by $\bigwedge^\bullet \mathcal{V}$ the graded symmetric algebra of $\mathcal{V}$. Let  $F\colon \mathcal{V}\rightarrow\mathcal{O}$ be a $\mathcal{O}$-linear map. Given a free basis $e_1,\ldots, e_n$ of $\mathcal{V}$, $F$ is completely determined by $n$-uplet $(f_1, \ldots,f_n)\subset \mathcal{O}$. 

\begin{definition}\cite{zbMATH00704831,Matsumura}
The \emph{Koszul complex  associated to $(f_1, \ldots,f_n)$} is the complex
\begin{equation}\label{eq:Koszul-complex}
    0\stackrel{}{\longrightarrow}\bigwedge^n \mathcal{V}\stackrel{\dd}{\longrightarrow}\bigwedge^{n-1} \mathcal{V}\stackrel{\dd}{\longrightarrow}\cdots\longrightarrow \bigwedge^2 \mathcal{V}\stackrel{\dd}{\longrightarrow} \mathcal{V}\stackrel{F}{\longrightarrow}\mathcal{O}.
\end{equation}
whose differential $\dd\colon \bigwedge^\bullet \mathcal{V}\stackrel{\dd}{\longrightarrow}\bigwedge^{\bullet-1}\mathcal{V}$ is the unique derivation of $\bigwedge^\bullet \mathcal{V}$ such that, $\dd|_\mathcal{V}=F.$ In other words, 

\begin{equation}
   \dd(e_{i_1}\wedge\cdots\wedge e_{i_k})=\sum_{j=1}^k(-1)^{j-1}F(e_{i_j})e_{i_1}\wedge\cdots \hat{e}_{i_j}\cdots \wedge e_{i_k}. 
\end{equation}
\end{definition}
It is easily checked  that this gives a well-defined complex.

\begin{definition}
   An ordered sequence of elements $f_1,\ldots, f_r\in \mathcal{O}$ is called a \emph{regular sequence} on a module $\mathcal{V}$ if the following conditions are satisfied
   \begin{enumerate}
      \item $f_i$ is not a zero divisor on the quotient $\mathcal{V}/\langle f_1,\ldots, f_{i-1}\rangle\mathcal{V}$,
       \item $\langle f_1,\ldots, f_r\rangle\mathcal{V}\neq \mathcal{V}$.
   \end{enumerate}
\end{definition}
\begin{proposition}
If $(f_1,\ldots,f_k)$ is regular sequence on $\mathcal{O}$, then the Koszul complex \eqref{eq:Koszul-complex} has no cohomology in degree less equal to $-1$.
\end{proposition}

\begin{proof}
Theorem 16.5. of \cite{Matsumura}.
\end{proof}

\begin{example}
For $\mathcal{O}=\mathbb{K}[x_1,\ldots,x_d]$ be the polynomial ring in $d$ indeterminates,  the Koszul complex which is associated to $(x_1,\ldots,x_d)$ induces a free resolution of $\mathbb K\simeq \mathcal{O}/\langle x_1,\ldots, x_d\rangle$.
\end{example}
\section{Geometric resolutions of a singular foliation}\label{sec:geo-reso}
Let us start with this definition.

\begin{definition} Let $\mathfrak{F}$ be a singular foliation on $M$. A \emph{complex of vector bundles over $\mathfrak{F}$} consists of a triple $(E_\bullet,\dd^\bullet,\rho)$, where
   \begin{enumerate}
       \item $ E_\bullet = (E_{-i})_{i \geq 1}$ is a family of vector bundles over $M$, indexed by negative integers.
       \item  $\dd^{(i+1)}\in \mathrm{Hom}(E_{-i-1}, E_{-i})$ is a vector bundle morphism over the identity of $M$ called the \emph{differential map}
       \item $\rho \colon E_{-1} \longrightarrow TM $ is a vector bundle morphism over the identity of $M$ called the \emph{anchor map} with $\rho(\Gamma(E_{-1}))=\mathfrak{F}$. 
   \end{enumerate}
   such that \begin{equation}\label{eq:geom-ch-complex}\xymatrix{ \cdots \ar[r] & E_{-i-1} \ar[r]^{{\dd^{(i+1)}}} \ar[d] & 
     E_{ -i} \ar[r]^{{\dd^{(i)}}}
     \ar[d] & E_{i-1} \ar[r] \ar[d] & \ar@{..}[r] & \ar[r]^{{\dd^{(2)}}}& E_{-1} \ar[r]^{\rho} \ar[d]& TM \ar[d] \\ 
      \ar@{=}[r] & \ar@{=}[r] M  &  \ar@{=}[r] M 
      &  \ar@{=}[r] M  &\ar@{..}[r] & \ar@{=}[r]   &  \ar@{=}[r] M  & M}\end{equation}
  which form a (chain) complex, i.e.
  $$ \dd^{(i)}\circ\dd^{(i+1)}=0  \hbox{ and }  \rho \circ \dd^{(2)}=0.$$
  \end{definition}
  
 \begin{remark}\label{rmk:cohmology}
  Two main cohomology groups can  be associated to a complex of vector bundles over $\mathfrak{F}$:

\begin{enumerate}
    \item {\textbf{Cohomology at the level of sections}}.
     The complex of vector bundles \eqref{eq:geom-ch-complex} induces a complex of sheaves of modules over functions. More explicitly, for every open subset $\mathcal U \subset M $, there is a complex: $$ \cdots {\longrightarrow} \Gamma_{\mathcal U}({E_{ -i-1}})
     \stackrel{\dd^{(i+1)}}{\longrightarrow} \Gamma_{\mathcal U}({E_{-i}})
     \stackrel{\dd^{(i)}}{\longrightarrow}{\Gamma_{\mathcal U}(E_{-i+1})}{\longrightarrow}\cdots  \stackrel{\dd^{(2)}}{\longrightarrow} \Gamma_{\mathcal U}(E_{-1})  
     \stackrel{\rho}{\longrightarrow} \mathfrak F_{\mathcal U} \subset \mathfrak X(\mathcal U).$$   Since $\mathrm{Im}(\dd^{(i+1)})\subseteq\ker \dd^{(i)}$ for every $i\in\mathbb N$, we are allowed to define the quotient spaces,
      $$ H^{-i}(E_{\bullet},\mathcal U) = \left\{ \begin{array}{ll} \frac{\ker\left(\Gamma_\mathcal{U}(E_{-1}) \overset{\rho}{\longrightarrow }\mathfrak{F}\right) 
      }{\mathrm{Im}\left(\Gamma_\mathcal{U}(E_{-2})\overset{\dd^{(2)}}{\longrightarrow}\Gamma_\mathcal U(E_{-1})\right)} & \hbox{for $i=1 $}\\&
      \\ \frac{\ker \left(\Gamma_\mathcal{U}(E_{-i})\overset{\dd^{(i)}}{\longrightarrow }\Gamma_\mathcal{U}(E_{i+1})\right)} {\mathrm{Im}\left(\Gamma_\mathcal{U}(E_{-i-1})\overset{\dd^{(i+1)}}{\longrightarrow}\Gamma_\mathcal{U}(E_{-i})\right)} &
      \hbox{ if $ i \geq 2$.}\end{array}\right.$$
  They are  modules over functions on $\mathcal U$. For each $i\geq 1$, $H^{-i}(E_{\bullet},\mathcal U)$ is called the \emph{$i$-th cohomology} of $(E_\bullet, \dd^\bullet, \rho) $ at the level of sections on $\mathcal{U}$.  
    \begin{enumerate}
        \item We way then $(E_\bullet, \dd^\bullet, \rho) $ is a \emph{geometric resolution} of $\mathfrak{F}$ if  for every open set $\mathcal{U}\subset M$ and $i\geq 1$, if it induces an exact complex on the level of sections, that is $$H^{-i}(E_{\bullet},\mathcal U')=\{0\}$$ for every open subset $\mathcal U'\subset \mathcal U$. 
        \item A geometric resolution $(E_\bullet,\dd^{\bullet}, \rho)$ of $\mathfrak{F}$ is said to be \emph{minimal} at a point $m\in M$ if for each $i\geq 2$ the linear map $\dd^{(i)}_{|_m}\colon E_{-i}\longrightarrow {E_{-i+1}}_{|_m}$ vanishes.
    \end{enumerate} 
    
    \item {\textbf{Cohomology at an arbitrary point $m \in M$}}. 
    Also, the complex of vector bundles \eqref{eq:geom-ch-complex}, at an arbitrary point $m \in M$, restricts to a complex of vector spaces $$ \quad  \cdots {\longrightarrow} {E_{ -i-1}}_{|_m}\stackrel{\dd^{(i+1)}_{|_m}}{\longrightarrow} {E_{-i}}_{|_m}  \stackrel{\dd^{(i)}_{|_m}}{\longrightarrow}{E_{-i+1}}_{|_m}{\longrightarrow}\cdots \stackrel{\dd^{(2)}_{|_m}} {\longrightarrow} E_{-1}  \stackrel{\rho_{|_m}}{\longrightarrow} T_m M .$$   We can look at the quotient vector spaces:
      
      $$ H^{-i}(E_{\bullet},m) = \left\{ \begin{array}{ll} \frac{\ker\left({E_{-1}}_{|_m} \overset{\rho_{|_m}}{\longrightarrow }T_mM\right) 
      }{\mathrm{Im}\left({E_{-2}}_{|_m}\overset{\dd^{(2)}{|_m}}{\longrightarrow}{E_{-1}}_{|_m}\right)} & \hbox{for $i=1 $}\\&
      \\ \frac{\ker \left({E_{-i}}_{|_m}\overset{\dd^{(i)}{|_m}}{\longrightarrow }{E_{i+1}}_{|_m}\right)} {\mathrm{Im}\left({E_{-i-1}}_{|_m}\overset{\dd^{(i+1)}{|_m}}{\longrightarrow}{E_{-i}}_{|_m}\right)} &
      \hbox{ if $ i \geq 2$.}\end{array}\right.$$
      
 Here we call $H^{-i}(E_{\bullet},m)$ the \emph{$i$-th cohomology of $(E_\bullet, \dd^\bullet, \rho) $ at the point $m$}.

 It is important to notice the following: even if $(E_\bullet, \dd^\bullet, \rho) $ is a geometric resolution of $\mathfrak{F}$, there are no reasons for $(E_\bullet, \dd^\bullet, \rho) $ to be exact at $m$. For example, if the resolution is minimal at some point $m\in M$, then $H^{-i}(E_{\bullet},m)\simeq {E_{-i}}_{|_m}$ for each $i\geq 2$ and $H^{-1}(E_{\bullet},m)\simeq\ker(\rho_{|_m})$.
\end{enumerate}
\end{remark}
\begin{remark}
Most of the definitions and operations on chain complexes of modules are adapted in a obvious manner to  complexes of vector bundles over $M$, both on the level of the sections or at a point. See also \cite{LLS,LavauSylvain} for more.
\end{remark}
\subsubsection{On existence of geometric resolutions}
Geometric resolutions of a singular foliation $\mathfrak{F}$ as defined in Remark \ref{rmk:cohmology}(a) are not guaranteed in all contexts. However, there always exists a projective resolution of $\mathfrak F$ as a $\mathcal{O}$-module (see \ref{app:proj-res}). But these resolutions do not induce always a geometric resolution of $\mathfrak F$, since the projective modules of a projective resolution may not correspond to vector bundles because they may not be locally finitely generated. When the latter condition is satisfied, the Serre-Swan theorem \cite{SwanRichardG,MoryeArchanaS} states that there is a one-to-one correspondence between locally finitely generated projective modules and sections of vector bundles. Under the assumptions of the latter theorem, geometric resolutions of 
 singular foliations are exactly projective resolutions at the sections level in the category of chain complexes by $\mathcal O$-modules, since sections of vector bundles over $M$ are projective $\mathcal O$-modules.\\
 
 The following proposition summarizes some contexts where geometric resolutions exist, see \cite{LLS} for their proofs.
 
 \begin{proposition}\phantom{\cite{LLS}}
 \begin{enumerate}
     \item Any algebraic singular foliation\footnote{A singular foliation which is generated by polynomial vector fields on $\mathbb{K}^d$.} on $\mathbb K^d $ admits geometric resolutions by trivial vector bundles and of length $\leq d+1 $.
    
    The same holds for a real analytic of holomorphic singular foliation, but only in a neighborhood of a point.
    
    \item  A locally real analytic singular foliation on a manifold of dimension $d$ admits a geometric resolution of length $\leq d+1$ on any relatively compact open subset of $M$.
 \end{enumerate}
 \end{proposition}
 
 
 Here we have some examples of geometric resolutions of singular foliations.
\begin{example}
Let $\mathfrak F_0=\{X\in\mathfrak X(V)\mid X(0)=0\}$ be the singular foliation made of all vector fields vanishing at the origin of a vector space $V$ (e.g. think of $\mathbb C^N$ or $\mathbb R^N$). The contraction by the Euler vector field $\displaystyle{\overrightarrow{E}=\sum_{i=1}^Nx_i\frac{\partial}{\partial x_i}}$ gives rise to a
complex of trivial vector bundles 

\begin{equation}\label{eq:Koszul1}
 \quad  \cdots {\longrightarrow} \wedge^3T^*V  \stackrel{\iota_{\overrightarrow{E}}}{\longrightarrow}\wedge^2T^*V   \stackrel{\iota_{\overrightarrow{E}}}{\longrightarrow}T^*V\stackrel{\iota_{\overrightarrow{E}}}{\longrightarrow} \mathbb{C}\times V=:\underline{\mathbb C},\end{equation} whose complex on the sections level is $(\Omega^\bullet(V), \iota_{\overrightarrow{E}})$. Here $(x_1,\ldots, x_N )$ are the
canonical coordinates on $V$. The latter is the Kozul complex associated to the coordinate  functions $x_1,\ldots,x_N$ of $V$. Since the $x_i$'s form a regular sequence, it is well known that $(\Omega^\bullet(V), \iota_{\overrightarrow{E}})$ is exact.\\
 
 The following complex of vector bundles over $V$ 
 
 \begin{equation}\label{eq:Koszul2}
 \quad  \cdots {\longrightarrow} \wedge^3T^*V\otimes TV \stackrel{\iota_{\overrightarrow{E}}\otimes\mathrm{id}}{\longrightarrow}\wedge^2T^*V  \otimes TV \stackrel{\iota_{\overrightarrow{E}}\otimes\mathrm{id}}{\longrightarrow}T^*V\otimes TV\stackrel{\iota_{\overrightarrow{E}}\otimes\mathrm{id}}{\longrightarrow}\underline{\mathbb C}\otimes TV.\end{equation}

 is a geometric resolution of $\mathfrak F_0$ since $\left(\Omega^\bullet(V)\otimes\mathfrak X(V), \iota_{\overrightarrow{E}}\otimes\mathrm{id}\right)$ is also exact (here  $\Omega^i(V):= \Gamma(\wedge^iT^* V )$ stands for the sheaf of $i$-forms on $V$).\\
 
 More generally, the construction we have made in \eqref{eq:Koszul2} is still possible by contracting with any vector field $\displaystyle{X=\sum_{i=1}^NX^i\frac{\partial}{\partial x_i}}\in \mathfrak{X}(V)$. The latter yields a complex of vector bundles that covers the singular foliation $\mathfrak F_X$ generated by the $X^i\frac{\partial}{\partial x_j}$'s. For instance, if $X$ is a polynomial vector field and $(X^1,\ldots,X^N)$ form a regular sequence, we get a geometric resolution of $\mathfrak F_X$.

\end{example}

\begin{example}
Let $\mathfrak F_2=\mathcal{I}^2_0\mathfrak X(\mathbb{K}^2)\subset \mathfrak F_0$ be the sub-singular  singular foliation made of vector fields vanishing at order $2$ at the origin of $\mathbb{K}^2$, where  $\mathcal{I}^2_0\subset \mathcal{O}(\mathbb{K}^2)$ is the ideal generated by the monomials $x^2, xy, y^2$. Note that the ideal $\mathcal{I}_0^2$ admits a free resolution of the form \begin{equation}\label{eq:ideal-reso1}
    0\longrightarrow \mathcal{O}(\mathbb{K}^2)\oplus\mathcal{O}(\mathbb{K}^2)\stackrel{\delta_1}{\longrightarrow} \mathcal{O}(\mathbb{K}^2)\oplus\mathcal{O}(\mathbb{K}^2)\oplus\mathcal{O}(\mathbb{K}^2)\stackrel{\delta_0}{\longrightarrow}\mathcal{I}_0^2\longrightarrow 0,
\end{equation}where for all $f,g,h\in\mathcal{O}(\mathbb{K}^2)$, $$\delta_0(f,g,h)=x^2f+xyg+y^2h\,\;\text{and}\,\; \delta_1(f,g)=(xf,xf-yg,xg).$$The free resolution \eqref{eq:ideal-reso1} has to take the form  
\begin{equation*}
    0\longrightarrow \Gamma(\mathcal I_{-2})\stackrel{\delta_1}{\longrightarrow} \Gamma(\mathcal{I}_{-1})\stackrel{\delta_0}{\longrightarrow}\mathcal{I}_0^2\longrightarrow 0,
\end{equation*}
for sum trivial vector bundles $\mathcal{I}_{-1}, \mathcal{I}_{-2}$ on $\mathbb{K}^2$. Thus, the following complex 
\begin{equation}
    0\longrightarrow \mathcal I_{-2}\otimes T\mathbb{K}^2\stackrel{\delta_1\otimes \mathrm{id}}{\longrightarrow} \mathcal{I}_{-1}\otimes
    T\mathbb{K}^2\stackrel{\delta_0\otimes \mathrm{id}}{\longrightarrow}\mathcal{I}_0^2\otimes T\mathbb{K}^2\longrightarrow 0
\end{equation}is a geometric resolution of $\mathcal
F_2$. Note that  $\mathcal{I}_{-1}$ can be identified with the tivial vector bundle $S^2((\mathbb{K}^2)^*)$.\\

More generally, let $\mathfrak F_k$ be the singular foliation made of  vector fields
vanishing at order $k$ at the origin of a vector space $V$ of dimension $N$ over $\mathbb{R}$ or $\mathbb{C}$. The  Hilbert’s syzygy theorem assures the existence of a free resolution of length $N+1$ of the ideal $\mathcal{I}^k_0$ made of functions on $V$ vanishing to
order $k$ at the origin. This resolution is of the form 
\begin{equation*}
    \cdots \longrightarrow \Gamma(\mathcal I_{-2})\stackrel{\delta_1}{\longrightarrow} \Gamma(\mathcal{I}_{-1})\stackrel{\delta_0}{\longrightarrow}\mathcal{I}_0^2\longrightarrow 0,
\end{equation*}for some family of trivial vector bundles $(\mathcal I_{-i})_{i\geq1}$ over $V$. We obtain a geometric resolution of $\mathfrak F^k$  of the form

\begin{equation*}
    \cdots \longrightarrow \Gamma(\mathcal I_{-2}\otimes TV)\stackrel{\delta_1\otimes \mathrm{id}}{\longrightarrow} \Gamma(\mathcal{I}_{-1}\otimes
    TV)\stackrel{\delta_0\otimes \mathrm{id}}{\longrightarrow}\mathcal{I}_0^2\otimes \mathfrak{X}(V)=\mathfrak F^k.
\end{equation*}
\end{example}

\end{appendices}
\bibliographystyle{alpha}
\bibliography{thesisphd}
\vfill
\begin{center}
    \textsc{Universit\'e de Lorraine, CNRS, IECL, F-57000 Metz, France.}
\end{center}
\newpage\hbox{}\thispagestyle{empty}\newpage

\end{document}